\documentclass[bj,authoryear,noshowframe]{imsart}

\RequirePackage{amsthm,amsmath,amsfonts,amssymb}
\RequirePackage[authoryear]{natbib}
\RequirePackage[colorlinks,citecolor=blue,urlcolor=blue]{hyperref}
\RequirePackage[resetlabels]{multibib} 
\RequirePackage{dsfont}
\RequirePackage{tikz}
\RequirePackage{pgfplots}
\RequirePackage{float}
\pgfplotsset{compat=1.18}
\usetikzlibrary{decorations.markings, shapes.geometric, arrows.meta, positioning, fit}

\tikzset{
  mynode/.style={ellipse, draw, minimum width=2.8cm, minimum height=1cm, align=center},
  myarrow/.style={->, thick},
  sectionbox/.style={draw, inner sep=0.3cm},
}

\startlocaldefs
\numberwithin{equation}{section}
\theoremstyle{plain}

\newtheorem{theorem}{Theorem}[section]
\newtheorem*{theorem*}{Theorem}
\newtheorem{lemma}[theorem]{Lemma}
\newtheorem{proposition}[theorem]{Proposition}
\newtheorem{corollary}[theorem]{Corollary}
\theoremstyle{definition}

\newtheorem*{remark}{Remark}

\newcommand{\bE}{\mathbb E}

\newcommand{\bN}{\mathbb N}

\newcommand{\bP}{\mathbb P}

\newcommand{\bR}{\mathbb R}
\newcommand{\bS}{\mathbb S}

\newcommand{\bU}{\mathbb U}

\newcommand{\cA}{\mathcal A}

\newcommand{\cC}{\mathcal C}

\newcommand{\cF}{\mathcal F}
\newcommand{\cG}{\mathcal G}
\newcommand{\cH}{\mathcal H}

\newcommand{\cJ}{\mathcal J}
\newcommand{\cK}{\mathcal K}
\newcommand{\cL}{\mathcal L}

\newcommand{\cN}{\mathcal N}
\newcommand{\cO}{\mathcal O}

\newcommand{\cU}{\mathcal U}

\newcommand{\cX}{\mathcal X}
\newcommand{\cY}{\mathcal Y}

\newcommand{\fA}{\mathfrak A}

\newcommand{\fW}{\mathfrak W}

\newcommand{\fZ}{\mathfrak Z}

\newcommand*{\defeq}{\mathrel{\vcenter{\baselineskip0.5ex \lineskiplimit0pt 
			\hbox{\scriptsize.}\hbox{\scriptsize.}}}=}
\newcommand*{\eqdef}{=\mathrel{\vcenter{\baselineskip0.5ex \lineskiplimit0pt 
			\hbox{\scriptsize.}\hbox{\scriptsize.}}}}

\newcommand{\argmin}{\operatorname*{arg min}}
\newcommand{\argminp}{\operatorname*{arg min^{+}}}
\newcommand{\argmax}{\operatorname*{arg max}}

\newcommand{\Argmax}{\operatorname*{Arg max}}

\newcommand{\Var}{\operatorname{Var}}
\newcommand{\Cov}{\operatorname{Cov}}

\newcommand{\boldcdot}{\vcenter{\hbox{\tiny$\bullet$}}}

\endlocaldefs

\begin{document}

\begin{frontmatter}
\title{The weak-feature-impact phase transition of the NPMLE in monotone binary regression}
\runtitle{The weak-feature-impact phase transition}

\begin{aug}
\author[A]{\fnms{Dario}~\snm{Kieffer}\ead[label=e1]{dario.kieffer@stochastik.uni-freiburg.de}}
\author[A]{\fnms{Angelika}~\snm{Rohde}\ead[label=e2]{angelika.rohde@stochastik.uni-freiburg.de}}
\address[A]{Albert-Ludwigs-Universit\"at Freiburg\printead[presep={,\ }]{e1,e2}}
\end{aug}

\begin{abstract}
Statistical literature provides pointwise limiting distributions of the nonparametric maximum likelihood estimator (NPMLE) in monotone binary regression for the two extremal cases: If the feature-label relation is strictly monotone and sufficiently smooth, it converges at a nonparametric rate with scaled Chernoff-type limiting distribution, and it converges at the parametric $\sqrt{n}$-rate if the underlying relation is flat. In this article, we provide the complete picture of the distributional metamorphosis of the NPMLE, revealing a new limiting distribution which provides a significantly better distributional approximation for small samples in case of a weak feature-label relationship. It is shown to continuously interpolate between the two extremal cases. The innovative way to determine this distribution is to generate it as a limit of the NPMLE in the newly introduced weak-feature-impact triangular array for a particular parameter-sample-size constellation. Moreover, a phase transition is likewise observed for the suitably rescaled $L^{1}$-error in this weak-feature-impact scenario. As a by-product, its limiting distribution for flat regression functions is obtained, which was unknown before. The proof develops a completely new strategy, notably not based on the switch relation. A novel type of local minimax lower bounds accompanies these results.
\end{abstract}


\end{frontmatter}

\addtocontents{toc}{\protect\setcounter{tocdepth}{-1}} 
\section{Introduction} \label{sec:setting}
Statistical literature provides the pointwise limiting distributions of the nonparametric maximum likelihood estimator (NPMLE) in monotone binary regression for the two extremal cases: If the feature-label relation is strictly monotone and sufficiently smooth, it converges pointwise at a nonparametric rate with scaled Chernoff-type limiting distribution, while it converges at the parametric $\sqrt{n}$-rate if the underlying relation is flat. Figure~\ref{fig:simulation} indicates that the flatter the slope of the strictly monotone regression function, the later the established Chernoff approximation kicks in (lightblue versus blue line compared to the Chernoff distribution function in black).
{
\setlength{\abovecaptionskip}{-8pt}
\setlength{\belowcaptionskip}{0pt}
\setlength{\intextsep}{15pt}
\begin{figure}[H]
\centering
\includegraphics[width=0.995\linewidth]{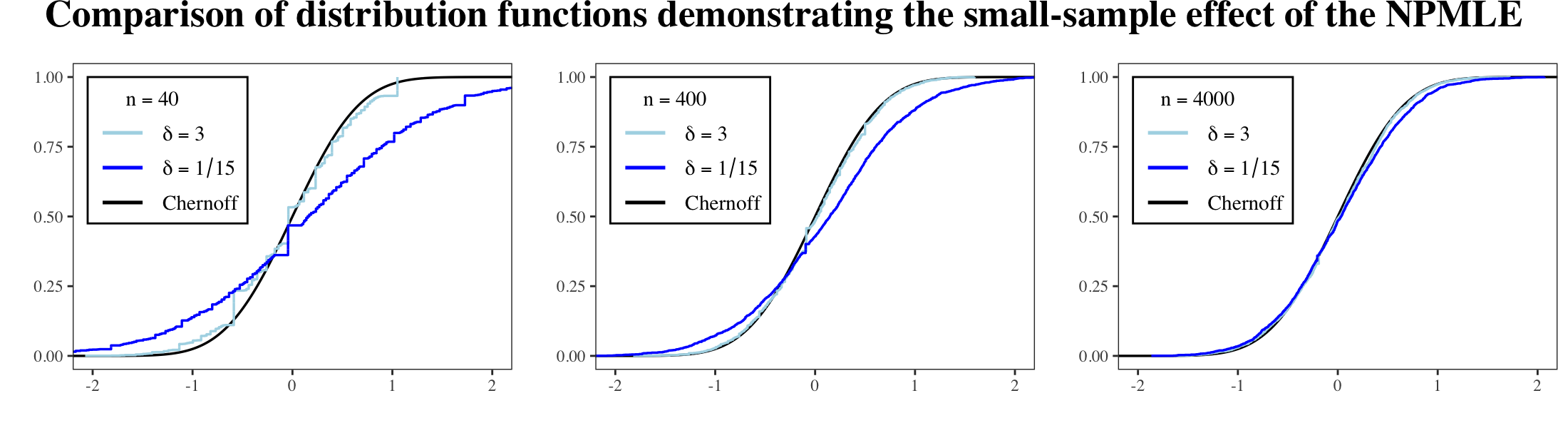}
\caption{{\small Let $X\sim \cU([-1,1])$, $\Phi(x) = \textrm{logit}^{-1}(\delta x)$ the regression function, $\hat{\Phi}_{n}$ the NPMLE for $\Phi$ and $x_{0} = 1/4$. Note that the regression function's derivative scales as $\delta$. The black line shows the Chernoff distribution function. The distribution function of the random variable $(2n)^{1/3}(4\Phi(x_{0})(1-\Phi(x_{0}))\Phi'(x_{0}))^{-1/3}(\hat{\Phi}_{n}(x_{0})-\Phi(x_{0}))$ is simulated for $\delta = 3$ (lightblue) and $\delta = 1/15$ (blue) based on $2000$ iterations of i.i.d.~samples of size $n$.}}\label{fig:simulation}
\end{figure}
}
\par
The goal of this article is to explain this small-sample phenomenon on a rigorous mathematical basis and to develop more appropriate distributional approximations for small sample sizes, both pointwise and in $L^{1}$, as well as to specify in particular to which extent these depend on unknown oracle information. Based on our new insights on the distribution of the NPMLE, we shall discuss their statistical implications and open problems deduced from them. Note that the situation of a weakly increasing regression function corresponds to a weak feature-label relationship, which occurs frequently in applications. For example, privacy preserving requirements may diminish the isolated effect of an explanatory variable $X$ on the response variable $Y$ considerably.

\subsection{State of the art}\label{sec:state of the art}
As the problem of estimating a monotone function arises naturally in many real world tasks and also builds the foundation 
for multiple statistical models, it has been studied extensively over the last decades, with \cite{Grenander1956mortality} 
being the first to consider the NPMLE for monotone densities, lending it the name \textit{Grenander estimator}. It was 
shown first in \cite{Rao1969unimodal} that this estimator is $n^{1/3}$-consistent with respect to the pointwise distance 
and asymptotically Chernoff-distributed if the density's first derivative does not vanish. This was then proven again in 
\cite{Groeneboom1985monotone} by a different technique utilizing inverse expressions based on the \textit{switch relation}, 
which became the most important tool for deriving limits of the NPMLE under various 
shape constraints. In that article, the $L^{1}$-limiting behavior was considered for the first time and a rigorous 
proof of the $L^{1}$-limit appeared in \cite{Groeneboom1999L1}, showing that the expectation of the $L^{1}$-distance 
converges with rate $n^{1/3}$ to zero and that the stabilized $L^{1}$-distance itself fluctuates with rate $n^{1/6}$ and is asymptotically normal. A generalization to the $L^{p}$-distance was given in \cite{Kulikov2005Lk}. 
Similar results regarding the pointwise distance appeared in the context of isotonic regression and least 
squares estimation (LSE) in \cite{Brunk1970isotonic} and for current status data in \cite{GroeneboomWellner1992}, utilizing 
that NPMLE and LSE coincide here. A unified study of various estimators, including the monotone NPMLE, was 
introduced in \cite{Kim1990cube}. The $L^{1}$-limit for isotonic regression with fixed design was derived in 
\cite{Durot2002sharp} and was later generalized to the $L^{p}$-distance in \cite{Durot2007Lp} and to the random design 
setting in \cite{Durot2008monotone}.\par 
Many more properties of the NPMLE under monotonicity constraints were derived, e.g.~the pointwise limiting behavior for functions with vanishing derivative up to some order 
$\beta$ in \cite{Wright1981asymptotic}, resulting in convergence rates $n^{\beta/(2\beta+1)}$, and for 
locally flat densities in \cite{Carolan1999flat}, yielding $\sqrt{n}$-consistency. Non-asymptotic 
properties were discovered in \cite{Birge1989nonasymptotic}, local minimax-optimality was derived in \cite{Cator2011optimality} and \cite{ChaGunSen2015} for the local and global estimation problem, respectively and \cite{Bellec2018oracle} derived sharp oracle inequalities in Euclidean norm for the LSE of isotonic vectors in $\bR^n$. The limiting behavior under the uniform distance was derived in 
\cite{Durot2012uniform} and the misspecified case was studied in \cite{Patilea2001misspecified} and \cite{Jankowski2014misspecified}. More information can be found in the overview articles \cite{Groeneboom2018overview}, 
\cite{Durot2018overview} and \cite{Groeneboom2014shape}. \par 
In \cite{Westling2020unified}, a unified approach to study generalized Grenander 
estimators was introduced. \cite{Mallick2023asymptotic} generated new pointwise limiting distributions in 
the nonparametric regime for $n$-dependent monotone functions with possibly locally changing shape, not reaching the parametric regime, however. Based on this, asymptotic confidence intervals that are uniformly valid over a large class of distributions are constructed. Using a new localization technique in 
isotonic regression and an anti-concentration inequality for the supremum of a Brownian motion with a Lipschitz drift, 
\cite{Han2022bounds} derived Berry-Esseen bounds for Chernoff-type limiting distributions. 
\cite{Cattaneo2024bootstrap} proposed a bootstrap adapting to the unknown order of the first non-zero derivative.

\subsection{The weak-feature-impact scenario}\label{subsec: 1.2} 
As it becomes apparent from Section~\ref{sec:state of the art}, the literature on the NPMLE does not adequately explain the small-sample effect observed in Figure~\ref{fig:simulation}. In this section, we introduce the so-called weak-feature-impact scenario along which we will rigorously discover and characterize the full distributional spectrum of the NPMLE. In order to motivate our setting, it is advisable to recall classical logistic regression, where the feature-label relation is given by $\bP(Y=1|X=x) = \textrm{logit}^{-1}(c+\delta^{\top}x)$. A natural generalization is to replace the logistic function by an arbitrary isotonic $[0,1]$-valued function $\Phi_{0}$, i.e.~$\bP(Y=1|X=x)=\Phi_{0}(\delta^{\top}x)$, where likewise the magnitude of the components of $\delta$ control the predictive power of the corresponding feature components in a global sense. Clearly, the extremal case of no impact is present if $\delta = 0$. The model reduces to monotone binary regression in the univariate case.
As the distribution of the NPMLE in monotone binary regression is essentially accessible subject to asymptotics, the magnitude of the global feature impact has to scale with the sample size $n$ to preserve the small-sample effect observed in Figure~\ref{fig:simulation} in an asymptotic sense (see Figure~\ref{fig:freeze}). \par
{
\setlength{\belowcaptionskip}{-8pt}
\setlength{\abovecaptionskip}{-8pt}
\setlength{\intextsep}{15pt}
\begin{figure}[H]
\centering
\includegraphics[width=\linewidth]{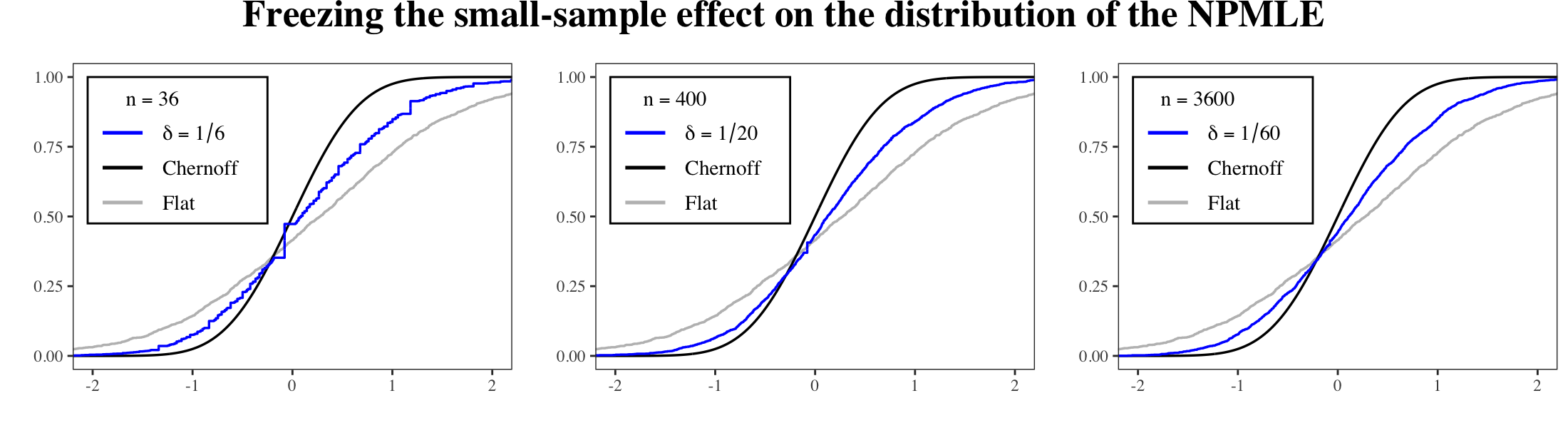}
\caption{In the same setting as in Figure~\ref{fig:simulation} with the additional grey line visualizing the distribution function of the limiting distribution for flat functions and again with regression function $\Phi(x) = \textrm{logit}^{-1}(\delta x)$, the small-sample effect on the distribution of the NPMLE freezes when $n\delta^{2} = \text{const.}$}\label{fig:freeze}
\end{figure}
} 
Let $(\Omega,\cA,\bP)$ denote a probability space and consider the triangular array $(X_{1},Y_{1}^{n}),\dots,(X_{n},Y_{n}^{n})$ of respective i.i.d.~copies of a random vector 
$(X,Y^{n}) \colon \Omega \to \bR \times \{0,1\}$, related via
\begin{equation}\label{eq: WFI1}
\bP\big(Y^n=1 | X\big)=\Phi_0(\delta_n X) \eqdef \Phi_n(X)
\end{equation}
for some isotonic function $\Phi_0$ and a stretching sequence $(\delta_n)_{n\in\bN}$ with $\delta_n\searrow 0$. We call the 
sequence $(\delta_n)_{n\in\bN}$ the \textit{level of feature impact}, as it represents --- in the style of the logistic regression model --- the rate at which the signal strength decreases and thus, controls the predictive power of the feature. Subsequently, this asymptotic scheme is referred to as \textit{weak-feature-impact scenario}. As Figure~\ref{fig:freeze} illustrates, it allows to strikingly describe the small-sample effect in an asymptotic setting. Indeed, an appropriate distributional approximation for the small-sample phenomenon observed in Figure~\ref{fig:simulation} appears to be an intermediate state in the metamorphosis from the strictly isotonic to the flat case. This requires us to fully understand this metamorphosis in mathematical terms. If $\Phi_0$ is continuously differentiable, the derivative of \eqref{eq: WFI1} with respect to the feature variable satisfies
\[
\Phi_n'(x)=\delta_{n}\Phi_{0}'(\delta_{n}x) =\delta_n\big(\Phi_0'(0)+o(1)\big)\longrightarrow 0 \quad \text{as } n \longrightarrow \infty. 
\]
Thus, if $\Phi_0'>0$, the level of feature impact characterizes the speed at which the derivative of the function 
$x\mapsto \bP(Y^n=1| X=x)$ approaches zero, uniformly on compacts. Note that a weak feature-label relation is a global property and hence, cannot be modeled locally solely.

\subsection{Overview of the results} \label{subsec:results}
In this article, we study the pointwise error $\hat\Phi_n(x)-\Phi_n(x)$ of the NPMLE (Section~\ref{sec:pointwise limit}) as well as its appropriately normalized $L^1$-error (Section \ref{sec:L1 limit}) in the weak-feature impact triangular array, which will be shown to bring different limiting distributions of the NPMLE into being. Assuming for the ease of representation that $\Phi_{0}' > 0 $, our first finding is that pointwise and in $L^{1}$, the full picture  is governed by the quantity $n\delta_{n}^{2}$: We identify three regimes, the slow regime $n\delta_{n}^{2} \longrightarrow \infty$, the fast regime $n\delta_{n}^{2} \longrightarrow 0$ and the intermediate regime $n\delta_{n}^{2} \longrightarrow c \in (0,\infty)$ together with the rate of consistency 
$$
\Big(\frac{n}{\delta_n}\Big)^{1/3}\wedge \sqrt{n} 
$$
both pointwise and in $L^{1}$, which is shown to coincide with a new type of minimax lower bounds (Theorem \ref{thm:lower bound pointwise} and Theorem \ref{thm:lower bound L1}). In the slow regime, the limiting distribution coincides with the limit for fixed strictly isotonic functions ($\delta_{n} = \text{const.}$), while it coincides with the limit for flat functions ($\delta_{n} = 0)$ in the fast regime.

\smallskip
\noindent
Our main contributions for the pointwise asymptotics (Section~\ref{sec:pointwise limit}) are as follows:
\begin{itemize}
\item We find a new distribution at the phase transition in the intermediate regime, which adequately fits the NPMLE for small sample sizes in case of a weak feature-label relationship (orange line in Figure~\ref{fig:new limit}). This distribution is different from both, the Chernoff distribution and the distribution known for flat regression functions. The innovative way to determine this distribution is to generate it as a limit of the NPMLE in the newly introduced weak-feature-impact triangular array for a particular parameter-sample-size constellation. 
\item Although valuable local adaptivity properties of the NPMLE have been documented in the literature, the question about the corresponding picture in terms of approximating distributions has been totally open. Theorem~\ref{thm:pointwise rate} answers this question by providing the complete picture of the distributional metamorphosis, while Theorem~\ref{thm:intermediate} confirms the continuity of the transformation between the two extremal cases (see Figure~\ref{fig:new limit}). 
\item These results are accompanied by a new type of local minimax lower bounds.
\end{itemize}
{
\setlength{\belowcaptionskip}{0pt}
\setlength{\abovecaptionskip}{-8pt}
\setlength{\intextsep}{8pt}
\begin{figure}[H]
\centering
\includegraphics[width=\linewidth]{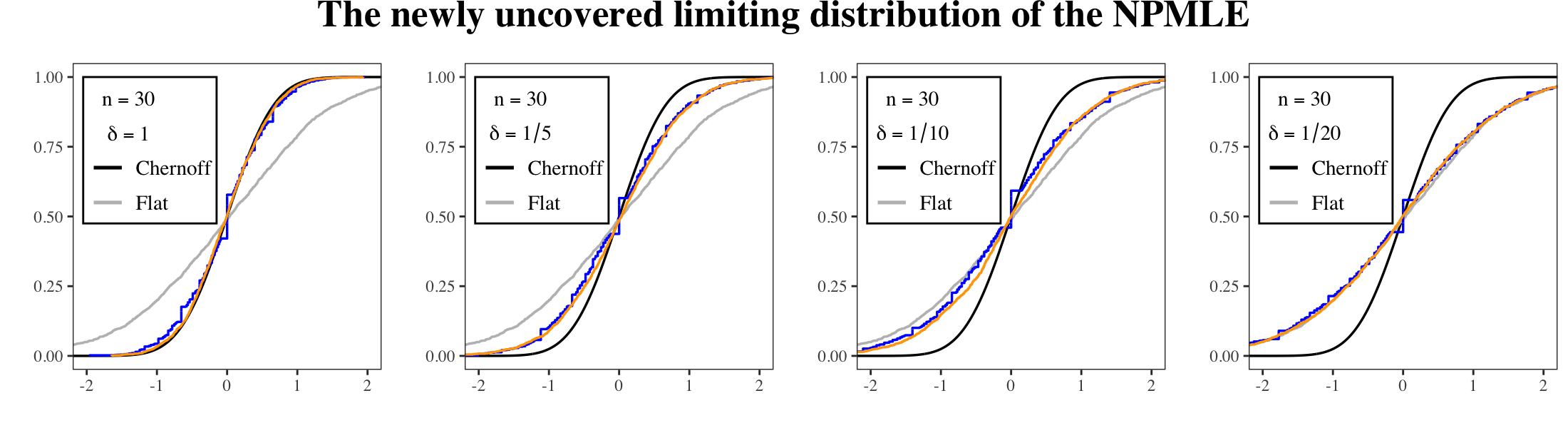}
\caption{Let $X\sim \cU([-1,1])$, $\Phi(x) = \textrm{logit}^{-1}(\delta x)$ and $x_{0} = 0$. The black line shows the Chernoff distribution function, while the grey and the orange line respectively show a simulated distribution function of the limit for flat functions and the newly uncovered limit. The blue line shows a simulated distribution function of the appropriately scaled NPMLE $\{(n/\Phi'(x_{0}))^{1/3} \wedge \sqrt{n}\}(\hat{\Phi}_{n}(x_{0})-\Phi(x_{0}))$. Exclusively if $n\delta^{2}$ is large, the approximation with the scaled Chernoff distribution is appropriate. When $n\delta^{2}$ is small, the limiting distribution for flat functions is a better approximation, even in the non-flat case. In between, however, both fail to be a good approximation. All simulations are calculated based on $2000$ iterations of i.i.d.~samples of size $n = 30$.}\label{fig:new limit}
\end{figure}
}\pagebreak
\noindent
Our main contributions for the $L^{1}$-asymptotics (Section~\ref{sec:L1 limit}) are as follows:
\begin{itemize}
\item In the fast regime (Theorem~\ref{thm:l1 rate} (ii)), the established limiting distribution is new and has not even been discovered in classical asymptotics for flat functions. 
Our proof in here develops a completely new strategy, notably not based on the switch relation.
\item The proof for the slow regime (Theorem~\ref{thm:l1 rate} (i)) uncovers a previously unknown interplay of the convergence rate $(n/\delta_{n})^{1/3}$ of the NPMLE and the newly derived convergence rate $(n\delta_{n}^{2})^{1/3}$ of the inverse process, which cannot be read off from classical asymptotics where these two rates coincide. 
\item Again, the results are accompanied by a new type of minimax lower bounds in $L^{1}$ (Theorem~\ref{thm:lower bound L1}) with the proof based on a non-standard choice of hypotheses in Assouad's lemma.
\end{itemize}
Moreover, our results across Section~\ref{sec:pointwise limit} and Section~\ref{sec:L1 limit} reveal that the limits in the slow regime coincide with the respective limits known from classical asymptotics of strictly increasing regression functions, while the scalings stand out from the classical cases and are different for local and global asymptotics. Whereas the rate of convergence is getting faster in the pointwise case, the rate of convergence $(n\delta_{n}^{2})^{1/6}$ of the stabilized $L^{1}$-distance towards the limiting distribution slows down and collapses at the phase transition.

\smallskip
The remaining part of the article is organized as follows. In Section~\ref{sec:NPMLE}, we introduce notation and present some basics of the NPMLE. In particular, a uniform version of Hellinger consistency is stated and 
uniform convergence on compacts in the weak-feature-impact scenario is deduced.
In Sections~\ref{sec:pointwise limit} and \ref{sec:L1 limit}, we state convergence 
rates and limiting distributions for the pointwise and the $L^{1}$-distance, respectively, as outlined in 
Section~\ref{subsec:results}, together with matching minimax lower bounds that are uniform over a family of appropriate subclasses of monotone functions. In Section~\ref{sec:statistical implications}, we discuss further statistical implications of our results, as well as some open problems.
Remaining proofs and auxiliary results are deferred to the supplement.

\section{Notation and preliminaries on the NPMLE} \label{sec:NPMLE} 
Let $P_{\Phi}$ denote the joint distribution of $(X,Y)$ with $\bP(Y=1| X)=\Phi(X)$ and feature-marginal $P_X$, and let $P_{\Phi}^{\otimes n}$ denote the $n$-fold product measure with expectation operator $\bE^{\otimes n}_{\Phi}$. For the remainder of the article, we write $F_X$ for the distribution function of $P_X$ and 
$\cX \subset \bR$ for its support. It is assumed that $P_X$ is Lebesgue-continuous and we write $p_{X}$ for the continuous version of the Lebesgue density on $\cX$ if it 
exists. For $F_n$ denoting the empirical distribution function of $X_{1},\dots,X_{n}$, we define 
$F_{n}^{-1} \colon [0,1] \to \bR$, $F_{n}^{-1}(a) \defeq \inf\{x \in \bR \mid F_{n}(x) \geq a\}$
as usual. Moreover, we write
\[
\cF \defeq \{ \Phi \colon \bR \to [0,1] \mid \Phi \text{ monotonically increasing}\}
\]
for the set of monotonically increasing functions from $\bR$ into the unit interval. For $\Phi\in\cF$, 
\[
p_{\Phi} \colon \bR \times \{0,1\} \to [0,1], \quad 
p_{\Phi}(x,y) \defeq \Phi(x)^{y}(1-\Phi(x))^{1-y}
\]
is the conditional probability mass function of $Y$ given $X$ if $(X,Y)\sim P_{\Phi}$. In the product experiment, the NPMLE for feature-label realizations 
$(x_{1},y_{1}),\dots,(x_{n},y_{n})$ is defined as 
\[
\hat{\Phi}_{n} 
\in \Argmax_{\Phi \in \cF} \prod_{i=1}^{n}p_{\Phi}(x_{i},y_{i}) 
= \Argmax_{\Phi \in \cF} \frac{1}{n}\sum_{i=1}^{n}\log p_{\Phi}(x_{i},y_{i}).
\]
Note that in the weak-feature-impact scenario, as introduced in Section \ref{subsec: 1.2}, the $n$ observations are realized according to $P_{\Phi_n}^{\otimes n}$ and the 
resulting NPMLE is an estimator for $\Phi_{n}$. Its existence and uniqueness at the sample points (in case the $x_i$ are pairwise different) can be proven as in Part~II Prop.~1.1 \& Prop.~1.2 of 
\cite{GroeneboomWellner1992}. 
As usual in the literature, we agree on $\hat{\Phi}_{n}$ being right-continuous and piecewise constant with jumping points 
being a subset of the sample points, i.e.~for the order statistic $x_{(1)},\dots,x_{(n)}$ of $x_{1},\dots,x_{n}$, 
\begin{align}\label{eq:Phi}
\hat{\Phi}_{n}|_{(-\infty,x_{(1)})} \defeq 0, \qquad 
\hat{\Phi}_{n}|_{[x_{(i)},x_{(i+1)})} \defeq \hat{\Phi}_{n}(x_{(i)}), \qquad 
\hat{\Phi}_{n}|_{[x_{(n)},\infty)} \defeq \hat{\Phi}_{n}(x_{(n)}) 
\end{align}
for $i=1,\dots,n-1$. Although there is no closed-form expression for $\hat{\Phi}_{n}$, it is possible to characterize the 
NPMLE under monotonicity constraints as follows: Let $y_{(1)},\dots,y_{(n)}$ be the corresponding ordering of the labels 
according to $x_{(1)},\dots,x_{(n)}$ (i.e.~if $x_{(i)} = x_{j}$ for some $1 \leq j \leq n$, then $y_{(i)} = y_{j}$), 
let 
\[
\cY_{n} 
\defeq \bigg\{\bigg(\frac{i}{n},\frac{1}{n}\sum_{j=1}^{i}y_{(j)}\bigg) \bigg| \ \, i \in \{1,\dots,n\}\bigg\} 
					\cup \big\{(0,0)\big\} 
\]
and let $G_{n} \colon [0,1] \to \bR$ denote the greatest convex minorant of $\cY_{n}$. Then, $\hat{\Phi}_{n}(x_{(i)})$ is 
given by the left-hand derivative of $G_{n}$ in the point $i/n$, i.e.
\begin{equation}\label{eq:char_phi}
\hat{\Phi}_{n}(x_{(i)}) 
= \sup_{s < \frac{i}{n}}\inf_{t \geq \frac{i}{n}}\frac{G_{n}(t)-G_{n}(s)}{t-s}. 
\end{equation}
In particular, $\hat{\Phi}_{n}$ coincides with the local average of the labels between two jumping points. \par
Generally, we write $g^{*}$ for the greatest convex minorant of a continuous function $g \colon I \to \bR$ for some 
interval $I \subset \bR$ and denote its left-hand derivative by $g^{*,\ell}$, which is given as in 
\eqref{eq:char_phi}, 
but with $G_{n}$ replaced by $g$. We refer to Chapter~3.3 of \cite{Groeneboom2014shape} for more details on this. 
From Lemma~3.2 of \cite{Groeneboom2014shape}, we obtain the \textit{switch relation}, giving an 
expression for the generalized inverse of $g^{*,\ell}$. Let
$\argminp$ denote the supremum of all minimizers.
\begin{lemma}[Switch relation]\label{lem:general switch relation}
For every $x$ in the interior of $I$ and any $a \in \bR$, we have 
\[
g^{*,\ell}(x) > a 
\quad \iff \quad 
\argminp_{u \in I}\{g(u) - au\} < x. 
\]
\end{lemma}
Similarly, two different characterizations of the generalized inverse of $\hat{\Phi}_{n}$ have been established in the 
literature, with \cite{Groeneboom1985monotone} being the first to introduce such an inverse process. Following 
Section~4.1 in \cite{Durot2008monotone}, let $\Upsilon_{n} \colon [0,1] \to \bR$ denote the polygonal chain with 
$(i/n,\Upsilon_{n}(i/n)) \in \cY_{n}$ for $i=1,\dots,n$ and let $g_{n} \colon [0,1] \to \bR$ denote the left-hand derivative 
of $G_{n}$. Then, by definition, we have $\hat{\Phi}_{n}(X_{(i)}) = g_{n}(i/n) = g_{n} \circ F_{n}(X_{(i)})$ for $i=1,\dots,n$. Define 
\begin{align}
U_{n} \colon [0,1] \to \bR, \quad
U_{n}(a) 
&\defeq \argminp_{x \in \cX}\bigg\{\frac{1}{n}\sum_{i=1}^{n}Y_{i}^{n}\mathds{1}_{\{X_{i} \leq x\}} - aF_{n}(x)\bigg\}, \nonumber \\
\tilde{U}_{n} \colon [0,1] \to \bR, \quad
\tilde{U}_{n}(a) 
&\defeq \argminp_{t \in [0,1]}\{\Upsilon_{n}(t) - at\} \label{eq:Un}
\end{align}
and note that $F_{n}^{-1} \circ \tilde{U}_{n}(a) = U_{n}(a)$, as $\tilde{U}_{n}$ maps into the set 
$\{i/n \mid i=0,\dots,n\}$ and the process inside the $\argminp_{x \in \cX}$ in the definition of $U_{n}$ changes its values only at the observation points. For completeness, the proof of the subsequent formulation of the switch relation is given in Section~\ref{proof:switch relation}.

\begin{lemma}\label{lem:switch relation}
For every $x \in \cX$ and any $a \in [0,1]$, we have 
\[
\hat{\Phi}_{n}(x) > a 
\quad \iff \quad 
U_{n}(a) < F_{n}^{-1}(F_{n}(x))
\quad \iff \quad 
F_{n}^{-1} \circ \tilde{U}_{n}(a) < F_{n}^{-1}(F_{n}(x))
\qquad P_{\Phi}^{\otimes n}-\text{a.s.}
\]
\end{lemma}

One particularly important property of the NPMLE, which paves the way for our later study, is Hellinger-consistency
uniformly in $\Phi$. 
Let $h$ denote the Hellinger metric, i.e.
\begin{equation}\label{eq:hellinger metric}
h^{2}(P_{\Phi},P_{\Psi}) 
= \frac{1}{2}\hspace{-0.55mm}\int_{\bR}\hspace{-1.2mm}\big(\sqrt{1-\Phi(x)} - \sqrt{1-\Psi(x)}\big)^{2} 
							+ \big(\sqrt{\Phi(x)} - \sqrt{\Psi(x)}\big)^{2}dP_{X}(x) \eqdef d^2(\Phi,\Psi)
\end{equation}
for any $\Phi, \Psi \in \cF$, inducing the semi-metric $d$ on $\cF$. 

\begin{proposition}[Uniform Hellinger consistency]\label{prop:hellinger consistency}
For any $\varepsilon>0$, the NPMLE satisfies 
\[
\sup_{\Phi \in \cF} P_{\Phi}^{\otimes n}\big(d(\hat{\Phi}_{n},\Phi) > \varepsilon\big) 
\longrightarrow 0 \quad \text{as } n \longrightarrow \infty.
\] 
\end{proposition}

The result might be well-known, yet we did not find it stated in the uniform version as formulated here.
Hence, we give a proof in Section~\ref{proof:hellinger consistency}. Because Hellinger distance dominates total variation, Proposition \ref{prop:hellinger consistency} reveals for any $\varepsilon>0$ likewise
\begin{equation}\label{eq: TV}
\sup_{\Phi \in \cF} P_{\Phi}^{\otimes n}\big(\| \hat{\Phi}_n-\Phi\|_{L^1(P_X)}> \varepsilon\big)
\longrightarrow 0 \quad \text{as } n \longrightarrow \infty.
\end{equation}
As a consequence, $d(\hat{\Phi}_n,\Phi_n)\longrightarrow_{\bP}0$ and $\| \hat{\Phi}_n-\Phi_n\|_{L^1(P_X)}\longrightarrow_{\bP}0$ in the weak-feature-impact scenario, irrespective of the level of feature impact. 
\begin{corollary}\label{cor:uniform consistency}
Let $\Phi_{0}$ be continuous in a neighborhood of zero. Then, for any compact interval $I$ in the interior of $\cX$, 
\[
\sup_{x \in I}|\hat{\Phi}_{n}(x) - \Phi_{n}(x)| \longrightarrow_{\bP} 0 \quad \text{as } n \longrightarrow \infty
\]
in the weak-feature-impact scenario.
\end{corollary}
The proof is given in Section~\ref{proof:uniform consistency}, where we 
design a tricky two-stage 
subsequence argument to deduce pointwise convergence from \eqref{eq: TV} in the weak-feature-impact scenario at any interior point of $\cX$. The result then follows from the fact that pointwise convergent $[0,1]$-valued isotonic functions with 
continuous limit also converge uniformly on compacts.\par
Throughout from now on, $P_{X}$ is assumed to be compactly supported on 
$\cX = [-T,T]$ for some $T > 0$ with continuous, strictly positive Lebesgue density $p_{X}$ on $\cX$.

\section{The local case} \label{sec:pointwise limit}

\subsection{Pointwise local minimax lower bounds}\label{sec:pointwise lower bound}
Recall that a crucial consequence of the weak-feature-impact scenario is that the level of feature impact controls the gradient of the feature-label relation uniformly on compacts, both from above and from below. This raises the question which rate of convergence is attainable in principle by any estimator over restricted classes reflecting this constraint.
For any function $f\in\cF$, let 
\begin{equation}\label{eq:modulus}
\| f\|_{\cX,L}\defeq\sup_{\substack{x,y\in\cX: \\x<y}}\frac{f(y)-f(x)}{y-x} \ \  \text{and} \ \ \omega_{\boldcdot}^{\cX}(f)\defeq\sup\big\{f(y)-f(x)\mid  x,y\in\cX,0<y-x\leq \boldcdot\big\}
\end{equation}
denote Lipschitz semi-norm and modulus of continuity of its restriction to $\cX$, respectively. For any $\delta \in [0,1]$, let 
\begin{equation}\label{eq: Fdelta}
\cF_{\delta} 
\defeq \Big\{\Phi \in \cF \ \Big| \ \| \Phi\|_{\cX,L}\leq \delta \text{ and } \inf_{\nu}\omega_{\nu}^{\cX}(\Phi)/\nu\geq \delta/2 \Big\}
\end{equation}
and note that $\cF_{\delta}$ contains only functions of $\cF$ with steepness of order $\delta$. Due to the condition on the modulus of continuity, the steepness of these functions is not only controlled from above, but also from below.
We remark that for continuously differentiable $\Phi_0$ with $\Phi_0'(0) > 0$, $\Phi_n=\Phi_0(\delta_n\boldcdot)\in \cF_{\kappa\delta_n}$ for any $\kappa \in (\Phi_{0}'(0),2\Phi_{0}'(0))$ and $n > n_{0}(\kappa)$ sufficiently large.

\begin{theorem}[Pointwise lower bound] \label{thm:lower bound pointwise}
For any $x_0$ contained in the interior of $\cX$, there exists a positive constant 
$C > 0$, such that 
\[
\liminf_{n\rightarrow\infty}\inf_{\delta\in [0,\frac{1}{4T}]}
	\inf_{T_n^{\delta}(x_0)}\sup_{\Phi\in\cF_{\delta}}
				P_{\Phi}^{\otimes n}\Big(\Big(\sqrt{n}\wedge \Big(\frac{n}{\delta}\Big)^{1/3}\Big)\big| T_n^{\delta}(x_0)-\Phi(x_0)\big| \geq C\Big) > 0,
\]
where the infimum is running over all estimators $T_n^{\delta}(x_0)=T_n^{\delta}\big(x_0,(x_1,y_1),\dots,(x_n,y_n)\big)$.
\end{theorem}

While this lower bound remains true when enlarging $\cF_{\delta}$ by dropping the lower bound constraint on the modulus of continuity, this sharper formulation suitably complements the subsequent convergence rate in the weak-feature-impact scenario, as the latter is not stated in a uniform sense.
Note that the range of $\delta$ necessarily has to be bounded from above as $\sharp\cF_{\delta} < 2$ for $\delta \geq 1/T$.
The proof is given in Section~\ref{proof:pointwise lower bound}. The lower bound exhibits an elbow at $\delta=\delta_n\sim n^{-1/2}$ separating two regimes --- the slow regime $(n/\delta)^{-1/3}$ in case $\delta \gg n^{-1/2}$ and the fast regime $n^{-1/2}$ for $\delta \ll n^{-1/2}$ with the intermediate regime at $\delta \asymp n^{-1/2}$. Note that by smoothing out the kinks in the respective lower bound hypotheses, the result continues to hold when restricted to continuously differentiable functions.

\subsection{The metamorphosis of the pointwise limiting distribution}\label{sec:metamorphosis}
In view of the valuable pointwise adaptivity properties of the NPMLE in \cite{Cator2011optimality}, it does not come as a surprise that the above stated faster rate (as compared to the $n^{-1/3}$-rate in the unrestricted case) is actually adaptively attained by the NPMLE in the weak-feature-impact scenario. In this section, we complement the rates with the complete picture in terms of limiting distributions. To state the result, let $\sigma_{\Phi_{0}} \defeq \sqrt{\Phi_{0}(0)(1-\Phi_{0}(0))}$, let $(Z(s))_{s \in \bR}$ denote a 
standard two-sided Brownian motion and let for $\beta \in \bN$ and any $c \geq 0$, 
\begin{align*}
f_{\beta} \colon \bR \to \bR, \quad 
f_{\beta}(s) &\defeq \sigma_{\Phi_{0}}Z(s) + \frac{\Phi_{0}^{(\beta)}(0)}{p_{X}(x_{0})^{\beta}(\beta+1)!}s^{\beta+1} \\
g_{\beta,c} \colon [0,1] \to \bR, \quad 
g_{\beta,c}(s) &\defeq \sigma_{\Phi_{0}}Z(s) + \sqrt{c}\,\Phi_{0}^{(\beta)}(0)\bE\big[(X-x_{0})^{\beta}
												\mathds{1}_{\{X \leq F_{X}^{-1}(s)\}}\big].
\end{align*}

\begin{theorem} \label{thm:pointwise rate}
For $\beta \in \bN$, let $x_{0}$ be an interior point of $ \cX$ and assume $\Phi_{0}$ to be $\beta$-times continuously 
differentiable in a neighborhood of zero with the $\beta$th derivative being the first non-vanishing derivative in zero. 
\begin{itemize}
\item[(i)] (Slow regime) If $n\delta_{n}^{2\beta} \longrightarrow \infty$, then 
\[
\Big(\frac{n}{\delta_{n}}\Big)^{\beta/(2\beta+1)}\big(\hat{\Phi}_{n}(x_{0}) - \Phi_{n}(x_{0})\big) 
\longrightarrow_{\cL} f_{\beta}^{*,\ell}(0) \quad \text{as }n \longrightarrow \infty.
\]

\item[(ii)] (Intermediate regime) 
If $n\delta_{n}^{2\beta} \longrightarrow c\in(0,\infty)$ and $F_{X}^{-1}$ is Hölder-continuous to the exponent $\alpha > 1/2$, then
\[
\sqrt{n}\big(\hat{\Phi}_{n}(x_{0}) - \Phi_{n}(x_{0})\big) 
\longrightarrow_{\cL} g_{\beta,c}^{*,\ell}(F_{X}(x_{0})) \quad \text{as } n \longrightarrow \infty.
\]
\item[(iii)](Fast regime) If $n\delta_{n}^{2\beta} \longrightarrow 0$, then
\[
\sqrt{n}\big(\hat{\Phi}_{n}(x_{0}) - \Phi_{n}(x_{0})\big) 
\longrightarrow_{\cL} g_{\beta,0}^{*,\ell}(F_{X}(x_{0})) \quad \text{as } n \longrightarrow \infty.
\]
\end{itemize}
\end{theorem}

The proof of the theorem is given in Section~\ref{proof:pointwise rate}. 
The result identifies three regimes, each of them with a different limiting distribution. Note that if $\Phi_0$ is continuously differentiable with $\Phi_0'(0)>0$, the convergence rate of the NPMLE equals, in correspondence to the minimax lower bound, 
$$\sqrt{n}\wedge \Big(\frac{n}{\delta_{n}}\Big)^{1/3}.$$ 
Corresponding to the elbow in this rate, the limiting distribution exhibits a phase transition. The elbow is shifted to 
$\delta_{n} = n^{-1/(2\beta)}$ if the $\beta$th derivative of $\Phi_{0}$ for some $\beta >1$ is the 
first non-vanishing derivative at zero. The distribution of $f_{\beta}^{*,\ell}(0)$ in (i) (slow regime) appeared first in \cite{Wright1981asymptotic} and is the well-known Chernoff-type limit (in the terminology of \cite{Han2022bounds}) of the NPMLE in 
classical asymptotics under these general conditions on the derivative of the function to estimate, in consonance with 
Theorem~2.2 in \cite{Mallick2023asymptotic}. 
We remark that by the switch relation (Lemma~\ref{lem:general switch relation}) and Lemma~\ref{lem:argmin brownian motion transformation},
\begin{align*}
\bP\big(f_1^{*,\ell}(0) \leq v\big) 
&=\bP\bigg(\argmin_{s \in \bR}\Big\{\sigma_{\Phi_{0}}\sqrt{p_{X}(x_{0})}Z(s) 
			+ \frac{\Phi_{0}'(0)p_{X}(x_{0})}{2}s^{2} - vp_{X}(x_{0})s\Big\} \geq 0\bigg) \\
&=\bP\bigg(\bigg(\frac{4\sigma_{\Phi_{0}}^{2}\Phi_{0}'(0)}{p_{X}(x_{0})}\bigg)^{1/3}
										\argmin_{s \in \bR}\big\{Z(s) + s^{2} \big\} \leq v\bigg)
\end{align*}
for any $v\in\bR$. 
That is, for $\beta=1$, the limit law $\cL\big(f_1^{*,\ell}(0)\big)$ coincides indeed with a scaled Chernoff distribution. 
\begin{itemize}
\item Without affecting the Chernoff-type limiting shape in the slow regime, the rate of consistency accelerates from the classical rate $n^{\beta/(2\beta+1)}$ to 
$
(n/\delta_{n})^{\beta/(2\beta+1)}
$
according to the level of feature impact. 
\item In the fast regime, rate of convergence and limiting distribution coincide with the one for flat functions in classical asymptotics as derived mutatis mutandis in Theorem~2.4 of \cite{Jankowski2014misspecified} for the Grenander estimator. 
\item With the new intermediate regime ($n\delta_n^{2\beta}\longrightarrow c\in (0,\infty)$) at the phase transition, the picture is completed by the new limiting distribution, which is different from the other two occurring distributions and does not show up in classical asymptotics. It is the so-far missing approximating limiting distribution, visualized in Figure~\ref{fig:new limit}.
\end{itemize}

The new limiting distribution in the intermediate regime depends continuously on $c$ in the topology of weak convergence. As the next result shows for $\beta = 1$, it converges to the limit in (iii) for $c \longrightarrow 0$, whereas a rescaled version converges to the limiting distribution in (i) for $c \longrightarrow \infty$.

\begin{theorem}[Continuity]\label{thm:intermediate}
Let $x_{0}$ be an interior point of $\cX$ and assume $\Phi_{0}$ to be continuously differentiable in a neighborhood of zero with non-vanishing derivative in zero. Then, 
\begin{alignat*}{2}
c^{-1/6}g_{1,c}^{*,\ell}(F_{X}(x_{0})) 
&\longrightarrow_{\cL} f_{1}^{*,\ell}(0) \quad &&\text{as } c \longrightarrow \infty, \\
g_{1,c}^{*,\ell}(F_{X}(x_{0})) 
&\longrightarrow_{\cL} g_{1,0}^{*,\ell}(F_{X}(x_{0})) \quad &&\text{as } c \longrightarrow 0.
\end{alignat*}
\end{theorem}
The proof is given in Section~\ref{proof:intermediate}.

\begin{remark}
It is insightful to relate Theorem~\ref{thm:pointwise rate} to the interesting recent work of \cite{Mallick2023asymptotic} (MSK for short). Both articles, theirs and ours, study asymptotic properties of the isotonic regression estimator under a sequence of models changing with the sample size. However, the frameworks considered in the two articles are non-nested: We describe a global property, whereas MSK focuses on a local property and correspondingly restricts their attention to pointwise limiting results. In the pointwise case, we characterize a new ``intermediate'' limiting distribution, which lies between the two cases well-established in the literature and which cannot be learnt from the work of MSK. On the other hand, while MSK only considers the slow regime, the class of drift terms of the Gaussian process in the limiting distribution is richer than the one considered in here, due to the fact that beyond the scaling, their triangular array allows also for varying qualitative properties of the elements of $(\Phi_{n})_{n \in \bN}$.
\end{remark}

\section{A global case}\label{sec:L1 limit}

As a weak feature-label relation constitutes a global property, it is natural to study its implications on a global distance like the $L^1$-error. In the classical $L^{1}$-asymptotics for strictly isotonic regression functions, the already parametric $\sqrt{n}$-convergence rate in the limit theorem of the appropriately centered $L^{1}$-error raises the question whether the phase transition was a purely pointwise phenomenon. Indeed, a new effect will be shown to emerge here, namely an interplay of the rate of convergence of the so-called inverse process and the $L^{1}$-convergence rate of the NPMLE, which do not coincide in the weak-feature-impact scenario any longer. 
To properly uncover this phenomenon, we first complement our pointwise lower bounds by $L^1$-risk local minimax lower bounds and prove that they are adaptively attained by the NPMLE in the weak-feature-impact scenario (Section~\ref{sec:L1 lower bound}). On this basis, the main result in Section~\ref{subsec:L1 limit} is the second order asymptotic of the $L^1$-error, which turns out to behave fundamentally different to the pointwise case and is considerably harder to derive. \par 
A generalization of the main result to the $L^{p}$-distance for $p > 1$ in the slow regime might be possible, but an essential step in the proof is to approximate the $L^{1}$-distance of $\hat{\Phi}_{n}$ and $\Phi_{n}$ by the integral of their inverse counterparts, where the inverse process of $\hat{\Phi}_{n}$ comes into play. In case of the $L^{p}$-distance, additional terms arise here and bounding the emerging expressions becomes heavily technical. In the fast regime, on the other hand, an essential argument breaks down completely (cf.~\eqref{eq: Step3-fast}), when the $L^{1}$-distance is replaced by the $L^{p}$-distance. An extension to $p > 1$ for the fast regime is thus far from straightforward.

\subsection{$L^{1}$-risk local minimax lower bounds and adaptivity of the NPMLE}\label{sec:L1 lower bound}

\begin{theorem}[${L^1}$-risk lower bound] \label{thm:lower bound L1}
Recalling the definition in \eqref{eq: Fdelta} of the restricted classes from Section~\ref{sec:pointwise lower bound}, we have 
\[
\liminf_{n\rightarrow\infty}\inf_{\delta\in [0,\frac{1}{4T}]}
	\inf_{T_n^{\delta}}\sup_{\Phi\in\cF_{\delta}}\Big(\sqrt{n}\wedge \Big(\frac{n}{\delta}\Big)^{1/3}\Big)
				\bE_{\Phi}^{\otimes n}\bigg[\int_{-T}^T\big| T_n^{\delta}(t)-\Phi(t)\big| d t \bigg]>0,
\]
where the infimum is running over all estimators $T_n^{\delta}=T_n^{\delta}\big(\boldcdot,(x_1,y_1),\dots,(x_n,y_n)\big)$.
\end{theorem}

The proof, which is based on Assouad's hypercube technique (cf.~Theorem~2.12 \cite{Tsybakov2008}), is deferred to Section~\ref{proof:L1 lower bound}. The construction of the hypotheses for the slow regime is visualized in the subsequent 
Figure~\ref{fig:hypotheses}. 
\vspace{-5mm}
{
\setlength{\belowcaptionskip}{-4pt}
\begin{figure}[htp]
  \centering
  \begin{tikzpicture}[scale=0.9]
  \draw[->, thin] (0,-0.1) -- (0,3.5) node[above] {};
  \draw[->, thin] (-0.1,0.0) -- (5.3,0.0) node[right] {};
  
  \draw (1.2,-0.1) -- (1.2,0.1);
  \node[below] at (1.2,-0.1) {\tiny{$x_{k}$}};
  
  \draw (2.8,-0.05) -- (2.8,0.05);
  \node[below] at (2.8,-0.1) {\tiny{$x_{k}+h_{n}$}};
  
  \draw (4.4,-0.1) -- (4.4,0.1);
  \node[below] at (4.4,-0.1) {\tiny{$x_{k+1}$}};
  
  \draw[thick, dotted] (2.8,1.0) -- (2.8,2.0);
  
  \draw (-0.1,0.0) node[left] {\tiny{$0$}};
  \draw (-0.1,1.0) -- (0.1,1.0);
  \draw (-0.1,2.0) -- (0.1,2.0);
  \draw (-0.1,3.0) -- (0.1,3.0);

  \foreach \x/\y/\xnew/\ynew/\arrow in {
    0.0/0.01/1.2/0.01/0,
    1.2/0.01/2.8/1.0/0,
    2.8/1.0/4.4/3.01/0,
    4.4/3.01/5.3/3.01/0
  }{
  \ifnum\arrow=1
    \draw[->, thick, orange] (\x,\y) -- (\xnew,\ynew);
  \else
    \draw[thick, orange] (\x,\y) -- (\xnew,\ynew);
  \fi
}
  
  \foreach \x/\y/\xnew/\ynew/\arrow in {
    0.0/-0.01/1.2/-0.01/0,
    1.2/-0.01/2.8/2.0/0,
    2.8/2.0/4.4/2.99/0,
    4.4/2.99/5.3/2.99/0
  }{
  \ifnum\arrow=1
    \draw[->, thick, blue] (\x,\y) -- (\xnew,\ynew);
  \else
    \draw[thick, blue] (\x,\y) -- (\xnew,\ynew);
  \fi
    }
  
  \filldraw[fill=black, draw=black] (1.2,0.0) circle (1.5pt);
  \filldraw[fill=black, draw=black] (4.4,3.0) circle (1.5pt);
  
  \node at (3.7,1.3) {\tiny{$\varphi_{k,n}$}};
  \node at (2.0,1.7) {\tiny{$\psi_{k,n}$}};
  \end{tikzpicture}
  \hspace{5em}
  \begin{tikzpicture}[scale=0.9]

  \draw[->, thin] (0,-0.1) -- (0,4.3) node[above] {};
  \draw[->, thin] (-0.1,0.0) -- (5.3,0.0) node[right] {};
  
  \draw (1.0,-0.1) -- (1.0,0.1);
  \draw (2.0,-0.1) -- (2.0,0.1);
  \draw (3.0,-0.1) -- (3.0,0.1);
  \draw (4.0,-0.1) -- (4.0,0.1);
  \draw (5.0,-0.1) -- (5.0,0.1);

  \draw (0.5,-0.05) -- (0.5,0.05);
  \draw (1.5,-0.05) -- (1.5,0.05);
  \draw (2.5,-0.05) -- (2.5,0.05);
  \draw (3.5,-0.05) -- (3.5,0.05);
  \draw (4.5,-0.05) -- (4.5,0.05);
  
  \draw (-0.1,0.0) node[left] {\tiny{$1/4$}};
  \draw (-0.1,4.0) node[left] {\tiny{$3/4$}};
  \draw (-0.1,4.0) -- (0.1,4.0);

  \foreach \x/\y/\xnew/\ynew/\arrow in {
    0.0/0.0/0.5/0.25/1,
    0.5/0.25/1.0/0.75/0,
    1.0/0.75/1.5/1.00/1,
    1.5/1.00/2.0/1.50/0,
    2.0/1.50/2.5/1.75/1,
    2.5/1.75/3.0/2.25/0,
    3.0/2.25/3.5/2.50/1,
    3.5/2.50/4.0/3.00/0,
    4.0/3.00/4.5/3.25/1,
    4.5/3.25/5.0/3.75/0
  }{
  \ifnum\arrow=1
    \draw[->, thick, orange] (\x,\y) -- (\xnew,\ynew);
  \else
    \draw[thick, orange] (\x,\y) -- (\xnew,\ynew);
  \fi
    }
    
  \foreach \x/\y/\xnew/\ynew/\dot in {
    0.0/0.0/0.5/0.50/1,
    0.5/0.50/1.0/0.75/0,
    1.0/0.75/1.5/1.25/1,
    1.5/1.25/2.0/1.50/0,
    2.0/1.50/2.5/2.00/1,
    2.5/2.00/3.0/2.25/0,
    3.0/2.25/3.5/2.75/1,
    3.5/2.75/4.0/3.00/0,
    4.0/3.00/4.5/3.50/1,
    4.5/3.50/5.0/3.75/0
  }{
  \ifnum\dot=1
    \draw[->, thick, blue] (\x,\y) -- (\xnew,\ynew);
    \filldraw[fill=black, draw=black] (\x,\y) circle (1pt);
  \else
    \draw[thick, blue] (\x,\y) -- (\xnew,\ynew);
  \fi
    }
  \end{tikzpicture}
  \caption{Left: Visualization of $\varphi_{n,k}$ and $\psi_{n,k}$, which are the base functions to construct 
  the hypotheses. They are defined to have either slope equal to $\delta$ on $(x_{k},x_{k}+h_{n})$ and slope equal to 
  $\delta/2$ on $(x_{k}+h_{n},x_{k+1})$ or the other way around for a partition of $[-T,T]$ with step width $2h_{n}$.
  Note that the pointwise 
  distance between these two functions at $x_{k} + h_{n}$ is of order $(n/\delta)^{-1/3}$, with 
  $h_{n} \sim (n\delta^{2})^{-1/3}$. 
  Right: For $m \sim (n\delta^{2})^{1/3}$, the hypotheses are obtained by choosing at each of the $m$ black bullets 
  either the blue path (i.e.~$\psi_{n,k}$) or the orange path 
  (i.e.~$\varphi_{n,k}$), resulting in $2^{m}$ graphs corresponding to different hypotheses 
  functions.}\label{fig:hypotheses}
\end{figure}
}

In preparation for the limiting distribution theory, the next proposition shows that this faster rate of convergence for the $L^1$-risk is actually adaptively attained by the NPMLE in the weak-feature-impact scenario. In particular, the transition from the nonparametric to the parametric regime shows up again at the level of feature impact $\delta=\delta_n\sim n^{-1/2}$. 

\begin{proposition}\label{prop: 4.2}
Suppose that $\Phi_0$ is continuously differentiable with $\Phi_0'(0)>0$. Then, 
\[
\Big(\sqrt{n}\wedge \Big(\frac{n}{\delta_n}\Big)^{1/3}\Big)
				\bE\bigg[\int_{-T}^T\big| \hat\Phi_n(t)-\Phi_n(
			t)\big| d t \bigg]=\cO(1)
\] 
in the weak-feature-impact scenario.
\end{proposition}

Although local adaptivity properties of the NPMLE for the global estimation problem were derived in \cite{ChaGunSen2015} and \cite{Bellec2018oracle}, Proposition~\ref{prop: 4.2} is not covered by those results. Note in particular that the sharp oracle inequality of \cite{Bellec2018oracle} in Euclidean norm for monotone vectors in $\bR^n$ has an additional logarithmic factor in the parametric regime that we do not observe here. 

\smallskip
\textit{Preview of the proof. } Arguing in the slow regime ($n\delta_n^2\longrightarrow\infty$) along the spirit of \cite{Durot2007Lp} and \cite{Durot2008monotone} (to actually get the bound in expectation rather than in probability), the proof is quite elucidating. On basis of Fubini's theorem and partial integration, the idea is to rewrite 
\begin{align*}
\bE\int_{-T}^T&\big| \hat\Phi_n(t)-\Phi_n(t)\big| dt
=\int_{-T}^T\int_{0}^{1}\bP\big(\hat{\Phi}_{n}(t) - \Phi_{n}(t) > x\big)dx\,dt+ \int_{-T}^T\int_{0}^{1}\bP\big(\Phi_n(t)-\hat{\Phi}_{n}(t) > x\big)dx\,dt
\end{align*}
to employ the switch relation (Lemma \ref{lem:switch relation}) in the probabilities inside the integrals, giving 
\[
\bP\big(\hat{\Phi}_{n}(t) - \Phi_{n}(t) > x\big) 
= \bP\big(F_{n}^{-1} \circ \tilde{U}_{n}(\Phi_{n}(t) + x) < F_{n}^{-1}(F_{n}(t))\big)
\]
(exemplarily for the left-hand side), and to derive by means of the slicing device and the \cite{DKW1956} inequality 
a tail bound (Lemma \ref{lem:tail bounds - 1}) for the process $F_n^{-1}\circ \tilde U_n-\Phi_n^{-1}$. 
This is the moment where the level of feature impact~$\delta_n$, i.e.~the exact dependence on the derivative $\Phi_n'$, starts to matter. Its occurence has to be traced back for being incorporated explicitly 
but notably in the tail inequality. Whereas NPMLE and inverse process both scale at the rate $n^{1/3}$ in the classical asymptotics, their convergence rates do not coincide in the weak-feature-impact scenario any longer: As the tail inequality in Lemma~\ref{lem:tail bounds - 1} reveals, the inverse process scales pointwise at the rate $(n\delta_{n}^{2})^{1/3}$. It is insightful to contrast its rate with 
the convergence rate $(n/\delta_{n})^{1/3}$ of the NPMLE. 
In the parametric regime ($n\delta_n^2=\cO(1)$), arguing by means of the inverse process is 
subtle as it is not everywhere convergent any longer. However, the interval of non-convergence turns out to have length of order $\delta_n$ only.
Combining this with sufficiently fast convergence outside this interval bounds the expected $L^1$-error in the fast regime. 
The complete proof is given in Section~\ref{proof:4.2}.

\subsection{Distributional transition of the $L^1$-error}\label{subsec:L1 limit}
Our final aim is to study the second order asymptotics of the stabilized $L^{1}$-error (cf.~Proposition~\ref{prop: 4.2})
\[
\Big(\sqrt{n}\wedge \Big(\frac{n}{\delta_n}\Big)^{1/3}\Big)
\int_{-T}^{T}|\hat{\Phi}_{n}(t) - \Phi_{n}(t)|dt,
\]
i.e.~to investigate the stochastic fluctuation around an appropriate centering $\mu_n=\cO(1)$. 
For this, let us define for ease of notation $X(a) \defeq \argmin_{s \in \bR}\{Z(s) + (s-a)^{2}\}$ for $a \in \bR$, as well as 
\[
\mu_{n} \defeq \bE[|X(0)|]\int_{-T}^{T}\bigg(\frac{4\Phi_n(t)(1-\Phi_n(t))\Phi_{0}'(\delta_{n}t)}{p_{X}(t)}\bigg)^{1/3}dt. 
\]
Note that $\cL(X(0))$ is the Chernoff distribution and that, indeed, $\mu_n=\cO(1)$. 
Next, set 
\begin{align}\label{eq:L1var}
\cC \defeq \int_{0}^{\infty}\Cov(|X(0)|,|X(a)-a|)da \quad \text{and} \quad \sigma^2\defeq8\,\cC\int_{-T}^T\frac{\Phi_0(0)(1-\Phi_0(0))}{p_X(t)}dt.
\end{align}

The next theorem shows that as in the pointwise asymptotics studied in Section~\ref{sec:pointwise limit}, the second order asymptotic of the NPMLE's stabilized $L^{1}$-error exhibits a phase transition. In contrast to the pointwise asymptotics, however, the rate of convergence diminishes in the slow regime and collapses at the phase transition. Indeed, it uncovers an interplay of the newly derived convergence rates $(n/\delta_{n})^{1/3}$ of the NPMLE and $(n\delta_{n}^{2})^{1/3}$ of the inverse process, which cannot be read off in the classical asymptotics. To the best of our knowledge, the limiting distribution in the fast regime is likewise new and has not even been discovered in classical asymptotics for flat functions. Our proof in here develops a new strategy, not based on the switch relation.

\begin{theorem}\label{thm:l1 rate}
Let $\Phi_{0}$ be differentiable in a neighborhood of zero with 
$\Phi_{0}'(0) > 0$. 
\begin{itemize}
\item[(i)] (Slow regime) Let $p_{X}$ be continuously differentiable on $[-T,T]$ (one-sided at $-T,T$) and assume that $\Phi_0'$ is Hölder-continuous in a neighborhood of zero. If $n\delta_n^2\longrightarrow\infty$, then
\[
(n\delta_{n}^{2})^{1/6}\bigg(\Big(\frac{n}{\delta_{n}}\Big)^{1/3}
		\int_{-T}^{T}|\hat{\Phi}_{n}(t) - \Phi_{n}(t)|dt - \mu_{n}\bigg)
\longrightarrow_{\cL} N\sim\cN(0,\sigma^{2}) \quad \text{as } n \longrightarrow \infty.
\]

\item[(ii)] (Fast regime) 
Let $\Phi_0'$ be continuous in a neighborhood of zero. If $n\delta_{n}^{2} \longrightarrow 0$, then 
\[
\sqrt{n}\int_{-T}^{T}|\hat{\Phi}_{n}(x) - \Phi_{n}(x)|dP_{X}(x)
\longrightarrow_{\cL} \max_{s \in [-T,T]}A(s) \quad \text{as } n \longrightarrow \infty,
\]
where $(A(s))_{s \in [-T,T]}$ is a 
continuous, centered Gaussian process with $A(-T) = -A(T)$ and 
\[
\Cov(A(s),A(t)) = \Phi_{0}(0)(1-\Phi_{0}(0))(1-2|F_{X}(s) - F_{X}(t)|) \quad \text{for } s,t \in [-T,T].
\] 
\end{itemize}
\end{theorem}

Note that statement (ii) can be turned into the convergence of 
$\frac{\sqrt{n}}{2T}\int_{-T}^{T}|\hat{\Phi}_{n}(t) - \Phi_{n}(t)|dt$
in case the feature is uniformly distributed. The proof of the theorem extends over Section~\ref{proof:4.4ii} and Section~\ref{proof:4.4i}, employing auxiliary results as visualized in Figure~\ref{fig:proofs}.

{
\setlength{\belowcaptionskip}{-16pt}
\begin{figure}[h]
\centering
\begin{tikzpicture}[node distance=0.3cm and 1.3cm]

\node[mynode, fill=blue!10] (proof43ii) {Proof of Theorem~\ref{thm:l1 rate} (ii) \\ (see Section~\ref{proof:4.4ii})};
\node[mynode, fill=blue!10, below=of proof43ii] (prop42) {Proof of Proposition~\ref{prop: 4.2} \\ (see Section~\ref{proof:4.2})};
\node[mynode, right=of prop42] (lemma101) {Lemma~\ref{lem:rate expectation}};
\node[mynode, right=of lemma101] (lemma61) {Lemma~\ref{lem:tail bounds - 1}};
\node[mynode, fill=blue!10, above=of lemma61] (thm43i) {Proof of Theorem~\ref{thm:l1 rate} (i) \\ (see Section~\ref{proof:4.4i})};
\node[mynode, right=of proof43ii] (cor24) {Corollary~\ref{cor:uniform consistency}};

\draw[myarrow] (cor24) -- (proof43ii);
\draw[myarrow] (lemma101) -- (prop42);
\draw[myarrow] (lemma61) -- (lemma101);
\draw[myarrow] (lemma61) -- (thm43i);
\draw[myarrow] (lemma101) -- (thm43i);
\end{tikzpicture}
\caption{Structural interrelation of the proofs of Section~\ref{sec:L1 limit} and their auxiliary results.}\label{fig:proofs}
\end{figure}
}

\begin{itemize}
\item Without affecting the limiting normal distribution and in contrast to the accelerating rate according to Proposition~\ref{prop: 4.2}, the rate of the stabilized $L^{1}$-error slows down to $(n\delta_{n}^{2})^{1/6}$ and collapses at the phase transition $\delta_n\sim n^{-1/2}$.
\item As already mentioned in Section~\ref{sec:L1 lower bound}, $(n\delta_n^2)^{1/3}$ is the convergence rate of the inverse process, which will be shown to actually drive the convergence in (i), and this inverse process is not convergent any longer if $n\delta_n^2=\cO(1)$. 
\item To derive the new limit in the fast regime (ii), arguing by means of the inverse process is therefore not reasonable any longer. Instead, we utilize Corollary \ref{cor:uniform consistency} to move over to an integral with respect to the empirical feature distribution in order to exploit the characterization \eqref{eq:char_phi}, which in turn allows to approximate the resulting empirical $L^1$-error by a supremum over a centered partial sum process. 
\item The limiting distributions in the intermediate regime ($n\delta_{n}^{2} \longrightarrow c \in (0,\infty)$) remain an open problem.
\end{itemize}

\subsection{Auxiliary results on the inverse process}\label{sec:inverse process}
The following result is a key ingredient for the proofs of 
Proposition~\ref{prop: 4.2} and Theorem~\ref{thm:l1 rate} (i). Recall the definition of the inverse process $\tilde{U}_{n}$ 
in \eqref{eq:Un} and define $\lambda_{n} \defeq \Phi_{n} \circ F_{X}^{-1}$.

\begin{lemma} \label{lem:tail bounds - 1}
Suppose that $\Phi_0$ is continuously differentiable with $\Phi_0'(0)>0$. Then, for any $q \geq 2$, there exist
constants $C = C(\Phi_{0},p_{X},q) > 0$ and $N_{0} = N_{0}(\Phi_{0},(\delta_{n})_{n\in\bN},q) \in \bN$, such that for every 
$n \geq N_{0}$, $a \in [0,1]$ and $x > 0$, 
\begin{itemize}
\item[(i)] 
$\displaystyle
\bP\big(|\tilde{U}_{n}(a) - \lambda_{n}^{-1}(a)| \geq x\big) 
\leq \mathds{1}_{\{x \in [0,(n\delta_{n}^{2})^{-1/3})\}} 
			+ \frac{C}{(n\delta_{n}^{2}x^{3})^{q/2}}\mathds{1}_{\{x \in [(n\delta_{n}^{2})^{-1/3},1]\}}$,

\item[(ii)] 
$ \displaystyle
\bP\big(|F_{n}^{-1}(\tilde{U}_{n}(a)) - \Phi_{n}^{-1}(a)| \geq x\big) 
\leq \mathds{1}_{\{x \in [0,(n\delta_{n}^{2})^{-1/3})\}} 
			+ \frac{C}{(n\delta_{n}^{2}x^{3})^{q/2}}\mathds{1}_{\{x \in [(n\delta_{n}^{2})^{-1/3},2T]\}}$.
\end{itemize}
\end{lemma}

The proof is given in Section~\ref{proof:tail bound - 1}. Interestingly, tight bounds on $\Phi_n'$, both from above and from below, enter its derivation. Therefore, the tail bound crucially depends on the fact that the level of feature impact $\delta_n$ actually precisely characterizes the speed with which the gradient of the feature-label relation approaches zero (uniformly on compacts). 

\begin{corollary}\label{cor:tail bounds - 1}
Suppose $\Phi_0$ to be continuously differentiable with $\Phi_0'(0)>0$. For $i=1,2$, let $(Z_{i,n})_{n\in\bN}$ be a 
sequence of $\bR$-valued random variables with $|Z_{i,n}|\leq c_n$ for some sequence $(c_n)_{n\in \bN}$. Then, for any 
$q \geq 2$ and any $r \in [1,3q/2)$, there exist $C = C(\Phi_{0},p_{X},q) > 0$ and $N_{0} = N_{0}(\Phi_{0},(\delta_{n})_{n\in\bN},q) \in \bN$, such that for every 
$n \geq N_{0}$, $a \in [0,1]$ and $Z_{i,n} \in [-a,1-a]$, 
\[
\bE\big[\big|\tilde{U}_{n}(a+Z_{1,n}) - \lambda_{n}^{-1}(a+Z_{2,n})\big|^{r}\big] 
\leq C\min\Big\{(n\delta_{n}^{2})^{-r/3} + \Big(\frac{c_{n}}{\delta_{n}}\Big)^{r}, 1\Big\}.
\]
\end{corollary}

The proof is deferred to Section~\ref{proof:cor tail bounds - 1}, utilizing monotonicity of both $\tilde{U}_{n}$ and 
$\lambda_{n}^{-1}$. For $c_n=0$, we obtain an upper bound on the pointwise risk of the inverse process. 

\section{Further statistical implications and open problems} \label{sec:statistical implications}
Apart from explaining the small-sample effect on the NPMLE and identifying $n\delta_{n}^{2}$ (corresponding to $(\sqrt{n}\cdot\Phi_{n}')^{2}$) as the key quantity which governs the full picture of limiting distributions in case of a non-vanishing derivative, our results reveal further statistical consequences sketched below. 

\medskip\noindent
\textbf{A new hypotheses test and its optimal power against local alternatives. } Based on $n$ i.i.d.~observations $(X_{1},Y_{1}),\dots,(X_{n},Y_{n})$ in the isotonic binary regression model with uniformly distributed  $X_{i} \sim \cU([-T,T])$, we consider the testing problem 
\[
H_{0} \ : \ \Phi_{|[-T,T]} = c_{0}
\qquad \text{versus} \qquad 
H_{1} \ : \ \Phi \in \cF \ \text{with} \ \Phi(0) = c_{0} \ \text{and} \ \inf_{\nu}\frac{\omega_{\nu}^{[-T,T]}(\Phi)}{\nu} > 0,
\]
with $c_{0} \in (0,1)$ and where $\omega_{\boldcdot}^{[-T,T]}$ is defined in \eqref{eq:modulus}. We propose the test statistic 
\[
S_{n} \defeq \frac{\sqrt{n}}{2T}\int_{-T}^{T} |\hat{\Phi}_{n}(x) - c_{0}|dx
\] 
and the corresponding test 
\[
T_{n}\big((X_{1},Y_{1}),\dots,(X_{n},Y_{n})\big) 
\defeq 
\begin{cases}
0, &\quad \text{if } S_{n} \leq \kappa_{\alpha}, \\
1, &\quad \text{else}, 
\end{cases}
\]
with $\kappa_{\alpha}$ the $(1-\alpha)$-quantile of the limiting distribution derived in Theorem~\ref{thm:l1 rate} (ii) with $\Phi_{0}(0) = c_{0}$. Note that the limiting distribution of the $L^{1}$-error and in particular the quantiles under the simple null were previously unknown. Hence, the test itself could not have been developed before. As a consequence of Theorem~\ref{thm:l1 rate}~(i), this test is consistent against local alternatives $\Phi_{n}(x) = \Phi_{0}(\delta_{n}x)$ with $\Phi_{0} \in H_{1}$ in the whole slow regime. Indeed, if $n\delta_{n}^{2} \longrightarrow \infty$,  
\begin{align*}
\bP_{\Phi_{n}}(T_{n}=1) 
\geq \ &\bP_{\Phi_{n}}\bigg(\frac{\sqrt{n}}{2T}\int_{-T}^{T} |\Phi_{n}(x) - c_{0}|dx 
				- \frac{\sqrt{n}}{2T}\int_{-T}^{T} |\hat{\Phi}_{n}(x) - \Phi_{n}(x)|dx  > \kappa_{\alpha}\bigg) \\
= \ &\bP_{\Phi_{n}}\bigg(\frac{(n\delta_{n}^{2})^{1/6}}{2T}\bigg(\Big(\frac{n}{\delta_{n}}\Big)^{1/3}\int_{-T}^{T}|\Phi_{n}(x) - c_{0}|dx - \mu_{n}\bigg) \\
	&\qquad\quad - \frac{(n\delta_{n}^{2})^{1/6}}{2T}\bigg(\Big(\frac{n}{\delta_{n}}\Big)^{1/3}\int_{-T}^{T}|\hat{\Phi}_{n}(x) - \Phi_{n}(x)|dx - \mu_{n}\bigg) > \kappa_{\alpha}\bigg) \\
= \ &\bP_{\Phi_{n}}\bigg(\frac{(n\delta_{n}^{2})^{1/6}}{2T}\big(O((n\delta_{n}^{2})^{1/3}) - \mu_{n}\big) \\
	&\qquad\quad - \frac{(n\delta_{n}^{2})^{1/6}}{2T}\bigg(\Big(\frac{n}{\delta_{n}}\Big)^{1/3}\int_{-T}^{T}|\hat{\Phi}_{n}(x) - \Phi_{n}(x)|dx - \mu_{n}\bigg) > \kappa_{\alpha}\bigg) 
	\longrightarrow 1 \quad \text{as } n \longrightarrow \infty, 
\end{align*}
whereas by Theorem~\ref{thm:l1 rate} (ii), the test $T_{n}$ has no power for $n\delta_{n}^{2} \longrightarrow 0$, as 
\begin{align*}
\bP_{\Phi_{n}}(T_{n}=1) 
\leq \bP_{\Phi_{n}}\bigg(\frac{\sqrt{n}}{2T}\int_{-T}^{T} |\hat{\Phi}_{n}(x) - \Phi_{n}(x)|dx + o(1) > \kappa_{\alpha}\bigg) 
\longrightarrow \alpha \quad \text{for } n \longrightarrow \infty.
\end{align*}
Note that no test can achieve non-trivial power against these local alternatives $\Phi_{n}(\cdot) = \Phi_{0}(\delta_{n} \ \cdot)$ in the fast regime, because, with $P_{0}$ representing the distribution under the null, 
\begin{align*}
h^{2}\Big(P_{0}^{\otimes n},P_{\Phi_{n}}^{\otimes n}\Big) 
= 2\Bigg(1-\bigg(1-\frac{h^{2}(P_{0},P_{\Phi_{n}})}{2}\bigg)^{n}\Bigg) 
\longrightarrow 0 \quad \text{iff } n\delta_{n}^{2} \longrightarrow 0 
\end{align*}
by elementary algebra, where $h$ denotes the Hellinger metric from  \eqref{eq:hellinger metric}. That is, our test achieves the optimal separation rate against these types of local alternatives.

\begin{remark}
In the context of hypotheses testing, a more natural formulation for the local alternatives is given by $\overline{\Phi}_{n}(x) \defeq c_{0} + \delta_{n}f(x)$, 
for some strictly isotonic function $f_{|[-T,T]} \colon [-T,T] \to [-c_{0},1-c_{0}]$. The results of Section~\ref{sec:pointwise limit} for $\beta = 1$, as well as the results of Section~\ref{sec:L1 limit} carry over mutatis mutandis for $\Phi_{n}$ replaced by $\overline{\Phi}_{n}$.
Solely for $\beta > 1$, the convergence rate and phase transition are affected, as in this case $\overline{\Phi}_{n}^{(\beta)}(x) = \delta_{n}f^{(\beta)}(x)$, i.e.~the rate at which $\overline{\Phi}_{n}^{(\beta)}$ converges to zero is now the same for all $\beta$. An appealing aspect of inserting $\delta_{n}$ inside the regression function in binary regression is that $\Phi_{n}(\bR)$ of the local alternatives remains unchanged, as is the case in logistic regression, and $\Phi_{n}$ converges for the opposite case $\delta_{n} \longrightarrow \infty$ to the step function, i.e.~the label is then measurable with respect to the feature.
\end{remark}

\noindent
\textbf{The new benchmark for distributional approximation. }
Subsequently, we use the abbreviation $\sigma_{x_{0}}^{2}(f) \defeq f(x_{0})(1-f(x_{0}))$, $f \in \cF$ and note that $\sigma_{x_{0}}^{2}(\Phi_{n})$ can be consistently estimated by the plug-in version $\sigma_{x_{0}}^{2}(\hat{\Phi}_{n})$ when $\Phi_{0}(0) \in (0,1)$. Recalling $\Phi_{n}'(x_{0}) \asymp \delta_{n}$ in case of a non-vanishing derivative $\Phi_{0}'(0) > 0$, an immediate consequence of Theorem~\ref{thm:intermediate} is that then the Chernoff approximation
\begin{align*}
\Bigg(\frac{n}{\sigma_{x_{0}}^{2}(\Phi_{n})\Phi_{n}'(x_{0})}\Bigg)^{1/3}\big(\hat{\Phi}_{n}(x_{0})-\Phi_{n}(x_{0})\big) 
\longrightarrow_{\cL} \bigg(\frac{4}{p_{X}(x_{0})}\bigg)^{1/3}\argmin_{t \in \bR}\big\{Z(t) + t^{2} \big\} \quad \text{for } n \longrightarrow \infty 
\end{align*}
continues to hold iff $n\!\cdot\!\Phi_{n}'(x_{0})^{2}\longrightarrow\infty$. Under this condition, if $\widehat{\Phi_n'}(x_0)-\Phi_n'(x_0)=o_{\mathbb{P}}(\Phi_n'(x_0))$ for some derivative estimator $\widehat{\Phi_n'}$ of $\Phi_{n}'$, the approximation remains valid for
$$
\Bigg(\frac{n}{\sigma_{x_{0}}^{2}(\hat{\Phi}_{n})\widehat{\Phi_{n}'}(x_{0})}\Bigg)^{1/3}\big(\hat{\Phi}_{n}(x_{0})-\Phi_{n}(x_{0})\big) 
$$
by Slutsky's lemma. Of course, this required speed $\widehat{\Phi_n'}(x_0)-\Phi_n'(x_0)=o_{\mathbb{P}}(\Phi_{n}'(x_{0}))$ limits applicability in case of small derivatives $\Phi_{n}'(x_{0})$ and especially does not allow to exploit the Chernoff approximation in the full slow regime as long as no additional data on the derivative is available. Note that the popular method by \cite{BerPolRom1999} to construct confidence intervals in case of unknown rates of convergence is not applicable here a priori, as the subsampling procedure leads to a reduced sample size, thus changing the relation of sample size and steepness of the regression function. \par 
However, as pointed out in~\cite{HallYat2007}, there do exist situations where additional data on the derivatives is present. That is, if even 
\begin{equation}\label{eq:derivative}
\widehat{\Phi_n'}(x_0)-\Phi_n'(x_0)=o_{\mathbb{P}}(n^{-1/2}), 
\end{equation}
for some (partly) external derivative estimator, the distributional approximation 
simultaneously in all regimes as derived in  Section~\ref{sec:pointwise limit} could be exploited in its full generality. Actually, as a consequence of Theorem~\ref{thm:pointwise rate} and Theorem~\ref{thm:intermediate}, if $\bU_{n,\Phi_{n}'(x_{0})}$ denotes the distribution of 
\begin{equation}\label{eq:51}
\bigg\{\Big(\frac{n}{\Phi_n'(x_{0})}\Big)^{1/3} \wedge \sqrt{n}\bigg\}\big(\hat{\Phi}_{n}(x_{0})-\Phi_n(x_{0})\big),
\end{equation}
then $d_{BL}\big(\bU_{n,\Phi_{n}'(x_{0})}, \mathbb{Q}_{n}\big)
\longrightarrow 0$ as $n\longrightarrow\infty$. Here, $d_{BL}$ denotes the dual bounded Lipschitz metric and $\mathbb{Q}_{n}$, solely parametrized in  $\sigma_{x_{0}}(\Phi_{n})$ and $n\Phi_{n}'(x_{0})^{2}$,  denotes the distribution of the rescaled greatest convex minorant
$$
\Big(\big(n\Phi_n'(x_{0})^{2}\big)^{-1/6} \wedge 1\Big)\bar g_{n\Phi_n'(x_0)^2}^{*,\ell}(F_{X}(x_{0})) 
$$
for 
$$\bar g_{c}(s) \defeq \sigma_{x_{0}}(\Phi_{n})Z(s) + \sqrt{c}\,\bE\Big[(X-x_{0})
												\mathds{1}_{\{X \leq F_{X}^{-1}(s)\}}\Big].$$ 
Indeed, as visualized in Figure~\ref{fig:new limit}, the new distribution $\mathbb{Q}_{n}$ (orange line) provides a significantly better distributional approximation of $\bU_{n,\Phi_{n}'(x_{0})}$ (blue line), as compared to both, the Chernoff distribution (black line), as well as the limiting distribution for flat functions (grey line). Hereby, the distribution $\mathbb{Q}_{n}$, which does not degenerate as $n \longrightarrow \infty$, provides a new benchmark for distributional approximation of the NPMLE. 
In case of \eqref{eq:derivative}, we can replace the quantity $\Phi_{n}'(x_{0})$ by its estimate in the rate in \eqref{eq:51}, without affecting the limiting distribution. Likewise, as $\mathbb{Q}_{n}$ depends continuously on $n\Phi_{n}'(x_{0})^{2}$ and $\sigma_{x_{0}}(\Phi_{n})$,
$$d_{BL}\big(\widehat{\mathbb{Q}}_{n},\mathbb{Q}_{n}\big) \longrightarrow_{\bP} 0 \quad \text{as } n \longrightarrow \infty,$$
in view of Theorem~\ref{thm:intermediate}, where $\widehat{\mathbb{Q}}_{n}$ arises from $\mathbb{Q}_{n}$ by inserting the respective estimators $\sigma_{x_{0}}(\hat{\Phi}_{n})$ and $n\widehat{\Phi_{n}'}(x_{0})^{2}$, noting that \eqref{eq:derivative} reveals 
$n\widehat{\Phi_{n}'}(x_{0})^{2} = n\Phi_{n}'(x_{0})^{2} + \big(1+\sqrt{n}\Phi_{n}'(x_{0})\big)o_{\bP}(1)$. In total, 
\begin{align*}
d_{BL}\Bigg(\cL\bigg(\bigg\{\bigg(\frac{n}{\widehat{\Phi_n'}(x_{0})}\bigg)^{1/3} \wedge \sqrt{n}\bigg\}\big(\hat{\Phi}_{n}(x_{0}) - \Phi_{n}(x_{0})\big)\bigg), \widehat{\mathbb{Q}}_{n}\Bigg)
\longrightarrow_{\bP} 0\quad \text{as }n\longrightarrow\infty.
\end{align*}
\indent
Finally, as $\Phi_{n}'(x_{0})$ is involved unpleasantly in the rate of convergence of the NPMLE as well as in the new approximating distribution $\mathbb{Q}_{n}$ and \eqref{eq:derivative} cannot be taken for granted in general, the convergence 
$$d_{BL}\big(\bU_{n,\Phi_{n}'(x_{0})}, \mathbb{Q}_{n}\big)
\longrightarrow 0 \quad \text{as } n\longrightarrow\infty$$ 
simultaneously in all regimes raises the desire about a valid bootstrap approximation, thus elegantly avoiding estimation of the derivative. That is, the question is to which extent it is possible to imitate $\hat{\Phi}_{n}(x_{0}) - \Phi_{n}(x_{0})$ by a bootstrap statistic $T_{n}^{*}$ such that 
\begin{align*}
d_{BL}\Bigg(\cL\bigg(\bigg\{\Big(\frac{n}{\Phi_n'(x_{0})}\Big)^{1/3} \wedge \sqrt{n}\bigg\}T_{n}^{*}\, \bigg\arrowvert \, (X_{1},Y_{1}^{n}),\dots,(X_{n},Y_{n}^{n}) \bigg), \mathbb{U}_{n,\Phi_n'(x_0)}\Bigg)
\longrightarrow_{\bP} 0\quad \text{as }n\longrightarrow\infty,
\end{align*}
aiming at adaptive confidence intervals for $\Phi_{n}(x_{0})$. This problem is under current investigation by the authors. The bootstrap for the $L^{1}$-error is likewise open.

\medskip\noindent
\textbf{A multivariate extension. } In the spirit of logistic regression, a natural extension for multivariate feature variables is given by the semiparametric relation $\bP(Y = 1|X = x) = \Phi_{0}(\beta^{\top}x)$ for some isotonic $[0,1]$-valued function $\Phi$ and $\beta \in \bS^{d-1}$, the unit sphere in $\bR^{d}$. Up to now, it is still an unsolved conjecture (cf.~\cite{BalGroHen2019}) that the component $\hat{\beta}_{n}$ of the semiparametric MLE $(\hat{\Phi}_{n},\hat{\beta}_{n})$ converges at the parametric rate. To the best of our knowledge, also the semiparametric efficiency theory is still open. 
As a further development, the semiparametric MLE can be studied in the weak-feature-impact triangular array, where 
$\bP(Y^{n} = 1|X = x) = \Phi_{n}(\beta^{\top}x) = \Phi_{0}(\delta_{n}\beta^{\top}x)$ with $\delta_{n} \searrow 0$. Here, especially the rate-of-convergence question for $\hat{\beta}_{n}$ in dependence on the level of feature impact is open once again. In view of our results, we expect the rate and limiting distribution of $\hat{\Phi}_{n}$ to be affected by the level of feature impact, but it might influence the convergence rate and limiting distribution of $\hat{\beta}_{n}$ as well. Note that the index parameter $\beta$ is clearly not identifiable in the extremal case where $\Phi_{0}$ is constant.

\section{Proof of Theorem~\ref{thm:intermediate}}\label{proof:intermediate}
The statement for $c \longrightarrow 0$ is immediate. For $c \longrightarrow \infty$, let $I_{c} \defeq (-F_{X}(x_{0})c^{1/3},(1-F_{X}(x_{0}))c^{1/3})$ and define 
\[
h_{c} \colon I_{c} \to \bR, \quad 
h_{c}(t) \defeq c^{1/6}\big(g_{1,c}(F_{X}(x_{0})+c^{-1/3}t)-g_{1,c}(F_{X}(x_{0}))\big).
\]
Note that 
\begin{align*}
h_c^{*,\ell}(t) 
= c^{1/6}\big(g_{1,c}(F_{X}(x_{0})+c^{-1/3}t)-g_{1,c}(F_{X}(x_{0}))\big)^{*,\ell} 
= c^{-1/6}g_{1,c}^{*,\ell}(F_{X}(x_{0})+c^{-1/3}t),
\end{align*}
and consequently, 
\[
h_{c}^{*,\ell}(0) 
= c^{-1/6}g_{1,c}^{*,\ell}(F_{X}(x_{0})).
\]
It thus suffices to prove the desired convergence for $h_{c}^{*,\ell}(0)$ for $c \longrightarrow \infty$. For ease of notation, let 
\[
m \colon [0,1] \to \bR, \quad
m(s) \defeq \int_{-\infty}^{F_{X}^{-1}(s)}(x-x_{0})p_{X}(x)dx 
= \bE[(X-x_{0})\mathds{1}_{\{X\leq F_{X}^{-1}(s)\}}]
\]
and note that 
\[
m(s)-m(F_{X}(x_{0}))
= \int_{x_{0}}^{F_{X}^{-1}(s)}(x-x_0)p_{X}(x)dx.
\]
Since $p_{X}$ is continuous at $x_{0}$ and $p_{X}(x_{0})>0$, a Taylor expansion of $F_{X}^{-1}$ around $F_{X}(x_{0})$ gives 
\[
F_{X}^{-1}(F_{X}(x_{0})+h)
= x_{0} + \frac{h}{p_{X}(x_{0})} + o(h).
\]
Consequently,
\begin{align*}
m(F_{X}(x_{0})+h)-m(F_{X}(x_{0}))
= \int_{x_{0}}^{F_{X}^{-1}(F_{X}(x_{0})+h)}(x-x_0)p_{X}(x)dx 
= \frac{h^{2}}{2p_{X}(x_{0})} + o(h^{2}).
\end{align*}
Therefore, for every fixed $R>0$,
\[
\sup_{|t|\leq R} \Big|c^{2/3}\big(m(F_{X}(x_{0})+c^{-1/3}t)-m(F_{X}(x_{0}))\big)-\frac{t^2}{2p_X(x_0)}\Big|
\longrightarrow 0 \quad \text{as } c \longrightarrow \infty.
\]
Now define
\[
\rho_{c} \colon I_{c} \to \bR, \qquad 
\rho_{c}(t)
\defeq c^{2/3}\big(m(F_{X}(x_{0})+c^{-1/3}t)-m(F_{X}(x_{0}))\big), 
\]
set $\tau \defeq \frac{\Phi_{0}'(0)}{2p_{X}(x_{0})}$ and consider the process $(\tilde{h}_{c}(t))_{t \in I_{c}}$, where
\[
\tilde{h}_{c}(t) 
\defeq \sigma_{\Phi_{0}} Z(t)+\Phi_{0}'(0)\rho_{c}(t).
\]
Then, by basic invariance properties of the law of Brownian motion, $\tilde{h}_{c}$ has the same law as $h_{c}$. Further, for 
\[
H \colon \bR \to \bR, \qquad 
H(t) \defeq \sigma_{\Phi_{0}} Z(t) + \tau t^{2}, 
\]
we have 
\[
\tilde{h}_{c}(t) \longrightarrow H(t) \quad \text{locally uniformly on } \bR \text{ as } c \longrightarrow \infty.
\]
For $a \in \bR$, Lemma~\ref{lem:general switch relation} now yields 
\begin{equation}\label{eq:50}
\{\tilde h_{c}^{*,\ell}(0)\leq a\} = \Big\{\argminp_{t \in I_{c}}\{\tilde{h}_{c}(t) - at\} \geq 0\Big\}
\ \text{ and } \ 
\{H^{*,\ell}(0)\leq a\} = \Big\{\argminp_{t \in \bR}\{H(t) - at\} \geq 0\Big\}.
\end{equation}
Since $H(t)-at = \sigma_{\Phi_{0}} Z(t) + \tau t^{2}-at \longrightarrow \infty$ as $|t| \longrightarrow \infty$ by the law of the iterated logarithm, the minimizer of this expression is almost surely finite and together with Theorem~2 in \cite{Pimentel2014location}, it is also almost surely unique. By argmin-continuous mapping (Theorem~3.2.2 of \cite{VaartWellner2023}), the local uniform convergence $\tilde{h}_{c} \longrightarrow H$ then implies
\[
\argminp_{t \in I_{c}}\{\tilde{h}_{c}(t) - at\} 
\longrightarrow_{\cL} \argminp_{t \in \bR}\{H(t) - at\} \quad \text{as } c \longrightarrow \infty.
\]
Hence, by an application of an obvious adjustment of Lemma~A.2 of \cite{Cattaneo2024bootstrap} to processes defined on a compact interval (see \cite{Cattaneo2025continuity} as well) and \eqref{eq:50}, $\tilde{h}_{c}^{*,\ell}(0) \longrightarrow_{\cL} H^{*,\ell}(0)$ for $c \longrightarrow \infty$ and it remains to identify the law of $H^{*,\ell}(0)$. By Lemma~\ref{lem:argmin brownian motion transformation}, 
\[
\argmin_{t\in\bR}\big\{\sigma_{\Phi_{0}} Z(t) + \tau t^{2} - at\big\}
=_{\cL}
\Big(\frac{\sigma_{\Phi_{0}}}{\tau}\Big)^{2/3}\argmin_{t \in \bR}\big\{Z(t) + t^{2}\big\} + \frac{a}{2\tau} 
\]
and so together with Lemma~\ref{lem:general switch relation}, 
\[
\bP\big(H^{*,\ell}(0)\le a\big) 
= \bP\bigg(\argmin_{t \in \bR}\big\{Z(t) + t^{2}\big\} \geq -\frac{a}{2(\sigma_{\Phi_{0}}^{2}\tau)^{1/3}}\bigg) 
= \bP\Big(2(\sigma_{\Phi_{0}}^{2}\tau)^{1/3}\argmin_{t \in \bR}\big\{Z(t) + t^{2}\big\} \leq a\Big).
\]
Thus,
\[
H^{*,\ell}(0) 
=_{\cL} 2(\sigma_{\Phi_{0}}^2\tau)^{1/3}\argmin_{t \in \bR}\big\{Z(t) + t^{2}\big\}
= \bigg(\frac{4\sigma_{\Phi_{0}}^{2}\Phi_{0}'(0)}{p_{X}(x_{0})}\bigg)^{1/3}
										\argmin_{t \in \bR}\big\{Z(t) + t^{2} \big\}
\]
and the assertion follows. \qed

\section{Proof of Theorem~\ref{thm:l1 rate} (ii)} \label{proof:4.4ii}
With $P_{n}=\frac{1}{n}\sum_{i=1}^n\delta_{X_i}$ denoting the empirical measure of $X_{1},\dots,X_{n}$, we shall first prove 
that 
\begin{equation}\label{eq: Step1-fast}
\sqrt{n}\int_{-T}^{T}|\hat{\Phi}_{n}(x) - \Phi_{n}(x)|dP_{X}(x) = \sqrt{n}\int_{-T}^{T}|\hat{\Phi}_{n}(x) - \Phi_{n}(x)|dP_{n}(x)+ o_{\bP}(1).
\end{equation}
To this aim, we decompose
\begin{align*}
\sqrt{n}&\int_{-T}^{T}|\hat{\Phi}_{n}(x) - \Phi_{n}(x)|dP_{X}(x) \\
&\qquad\quad= \sqrt{n}\int_{-T}^{T}|\hat{\Phi}_{n}(x) - \Phi_{n}(x)|d(P_{X}-P_{n})(x) 
							+ \sqrt{n}\int_{-T}^{T}|\hat{\Phi}_{n}(x) - \Phi_{n}(x)|dP_{n}(x) 
\end{align*}
and have to verify that the first term on the right-hand side converges to zero in probability. For this, let 
$\varepsilon>0$ be 
arbitrary. Setting $\Psi_{n}(\boldcdot) \defeq |\hat{\Phi}_{n}(\boldcdot)-\Phi_{n}(\boldcdot)|$, 
$I_{\eta} \defeq [-T+\eta,T-\eta]$ and writing $\| \boldcdot \|_{I_{\eta}}$ for the $\sup$-norm on $I_{\eta}$, 
we have for any $\eta\in (0,T)$, 
\begin{align*}
&\bP\bigg(\sqrt{n}\bigg|\int_{-T}^{T}|\hat{\Phi}_{n}(x)-\Phi_{n}(x)|d(P_{X}-P_{n})(x)\bigg| > \varepsilon\bigg) \\
&\qquad\quad\leq \bP\bigg(\sqrt{n}\bigg|\int_{I_{\eta}}\Psi_{n}(x)d(P_{X}-P_{n})(x)\bigg| > \varepsilon/2, 
															\|\Psi_{n}\|_{I_{\eta}} \leq \eta\bigg) 
					+ \bP\big(\|\Psi_{n}\|_{I_{\eta}} > \eta\big) \\
&\qquad\qquad\qquad\quad+ \bP\bigg(\sqrt{n}\bigg|\int_{[-T,T] \setminus I_{\eta}}\Psi_{n}(x)d(P_{X}-P_{n})(x)\bigg| > \varepsilon/2\bigg).
\end{align*}
By Corollary~\ref{cor:uniform consistency}, $\bP\big(\|\Psi_{n}\|_{I_{\eta}} > \eta\big) \longrightarrow 0$ as 
$n \longrightarrow \infty$. From Markov's inequality, we get 
\begin{align*}
\bP\bigg(\sqrt{n}\bigg|\int_{I_{\eta}}\Psi_{n}(x)d(P_{X}-P_{n})(x)\bigg| > \varepsilon/2, 
															\|\Psi_{n}\|_{I_{\eta}} \leq \eta\bigg) 
&\leq \bP\bigg(\sup_{g \in \cG_{n,\eta}}\bigg|\sqrt{n}\int_{I_{\eta}}g(x)d(P_{n}-P_{X})(x)\bigg| > \varepsilon/2\bigg) \\
&\leq \frac{2}{\varepsilon}\bE\bigg[\sup_{g \in \cG_{n,\eta}} \bigg|\frac{1}{\sqrt{n}}\sum_{i=1}^{n}g(X_{i})-\bE[g(X_{i})]\bigg|\bigg] 
\end{align*}
for the class
$\cG_{n,\eta} \defeq \big\{g \colon I_{\eta} \to [0,1] 
\mid g=|f-\Phi_{n}| \text{ for } f \in \cF, \ \|g\|_{I_{\eta}} \leq \eta\big\}$. 
Note that any $g \in \cG_{n,\eta}$ satisfies $\bE[g(X)^{2}] \leq \eta^{2}$ and $\|g\|_{I_{\eta}} \leq \eta$. 
Theorem~2.14.17' of \cite{VaartWellner2023} then reveals for some universal constant $C > 0$, 
\begin{align*}
\bE\bigg[\sup_{g \in \cG_{n,\eta}}\bigg|\frac{1}{\sqrt{n}}\sum_{i=1}^{n}g(X_{i})-\bE[g(X_{i})]\bigg|\bigg] 
\leq CJ_{[]}\big(\eta,\cG_{n,\eta},L^{2}(P_{X})\big) \bigg(1 + \frac{J_{[]}\big(\eta,\cG_{n,\eta},L^{2}(P_{X})\big)}{\eta^{2}\sqrt{n}}\eta\bigg).
\end{align*}
where $J_{[]}\big(\eta,\cG_{n,\eta},L^{2}(P_{X})\big) 
\defeq \int_{0}^{\eta}\sqrt{1+\log(N_{[]}(\nu,\cG_{n,\eta},L^{2}(P_{X}))}d\nu$ denotes the bracketing integral with $\nu$-bracketing number 
$N_{[]}(\nu,\cG_{n,\eta},L^{2}(P_{X}))$ of $\cG_{n,\eta}$ in $L^{2}(P_{X})$. 
It remains to specify a bound for the entropy with bracketing. We prove in Lemma~\ref{lem:bracketing monotone functions} 
that 
\[
\log\big(N_{[]}\big(\nu,\cG_{n,\eta},L^{2}(P_{X})\big)\big)
\leq K\frac{(\eta+\delta_{n})}{\nu}\quad \forall\,\nu\in [0,\eta]
\]
for some constant $K > 0$ independent of $n$, $\eta$ and $\nu$, whence $J_{[]}\big(\eta,\cG_{n,\eta},L^{2}(P_{X})\big)$ is 
bounded by $K\sqrt{(\eta+\delta_{n})\eta}$ and therefore, 
\[
\limsup_{n\rightarrow\infty}\,\bE\bigg[\sup_{g \in \cG_{n,\eta}}\bigg|\frac{1}{\sqrt{n}}\sum_{i=1}^{n}g(X_{i})-\bE[g(X_{i})]\bigg|\bigg]
=\cO\big(\eta\big) \quad \text{as } \eta \longrightarrow 0.
\]
Now note that 
\begin{align*}
\bP\bigg(\sqrt{n}\bigg|\int_{[-T,T] \setminus I_{\eta}}\Psi_{n}(x)d(P_{X}-P_{n})(x)\bigg| > \varepsilon/2\bigg) 
&\leq \bP\bigg(\sqrt{n}\bigg|\int_{-T}^{-T+\eta}\Psi_{n}(x)d(P_{X}-P_{n})(x)\bigg| > \varepsilon/4\bigg) \\
		&\qquad+ \bP\bigg(\sqrt{n}\bigg|\int_{T-\eta}^{T}\Psi_{n}(x)d(P_{X}-P_{n})(x)\bigg| > \varepsilon/4\bigg). 
\end{align*}
Similar as before, but now for the class $\cG_{n,\eta}' \defeq \big\{g \colon [-T,-T+\eta] \to [0,1] \mid g=|f-\Phi_{n}| \text{ for } f \in \cF\big\}$,
\[
\bP\bigg(\sqrt{n}\bigg|\int_{-T}^{-T+\eta}\Psi_{n}(x)d(P_{X}-P_{n})(x)\bigg| > \varepsilon/4 \bigg) 
\leq \frac{4}{\varepsilon}\bE\bigg[\sup_{g \in \cG_{n,\eta}'}
												\bigg|\frac{1}{\sqrt{n}}\sum_{i=1}^{n}g(X_{i})-\bE[g(X_{i})]\bigg|\bigg]. 
\]
Note that any $g \in \cG_{n,\eta}'$ satisfies $\bE[g(X)^{2}] \leq \eta\|p_{X}\|_{\infty}$ and $\|g\|_{[-T,-T+\eta]} \leq 1$.
Theorem~2.14.17' of \cite{VaartWellner2023} then reveals for some universal constant $C > 0$, 
\begin{align*}
&\bE\bigg[\sup_{g \in \cG_{n,\eta}'}\bigg|\frac{1}{\sqrt{n}}\sum_{i=1}^{n}g(X_{i})-\bE[g(X_{i})]\bigg|\bigg] \\
&\qquad\qquad\qquad\leq CJ_{[]}\big(\sqrt{\eta\|p_{X}\|_{\infty}},\cG_{n,\eta}',L^{2}(P_{X})\big)
				\bigg(1 + \frac{J_{[]}\big(\sqrt{\eta\|p_{X}\|_{\infty}},
								\cG_{n,\eta}',L^{2}(P_{X})\big)}{\eta\|p_{X}\|_{\infty}\sqrt{n}}\bigg). 
\end{align*}
Again, from Lemma~\ref{lem:bracketing monotone functions}, 
\[
\log\big(N_{[]}\big(\nu,\cG_{n,\eta}',L^{2}(P_{X})\big)\big)
\leq K\frac{(1+\delta_{n})}{\nu}\quad \forall\,\nu\in \big[0,\sqrt{\eta\|p_{X}\|_{\infty}}\big]
\]
for some constant $K > 0$ independent of $n$, $\eta$ and $\nu$ and so 
$J_{[]}\big(\sqrt{\eta\|p_{X}\|_{\infty}},\cF_{\eta},L^{2}(P_{X})\big)$ is bounded by 
$K\sqrt{1+\delta_{n}}\eta^{1/4}$. Therefore, 
\[
\limsup_{n\rightarrow\infty}\,\bE\bigg[\sup_{g \in \cG_{n,\eta}'}\bigg|\frac{1}{\sqrt{n}}\sum_{i=1}^{n}g(X_{i})-\bE[g(X_{i})]\bigg|\bigg]
=\cO\big(\eta^{1/4}\big) \quad \text{as } \eta \longrightarrow 0.
\]
Identical arguments hold for $\bP\big(\sqrt{n}|\int_{T-\eta}^{T}\Psi_{n}(x)d(P_{X}-P_{n})(x)| > \varepsilon/4\big)$ and so \eqref{eq: Step1-fast} is verified. \par
Next, we shall prove that we may replace $\Phi_n$ by the constant $\Phi_0(0)$ in the $L^1$-distance within an error of 
negligible order. Here, the requirement $n\delta_n^2\longrightarrow 0$ is getting essential. By the reverse triangle 
inequality, a Taylor expansion of $\Phi_{n}$ around $0$ reveals 
\begin{align*}
\sqrt{n}\bigg|\int_{-T}^{T}|\hat{\Phi}_{n}(x) - \Phi_{n}(x)|dP_{n}(x) 
						- \int_{-T}^{T}|\hat{\Phi}_{n}(x) - \Phi_{0}(0)|dP_{n}(x)\bigg| 
&\leq \sqrt{n}\int_{-T}^{T}|\Phi_{n}(x) - \Phi_{0}(0)|dP_{n}(x) \\
&= \delta_{n}\frac{1}{\sqrt{n}}\sum_{i=1}^{n}\Phi_{0}'(\delta_{n}\xi_{i}^{n})|X_{i}|
\end{align*}
for suitable $\xi_{i}^{n}$ between $0$ and $X_{i}$.
Markov's inequality combined with the assumption that $n\delta_{n}^{2} \longrightarrow 0$ then yields 
$$
\sqrt{n}\int_{-T}^{T}|\hat{\Phi}_{n}(x) - \Phi_{n}(x)|dP_{n}(x) 
= \sqrt{n}\int_{-T}^{T}|\hat{\Phi}_{n}(x) - \Phi_{0}(0)|dP_{n}(x) + o_{\bP}(1) 
$$
and with \eqref{eq: Step1-fast}, we get
$\sqrt{n}\int_{-T}^{T}|\hat{\Phi}_{n}(x) - \Phi_{n}(x)|dP_{X}(x) 
= \sqrt{n}\int_{-T}^{T}|\hat{\Phi}_{n}(x) - \Phi_{0}(0)|dP_{n}(x) + o_{\bP}(1)$. Now, as the NPMLE is an increasing function and by Lemma~\ref{lem:jumping point}, as illustrated in 
Figure~\ref{fig:integral}, 
\begin{align}
\int_{-T}^{T} |\hat{\Phi}_{n}(x)-\Phi_{0}(0)| dP_{n}(x) 
&= \hspace{-0.54em}\sup_{s \in [-T,T]}\hspace{-0.2em}\bigg\{\int_{s}^{T}\big(\hat{\Phi}_{n}(x)-\Phi_{0}(0)\big)dP_{n}(x) 
					- \hspace{-0.2em}\int_{-T}^{s}\big(\hat{\Phi}_{n}(x)-\Phi_{0}(0)\big)dP_{n}(x)\bigg\}\nonumber \\
&= \hspace{-0.54em}\sup_{s \in [-T,T]}\hspace{-0.2em}\bigg\{\int_{-T}^{T}\big(\hat{\Phi}_{n}(x)-\Phi_{0}(0)\big)\big(1-2\mathds{1}_{\{x \leq s\}}\big)dP_{n}(x)\bigg\}. \label{eq: Step3-fast}
\end{align} 

\vspace{-6mm}
\begin{figure}[H]
  \centering
  \begin{tikzpicture}[scale=1.1]

  \draw[->, thin] (0,-1.5) -- (0,1.5) node[above] {};
  \draw[->, thin] (-0.1,0.2) -- (5.2,0.2) node[right] {};
  
  \draw (0.3, 0.1) -- (0.3, 0.3); 
  \node[below] at (0.3, 0.1) {\tiny{$-T$}};

  \draw (4.9, 0.1) -- (4.9, 0.3);
  \node[below] at (4.9, 0.1) {\tiny{$T$}};
  
  \draw[dotted] (1.7, -0.9) -- (1.7, 1.0);
  \node[below] at (1.7, -0.9) {\tiny{$s$}};
  
  \draw (-0.1,-1.0) node[left] {\tiny{$0$}};
  
  \draw (-0.1,1.0) node[left] {\tiny{$1$}};
  \draw (-0.1,1.0) -- (0.1,1.0);

  \draw (-0.1,0.2) node[left] {\tiny{$\Phi_{0}(0)$}};

  \foreach \x/\y/\len/\openstart/\openend in {
    -0.1/-1.0/0.7/0/0,
    0.6/-0.6/0.9/1/0,
    1.5/-0.3/0.6/1/0,
    2.1/-0.1/0.8/1/0,
    2.9/0.35/0.5/1/0,
    3.4/0.5/0.8/1/0,
    4.2/0.7/0.9/1/1
  }{
    \draw[thin, black] (\x,\y) -- ({\x+\len},\y);

    \ifnum\openstart=1
      \filldraw[fill=black, draw=black] (\x,\y) circle (1.25pt);
    \fi

    \ifnum\openend=0
      \filldraw[fill=white, draw=black] (\x + \len,\y) circle (1.25pt);
    \fi
    }

  \fill[blue!30, opacity=0.5]
    (0.3,0.2) -- (0.3,-1.0) --
    (0.6,-1.0) -- (0.6,-0.6) --
    (1.5,-0.6) -- (1.5,-0.3) --
    (1.7,-0.3) -- (1.7,0.2);
   
  \fill[orange!30, opacity=0.5]
    (1.7,0.2) -- (1.7,-0.3) --
    (2.1,-0.3) -- (2.1,-0.1) --
    (2.9,-0.1) -- (2.9,0.2); 

  \fill[blue!30, opacity=0.5]
    (2.9,0.2) -- (2.9,0.35) --
    (3.4,0.35) -- (3.4,0.5) --
    (4.2,0.5) -- (4.2,0.7) --
    (4.9,0.7) -- (4.9,0.2);

  \node at (4.9,1.0) {\tiny{$\hat{\Phi}_{n}$}};
  \end{tikzpicture}
  \hfill
  \begin{tikzpicture}[scale=1.1]

  \draw[->, thin] (0,-1.5) -- (0,1.5) node[above] {};
  \draw[->, thin] (-0.1,0.2) -- (5.2,0.2) node[right] {};
  
  \draw (0.3, 0.1) -- (0.3, 0.3);
  \node[below] at (0.3, 0.1) {\tiny{$-T$}};

  \draw (4.9, 0.1) -- (4.9, 0.3); 
  \node[below] at (4.9, 0.1) {\tiny{$T$}}; 
  
  \draw[dotted] (2.9, -0.9) -- (2.9, 1.0); 
  \node[below] at (2.9, -0.9) {\tiny{$s$}}; 
  
  \draw (-0.1,-1.0) node[left] {\tiny{$0$}};
  
  \draw (-0.1,1.0) node[left] {\tiny{$1$}};
  \draw (-0.1,1.0) -- (0.1,1.0);

  \draw (-0.1,0.2) node[left] {\tiny{$\Phi_{0}(0)$}};

  \foreach \x/\y/\len/\openstart/\openend in {
    -0.1/-1.0/0.7/0/0,
    0.6/-0.6/0.9/1/0,
    1.5/-0.3/0.6/1/0,
    2.1/-0.1/0.8/1/0,
    2.9/0.35/0.5/1/0,
    3.4/0.5/0.8/1/0,
    4.2/0.7/0.9/1/1
  }{
    \draw[thin, black] (\x,\y) -- ({\x+\len},\y);
        
    \ifnum\openstart=1
      \filldraw[fill=black, draw=black] (\x,\y) circle (1.25pt);
    \fi
    
    \ifnum\openend=0
      \filldraw[fill=white, draw=black] (\x + \len,\y) circle (1.25pt);
    \fi
    }
  
  \fill[blue!30, opacity=0.5]
    (0.3,0.2) -- (0.3,-1.0) --
    (0.6,-1.0) -- (0.6,-0.6) --
    (1.5,-0.6) -- (1.5,-0.3) --
    (2.1,-0.3) -- (2.1,-0.1) --
    (2.9,-0.1) -- (2.9,0.2);

  \fill[blue!30, opacity=0.5]
    (2.9,0.2) -- (2.9,0.35) --
    (3.4,0.35) -- (3.4,0.5) --
    (4.2,0.5) -- (4.2,0.7) --
    (4.9,0.7) -- (4.9,0.2);

  \node at (4.9,1.0) {\tiny{$\hat{\Phi}_{n}$}};
  \end{tikzpicture}
  \caption{The colored area represents 
  $\int_{s}^{T} (\hat{\Phi}_{n}(x)-\Phi_{0}(0)) dP_{n}(x) - \int_{-T}^{s} (\hat{\Phi}_{n}(x)-\Phi_{0}(0)) dP_{n}(x)$, 
  where the blue color signals a positive area w.r.t~$P_{n}$ and the orange color signals a negative area w.r.t~$P_{n}$. 
  As we see, the area is maximized in the situation visualized on the right side and is equal to 
  $\int_{-T}^{T} |\hat{\Phi}_{n}(x)-\Phi_{0}(0)| dP_{n}(x)$.}\label{fig:integral}
\end{figure}

\vspace{-2mm}
Let $T_{1}^{n},\dots,T_{j_{n}}^{n}$ denote the jumping points of $\hat{\Phi}_{n}$ (which are random, both in number and location) and set $T_{0}^{n} \defeq X_{(1)}$, $T_{j_{n}+1}^{n} \defeq X_{(n)}$ and 
$T_{j_{n}+2}^{n} \defeq T$. Then, \eqref{eq: Step3-fast} can be rewritten as
\begin{align*}
\sup_{s \in [-T,T]}&\bigg\{\int_{-T}^{T}\big(\hat{\Phi}_{n}(x)-\Phi_{0}(0)\big)\big(1-2\mathds{1}_{\{x \leq s\}}\big)dP_{n}(x)\bigg\} \\
&= \sup_{s \in [-T,T]}\Bigg\{\sum_{j=0}^{j_{n}+1}\big(\hat{\Phi}_{n}(T_{j+1}^{n})-\Phi_{0}(0)\big)
				\big(F_{n}(T_{j+1}^{n})-F_{n}(T_{j}^{n})\big)\big(1-2\mathds{1}_{\{T_{j+1}^{n} \leq s\}}\big)\Bigg\}.
\end{align*}
Exploiting the characterization of the NPMLE as local sample average between two jumping points, which can be deduced from \eqref{eq:Phi} and \eqref{eq:char_phi} (cf.~\cite{Brunk1958Estimation}), i.e.
\[
\hat{\Phi}_{n}|_{(-\infty,T_{0}^{n})} 
= 0, \quad 
\hat{\Phi}_{n}|_{[T_{j_{n}+1}^{n},\infty)} 
= \hat{\Phi}_{n}(X_{(n)}), \quad 
\hat{\Phi}_{n}|_{[T_{j}^{n},T_{j+1}^{n})}
= \frac{\sum_{\ell=1}^{n}Y_{\ell}^{n}\mathds{1}_{\{T_{j}^{n} \leq X_{\ell} < T_{j+1}^{n}\}}}
											{\sum_{\ell=1}^{n}\mathds{1}_{\{T_{j}^{n} < X_{\ell} \leq T_{j+1}^{n}\}}}
\]
for $j=0,\dots,j_{n}$, where we also agree on $\hat{\Phi}_{n}(T_{j_{n}+2}^{n}) = \hat{\Phi}_{n}(X_{(n)})$, we obtain 
\begin{align*}
\sum_{j=0}^{j_{n}+1}\hat{\Phi}_{n}(T_{j+1}^{n})\big(F_{n}(T_{j+1}^{n})-F_{n}(T_{j}^{n})\big)
				\big(1-2\mathds{1}_{\{T_{j+1}^{n} \leq s\}}\big) 
&= \frac{1}{n}\sum_{\ell=1}^{n}Y_{\ell}^{n}\sum_{j=0}^{j_{n}+1}\mathds{1}_{\{T_{j}^{n} \leq X_{\ell} < T_{j+1}^{n}\}}
										\big(1 - 2\mathds{1}_{\{T_{j+1} \leq s\}}\big) \\
&= \frac{1}{n}\sum_{\ell=1}^{n}Y_{\ell}^{n}\big(1-2\mathds{1}_{\{X_{\ell} \leq s\}}\big)
					+ \frac{2}{n}\sum_{\ell=1}^{n}Y_{\ell}^{n}\mathds{1}_{\{X_{\ell} = s\}}.
\end{align*}
Further, we have 
\[
\sum_{j=0}^{j_{n}+1}\Phi_{0}(0)\big(F_{n}(T_{j+1}^{n}) - F_{n}(T_{j}^{n})\big)
						\big(1 - 2\mathds{1}_{\{T_{j+1}^{n} \leq s\}} \big) 
= \Phi_{0}(0)\big(1-2F_{n}(s)\big) + \frac{\Phi_{0}(0)}{n}, 
\]
as well as 
\[
\Bigg|\sup_{s \in [-T,T]}\Bigg\{\frac{2}{n}\sum_{\ell=1}^{n}Y_{\ell}^{n}\mathds{1}_{\{X_{\ell} = s\}} 
							- \frac{\Phi_{0}(0)}{n}\Bigg\}\Bigg| 
= o_{\bP}(n^{-1/2}).
\]
Now for $A_{n} \colon [-T,T] \to \bR$ denoting the continuous, piecewise linear process that satisfies 
\[
A_{n}(X_{i}) 
= \frac{1}{\sqrt{n}}\sum_{\ell=1}^{n}(Y_{\ell}^{n}-\Phi_{0}(0))
						\big(1-2\mathds{1}_{\{X_{\ell} \leq X_{i}\}}\big)
\]
for $i \in \{1,\dots,n\}$ and noting that $A_{n}$ attains its maximum at the observation points, combining the previous 
results shows that $\sqrt{n}\int_{\bR} |\hat{\Phi}_{n}(x) - \Phi_{n}(x)| dP_{X}(x)$ has the same asymptotic distribution as 
\[
\sup_{s \in [-T,T]}\Bigg\{\frac{1}{\sqrt{n}}\sum_{\ell=1}^{n}(Y_{\ell}^{n}-\Phi_{0}(0))
						\big(1-2\mathds{1}_{\{X_{\ell} \leq s\}}\big)\Bigg\}
= \sup_{s \in [-T,T]}\{A_{n}(s)\} 
= \max_{s \in [-T,T]}\{A_{n}(s)\}, 
\]
where we used continuity of $A_{n}$ and the fact that the process inside the $\sup$ on the left-hand side changes its value only 
at the observation points. Lemma \ref{lem: 4.5} yields $A_{n} \longrightarrow_{\cL} A$ in 
the space $\cC([-T,T])$ of continuous functions on $[-T,T]$, equipped with the topology of uniform convergence. 
The assertion then follows from the continuous mapping theorem. \hfill \qed

%
%

\begin{acks}[Acknowledgments]
The results are part of the first author's PhD thesis. 
This work was supported by the Research Unit 5381, DFG-grant RO3766/8-1 and the CRC 1597, Project-ID 499552394. Constructive comments and suggestions by the referees are gratefully acknowledged. 
\end{acks}
%

\begin{supplement}
\stitle{Supplement to ``The weak-feature-impact phase transition of the NPMLE in monotone binary regression''}
\sdescription{The supplement contains proofs and auxiliary results, namely the proofs of Lemma~\ref{lem:switch relation}, Propsition~\ref{prop:hellinger consistency}, Corollary~\ref{cor:uniform consistency}, 
Theorem~\ref{thm:lower bound pointwise}, Theorem~\ref{thm:pointwise rate}, Theorem~\ref{thm:lower bound L1}, Proposition~\ref{prop: 4.2}, Theorem~\ref{thm:l1 rate} (i), Lemma~\ref{lem:tail bounds - 1}, Corollary~\ref{cor:tail bounds - 1} and necessary intermediate results of 
the proofs of Theorem~\ref{thm:pointwise rate} and Theorem~\ref{thm:l1 rate}.}
\end{supplement}

\newpage

\title{Supplement to ``The weak-feature-impact phase transition of the NPMLE in monotone binary regression''}

\begin{aug}
\author[A]{\fnms{Dario}~\snm{Kieffer}\ead[label=e1]{dario.kieffer@stochastik.uni-freiburg.de}}
\and
\author[A]{\fnms{Angelika}~\snm{Rohde}\ead[label=e2]{angelika.rohde@stochastik.uni-freiburg.de}}
\address[A]{Albert-Ludwigs-Universit\"at Freiburg\printead[presep={,\ }]{e1,e2}}
\end{aug}

\addtocontents{toc}{\protect\setcounter{tocdepth}{2}} 
\tableofcontents
\appendix

\section{Proofs of Section~\ref{sec:NPMLE}} \label{sec:consistency proof}
In this section, the proofs of Lemma~\ref{lem:switch relation}, Proposition~\ref{prop:hellinger consistency}, necessary auxiliary results, as well as the proof of Corollary~\ref{cor:uniform consistency} are given.

\subsection{Proof of Lemma~\ref{lem:switch relation}}\label{proof:switch relation}
Using the notation introduced in Section~\ref{sec:NPMLE}, we have $\hat{\Phi}_{n}(x) = g_{n} \circ F_{n}(x)$, where $g_{n}$ denoted the left-derivative of $G_{n}$, which is the greatest convex minorant of $\Upsilon_{n}$. Thus, for $x \in \cX$ and $a \in [0,1]$, 
\[
\hat{\Phi}_{n}(x) > a 
\quad \iff \quad 
g_{n} \circ F_{n}(x) > a.
\]
By Lemma~\ref{lem:general switch relation} and the definition of $\tilde{U}_{n}$, we obtain 
\[
g_{n}(F_{n}(x)) > a
\quad \iff \quad 
\argminp_{u \in [0,1]}\{\Upsilon_{n}(u)-au\} < F_{n}(x)
\quad \iff \quad 
\tilde{U}_{n}(a) < F_{n}(x).
\]
As $F_{n}^{-1}$ is monotonically increasing and both $F_{n}$ and $\tilde{U}_{n}$ map into the set $\{i/n \mid i=0,\dots,n\}$, this is equivalent to 
\[
F_{n}^{-1} \circ \tilde{U}_{n}(a) < F_{n}^{-1} \circ F_{n}(x) 
\]
and the assertion follows from $F_{n}^{-1} \circ \tilde{U}_{n}(a) = U_{n}(a)$. \hfill \qed

\subsection{Proof of Proposition~\ref{prop:hellinger consistency}} \label{proof:hellinger consistency}
Before we start with the actual proof, 
let us introduce for every $\Psi \in \cF$ the functions 
\begin{align*}
f_{\Psi,\Phi} \colon \bR \times \{0,1\} \to \bar{\bR}, 
&\quad
f_{\Psi,\Phi}(x,y) \defeq \frac{p_{\Psi}(x,y) + p_{\Phi}(x,y)}{2p_{\Phi}(x,y)}, \\
m_{\Psi,\Phi} \colon \bR \times \{0,1\} \to \bar{\bR}, 
&\quad 
m_{\Psi,\Phi}(x,y) 
\defeq \log(f_{\Psi,\Phi}(x,y)), 
\end{align*}
where $\bar{\bR} \defeq \bR \cup \{-\infty,\infty\}$ and we agree on $\frac{c}{0} \defeq \infty$ for $c \geq 0$. Further, we define for every $n \in \bN$ 
\[
M_{n}(\Psi,\Phi) \defeq \frac{1}{n}\sum_{i=1}^{n}m_{\Psi,\Phi}(x_{i},y_{i}) 
\]
and their expectation under $P_{\Phi}^{\otimes n}$,  
\[
M(\Psi,\Phi) \defeq \bE_{\Phi}[m_{\Psi,\Phi}(X,Y)].
\]
Note that for the set $A_{\Phi} \defeq \{p_{\Phi} > 0\}$, we have $P_{\Phi}(A_{\Phi}) = 1$, so $f_{\Psi,\Phi}$ and $m_{\Psi,\Phi}$ are finite outside of a set of measure zero and we can work on the intersection with $A_{\Phi}$ for the remainder of this proof whenever necessary. Note further that $\Phi$ is identifiable by definition and that $M_{n}(\Phi,\Phi) = M(\Phi,\Phi) = 0$ by definition of 
$m_{\Psi,\Phi}$. The following Lemma guarantees $M_{n}(\hat{\Phi}_{n},\Phi) \geq M_{n}(\Phi,\Phi) = 0$ for every 
$n \in \bN$, which 
is a weaker statement than $\hat{\Phi}_{n}$ nearly maximizing $M_{n}$, but still suffices for the consistency proof. 

\begin{lemma} \label{lem:maximizer}
For every $n \in \bN$, we have $M_{n}(\hat{\Phi}_{n},\Phi) \geq 0$ $P_{\Phi}^{\otimes n}$-almost surely.  
\end{lemma}

\begin{proof}
By concavity of the logarithm and the definition of $\hat{\Phi}_{n}$ as the maximizer of the log-likelihood, we have $P_{\Phi}^{\otimes n}$-almost surely
\begin{align*}
M_{n}(\hat{\Phi}_{n},\Phi) 
&= \frac{1}{n}\sum_{i=1}^{n}\log\bigg(\frac{p_{\hat{\Phi}_{n}}(x_{i},y_{i}) 
										+ p_{\Phi}(x_{i},y_{i})}{2p_{\Phi}(x_{i},y_{i})}\bigg) \\
&\geq \frac{1}{n}\sum_{i=1}^{n}\frac{1}{2}\log\bigg(\frac{p_{\hat{\Phi}_n}(x_{i},y_{i})}
																{p_{\Phi}(x_{i},y_{i})}\bigg) \\
&= \frac{1}{2n}\sum_{i=1}^{n}\log(p_{\hat{\Phi}_{n}}(x_{i},y_{i})) - \log(p_{\Phi}(x_{i},y_{i})) \\
&\geq \frac{1}{2n}\sum_{i=1}^{n}\log(p_{\Phi}(x_{i},y_{i})) - \log(p_{\Phi}(x_{i},y_{i})) 
= 0.
\end{align*}
\end{proof}

The following Lemma guarantees that $\Phi$ is a well-separated point of maximum of $M(\cdot,\Phi)$. 

\begin{lemma} \label{lem:well separated maximum}
For every $\Psi, \Phi \in \cF$, we have $M(\Psi,\Phi) \leq -\frac{d^{2}(\Psi,\Phi)}{8}$.
In particular, 
\[
\sup_{\Psi : d(\Psi,\Phi) \geq \varepsilon} M(\Psi,\Phi) \leq -\frac{\varepsilon^{2}}{8} 
\quad \text{for every } \varepsilon > 0.
\]
\end{lemma}

\begin{proof}
By some basic calculations and Lemma~\ref{lem:log inequalities} (ii), we obtain 
\begin{align*}
M(\Psi,\Phi) 
&= \int_{\bR \times \{0,1\}}m_{\Psi,\Phi}(x,y)dP_{\Phi}(x,y) \\
&= \int_{\bR}\int_{\{0,1\}}m_{\Psi,\Phi}(x,y)p_{\Phi}(x,y)d\zeta(y)dP_{X}(x) \\
&= \int_{\bR}\int_{\{0,1\}} \log\bigg(\frac{p_{\Psi}(x,y) + p_{\Phi}(x,y)}{2p_{\Phi}(x,y)}\bigg)
																		p_{\Phi}(x,y)d\zeta(y)dP_{X}(x) \\
&\leq \int_{\bR}\int_{\{0,1\}} 2\Bigg(\sqrt{\frac{p_{\Psi}(x,y) 
											+ p_{\Phi}(x,y)}{2p_{\Phi}(x,y)}} - 1\Bigg)p_{\Phi}(x,y)d\zeta(y)dP_{X}(x) \\
&= \int_{\bR}2\int_{\{0,1\}} \sqrt{\frac{p_{\Psi}(x,y) 
											+ p_{\Phi}(x,y)}{2p_{\Phi}(x,y)}}p_{\Phi}(x,y)d\zeta(y)dP_{X}(x) - 2 \\
&= -\int_{\bR}\int_{\{0,1\}} \bigg(\sqrt{\frac{p_{\Psi}(x,y) + p_{\Phi}(x,y)}{2}} 
												- \sqrt{p_{\Phi}(x,y)}\bigg)^{2} d\zeta(y)dP_{X}(x) 
\end{align*}
and by Lemma~\ref{lem:hellinger inequalities}, 
\[
\bigg(\sqrt{\frac{p_{\Psi}(x,y) + p_{\Phi}(x,y)}{2}} - \sqrt{p_{\Phi}(x,y)}\bigg)^{2} 
\geq \frac{1}{16}\big(\sqrt{p_{\Psi}(x,y)} - \sqrt{p_{\Phi}(x,y)}\big)^{2}.
\]
Consequently, 
\begin{align*}
M(\Psi,\Phi) 
\leq -\frac{1}{16}\int_{\bR}\int_{\{0,1\}}\big(\sqrt{p_{\Psi}(x,y)} - \sqrt{p_{\Phi}(x,y)}\big)^{2}d\zeta(y)dP_{X}(x) 
= -\frac{1}{8}h^{2}(p_{\Psi},p_{\Phi}) 
= -\frac{1}{8}d^{2}(\Psi,\Phi).
\end{align*}
Now for any $\varepsilon > 0$ and every $\Psi \in \cF$ satisfying $d(\Psi,\Phi) \geq \varepsilon$, we have 
$M(\Psi,\Phi) \leq -\frac{\varepsilon^{2}}{8}$ and the assertion follows.
\end{proof}

Note that the previous result implies $\Phi \in \argmax_{\Psi \in \cF}M(\Psi,\Phi)$. 
Moreover, we obtain $M(\Psi,\Phi) = 0$ if and only if $\Psi = \Phi$ almost surely. \par 
As an intermediate step, before we prove that the difference between $M_{n}$ and $M$ converges uniformly in probability over $\cF$, we derive an upper bound uniformly in $\Phi$ on the bracketing numbers of the set of functions $m_{\Psi,\Phi}$. 

\begin{proposition} \label{prop:bracketing}
Let $\cG_{\Phi} \defeq \{m_{\Psi,\Phi} \ | \ \Psi \in \cF \}$. Then, there exists a constant $C > 0$, such that for all 
$\delta > 0$,
\[
\sup_{\Phi\in\cF}N_{[]}\big(\delta,\cG_{\Phi},L^{1}(P_{\Phi})\big) 
\leq N_{[]}(\delta/2,\cF,L^{1}(P_{X}))\leq C^{1/\delta}.
\]
\end{proposition}

\begin{proof}
The second inequality is an immediate consequence of Theorem~2.7.9 in \cite{VaartWellner2023}, 
where the constructed brackets in particular belong to $\cF$. 
For arbitrary $\Psi \in \cF$, let 
$[\Psi_{L},\Psi^{U}]$ denote a corresponding $\delta$-bracket for $\Psi$, where $\Psi_{L}, \Psi^{U} \in \cF$. Let 
\begin{align*}
p_{L} \colon \bR \times \{0,1\} \to \bar{\bR}, &\quad p_{L}(x,y) \defeq \Psi_{L}(x)^{y}(1-\Psi^{U}(x))^{1-y}, \\
p^{U} \colon \bR \times \{0,1\} \to \bar{\bR}, &\quad p^{U}(x,y) \defeq \Psi^{U}(x)^{y}(1-\Psi_{L}(x))^{1-y}
\end{align*}
and define 
\begin{align*}
f_{\Phi,L} \colon \bR \times \{0,1\} \to \bar{\bR}, 	
&\quad f_{\Phi,L}(x,y) \defeq \frac{p_{L}(x,y) + p_{\Phi}(x,y)}{2p_{\Phi}(x,y)}, \\
f_{\Phi}^{U} \colon \bR \times \{0,1\} \to \bar{\bR}, 
&\quad f_{\Phi}^{U}(x,y) \defeq \frac{p^{U}(x,y) + p_{\Phi}(x,y)}{2p_{\Phi}(x,y)}.
\end{align*}
Then, for every $x \in \bR$, 
\begin{align*}
f_{\Phi,L}(x,0) 
= \frac{1}{2} + \frac{1-\Psi^{U}(x)}{2(1-\Phi(x))} 
&\leq \frac{1}{2} + \frac{1-\Psi(x)}{2(1-\Phi(x))} 
= f_{\Psi,\Phi}(x,0), \\
f_{\Phi,L}(x,1) = \frac{1}{2} + \frac{\Psi_{L}(x)}{2\Phi(x)} 
&\leq \frac{1}{2} + \frac{\Psi(x)}{2\Phi(x)} 
= f_{\Psi,\Phi}(x,1), \\
f_{\Phi}^{U}(x,0) 
= \frac{1}{2} + \frac{1-\Psi_{L}(x)}{2(1-\Phi(x))} 
&\geq \frac{1}{2} + \frac{1-\Psi(x)}{2(1-\Phi(x))} 
= f_{\Psi,\Phi}(x,0), \\
f_{n}^{U}(x,1) = \frac{1}{2} + \frac{\Phi^{U}(x)}{2\Phi(x)} 
&\geq \frac{1}{2} + \frac{\Psi(x)}{2\Phi(x)} 
= f_{\Psi,\Phi}(x,1),
\end{align*}
i.e. for every $(x,y) \in \bR \times \{0,1\}$, we have 
\[
f_{\Phi,L}(x,y) \leq f_{\Phi,\Psi}(x,y) \leq f_{\Phi}^{U}(x,y).
\]
Defining 
\begin{align*}
m_{\Phi,L} \colon \bR \times \{0,1\} \to \bar{\bR}, &\quad m_{\Phi,L}(x,y) \defeq \log(f_{\Phi,L}(x,y)), \\
m_{\Phi}^{U} \colon \bR \times \{0,1\} \to \bar{\bR}, &\quad m_{\Phi}^{U}(x,y) \defeq \log(f_{\Phi}^{U}(x,y)),
\end{align*}
we have 
\[
m_{\Phi,L}(x,y) \leq m_{\Psi,\Phi}(x,y) \leq m_{\Phi}^{U}(x,y)
\]
by definition of $m_{\Psi,\Phi}$. Moreover, from Lemma~\ref{lem:log inequalities} (i), we obtain 
\begin{align*}
\|m_{\Phi}^{U} - m_{\Phi,L}\|_{1,P_{\Phi}} 
&= \Big\|\log\Big(\frac{1}{2} + \frac{p^{U}}{2p_{\Phi}}\Big) 
				- \log\Big(\frac{1}{2} + \frac{p_{L}}{2p_{\Phi}}\Big)\Big\|_{1,P_{\Phi}} \\
&\leq 2\Big\|\frac{p^{U}}{2p_{\Phi}} - \frac{p_{L}}{2p_{\Phi}}\Big\|_{1,P_{\Phi}} \\
&= \int_{\bR}\int_{\{0,1\}} \Big|p^{U}(x,y) - p_{L}(x,y)\Big|\mathds{1}_{\{p_{\Phi}(x,y) > 0\}} d\zeta(y)dP_{X}(x) \\
&\leq \int_{\bR} |\Phi^{U}(x) - \Phi_{L}(x)| + |1 - \Phi^{U}(x) - (1-\Phi_{L}(x))| dP_{X}(x) \\
&= 2\|\Phi^{U} - \Phi_{L}\|_{1,P_{X}} 
\leq 2\delta.
\end{align*}
Thus, $[m_{\Phi,L},m_{\Phi}^{U}]$ is a $2\delta$-bracket enclosing $m_{\Psi,\Phi}\in\cG_{\Phi}$, where both, $m_{\Phi,L}$ 
and $m_{\Phi}^{U}$, are contained in $\cG_{\Phi}$ by construction. Consequently, 
\[
N_{[]}\big(\delta,\cG_{\Phi},L^{1}(P_{\Phi})\big) 
\leq N_{[]}(\delta/2,\cF,L^{1}(P_{X})). 
\]
\end{proof}

Uniformly in $\Phi$, the next Lemma states uniform convergence in probability of the difference 
$M_{n}(\cdot,\Phi)-M(\cdot,\Phi)$ over $\cF$, which will later 
allow us to derive convergence of the approximate maximizers of $M_{n}(\cdot,\Phi)$ and $M(\cdot,\Phi)$. The proof makes 
use of 
Proposition~\ref{prop:bracketing} and is based on a typical Glivenko-Cantelli argument (cf.~Lemma~3.1 in \cite{Geer2010}), which we had to modify for our setting to take into account the $\Phi$-dependent function 
classes. 

\begin{lemma} \label{lem:uniform convergence}
For every $\varepsilon>0$, we have 
\[
\sup_{\Phi\in\cF}P_{\Phi}^{\otimes n}\bigg(\sup_{\Psi \in \cF}|M_{n}(\Psi,\Phi) - M(\Psi,\Phi)| > \varepsilon\bigg)
\longrightarrow 0 \quad \text{as } n \longrightarrow \infty.
\]
\end{lemma}

\begin{proof}
First of all, note that for $\cG_{\Phi}$ defined as in Proposition~\ref{prop:bracketing}, we have 
\[
\sup_{\Psi \in \cF}|M_{n}(\Psi,\Phi) - M(\Psi,\Phi)| 
= \sup_{g \in \cG_{\Phi}} \bigg|\frac{1}{n}\sum_{i=1}^{n}g(x_{i},y_{i}) - \bE_{\Phi}[g(X,Y)]\bigg|.
\]
From Lemma~\ref{prop:bracketing}, we know that there exists $C > 0$, independent of $\Phi$, such that
\[
N_{[]}\big(\delta,\cG_{\Phi},L^{1}(P_{\Phi})\big) \leq C^{1/\delta} \qquad \text{for all } \delta > 0 \text{ and all } \Phi \in \cF.
\]
Thus, for every $\delta > 0$, there exists a $\delta$-bracketing set 
$\{[g_{j,L}^{\Phi},g_{j}^{U,\Phi}]\}_{j=1,\dots,N(\delta)}$ for $\cG_{\Phi}$ with respect to $P_{\Phi}$, satisfying 
$N(\delta) \leq C^{1/\delta}$ and $g_{j,L}^{\Phi}, g_{j}^{U,\Phi} \in \cG_{\Phi}$ for $j=1,\dots,N(\delta)$, for every 
$\Phi\in\cF$. More specifically, this means 
\[
\Big\|g_{j}^{U,\Phi} - g_{j,L}^{\Phi}\Big\|_{1,P_{\Phi}} \leq \delta 
\]
for $j=1,\dots,N(\delta)$ and that for every $g \in \cG_{\Phi}$, there exists $j \in \{1,\dots,N(\delta)\}$, such that 
\[
g_{j,L}^{\Phi} \leq g \leq g_{j}^{U,\Phi}.
\]
Thus, for every $g \in \cG_{\Phi}$, 
\begin{align*}
&\frac{1}{n}\sum_{i=1}^{n}g(x_{i},y_{i}) - \bE_{\Phi}[g(X,Y)] \\
&\qquad\qquad \leq \frac{1}{n}\sum_{i=1}^{n}g_{j}^{U,\Phi}(x_{i},y_{i}) - \bE_{\Phi}\Big[g_{j}^{U,\Phi}(X,Y)\Big] 
																+ \bE_{\Phi}\Big[g_{j}^{U,\Phi}(X,Y)\Big] - \bE_{\Phi}[g(X,Y)] \\
&\qquad\qquad \leq \frac{1}{n}\sum_{i=1}^{n}g_{j}^{U,\Phi}(x_{i},y_{i}) - \bE_{\Phi}\Big[g_{j}^{U,\Phi}(X,Y)\Big] 
																+ \Big\|g_{j}^{U,\Phi} - g\Big\|_{1,P_{\Phi}} \\
&\qquad\qquad \leq \frac{1}{n}\sum_{i=1}^{n}g_{j}^{U,\Phi}(x_{i},y_{i}) - \bE_{\Phi}\Big[g_{j}^{U,\Phi}(X,Y)\Big] + \delta.
\end{align*}
Similarly, we obtain 
\[
\frac{1}{n}\sum_{i=1}^{n}g(x_{i},y_{i}) - \bE_{\Phi}[g(X,Y)] 
\geq \frac{1}{n}\sum_{i=1}^{n}g_{j,L}^{\Phi}(x_{i},y_{i}) - \bE_{\Phi}\Big[g_{j,L}^{\Phi}(X,Y)\Big] - \delta, 
\]
implying 
\begin{align*}
&\bigg|\frac{1}{n}\sum_{i=1}^{n}g(x_{i},y_{i}) - \bE_{\Phi}[g(X,Y)]\bigg| 
\leq \max \bigg\{\bigg|\frac{1}{n}\sum_{i=1}^{n}g_{j}^{U,\Phi}(x_{i},y_{i}) - \bE\Big[g_{j}^{U,\Phi}(X,Y)\Big]\bigg|, \\
&\qquad\qquad\qquad\qquad\qquad\qquad\qquad\qquad\qquad 
		\bigg|\frac{1}{n}\sum_{i=1}^{n}g_{j,L}^{\Phi}(x_{i},y_{i}) - \bE_{\Phi}\Big[g_{j,L}^{\Phi}(X,Y)\Big]\bigg|\bigg\} + \delta.
\end{align*}
Defining 
$\cG_{\Phi,\delta}' \defeq \big\{ g_{j,L}^{\Phi} \ | \ j=1,\dots,N(\delta) \big\} 
													\cup \big\{ g_{j}^{U,\Phi} \ | \ j=1,\dots,N(\delta) \big\}$, 
we know from Proposition~\ref{prop:bracketing}, that $\cG_{\Phi,\delta}' \subset \cG_{\Phi}$ and obtain 
\begin{align*}
&\sup_{g \in \cG_{\Phi}} \bigg|\frac{1}{n}\sum_{i=1}^{n}g(x_{i},y_{i}) - \bE_{\Phi}[g(X,Y)]\bigg| \\
&\qquad\qquad \leq \max_{j=1,\dots,N(\delta)}
		\max \bigg\{\bigg|\frac{1}{n}\sum_{i=1}^{n}g_{j}^{U,\Phi}(x_{i},y_{i}) - \bE_{\Phi}\Big[g_{j}^{U,\Phi}(X,Y)\Big]\bigg|, \\
&\qquad\qquad\qquad\qquad\qquad\qquad\qquad\qquad\quad 
			\bigg|\frac{1}{n}\sum_{i=1}^{n}g_{j,L}^{\Phi}(x_{i},y_{i}) - \bE_{\Phi}\Big[g_{j,L}^{\Phi}(X,Y)\Big]\bigg|\bigg\} + \delta \\
&\qquad\qquad = \max_{g \in \cG_{\Phi,\delta}'}\bigg|\frac{1}{n}\sum_{i=1}^{n}g(x_{i},y_{i}) - \bE_{\Phi}[g(X,Y)]\bigg| 
																											+ \delta.
\end{align*}
Now for every $\varepsilon > 0$ and $\delta = \frac{\varepsilon}{2}$, we have $N_{n} \leq C^{1/\delta} = C^{2/\varepsilon}$ 
and we obtain from an application of Chebyshev's inequality, that 
\begin{align*}
&P_{\Phi}^{\otimes n}\bigg(\sup_{g \in \cG_{\Phi,\delta}}\bigg|\frac{1}{n}\sum_{i=1}^{n}g(x_{i},y_{i}) - \bE_{\Phi}[g(X,Y)]\bigg| 
																							\geq \varepsilon\bigg) \\
&\qquad\qquad\qquad \leq P_{\Phi}^{\otimes n}\bigg(\max_{g \in \cG_{\Phi,\delta}'}
					\bigg|\frac{1}{n}\sum_{i=1}^{n}g(x_{i},y_{i}) - \bE_{\Phi}[g(X,Y)]\bigg| + \frac{\varepsilon}{2} 
																							\geq \varepsilon\bigg) \\
&\qquad\qquad\qquad \leq \sum_{g \in \cG_{\Phi,\delta}'}P_{\Phi}^{\otimes n}\bigg(\bigg|\frac{1}{n}\sum_{i=1}^{n}g(x_{i},y_{i}) 
															- \bE_{\Phi}[g(X,Y)]\bigg| \geq \frac{\varepsilon}{2}\bigg) \\
&\qquad\qquad\qquad \leq \sum_{g \in \cG_{\Phi,\delta}'}\frac{4}{\varepsilon^{2}}\frac{\Var_{\Phi}(g(X,Y))}{n} \\
&\qquad\qquad\qquad \leq C^{2/\varepsilon}\frac{4}{\varepsilon^{2}}\frac{1}{n}
																	\sup_{g \in \cG_{\Phi,\delta}}\Var_{\Phi}(g(X,Y)).
\end{align*}
Assuming the variance is uniformly bounded in $\Phi$, the assertion follows immediately. To this aim, note first that for 
arbitrary $\Psi \in \cF$, 
\begin{align*}
\Var_{\Phi}(m_{\Psi,\Phi}(X,Y)) 
&\leq \int_{\bR}\int_{\{0,1\}} \log( f_{\Psi,\Phi}(x,y))^{2}p_{\Phi}(x,y)d\zeta(y)dP_{X}(x) \\
&= \int_{\bR}\int_{\{0,1\}} \log(f_{\Psi,\Phi}(x,y))^{2}
											p_{\Phi}(x,y)\mathds{1}_{\{f_{\Psi,\Phi}(x,y) \geq 1\}}d\zeta(y)dP_{X}(x) \\
	&\qquad + \int_{\bR}\int_{\{0,1\}} \log(f_{\Psi,\Phi}(x,y))^{2}
											p_{\Phi}(x,y)\mathds{1}_{\{f_{\Psi,\Phi}(x,y) < 1\}}d\zeta(y)dP_{X}(x).
\end{align*}
By applying Lemma~\ref{lem:log inequalities} (ii), as well as using the fact that $0 \leq p_{\Psi} \leq 1$ for every 
$\Psi \in \cF$, we obtain 
\begin{align*}
&\int_{\bR}\int_{\{0,1\}} \log( f_{\Psi,\Phi}(x,y))^{2}
											p_{\Phi}(x,y)\mathds{1}_{\{f_{\Psi,\Phi}(x,y) \geq 1\}}d\zeta(y)dP_{X}(x) \\
&\qquad \leq 4\int_{\bR}\int_{\{0,1\}} \Big(\sqrt{f_{\Psi,\Phi}(x,y)}-1\Big)^{2}
											p_{\Phi}(x,y)\mathds{1}_{\{f_{\Psi,\Phi}(x,y) \geq 1\}}d\zeta(y)dP_{X}(x) \\
&\qquad \leq 4\int_{\bR}\int_{\{0,1\}} \Big(f_{\Psi,\Phi}(x,y) - 2\sqrt{f_{\Psi,\Phi}(x,y)} + 1\Big)
											p_{\Phi}(x,y)d\zeta(y)dP_{X}(x) \\
&\qquad \leq 4\int_{\bR}\int_{\{0,1\}} \Big(\frac{p_{\Psi}(x,y) + p_{\Phi}(x,y)}{2} + p_{\Phi}(x,y)\Big)d\zeta(y)dP_{X}(x) \\ 
&\qquad \leq 4\int_{\bR}\int_{\{0,1\}} 2d\zeta(y)dP_{X}(x).
\end{align*}
Similarly, by an application of Lemma~\ref{lem:log inequalities} (iii) and by using $p_{\Psi}(x,y) + p_{\Phi}(x,y) \geq p_{\Phi}(x,y)$, 
\begin{align*}
&\int_{\bR}\int_{\{0,1\}} \log(f_{\Psi,\Phi}(x,y))^{2}
											p_{\Phi}(x,y)\mathds{1}_{\{f_{\Psi,\Phi}(x,y) < 1\}}d\zeta(y)dP_{X}(x) \\
&\qquad \leq \int_{\bR}\int_{\{0,1\}} \bigg(1 - \frac{1}{f_{\Psi,\Phi}(x,y)}\bigg)^{2}
											p_{\Phi}(x,y)\mathds{1}_{\{f_{\Psi,\Phi}(x,y) < 1\}}d\zeta(y)dP_{X}(x) \\
&\qquad \leq \int_{\bR}\int_{\{0,1\}} \bigg(1 - \frac{2}{f_{\Psi,\Phi}(x,y)} + \frac{1}{f_{\Psi,\Phi}(x,y)^{2}}\bigg)
											p_{\Phi}(x,y)d\zeta(y)dP_{X}(x) \\
&\qquad \leq \int_{\bR}\int_{\{0,1\}} \bigg(1 + \frac{1}{f_{\Psi,\Phi}(x,y)^{2}}\bigg)p_{\Phi}(x,y)d\zeta(y)dP_{X}(x) \\
&\qquad \leq \int_{\bR}\int_{\{0,1\}} \bigg(1 + 4\frac{p_{\Phi}(x,y)^{2}}{p_{\Phi}(x,y)^{2}}\bigg)
												p_{\Phi}(x,y)d\zeta(y)dP_{X}(x) \\
&\qquad = 5\int_{\bR}\int_{\{0,1\}} p_{\Phi}(x,y)d\zeta(y)dP_{X}(x).  
\end{align*}
Combining these results, we have shown that $\Var_{\Phi}(m_{\Psi,\Phi}(X,Y)) \leq 21$.
\end{proof}

Based on Lemmas~\ref{lem:maximizer}, \ref{lem:well separated maximum} and \ref{lem:uniform convergence}, we can now prove 
Proposition~\ref{prop:hellinger consistency}, following the proof of Theorem~5.7 in \cite{Vaart1998}. 

\begin{proof}[Proof of Proposition~\ref{prop:hellinger consistency}]
For every $\varepsilon > 0$, Lemma~\ref{lem:well separated maximum} shows that 
$M(\Psi,\Phi) \leq -\frac{\varepsilon^{2}}{8}$ for all $\Psi \in \cF$ with $d(\Psi,\Phi) \geq \varepsilon$. Thus, 
\[
\big\{d(\hat{\Phi}_n,\Phi) \geq \varepsilon\big\} 
\subset \bigg\{M(\hat{\Phi}_{n},\Phi) \leq -\frac{\varepsilon^{2}}{8}\bigg\} 
= \bigg\{-M(\hat{\Phi}_{n},\Phi) \geq \frac{\varepsilon^{2}}{8}\bigg\} \quad P_{\Phi}^{\otimes n}-\text{a.s.}
\]
From Lemma~\ref{lem:maximizer}, we obtain 
\[
-M(\hat{\Phi}_{n},\Phi) 
\leq M_{n}(\hat{\Phi}_{n},\Phi) - M(\hat{\Phi}_{n},\Phi)
\leq \sup_{\Psi \in \cF}|M_{n}(\Psi,\Phi) - M(\Psi,\Phi)| \quad P_{\Phi}^{\otimes n}-\text{a.s.}
\]
Consequently, 
\[
\bigg\{-M(\hat{\Phi}_{n},\Phi) \geq \frac{\varepsilon^{2}}{8}\bigg\} 
\subset \bigg\{\sup_{\Psi \in \cF} |M_{n}(\Psi,\Phi) - M(\Psi,\Phi)| \geq \frac{\varepsilon^{2}}{8}\bigg\} \quad P_{\Phi}^{\otimes n}-\text{a.s.}
\]
and by Lemma~\ref{lem:uniform convergence}, we have for all $\varepsilon > 0$, 
\[
\sup_{\Phi\in\cF}P_{\Phi}^{\otimes n}\big(d(\hat{\Phi}_{n},\Phi) \geq \varepsilon\big) 
\leq \sup_{\Phi\in\cF} P_{\Phi}^{\otimes n}\bigg(\sup_{\Psi \in \cF} |M_{n}(\Psi,\Phi) - M(\Psi,\Phi)| 
																	\geq \frac{\varepsilon^{2}}{8}\bigg) 
\longrightarrow 0 \quad \text{as } n \longrightarrow \infty.
\]
\end{proof}

\subsection{Proof of Corollary~\ref{cor:uniform consistency}} \label{proof:uniform consistency}
The idea of the proof is to show for every subsequence of $D_{n} \defeq \hat{\Phi}_{n} - \Phi_{n}$ that 
there exists a subsubsequence converging uniformly to $0$ in probability under~$\bP$. 
To make this precise, we start with an arbitrary subsequence of $(D_{n})$, which we will denote by $(D_{n})$ again for ease 
of notation. Then, by \eqref{eq: TV} and the characterization of convergence in probability in terms of almost surely convergent 
subsequences, there exists a subsubsequence $(n_{j})_{j \in \bN}$ such that 
\[
\int_{\bR}|\hat{\Phi}_{n_{j}}(x) - \Phi_{n_{j}}(x)|dP_{X}(x) \longrightarrow 0 \quad \bP\text{-a.s.} 
\quad \text{as } j \longrightarrow \infty. 
\]
Define 
\[
S_{\bP} \defeq \bigg\{\omega \in \Omega \ \bigg| \int_{\bR}|\hat{\Phi}_{n_{j}}(\omega,x) - \Phi_{n_{j}}(x)|dP_{X}(x) 
\longrightarrow 0 \text{ as } j \longrightarrow \infty \bigg\}
\]
and consider for fixed $\omega \in S_{\bP}$ an arbitrary subsequence of $D_{n_{j}}(\omega,\cdot)$, which we denote by 
$D_{n_{j}}(\omega,\cdot)$ again. Then, by an application of Markov's inequality with respect to $P_{X}$ on $\bR$, we obtain 
for every $\varepsilon > 0$, 
\begin{align*}
P_{X}(|D_{n_{j}}(\omega,\cdot)| > \varepsilon) 
&\leq \frac{1}{\varepsilon}\bE_{X}\big[|\hat{\Phi}_{n_{j}}(\omega,\cdot) - \Phi_{n_{j}}(\cdot)|\big] \\
&= \frac{1}{\varepsilon}\int_{\bR}|\hat{\Phi}_{n_{j}}(\omega,x) - \Phi_{n_{j}}(x)|dP_{X}(x) 
\longrightarrow 0 \quad \text{as } j \longrightarrow \infty,
\end{align*}
by definition of $S_{\bP}$. In different notation, this means 
\[
|D_{n_{j}}(\omega,\cdot)| = |\hat{\Phi}_{n_{j}}(\omega,\cdot) - \Phi_{n_{j}}(\cdot)| 
								\longrightarrow_{P_{X}} 0 \quad \text{as } j \longrightarrow \infty.
\]
But then, again, there exists another increasing sequence $(j_{k}^{\omega})_{k \in \bN}$, 
depending on $\omega$, satisfying $j_{k}^{\omega} \longrightarrow \infty$ for $k \longrightarrow \infty$, such that 
\[
|D_{n_{j_{k}^{\omega}}}(\omega,\cdot)| = |\hat{\Phi}_{n_{j_{k}^{\omega}}}(\omega,\cdot) - \Phi_{n_{j_{k}^{\omega}}}(\cdot)| 
									\longrightarrow 0 \quad P_{X}\text{-a.s.} \quad \text{as } k \longrightarrow \infty.
\]
Now, similar as before, we define 
\[
S_{P_{X}}(\omega) \defeq \big\{x \in \cX^{o} \mid |\hat{\Phi}_{n_{j_{k}^{\omega}}}(\omega,\cdot) 
																- \Phi_{n_{j_{k}^{\omega}}}(\cdot)| 
													\longrightarrow 0 \text{ as } k \longrightarrow \infty \big\}, 
\]
where $\cX^{o}$ denotes the interior of $\cX$. Then, for arbitrary but fixed $x_{0} \in \cX^{o} \setminus S_{P_{X}}(\omega)$ 
and for all $\varepsilon > 0$, the fact that $P_{X}$ has a 
Lebesgue density being positive on $\cX^{o}$ implies the existence of $x_{1}, x_{2} \in S_{P_{X}}(\omega)$ with 
$x_{1} < x_{0} < x_{2}$. Moreover, from Lemma~\ref{lem:pointwise convergence}, we know that there exists $K \in \bN$, such 
that 
\[
|\Phi_{n_{j_{k}^{\omega}}}(x_{2})-\Phi_{n_{j_{k}^{\omega}}}(x_{1})| < \varepsilon/5
\]
for every $k > K$. By choosing $K \in \bN$ sufficiently large, we also have 
\[
|\hat{\Phi}_{n_{j_{k}^{\omega}}}(\omega,x_{1}) - \Phi_{n_{j_{k}^{\omega}}}(x_{1})| < \varepsilon/5
\quad \text{ and } \quad 
|\hat{\Phi}_{n_{j_{k}^{\omega}}}(\omega,x_{2}) - \Phi_{n_{j_{k}^{\omega}}}(x_{2})| < \varepsilon/5
\]
for all $k > K$, and obtain that $|D_{n_{j_{k}^{\omega}}}(\omega,x_{0})|$ is bounded by 
\begin{align*}
&|\hat{\Phi}_{n_{j_{k}^{\omega}}}(\omega,x_{0}) - \Phi_{n_{j_{k}^{\omega}}}(x_{0})| \\
&\leq |\hat{\Phi}_{n_{j_{k}^{\omega}}}(\omega,x_{0}) - \hat{\Phi}_{n_{j_{k}^{\omega}}}(\omega,x_{1})| 
		+ |\hat{\Phi}_{n_{j_{k}^{\omega}}}(\omega,x_{1}) - \Phi_{n_{j_{k}^{\omega}}}(x_{1})| 
								+ |\Phi_{n_{j_{k}^{\omega}}}(x_{1}) - \Phi_{n_{j_{k}^{\omega}}}(x_{0})| \\
&\leq |\hat{\Phi}_{n_{j_{k}^{\omega}}}(\omega,x_{2}) - \hat{\Phi}_{n_{j_{k}^{\omega}}}(\omega,x_{1})| 
		+ |\hat{\Phi}_{n_{j_{k}^{\omega}}}(\omega,x_{1}) - \Phi_{n_{j_{k}^{\omega}}}(x_{1})| 
								+ |\Phi_{n_{j_{k}^{\omega}}}(x_{1}) - \Phi_{n_{j_{k}^{\omega}}}(x_{2})| \\
&< |\hat{\Phi}_{n_{j_{k}^{\omega}}}(\omega,x_{2}) - \Phi_{n_{j_{k}^{\omega}}}(x_{2})| 
		+ |\Phi_{n_{j_{k}^{\omega}}}(x_{2}) - \Phi_{n_{j_{k}^{\omega}}}(x_{1})| 
								+ |\Phi_{n_{j_{k}^{\omega}}}(x_{1}) - \hat{\Phi}_{n_{j_{k}^{\omega}}}(\omega,x_{1})| 
								+ \frac{2\varepsilon}{5}, 
\end{align*}
which is bounded by $\varepsilon$ and 
where we used the fact that both $\hat{\Phi}_{n}$ and $\Phi_{n}$ are increasing in $x$. Thus, we have shown 
\[
|D_{n_{j_{k}^{\omega}}}(\omega,x)| = |\hat{\Phi}_{n_{j_{k}^{\omega}}}(\omega,x) - \Phi_{n_{j_{k}^{\omega}}}(x)| 
														\longrightarrow 0 \quad \text{as } k \longrightarrow \infty
\]
not only for $x \in S_{P_{X}}(\omega)$, but for all $x \in \cX^{o}$. 
Utilizing that pointwise convergent $[0,1]$-valued isotonic functions with continuous limit also converge uniformly on 
compacts, we obtain for any compact interval $I \subset \cX^{o}$, 
\[
\sup_{x \in I}|\hat{\Phi}_{n_{j_{k}^{\omega}}}(\omega,x) - \Phi_{n_{j_{k}^{\omega}}}(x)| 
										\longrightarrow 0 \quad \text{as } k \longrightarrow \infty.
\]
But this means, that for any arbitrary subsequence of $D_{n_{j}}(\omega,\cdot)$, we found a subsubsequence converging to 
zero uniformly on $I$, implying by the subsequence argument, 
\[
\sup_{x \in I}|\hat{\Phi}_{n_{j}}(\omega,x) - \Phi_{n_{j}}(x)| 
										\longrightarrow 0 \quad \text{as } j \longrightarrow \infty.
\]
But because $\omega \in S_{\bP}$ was arbitrary, we have actually shown by definition of $S_{\bP}$, that 
\[
\sup_{x \in I}|\hat{\Phi}_{n_{j}}(\cdot,x) - \Phi_{n_{j}}(x)| 
									\longrightarrow 0 \quad \bP\text{-a.s.} \quad \text{as } j \longrightarrow \infty, 
\]
implying
\[
\sup_{x \in I}|\hat{\Phi}_{n_{j}}(\cdot,x) - \Phi_{n_{j}}(x)| 
									\longrightarrow_{\bP} 0 \quad \text{as } j \longrightarrow \infty.
\]
Applying the subsequence argument again, we conclude
\[
\sup_{x \in I}|\hat{\Phi}_{n}(\cdot,x) - \Phi_{n}(x)| 
									\longrightarrow_{\bP} 0 \quad \text{as } n \longrightarrow \infty.
\] \hfill \qed

\section{Remaining proofs of Section~\ref{sec:pointwise limit}} 
In this section, we prove Theorem~\ref{thm:lower bound pointwise}, Theorem~\ref{thm:pointwise rate} and the auxiliary results used in the proof of Theorem~\ref{thm:pointwise rate}. 

\subsection{Proof of Theorem~\ref{thm:lower bound pointwise}} \label{proof:pointwise lower bound}
Assume there exist $\Phi_{0,n}, \Phi_{1,n} \in \cF_{\delta}$ with 
\begin{equation}\label{eq: US1}
|\Phi_{0,n}(x_{0}) - \Phi_{1,n}(x_{0})| \geq 2C\max\Big\{n^{-1/2},\Big(\frac{n}{\delta}\Big)^{-1/3}\Big\} 
\end{equation}
for some $C > 0$. Provided that
\begin{align} \label{eq:hellinger}
h^{2}\big(P_{0,n}^{\otimes n},P_{1,n}^{\otimes n}\big) \leq \alpha < 2
\end{align}
with $P_{0,n}^{\otimes n} \defeq P_{\Phi_{0,n}}^{\otimes n}$ and 
$P_{1,n}^{\otimes n} \defeq P_{\Phi_{1,n}}^{\otimes n}$,
the general reduction scheme of Chapter~2.2 in \cite{Tsybakov2008} and Theorem~2.2 (ii) in \cite{Tsybakov2008} then reveal
\[
\inf_{T_n^{\delta}(x_0)}\sup_{\Phi\in\cF_{\delta}}
						P_{\Phi}^{\otimes n}\Big(\Big(\sqrt{n}\wedge \Big(\frac{n}{\delta}\Big)^{1/3}\Big)\big| T_n^{\delta}(x_0)-\Phi(x_0)\big| \geq C\Big)
\geq \frac{1}{2}\big(1-\sqrt{\alpha(1-\alpha)/4}\big)
> 0.
\]
In what follows, we construct $\Phi_{0,n}$ and $\Phi_{1,n}$ with properties \eqref{eq: US1} and \eqref{eq:hellinger} for 
$\delta \geq n^{-1/2}$ and $\delta < n^{-1/2}$ separately, noting that 
$\max\{n^{-1/2},(\frac{n}{\delta})^{-1/3}\} = (\frac{n}{\delta})^{-1/3}$ if and 
only if $\delta \geq n^{-1/2}$. In both cases, the constructed hypotheses will satisfy $\Phi_{0,n} \geq \Phi_{1,n}$ (hence $1-\Phi_{0,n} \leq 1-\Phi_{1,n}$). Thus, 
\begin{align*}
&h^{2}\big(P_{0,n}^{\otimes n},P_{1,n}^{\otimes n}\big) 
\leq nh^{2}\big(P_{0,n},P_{1,n}\big) \\
&\quad= \frac{n}{2}\int_{-T}^{T}\Big(\sqrt{\Phi_{0,n}(x)} - \sqrt{\Phi_{1,n}(x)}\Big)^{2} 
								+ \Big(\sqrt{1-\Phi_{0,n}(x)} - \sqrt{1-\Phi_{1,n}(x)}\Big)^{2}dP_{X}(x) \\
&\quad= \frac{n}{2}\int_{-T}^{T}
		\bigg(\frac{\Phi_{0,n}(x)-\Phi_{1,n}(x)}{\sqrt{\Phi_{0,n}(x)} + \sqrt{\Phi_{1,n}(x)}}\bigg)^{2} 
			+ \bigg(\frac{\Phi_{0,n}(x)-\Phi_{1,n}(x)}{\sqrt{1-\Phi_{0,n}(x)} + \sqrt{1-\Phi_{1,n}(x)}}\bigg)^{2}dP_{X}(x) \\
&\quad\leq \frac{n}{8}\int_{-T}^{T}(\Phi_{0,n}(x)-\Phi_{1,n}(x))^{2}
		\bigg(\frac{1}{\Phi_{1,n}(x)} + \frac{1}{1-\Phi_{0,n}(x)}\bigg)dP_{X}(x).
\end{align*}
$\bullet$ We start with the case $\delta < n^{-1/2}$. Let $0 < C < 1/\sqrt{2}$, $\eta_{n,\delta} \defeq 1/2 - \delta T - Cn^{-1/2}$ and define 
\begin{align*}
\Phi_{0,n} \colon \bR \to [0,1], 
&\quad \Phi_{0,n}|_{[-T,T]}(x) \defeq \delta(x+T) + \eta_{n,\delta} + 2Cn^{-1/2}, \\
\Phi_{1,n} \colon \bR \to [0,1], 
&\quad \Phi_{1,n}|_{[-T,T]}(x) \defeq \delta(x+T) + \eta_{n,\delta},  
\end{align*}
where both functions are defined outside $[-T,T]$ by their values at the respective boundaries. Obviously, 
$\Phi_{0,n}, \Phi_{1,n} \in \cF_{\delta}$ and
\[
|\Phi_{0,n}(x_{0})-\Phi_{1,n}(x_{0})| 
= 2Cn^{-1/2}.
\]
Next, for $n \geq 16(C+T)^{2}$ and all $x \in [-T,T]$, 
\begin{align*}
\Phi_{1,n}(x) &\geq \Phi_{1,n}(-T) = \eta_{n,\delta} \geq 1/2 - n^{-1/2}(C+T) \geq 1/4, \\
1-\Phi_{0,n}(x) &\geq 1-\Phi_{0,n}(T) = 1 - 2T\delta - \eta_{n,\delta} - 2Cn^{-1/2} = \eta_{n,\delta} \geq 1/4.
\end{align*}
Consequently, for $\alpha = 4C^{2}$, 
\[
h^{2}\big(P_{0,n}^{\otimes n},P_{1,n}^{\otimes n}\big) 
\leq \frac{n}{8}\int_{-T}^{T}8(\Phi_{0,n}(x)-\Phi_{1,n}(x))^{2}dP_{X}(x) 
= 4C^{2} 
= \alpha 
< 2. 
\]
Thus, \eqref{eq: US1} and \eqref{eq:hellinger} are satisfied for all $\delta \in [0,n^{-1/2})$ and $n \geq 16(C+T)^{2}$, whence 
\[
\inf_{T_n^{\delta}(x_0)}\sup_{\Phi\in\cF_{\delta}}
						P_{\Phi}^{\otimes n}\Big(\Big(\sqrt{n}\wedge \Big(\frac{n}{\delta}\Big)^{1/3}\Big)\big| T_n^{\delta}(x_0)-\Phi(x_0)\big| \geq C\Big)
> \frac{1}{2}\big(1-\sqrt{\alpha(1-\alpha)/4}\big)
> 0 .
\]

$\bullet$ For $\delta \geq n^{-1/2}$, assume $0 < C < \min\{(4T)^{1/3}/8,(32\|p_{X}\|_{\infty})^{-1/3}\}$, define $\eta_{\delta} \defeq 1/2 -\frac{\delta}{2}(x_{0}+T)$ and 
set 
\begin{align*}
\Phi_{0,n}(x) 
\defeq \eta_{\delta} &+ \frac{\delta}{2}(x+T)\mathds{1}_{\{x \in [-T,x_{0}-4C(n\delta^{2})^{-1/3})\}} \\
		&+ \delta\Big(x + \frac{T-x_{0}+4C(n\delta^{2})^{-1/3}}{2}\Big)
														\mathds{1}_{\{x \in [x_{0}-4C(n\delta^{2})^{-1/3},x_{0})\}} \\
		&+ \frac{\delta}{2}\big(x + T+4C(n\delta^{2})^{-1/3}\big)
														\mathds{1}_{\{x \in [x_{0},T]\}}, 
\end{align*}
as well as 
\begin{align*}
\Phi_{1,n}(x) 
\defeq \eta_{\delta} &+ \frac{\delta}{2}(x+T)\mathds{1}_{\{x \in [-T,x_{0})\}} \\
		&+ \delta\Big(x+\frac{T-x_{0}}{2}\Big)\mathds{1}_{\{x \in [x_{0},x_{0}+4C(n\delta^{2})^{-1/3})\}} \\
		&+ \frac{\delta}{2}\big(x + T+4C(n\delta^{2})^{-1/3}\big)\mathds{1}_{\{x \in [x_{0}+4C(n\delta^{2})^{-1/3},T]\}}, 
\end{align*}
where both functions are defined outside $[-T,T]$ by their values at the respective boundaries. 
\begin{figure}
\centering
\begin{tikzpicture}[scale=1.1]

\draw[->, thin] (0,-0.3) -- (0,3.5) node[above] {};
\draw[->, thin] (-0.1,-0.2) -- (5.0,-0.2) node[right] {};
 
\draw (0.2,-0.3) -- (0.2,-0.1); 
\node[below] at (0.2,-0.3) {\tiny{$-T$}}; 

\draw (2.2,-0.25) -- (2.2,-0.15); 
\draw (3.2,-0.25) -- (3.2,-0.15); 

\draw (2.7,-0.25) -- (2.7,-0.15); 
\node[below] at (2.7,-0.3) {\tiny{$x_{0}$}};
 
\draw (4.7,-0.3) -- (4.7,-0.1);
\node[below] at (4.7,-0.3) {\tiny{$T$}};
 
\draw[thick, dotted] (2.7,1.24) -- (2.7,2.01);

\draw (-0.1,0.74) node[left] {\tiny{$1/4$}};
\draw (-0.1,0.74) -- (0.1,0.74);
\draw (-0.1,1.24) -- (0.1,1.24);
\draw (-0.1,2.0) -- (0.1,2.0);
\draw (-0.1,2.5) -- (0.1,2.5);
\draw (-0.1,2.5) node[left] {\tiny{$3/4$}};

\foreach \x/\y/\xnew/\ynew in {
  0.2/0.01/2.2/1.01,
  2.2/1.01/2.7/2.01,
  2.7/2.01/4.7/3.01
}{
\draw[thick, orange] (\x,\y) -- (\xnew,\ynew);
}
  
\foreach \x/\y/\xnew/\ynew in {
  0.2/-0.01/2.7/1.24,
  2.7/1.24/3.2/2.24,
  3.2/2.24/4.7/2.99
}{
\draw[thick, blue] (\x,\y) -- (\xnew,\ynew);
}

\node at (2.0,1.7) {\tiny{$\Phi_{0,n}$}};
\node at (3.5,1.5) {\tiny{$\Phi_{1,n}$}};
\end{tikzpicture}
\caption{Visualization of $\Phi_{0,n}$ and $\Phi_{1,n}$ in case $\delta \geq n^{-1/2}$. Note that $|\Phi_{0,n}(x_0)-\Phi_{1,n}(x_0)| = 2C(n/\delta)^{-1/3}$.}\label{fig:hypotheses 2}
\end{figure}
Obviously, $\Phi_{0,n}, \Phi_{1,n} \in \cF_{\delta}$. A visualization of the hypotheses is given in Figure~\ref{fig:hypotheses 2}. Note that 
\[
|\Phi_{0,n}(x_{0})-\Phi_{1,n}(x_{0})| 
= 2C\Big(\frac{n}{\delta}\Big)^{-1/3}.
\]
Note further that for $n \geq 16^{3}C^{3}$ and all $x \in [x_{0}-4C(n\delta^{2})^{-1/3},x_{0}+4C(n\delta^{2})^{-1/3}]$, 
\begin{align*}
\Phi_{1,n}(x) 
&\geq \Phi_{1,n}(x_{0}-4C(n\delta^{2})^{-1/3}) 
= \frac{1}{2} - 4C\Big(\frac{n}{\delta}\Big)^{-1/3}
\geq \frac{1}{4}, \\
1-\Phi_{0,n}(x) 
&\geq 1-\Phi_{0,n}(x_{0}+4C(n\delta^{2})^{-1/3}) 
= \frac{1}{2} - 4C\Big(\frac{n}{\delta}\Big)^{-1/3} 
\geq \frac{1}{4}.
\end{align*}
Thus, for $\alpha = 8^{2}C^{3}\|p_{X}\|_{\infty}$, 
\begin{align*}
h^{2}\big(P_{0,n}^{\otimes n},P_{1,n}^{\otimes n}\big) 
&\leq n\int_{x_{0}-4C(n\delta^{2})^{-1/3}}^{x_{0}+4C(n\delta^{2})^{-1/3}}
											(\Phi_{0,n}(x)-\Phi_{1,n}(x))^{2}dP_{X}(x) \\
&= n\int_{x_{0}-4C(n\delta^{2})^{-1/3}}^{x_{0}}
							\Big(\frac{\delta}{2}(x-x_{0})+2C\Big(\frac{n}{\delta}\Big)^{-1/3}\Big)^{2}dP_{X}(x) \\
	&\qquad\qquad+ n\int_{x_{0}}^{x_{0}+4C(n\delta^{2})^{-1/3}}
							\Big(\frac{\delta}{2}(x_{0}-x)+2C\Big(\frac{n}{\delta}\Big)^{-1/3}\Big)^{2}dP_{X}(x) \\
&\leq n\int_{x_{0}-4C(n\delta^{2})^{-1/3}}^{x_{0}}
							\frac{\delta^{2}}{4}(x-x_{0})^{2}+4C^{2}\Big(\frac{n}{\delta}\Big)^{-2/3}dP_{X}(x) \\
	&\qquad\qquad+ n\int_{x_{0}}^{x_{0}+4C(n\delta^{2})^{-1/3}}
							\frac{\delta^{2}}{4}(x-x_{0})^{2}+4C^{2}\Big(\frac{n}{\delta}\Big)^{-2/3}dP_{X}(x) \\
&\leq n\int_{x_{0}-4C(n\delta^{2})^{-1/3}}^{x_{0}}8C^{2}\Big(\frac{n}{\delta}\Big)^{-2/3}dP_{X}(x) \\
	&\qquad\qquad+ n\int_{x_{0}}^{x_{0}+4C(n\delta^{2})^{-1/3}}8C^{2}\Big(\frac{n}{\delta}\Big)^{-2/3}dP_{X}(x) \\
&\leq 8^{2}C^{3}n(n\delta^{2})^{-1/3}\Big(\frac{n}{\delta}\Big)^{-2/3}\|p_{X}\|_{\infty} 
= 8^{2}C^{3}\|p_{X}\|_{\infty} 
= \alpha < 2
\end{align*}
and we have 
\[
\inf_{T_n^{\delta}(x_0)}\sup_{\Phi\in\cF_{\delta}}
						P_{\Phi}^{\otimes n}\Big(\Big(\sqrt{n}\wedge \Big(\frac{n}{\delta}\Big)^{1/3}\Big)\big| T_n^{\delta}(x_0)-\Phi(x_0)\big| \geq C\Big)
> \frac{1}{2}\big(1-\sqrt{\alpha(1-\alpha)/4}\big)
> 0 
\]
for all $\delta \in [n^{-1/2},1]$ and all $n \geq 12^{3}C^{3}$. \par 
In Summary, 
\[
\inf_{T_n^{\delta}(x_0)}\sup_{\Phi\in\cF_{\delta}}
						P_{\Phi}^{\otimes n}\Big(\Big(\sqrt{n}\wedge \Big(\frac{n}{\delta}\Big)^{1/3}\Big)\big| T_n^{\delta}(x_0)-\Phi(x_0)\big| \geq C\Big)
\geq \frac{1}{2}\big(1-\sqrt{\alpha(1-\alpha)/4}\big)
> 0 
\]
for $\alpha = \max\{4C^{2},8^{2}C^{3}\|p_{X}\|_{\infty}\}$, all $\delta \in [0,1]$ and 
$n > \max\{12^{3}C^{3}, 16(C+T)^{2}\}$ and so the assertion follows. \hfill \qed

\subsection{Proof of Theorem~\ref{thm:pointwise rate}}\label{proof:pointwise rate}
Note first that for every $v \in \bR$ and any sequence $(r_{n})_{n \in \bN}$ of real numbers, the switch 
relation (Lemma~\ref{lem:switch relation}) reveals 
\begin{equation}\label{eq:switch}
\begin{aligned}
\bP\big(r_{n}\big(\hat{\Phi}_{n}&(x_{0}) - \Phi_{n}(x_{0})\big) \leq v\big) 
= \bP\big(\hat{\Phi}_{n}(x_{0}) \leq \Phi_{n}(x_{0}) + r_{n}^{-1}v\big) \\
&= \bP\bigg(\argminp_{s \in [-T,T]}\bigg\{\frac{1}{n}\sum_{i=1}^{n}\big(Y_{i}^{n}-\Phi_{n}(x_{0})\big)
																			\mathds{1}_{\{X_{i} \leq s\}} 
				- r_{n}^{-1}\frac{v}{n}\sum_{i=1}^{n}\mathds{1}_{\{X_{i} \leq s\}}\bigg\} \geq F_{n}^{-1}(F_{n}(x_{0}))\bigg). 
\end{aligned} 
\end{equation} \par
$\bullet$ For the proof of (ii) and (iii), let $r_{n} = \sqrt{n}$ and define 
\[
h_{n} \colon [-T,T] \times \{0,1\} \times [-T,T] \to \bR, 
\quad h_{n}(x,y,t) \defeq (y-\Phi_{n}(x_{0}))\mathds{1}_{\{x \leq t\}}, 
\]
as well as $H_{n}(t) \defeq \bE[h_{n}(X,Y^{n},t)]$. Note that multiplying a function inside the $\argminp$ by $\sqrt{n}$ 
does not change the location of its minimum. Hence, by \eqref{eq:switch} and by utilizing that $F_{X}$ is a strictly 
isotonic bijection between $[-T,T]$ and $[0,1]$, we obtain 
\begin{align*}
&\bP\big(\sqrt{n}\big(\hat{\Phi}_{n}(x_{0}) - \Phi_{n}(x_{0})\big) \leq v\big) \\
&\quad= \bP\bigg(\argminp_{s \in [-T,T]}\bigg\{\frac{1}{\sqrt{n}}\sum_{i=1}^{n}\big(h_{n}(X_{i},Y_{i}^{n},s) - H_{n}(s)\big) 
					+ \sqrt{n}H_{n}(s) - \frac{v}{n}\sum_{i=1}^{n}\mathds{1}_{\{X_{i} \leq s\}}\bigg\} \geq F_{n}^{-1}(F_{n}(x_{0})) \bigg) \\
&\quad= \bP\bigg(\argminp_{s \in [0,1]}\bigg\{\frac{1}{\sqrt{n}}
							\sum_{i=1}^{n}\big(h_{n}(X_{i},Y_{i}^{n},F_{X}^{-1}(s)) - H_{n}(F_{X}^{-1}(s))\big) \\
&\quad\qquad\qquad\qquad\qquad\qquad\qquad	
					+ \sqrt{n}H_{n}(F_{X}^{-1}(s)) 
					- v\frac{1}{n}\sum_{i=1}^{n}\mathds{1}_{\{X_{i} \leq F_{X}^{-1}(s)\}}\bigg\} \geq F_{X}(F_{n}^{-1}(F_{n}(x_{0})))\bigg). 
\end{align*}
By Lemma~\ref{lem:pointwise inside argmin - fast}, the sequence inside the $\argmin$ 
converges weakly in $\ell^{\infty}([0,1])$ to 
\begin{equation}\label{eq:limit}
\big(\sigma_{\Phi_{0}}W(s) + \sqrt{c}\Phi_{0}^{(\beta)}(0)\bE\big[(X-x_{0})^{\beta}
												\mathds{1}_{\{X \leq F_{X}^{-1}(s)\}}\big] - vs\big)_{s \in [0,1]} 
\end{equation}
as long as $n\delta_{n}^{2\beta} \longrightarrow c \in [0,\infty)$ and so Proposition~\ref{prop:pointwise argmin - fast} 
yields convergence in distribution of the respective $\argmin$'s. By \cite{Stryhn1996triangular} for $c=0$ and more general, 
by an application of an obvious adjustment of Lemma~A.2 in \cite{Cattaneo2024bootstrap} to processes defined on a compact 
interval (see \cite{Cattaneo2025continuity} as well), we obtain under the condition in (ii) for any $c \geq 0$ 
that the $\argmin$ of the process in \eqref{eq:limit} has a continuous distribution function. In particular, $F_{X}(x_{0})$ 
is a continuity point of this distribution and $F_{X}(F_{n}^{-1}(F_{n}(x_{0})))$ converges stochastically to $F_{X}(x_{0})$ for $n \longrightarrow \infty$. Thus, 
\begin{align*}
\bP\big(\sqrt{n}\big(\hat{\Phi}_{n}(x_{0}) - \Phi_{n}(x_{0})\big) \leq v\big) 
\longrightarrow \ &\bP\bigg(\argmin_{s \in [0,1]}\{g_{\beta,c}(s) - vs\} \geq F_{X}(x_{0})\bigg) 
=\bP\big(g_{\beta,c}^{*,\ell}(F_{X}(x_{0})) \leq v\big), 
\end{align*}
where the last equality is a consequence of the switch relation (Lemma~\ref{lem:general switch relation}). 
As this is true for every $v \in \bR$, statements (ii) and (iii) now follow immediately. \par 
$\bullet$ Statement (i) could be deduced by appropriately specifying and verifying the technical ingredients from 
Theorem~2.2 in \cite{Mallick2023asymptotic}, which itself follows the so-called direct approach along the lines of 
\cite{Wright1981asymptotic}. Here, however, we prove (i) based on the switch relation in line with the proof of (ii) and 
(iii), as it highlights the occurence of the convergence rate of the inverse process, which also plays an important role in 
Section~\ref{sec:L1 limit}. Let us start by introducing the following functions 
\begin{align*}
g \colon [-T,T] \times [-T,T] \to \bR, 
&\quad g(x,t) \defeq \mathds{1}_{\{x \leq t\}} - \mathds{1}_{\{x \leq x_{0}\}}, \\
f_{n} \colon [-T,T] \times \{0,1\} \times [-T,T] \to \bR, 
&\quad f_{n}(x,y,t) \defeq (y-\Phi_{n}(x_{0}))g(x,t) 
\end{align*}
and let $r_{n} = (n/\delta_{n})^{\beta/(2\beta+1)}$. As in \eqref{eq:switch} and by noting that adding 
expressions which are independent of $s$ does not change the 
location of the minimum of a function in $s$, we obtain by an addition of zero 
\begin{align*}
&\bP\big(r_{n}\big(\hat{\Phi}_{n}(x_{0}) - \Phi_{n}(x_{0})\big) \leq v\big) \\
&\quad= \bP\bigg(\argminp_{s \in [-T,T]}\bigg\{\frac{1}{n}\sum_{i=1}^{n}f_{n}(X_{i},Y_{i}^{n},s) 
				- \frac{v}{nr_{n}}\sum_{i=1}^{n}g(X_{i},s)\bigg\} \geq F_{n}^{-1}(F_{n}(x_{0}))\bigg) \\
&\quad= \bP\bigg(\argminp_{s \in [x_{0}-T,x_{0}+T]}\bigg\{\frac{1}{n}\sum_{i=1}^{n}f_{n}(X_{i},Y_{i}^{n},x_{0}+s) 
				- \frac{v}{nr_{n}}\sum_{i=1}^{n}g(X_{i},x_{0}+s)\bigg\} \geq F_{n}^{-1}(F_{n}(x_{0})) - x_{0}\bigg).
\end{align*}
Defining $E_{n}(t) \defeq \bE[f_{n}(X_{i},Y_{i}^{n},t)]$ for $t \in [-T,T]$, 
$a_{n} \defeq (n\delta_{n}^{2\beta})^{-1/(2\beta+1)}$ and $b_{n} \defeq (n^{\beta+1}\delta_{n}^{\beta})^{1/(2\beta+1)}$, 
an addition of zero and multiplying with $b_{n}$ inside the $\argmin$ yields 
\begin{align*}
&\bP\big(r_{n}\big(\hat{\Phi}_{n}(x_{0}) - \Phi_{n}(x_{0})\big) \leq v\big) \\
&= \bP\bigg(\argminp_{s \in [x_{0}-T,x_{0}+T]}\bigg\{\frac{1}{n}\sum_{i=1}^{n}
									\big(f_{n}(X_{i},Y_{i}^{n},x_{0}+s) - E_{n}(x_{0}+s)\big) \\
&\qquad\qquad\qquad\qquad\qquad+ E_{n}(x_{0}+s) - \frac{v}{nr_{n}}\sum_{i=1}^{n}
																g(X_{i},x_{0}+s)\bigg\} - \big(F_{n}^{-1}(F_{n}(x_{0})) - x_{0}\big) \geq 0 \bigg) \\
&= \bP\bigg(a_{n}\argminp_{s \in [a_{n}^{-1}(x_{0}-T),a_{n}^{-1}(x_{0}+T)]}\bigg\{\frac{b_{n}}{n}
												\sum_{i=1}^{n}\big(f_{n}(X_{i},Y_{i}^{n},x_{0}+a_{n}s) 
																			- E_{n}(x_{0}+a_{n}s)\big) \\ 
&\qquad\qquad\qquad\qquad\qquad+ b_{n}E_{n}(x_{0}+a_{n}s) 
											- v\frac{b_{n}}{nr_{n}}\sum_{i=1}^{n}g(X_{i},x_{0}+a_{n}s)\bigg\} - \big(F_{n}^{-1}(F_{n}(x_{0})) - x_{0}\big) \geq 0 \bigg).
\end{align*}
By Lemma~\ref{lem:pointwise inside argmin}, the sequence inside the $\argmin$ restricted to $[-S,S]$ converges weakly in the space $\ell^{\infty}([-S,S])$ to 
\[
\bigg(\sigma_{\Phi_{0}}\sqrt{p_{X}(x_{0})}Z(s) 
				+ \frac{1}{(\beta+1)!}\Phi_{0}^{(\beta)}(0)p_{X}(x_{0})s^{\beta+1} - vp_{X}(x_{0})s\bigg)_{s \in [-S,S]}, 
\]
for every $S > 0$, as long as $(n\delta_{n}^{2\beta}) \longrightarrow \infty$. 
From Proposition~\ref{prop:pointwise argmin}, we 
then obtain convergence in distribution of the respective $\argmin$'s and by Lemma~A.2 of \cite{Cattaneo2024bootstrap}, 
the $\argmin$ of this process has a continuous distribution function. Further, $a_{n}^{-1}\big(F_{n}^{-1}(F_{n}(x_{0})) - x_{0}\big)$ converges stochastically to zero for $n \longrightarrow \infty$ and thus, 
\begin{align*}
\bP\big(r_{n}\big(\hat{\Phi}_{n}(x_{0}) &- \Phi_{n}(x_{0})\big) \leq v\big) \\
\longrightarrow \ &\bP\bigg(\argmin_{s \in \bR}\bigg\{\sigma_{\Phi_{0}}\sqrt{p_{X}(x_{0})}Z(s) 
	+ \frac{\Phi_{0}^{(\beta)}(0)p_{X}(x_{0})}{(\beta+1)!}s^{\beta+1} - vp_{X}(x_{0})s\bigg\} \geq 0\bigg) \\
&= \bP\bigg(\frac{1}{p_{X}(x_{0})}\argmin_{s \in \bR}\bigg\{\sigma_{\Phi_{0}}Z(s) 
				+ \frac{\Phi_{0}^{(\beta)}(0)}{p_{X}(x_{0})^{\beta}(\beta+1)!}s^{\beta+1} - vs\bigg\} \geq 0\bigg), 
\end{align*}
as $n \longrightarrow \infty$ and by the switch relation (Lemma~\ref{lem:general switch relation}), for every $v \in \bR$, 
\[
\bP\Big(\Big(\frac{n}{\delta_{n}}\Big)^{1/3}\big(\hat{\Phi}_{n}(x_{0}) - \Phi_{n}(x_{0})\big) \leq v\Big) 
\longrightarrow \bP\big(f_{\beta}^{*,\ell}(0) \leq v\big) 
\quad \text{as } n \longrightarrow \infty.
\] \hfill \qed

\subsection{Auxiliary results for the proof of Theorem~\ref{thm:pointwise rate}}
For the results related to the proof of Theorem~\ref{thm:pointwise rate} (i), let us recall the definitions 
\begin{align*}
g \colon [-T,T] \times [-T,T] \to \bR, 
&\quad g(x,t) \defeq \mathds{1}_{\{x \leq t\}} - \mathds{1}_{\{x \leq x_{0}\}}, \\
f_{n} \colon [-T,T] \times \{0,1\} \times [-T,T] \to \bR, 
&\quad f_{n}(x,y,t) \defeq (y-\Phi_{n}(x_{0}))g(x,t)
\end{align*}
and $E_{n}(t) \defeq \bE[f_{n}(X_{i},Y_{i}^{n},t)]$ for every $t \in [-T,T]$.
Furthermore, let $\beta \in \bN_{\geq 1}$, let 
\[
r_{n} \defeq \Big(\frac{n}{\delta_{n}}\Big)^{\beta/(2\beta+1)}, 
\qquad 
a_{n} \defeq (n\delta_{n}^{2\beta})^{-1/(2\beta+1)}, 
\qquad 
b_{n} \defeq (n^{\beta+1}\delta_{n}^{\beta})^{1/(2\beta+1)}
\]
and let $Z(s)$ denote a standard two-sided Brownian motion on $\bR$. We also define the stochastic processes 
\begin{align*}
\fZ_{n}^{1}(s) 
	&\defeq \frac{b_{n}}{n}\sum_{i=1}^{n}\big(f_{n}(X_{i},Y_{i}^{n},x_{0}+a_{n}s) - E_{n}(x_{0}+a_{n}s)\big), \\
\fZ_{n}^{2}(s) 
	&\defeq b_{n}E_{n}(x_{0}+a_{n}s), \\
\fZ_{n}^{3}(s)
	&\defeq v\frac{b_{n}}{nr_{n}}\sum_{i=1}^{n}g(X_{i},x_{0}+a_{n}s), \\
\fZ^{1}(s) 
	&\defeq \sqrt{\Phi_{0}(0)(1-\Phi_{0}(0))p_{X}(x_{0})}Z(s), \\
\fZ^{2}(s) 
	&\defeq \frac{1}{(\beta+1)!}\Phi_{0}^{(\beta)}(0)p_{X}(x_{0})s^{\beta+1}, \\
\fZ^{3}(s)
	&\defeq vp_{X}(x_{0})s 
\end{align*}
and set 
\[
\fZ_{n}(s) 
	\defeq \fZ_{n}^{1}(s) + \fZ_{n}^{2}(s) - \fZ_{n}^{3}(s), 
\quad 
\fZ(s) 
	\defeq \fZ^{1}(s) + \fZ^{2}(s) - \fZ^{3}(s)
\]
for $s \in [a_{n}^{-1}(x_{0}-T),a_{n}^{-1}(x_{0}+T)]$. Moreover, let 
\[
\hat{s}_{n} \defeq \argminp_{s \in [a_{n}^{-1}(x_{0}-T),a_{n}^{-1}(x_{0}+T)]} \fZ_{n}(s) 
\quad \text{and} \quad 
\hat{s} \defeq \argmin_{s \in \bR} \fZ(s)
\]
denote the minimizers of $\fZ_{n}(s)$ and $\fZ(s)$ respectively. 

\begin{lemma} \label{lem:pointwise inside argmin}
Let $\beta \in \bN$, $x_{0}$ an interior point of $ \cX$ and assume $\Phi_{0}$ to be $\beta$-times continuously 
differentiable in a neighborhood of $0$ with the $\beta$th derivative being the first non-vanishing derivative in $0$. 
Then, as long as $n\delta_{n}^{2\beta} \longrightarrow \infty$, 
\[
(\fZ_{n}(s))_{s \in [-S,S]}\longrightarrow_{\cL}(\fZ(s))_{s\in[-S,S]}\ \text{ in } \ell^{\infty}([-S,S])
\] 
for every $S > 0$.
\end{lemma}

\begin{proof}
Let $S>0$ be fixed but arbitrary and denote $\|f\|_{[-S,S]}:=\sup_{s\in[-S,S]}|f(s)|$ for any continuous $f:[-S,S]\rightarrow \bR$.

\medskip
\noindent
\textsc{Claim I: } $\|\fZ_{n}^{2}-\fZ^2\|_{[-S,S]}\longrightarrow_{\bP}0$. 

\smallskip
\noindent
\textit{Proof of Claim I. } By a Taylor expansion with Lagrange remainder of 
$\Phi_{n}(x)$ around $x_{0}$ up to order $\beta$, there exists $\xi_{n}(x)$ between $x_{0}$ and $x$ such that 
\[
\Phi_{n}(x) - \Phi_{n}(x_{0}) 
= \frac{1}{\beta!}\Phi_{n}^{(\beta)}(\xi_{n}(x))(x-x_{0})^{\beta}
= \delta_{n}^{\beta}\frac{1}{\beta!}\Phi_{0}^{(\beta)}(\delta_{n}\xi_{n}(x))(x-x_{0})^{\beta}.
\]
Thus, by using $b_{n}\delta_{n}^{\beta}a_{n}^{\beta+1}=1$, 
\begin{align*}
&\sup_{s \in [-S,S]}|\fZ_{n}^{2}(s) - \fZ_{n}^{2}(s)| \\
&= \sup_{s \in [-S,S]}\bigg|b_{n}\int_{x_{0}}^{x_{0}+a_{n}s}\frac{\delta_{n}^{\beta}}{\beta!}\Phi_{0}^{(\beta)}(\delta_{n}\xi_{n}(x))(x-x_{0})^{\beta}p_{X}(x)dx 
						- \frac{\Phi_{0}^{(\beta)}(0)p_{X}(x_{0})}{(\beta+1)!}s^{\beta+1}\bigg| \\
&= \frac{1}{\beta!}\sup_{s \in [-S,S]}\bigg|b_{n}\delta_{n}^{\beta}a_{n}^{\beta+1}
						\int_{0}^{s}\Phi_{0}^{(\beta)}(\delta_{n}\xi_{n}(x_{0}+a_{n}x))x^{\beta}p_{X}(x_{0}+a_{n}x)dx - \Phi_{0}^{(\beta)}(0)p_{X}(x_{0})\int_{0}^{s}x^{\beta}dx\bigg| \\
&= \frac{1}{\beta!}\sup_{s \in [-S,S]}\bigg|
						\int_{0}^{s}\big(\Phi_{0}^{(\beta)}(\delta_{n}\xi_{n}(x_{0}+a_{n}x))p_{X}(x_{0}+a_{n}x) 
												- \Phi_{0}^{(\beta)}(0)p_{X}(x_{0})\big)x^{\beta}dx\bigg|\\
&\leq \frac{2S^{\beta+1}}{\beta!}\Big\|\Phi_{0}^{(\beta)}(\delta_{n}\xi_{n}(x_{0}+a_{n}\boldcdot))p_{X}(x_{0}+a_{n}\boldcdot) 
												- \Phi_{0}^{(\beta)}(0)p_{X}(x_{0})\Big\|_{[-S,S]}
\end{align*}
which converges to zero as $n \longrightarrow \infty$ by continuity of $\Phi_{0}^{(\beta)}$ and $p_{X}$, the convergence $a_n \longrightarrow 0$ and $\delta_n\longrightarrow 0$, as well as $\xi_{n}(x_{0}+a_{n}\boldcdot)\in[-S,S]$.

\medskip
\noindent
\textsc{Claim II: } $\| \fZ_{n}^{3}-\fZ^3\|_{[-S,S]}\longrightarrow_{\bP}0$.

\smallskip
\noindent
\textit{Proof of Claim II. }
Define 
\[
g_{n,s} \colon [-S,S] \to \bR, 
\quad 
g_{n,s}(x) \defeq va_{n}^{-1}\big(\mathds{1}_{\{x \leq x_{0}+a_{n}s\}} - \mathds{1}_{\{x \leq x_{0}\}}\big)
\]
for every $s \in [-S,S]$, set $\cG_{n} \defeq \{g_{n,s} \mid s \in [-S,S]\}$ for every $n \in \bN$ and note that 
\[
\fZ_{n}^{3}(s) = \frac{1}{n}\sum_{i=1}^{n}g_{n,s}(X_{i}).
\]
From
\[
\bE[\fZ_{n}^{3}(s)] 
= va_{n}^{-1}\bE[\mathds{1}_{\{X \leq x_{0}+a_{n}s\}} - \mathds{1}_{\{X \leq x_{0}\}}] 
= va_{n}^{-1}\big(F_{X}(x_{0}+a_{n}s) - F_{X}(x_{0})\big),
\]
we deduce with $s_{n}^{*} \in [-S,S]$ denoting the maximizier of the function inside the subsequent supremum,
\begin{equation}\label{eq: conv3}
\begin{aligned}
\sup_{s \in [-S,S]}\big|\bE[\fZ_{n}^{3}(s)] - vp_{X}(x_{0})s\big| 
&= \Big|va_{n}^{-1}(F_{X}(x_{0}+a_{n}s_{n}^{*}) - F_{X}(x_{0})) - vp_{X}(x_{0})s_{n}^{*}\Big| \\ 
&= |vs_{n}^{*}|\Big|\frac{1}{a_ns_{n}^{*}}(F_{X}(x_{0}+a_{n}s_{n}^{*}) 
														- F_{X}(x_{0})) - p_{X}(x_{0})\Big| \mathds{1}_{\{s_n^*\not=0\}} \\ 
&\leq |vS|\Big|\frac{1}{a_ns_{n}^{*}}(F_{X}(x_{0}+a_{n}s_{n}^{*}) - F_{X}(x_{0})) - p_{X}(x_{0})\Big|\mathds{1}_{\{s_n^*\not=0\}} \\
&\quad\longrightarrow 0 \quad \text{as }\, n \longrightarrow \infty
\end{aligned}
\end{equation}
by the fundamental theorem of calculus. \par
Now we bound the $\nu$-bracketing number 
$N_{[]}(\nu,\cG_{n},L^{1}(P_{X}))$. To this end, let $\nu > 0$ be arbitrary, set $N(\nu) \defeq \frac{2S}{\nu}2vp_{X}(x_{0})$ 
and define for $i=1,\dots,\lfloor N(\nu) \rfloor$, 
\[
s_{0} \defeq -S, 
\qquad 
s_{i} \defeq s_{i-1} + \frac{\nu}{2vp_{X}(x_{0})}, 
\qquad 
s_{\lfloor N(\nu) \rfloor + 1} \defeq S.
\]
Then, $-S = s_{0} < s_{1} < \cdots < s_{\lfloor N(\nu) \rfloor + 1} = S$ and 
$s_{i}-s_{i-1} \leq \frac{\nu}{2vp_{X}(x_{0})}$ for $1 \leq i \leq \lfloor N(\nu) \rfloor + 1$ and for every 
$s \in [-S,S]$, there exists $i \in \{1,\dots,\lfloor N(\nu) \rfloor + 1\}$ such that $s_{i-1} \leq s \leq s_{i}$. 
Consequently, we have $g_{n,s_{i-1}}(x) \leq g_{n,s}(x) \leq g_{n,s_{i}}(x)$ for every $x \in \bR$ and 
\begin{align*}
&\int_{\bR}|g_{n,s_{i}}(x) - g_{n,s_{i-1}}(x)|dP_{X}(x) \\
&= \int_{\bR}g_{n,s_{i}}(x) - vs_{i}p_{X}(x_{0}) + vs_{i-1}p_{X}(x_{0}) - g_{n,s_{i-1}}(x)dP_{X}(x) 
			+ v(s_{i}-s_{i-1})p_{X}(x_{0}) \\
&\leq 2\sup_{s \in [-S,S]}|\bE[\fZ_{n}^{3}(s)] - vp_{X}(x_{0})s| 
			+ \frac{\nu}{2}.
\end{align*}
By \eqref{eq: conv3}, $\| \bE[\fZ_{n}^{3}] -\fZ^3\|_{[-S,S]}<\nu/4$
for $n$ large enough, whence $[g_{n,s_{i-1}},g_{n,s_{i}}]_{i=1,\dots,\lfloor N(\nu) \rfloor + 1}$ 
define $\nu$-brackets for $\cG_{n}$ with respect to $L^{1}(P_{X})$ and 
\[
N_{[]}(\nu,\cG_{n},L^{1}(P_{X})) 
\leq \lfloor N(\nu) \rfloor + 1
\leq 1 + \frac{2S}{\nu}2vp_{X}(x_{0})
\]
for $n$ sufficiently large. Moreover, 
\begin{align}\label{eq: 7.3}
\Var(g_{n,s}(X)) 
\leq v^{2}a_{n}^{-2}\big(\bE[(\mathds{1}_{\{X \leq x_{0}+a_{n}s\}} - \mathds{1}_{\{X \leq x_{0}\}})^{2}]\big) 
\leq 2v^{2}a_{n}^{-2}
\end{align}
for every $s \in [-S,S]$. 
By definition of $\cG_{n}$, 
\[
\big\|\fZ_{n}^{3} - \bE[\fZ_{n}^{3}]\big\|_{[-S,S]}
= \sup_{g \in \cG_{n}}\bigg|\frac{1}{n}\sum_{i=1}^{n}g(X_{i}) - \bE[g(X)]\bigg|. 
\]
Therefore, for every $\varepsilon > 0$, we obtain with the $(\varepsilon/2)$-brackets $g_{n}^{1},\dots,g_{n}^{N(\varepsilon/2)} \in \cG_{n}$ by the union bound, Chebychev's inequality and \eqref{eq: 7.3},
\begin{align*}
\bP\Big(\big\|\fZ_{n}^{3} - \bE[\fZ_{n}^{3}]\big\|_{[-S,S]}\geq \varepsilon\Big) 
&\leq \bP\bigg(\max_{j=1,\dots, N(\varepsilon/2)}\Big|\frac{1}{n}\sum_{i=1}^{n}g_{n}^{j}(X_{i}) 
											- \bE[g_{n}^{j}(X)]\Big| \geq \frac{\varepsilon}{2}\bigg)\\
&\leq \big(N(\varepsilon/2)+1\big)\frac{8}{\varepsilon^{2}}v^{2}
							\Big(\frac{\delta_{n}^{4\beta}}{n^{2\beta-1}}\Big)^{1/(2\beta+1)} 
\longrightarrow 0 
\end{align*}
as $n \longrightarrow \infty$. Together with \eqref{eq: conv3}, this reveals
\begin{align*}
\| \fZ_{n}^{3}-\fZ^3\|_{[-S,S]}&\leq \big\|\fZ_{n}^{3} - \bE[\fZ_{n}^{3}]\big\|_{[-S,S]}
		+ \big\| \bE[\fZ_{n}^{3}] -\fZ^3\big\|_{[-S,S]}
\longrightarrow_{\bP} 0 \quad \text{as } n \longrightarrow \infty
\end{align*}

\medskip
\noindent
\textsc{Claim III: } $(\fZ_{n}^{1}(s))_{s \in [-S,S]}\longrightarrow_{\cL}(\fZ^1(s))_{s \in [-S,S]}$ in $\ell^{\infty}([-S,S])$.

\smallskip
\noindent
\textit{Proof of Claim III. } By Theorem~1.5.4 in \cite{VaartWellner2023}, it is sufficient to show that the sequence of 
stochastic processes $\fZ_{n}^{1}$ is asymptotically tight and 
that for every finite subset $\{s_{1},\dots,s_{k}\} \subset [-S,S]$, the marginals 
$(\fZ_{n}^{1}(s_{1}),\dots,\fZ_{n}^{1}(s_{k}))$ converge weakly to $(\fZ^{1}(s_{1}),\dots,\fZ^{1}(s_{k}))$.

\smallskip
\noindent
\textit{Convergence of finite-dimensional distributions. }
Let $k \in \bN$ be arbitrary, let $\{s_{1},\dots,s_{k}\} \subset [-S,S]$ denote an arbitrary finite subset of 
$[-S,S]$ and note that 
\begin{align*}
\begin{pmatrix}
\fZ_{n}^{1}(s_{1}) \\
\vdots \\
\fZ_{n}^{1}(s_{k})
\end{pmatrix} 
= 
\sum_{i=1}^{n}\frac{b_{n}}{n}
\begin{pmatrix}
f_{n}(X_{i},Y_{i}^{n},x_{0} + a_{n}s_{1}) - \bE[f_{n}(X_{i},Y_{i}^{n},x_{0} + a_{n}s_{1})] \\
\vdots \\
f_{n}(X_{i},Y_{i}^{n},x_{0} + a_{n}s_{k}) - \bE[f_{n}(X_{i},Y_{i}^{n},x_{0} + a_{n}s_{k})]
\end{pmatrix}
.
\end{align*}
As a shorthand notation, let us introduce 
\begin{align*}
V_{i}^{n}
\defeq 
\frac{b_{n}}{n}
\begin{pmatrix}
f_{n}(X_{i},Y_{i}^{n},s_{1}) \\
\vdots \\
f_{n}(X_{i},Y_{i}^{n},s_{k})
\end{pmatrix}
\end{align*}
for $i=1,\dots,n$. 
Note that 
$\|V_{i}^{n}\|_2^{2} \leq k\frac{b_{n}^{2}}{n^{2}} = k(\frac{\delta_{n}}{n})^{2\beta/(2\beta+1)}$ by definition of $f_{n}$ 
and $b_{n}$, hence for every $\varepsilon > 0$, 
\begin{align*}
\sum_{i=1}^{n}\bE\big[\|V_{i}^{n}\|_2^{2}\mathds{1}_{\{\|V_{i}^{n}\|_2 > \varepsilon\}}\big] 
&\leq k\Big(\frac{\delta_{n}}{n}\Big)^{2\beta/(2\beta+1)}
		\sum_{i=1}^{n}\bE\big[\mathds{1}_{\{\|V_{i}^{n}\|_2^{2} > \varepsilon^{2}\}}\big] \\
&\leq k\Big(\frac{\delta_{n}}{n}\Big)^{2\beta/(2\beta+1)}
		\sum_{i=1}^{n}\bE\big[\mathds{1}_{\{k > (\frac{n}{\delta_{n}})^{2\beta/(2\beta+1)}\varepsilon^{2}\}}\big] \\
&= ka_{n}^{-1}\mathds{1}_{\{k > (\frac{n}{\delta_{n}})^{2\beta/(2\beta+1)}\varepsilon^{2}\}} 
\longrightarrow 0 \quad \text{as } n \longrightarrow \infty, 
\end{align*}
where we used $n/\delta_{n} \longrightarrow \infty$. For the sum of the covariance 
matrices of $V_{i}$, note that for $j,\ell \in \{1,\dots,k\}$, 
\begin{align*}
\bigg(\sum_{i=1}^{n}\Cov(V_{i}^{n})\bigg)_{j\ell} \hspace{-0.36em}
= \frac{b_{n}^{2}}{n}\big(\bE[f_{n}(X,Y^{n},x_{0}+a_{n}s_{j})f_{n}(X,Y^{n},x_{0}+a_{n}s_{\ell})] 
- E_{n}(x_{0} + a_{n}s_{j})E_{n}(x_{0} + a_{n}s_{\ell})\big).
\end{align*}
Recall from Claim I that $b_{n}E_{n}(x_{0} + a_{n}s) = \fZ_{n}^{2}(s) 
\longrightarrow \frac{1}{(\beta+1)!}\Phi_{0}^{(\beta)}(0)p_{X}(x_{0})s^{\beta+1}$ for any $s \in [-S,S]$ as $n \longrightarrow \infty$, whence
\[
\frac{b_{n}^{2}}{n}E_{n}(x_{0} + a_{n}s_{j})E_{n}(x_{0} + a_{n}s_{\ell}) 
= \frac{1}{n}\fZ_{n}^{2}(s_{j})\fZ_{n}^{2}(s_{\ell}) \longrightarrow 0 \quad \text{as } n \longrightarrow \infty. 
\]
Observe further that 
\begin{align*}
&f_{n}(X,Y^{n},x_{0}+a_{n}s_{j})f_{n}(X,Y^{n},x_{0}+a_{n}s_{\ell}) \\
&\qquad= (Y^{n}-\Phi_{n}(x_{0}))^{2}
					\big(\mathds{1}_{\{X \leq x_{0} + a_{n}s_{j}\}} - \mathds{1}_{\{X \leq x_{0}\}}\big)
						\big(\mathds{1}_{\{X \leq x_{0} + a_{n}s_{\ell}\}} - \mathds{1}_{\{X \leq x_{0}\}}\big) \\
&\qquad= (Y^{n}-\Phi_{n}(x_{0}))^{2}
					\big(\mathds{1}_{\{x_{0} < X \leq x_{0} + a_{n}\min\{s_{j},s_{\ell}\}\}} 
							+ \mathds{1}_{\{x_{0} - a_{n}\min\{-s_{j},-s_{\ell}\} < X \leq x_{0}\}}\big)
\end{align*}
and consequently, 
\begin{align*}
&\bE[f_{n}(X,Y^{n},x_{0}+a_{n}s_{j})f_{n}(X,Y^{n},x_{0}+a_{n}s_{\ell})] \\
&\qquad= \bE\big[(Y^{n}-\Phi_{n}(x_{0}))^{2}
					\big(\mathds{1}_{\{x_{0} < X \leq x_{0} + a_{n}\min\{s_{j},s_{\ell}\}\}} 
							+ \mathds{1}_{\{x_{0} - a_{n}\min\{-s_{j},-s_{\ell}\} < X \leq x_{0}\}}\big)\big] \\
&\qquad= \mathds{1}_{\{s_{j},s_{\ell} > 0\}}
				\bE\big[(Y^{n}-\Phi_{n}(x_{0}))^{2}
							\mathds{1}_{\{x_{0} < X \leq x_{0} + a_{n}\min\{s_{j},s_{\ell}\}\}}\big] \\
&\qquad\qquad\qquad+ \mathds{1}_{\{s_{j},s_{\ell} < 0\}}
						\bE\big[(Y^{n}-\Phi_{n}(x_{0}))^{2}
									\mathds{1}_{\{x_{0} - a_{n}\min\{|s_{j}|,|s_{\ell}|\} < X \leq x_{0}\}}\big].
\end{align*}
From now on, we will only consider the case $s_{j},s_{\ell} > 0$ as the case $s_{j},s_{\ell} < 0$ follows analogously. 
Note first that by the tower property of conditional expectation,
\begin{align*}
&\bE\big[(Y^{n}-\Phi_{n}(x_{0}))^{2}\mathds{1}_{\{x_{0} < X \leq x_{0} + a_{n}\min\{s_{j},s_{\ell}\}\}}\big] \\
&\qquad= \bE\big[\big((1-\Phi_{n}(x_{0}))^{2}\Phi_{n}(X) + (\Phi_{n}(x_{0}))^{2}(1-\Phi_{n}(X))\big)
									\mathds{1}_{\{x_{0} < X \leq x_{0} + a_{n}\min\{s_{j},s_{\ell}\}\}}\big] \\
&\qquad= \bE\big[\big(\Phi_{n}(X)(1 - 2\Phi_{n}(x_{0})) + \Phi_{n}(x_{0})^{2}\big)
									\mathds{1}_{\{x_{0} < X \leq x_{0} + a_{n}\min\{s_{j},s_{\ell}\}\}}\big] \\
&\qquad= \bE\big[\Phi_{n}(X)(1 - 2\Phi_{n}(x_{0}))
								\mathds{1}_{\{x_{0} < X \leq x_{0} + a_{n}\min\{s_{j},s_{\ell}\}\}}\big] \\
&\qquad\qquad\qquad\qquad + \Phi_{n}(x_{0})^{2}(F_{X}(x_{0} + a_{n}\min\{s_{j},s_{\ell}\}) - F_{X}(x_{0})).
\end{align*}
Now, because of $\frac{b_{n}^{2}}{n} = a_{n}^{-1}$, we obtain 
\[
\frac{b_{n}^{2}}{n}\Phi_{n}(x_{0})^{2}\big(F_{X}(x_{0} + a_{n}\min\{s_{j},s_{\ell}\}) - F_{X}(x_{0})\big)
\longrightarrow \Phi_{0}(0)^{2}\min\{s_{j},s_{\ell}\}p_{X}(x_{0}) 
\]
as $n \longrightarrow \infty$. Defining $J_{n}(t) \defeq \int_{x_{0}}^{t}\Phi_{n}(x)p_{X}(x)dx$, 
we further have 
\begin{align*}
&\frac{b_{n}^{2}}{n}
	\bE\big[\Phi_{n}(X)(1 - 2\Phi_{n}(x_{0}))\mathds{1}_{\{x_{0} < X \leq x_{0} + a_{n}\min\{s_{j},s_{\ell}\}\}}\big] \\
&\qquad\qquad\qquad= a_{n}^{-1}(1 - 2\Phi_{n}(x_{0}))J_{n}(x_{0} + a_{n}\min\{s_{j},s_{\ell}\})
\end{align*}
and by a Taylor expansion with Lagrange remainder of $J_{n}(t)$ around $x_{0}$, we obtain for a suitable intermediate point 
$\eta_{n}$ between $x_{0}$ and $x_{0} + a_{n}\min\{s_{j},s_{\ell}\}$ by the fundamental theorem of calculus
\begin{align*}
a_{n}^{-1}&(1 - 2\Phi_{n}(x_{0}))J_{n}(x_{0} + a_{n}\min\{s_{j},s_{\ell}\}) \\
=&\, a_{n}^{-1}(1 - 2\Phi_{n}(x_{0}))(J_{n}(x_{0}) + J_{n}'(\eta_{n})a_{n}\min\{s_{j},s_{\ell}\}) \\
=&\, (1 - 2\Phi_{n}(x_{0}))J_{n}'(\eta_{n})\min\{s_{j},s_{\ell}\} \\
=&\, (1 - 2\Phi_{n}(x_{0}))\Phi_{n}(\eta_{n})p_{X}(\eta_{n})\min\{s_{j},s_{\ell}\} \longrightarrow (1 - 2\Phi_{0}(0))\Phi_{0}(0)p_{X}(x_{0})\min\{s_{j},s_{\ell}\},
\end{align*}
where we used that $\eta_{n} \longrightarrow x_{0}$ as $n \longrightarrow \infty$. Combining the previous 
calculations, we see
\begin{align*}
\bigg(&\sum_{i=1}^{n}\Cov(V_{i}^{n})\bigg)_{j\ell} \\
&\hspace*{-2mm}\longrightarrow \big((1 - 2\Phi_{0}(0))\Phi_{0}(0)p_{X}(x_{0})\min\{|s_{j}|,|s_{\ell}|\} 
						+ \Phi_{0}(0)^{2}\min\{|s_{j}|,|s_{\ell}|\}p_{X}(x_{0})\big)\mathds{1}_{\{s_{j}s_{\ell} > 0\}} \\
&\phantom{\longrightarrow}\ = \Phi_{0}(0)(1 - \Phi_{0}(0))p_{X}(x_{0})\min\{|s_{j}|,|s_{\ell}|\}\mathds{1}_{\{s_{j}s_{\ell} > 0\}} \\
&\phantom{\longrightarrow}\ = \Cov\big(\sqrt{\Phi_{0}(0)(1-\Phi_{0}(0))p_{X}(x_{0})}Z(s_{j}),\sqrt{\Phi_{0}(0)(1-\Phi_{0}(0))p_{X}(x_{0})}Z(s_{\ell})\big) \\
&\phantom{\longrightarrow}\ = \Cov(\fZ^{1}(s_{j}),\fZ^{1}(s_{\ell})).
\end{align*}
Finally, the Lindeberg-Feller central limit theorem yields 
\[
(\fZ_{n}^{1}(s_{1}),\dots,\fZ_{n}^{1}(s_{k})) 
\longrightarrow_{\cL} (\fZ^{1}(s_{1}),\dots,\fZ^{1}(s_{k})) 
\quad \text{as } n \longrightarrow \infty. 
\]

\smallskip
\noindent
\textit{Asymptotic tightness.} By convergence of the finite dimensional distributions, it is sufficient by 
Theorem~1.5.7 of \cite{VaartWellner2023} to prove asymptotic uniform equicontinuity in probability. For this, let 
$\Delta>0$ be fixed but arbitrary and note that by Markov's inequality, 
\begin{align}\label{eq:markov2}
\bP\bigg(\sup_{|s-t|<\eta}|\fZ_{n}^{1}(s)-\fZ_{n}^{1}(t)| > \Delta\bigg) 
\leq \frac{1}{\Delta}\bE\bigg[\sup_{|s-t|<\eta}|\fZ_{n}^{1}(s)-\fZ_{n}^{1}(t)|\bigg].
\end{align}
Define
\[
f_{n,s,t} \colon [-S,S] \times \{0,1\} \to \bR, 
\quad 
f_{n,s,t}(x,y) \defeq (y-\Phi_{n}(x_{0}))\big(\mathds{1}_{\{x \leq x_{0}+a_{n}s\}}-\mathds{1}_{\{x \leq x_{0}+a_{n}t\}}\big)
\]
for $s,t \in [-S,S]$, and $\cF_{n,\eta} \defeq \{f_{n,s,t} \mid s,t \in [-S,S], |s-t| < \eta\}$. For any 
$\varepsilon_{n} > 0$ and $M_{n} > 0$ satisfying $\bE[f^{2}] < \varepsilon_{n}^{2}$ and $\|f\|_{\infty} \leq M_{n}$ for every 
$f \in \cF_{n,\eta}$, Theorem~2.14.17' of \cite{VaartWellner2023} reveals for a universal constant $C > 0$, 
\begin{align}\label{eq: 7.4}
\bE\bigg[\sup_{|s-t|<\eta}|\fZ_{n}^{1}(s)-&\fZ_{n}^{1}(t)|\bigg]
= b_{n}n^{-1/2}\bE\bigg[\sup_{f_{n} \in \cF_{n,\eta}}\bigg|\frac{1}{\sqrt{n}}\sum_{i=1}^{n}f_{n}(X_{i},Y_{i}^{n}) 
																	- \bE[f_{n}(X_{i},Y_{i}^{n})] \bigg|\bigg] \nonumber\\
&\leq Cb_{n}n^{-1/2}J_{[]}\big(\varepsilon_{n},\cF_{n,\eta},L^{2}(P_{\Phi_n})\big)
		\bigg(1 + \frac{J_{[]}\big(\varepsilon_{n},\cF_{n,\eta},L^{2}(P_{\Phi_n})\big)}
																	{\varepsilon_{n}^{2}n^{1/2}}M_{n}\bigg) 
\end{align}
with 
$J_{[]}\big(\varepsilon_{n},\cF_{n,\eta},L^{2}(P_{\Phi_{n}})\big) 
\defeq \int_{0}^{\varepsilon_{n}}\sqrt{1+\log(N_{[]}(\nu,\cF_{n,\eta},L^{2}(P_{\Phi_{n}}))}d\nu$.
It remains to specify $M_n$, $\varepsilon_n$ and a bound for the entropy with bracketing. For arbitrary $f \in \cF_{n,\eta}$, there exist $s,t \in [-S,S]$, satisfying $|s-t| < \eta$, such that 
\begin{align*}
\bE[f(X,Y^{n})^{2}] 
\leq \bE[(\mathds{1}_{\{X \leq x_{0}+a_{n}s\}}-\mathds{1}_{\{X \leq x_{0}+a_{n}t\}})^{2}] 
= |F_{X}(x_{0} + a_{n}s) - F_{X}(x_{0} + a_{n}t)| 
\leq a_{n}\eta\|p_{X}\|_{\infty} 
\end{align*}
and 
$
\|f\|_{\infty} 
= \|f_{n,s,t}\|_{\infty} 
\leq 1$.
Thus, with $\varepsilon_{n} = \sqrt{a_{n}\eta\|p_{X}\|_{\infty}}$ and $M_{n} = 1$, \eqref{eq: 7.4} is equal to
\[
b_{n}n^{-1/2}J_{[]}\big(\sqrt{a_{n}\eta\|p_{X}\|_{\infty}},\cF_{n,\eta},L^{2}(P_{\Phi_n})\big) \bigg(1 + \frac{J_{[]}\big(\sqrt{a_{n}\eta\|p_{X}\|_{\infty}},
								\cF_{n,\eta},L^{2}(P_{\Phi_n})\big)}{a_{n}\eta\|p_{X}\|_{\infty}\sqrt{n}}\bigg).
\]
By Lemma~\ref{lem:bracketing} (i), it follows that for some constant $K > 0$ independent from the variable parameters in the 
following expressions and which may change from line to line, 
\[
N_{[]}\big(\nu,\cF_{n,\eta},L^{2}(P_{\Phi_n})\big) 
\leq N_{[]}\big(\nu,\cF_{n,2S},L^{2}(P_{\Phi_n})\big) 
\leq a_{n}^{2}\frac{K}{\nu^{4}}.
\]
Thus, by utilizing that $\frac{d}{d x}x(\log(K/x^{4}) + 4) = \log(K/x^{4})$, 
\begin{align*}
J_{[]}\big(\sqrt{a_{n}\eta\|p_{X}\|_{\infty}},\cF_{n,\eta},L^{2}(P_{\Phi_{n}})\big) 
\leq K\int_{0}^{\sqrt{a_{n}\eta\|p_{X}\|_{\infty}}}\log\Big(a_{n}^{2}\frac{K}{\nu^{4}}\Big)d\nu 
\leq K\sqrt{a_{n}\eta}\log\Big(\frac{K}{\eta^{2}}\Big)
\end{align*}
and Claim III follows from \eqref{eq:markov2} together with 
\[
\limsup_{n\to\infty}\bE\bigg[\sup_{|s-t|<\eta}|\fZ_{n}^{1}(s)-\fZ_{n}^{1}(t)|\bigg] 
\leq \limsup_{n\to\infty}Kb_{n}n^{-1/2}\sqrt{a_{n}\eta}\log\Big(\frac{K}{\eta^{2}}\Big) 
= K\sqrt{\eta}\log\Big(\frac{K}{\eta^{2}}\Big) .
\]

\medskip
\noindent
The assertion now follows from the fact that $\fZ_{n} = \fZ_{n}^{1}+ \fZ_{n}^{2} - \fZ_{n}^{3}$ as well as that 
$\fZ_{n}^{2}$ and $\fZ_{n}^{3}$ converge to nonrandom functions.
\end{proof}

\begin{lemma} \label{lem:pointwise argmin bounded}
Under the same assumptions as in Lemma~\ref{lem:pointwise inside argmin}, the minimizers $\hat{s}_{n}$ of 
$\fZ_{n}$ form a tight sequence.
\end{lemma}

\begin{proof}
Let $s_{n} \defeq \argmin_{s \in \bR}\fZ_{n}^{2}(s)$ and note that $s_{n} = 0$, which follows from 
\begin{align*}
E_{n}(x_{0} + a_{n}s) 
&= \bE[(Y^{n}-\Phi_{n}(x_{0}))(\mathds{1}_{\{X \leq x_{0} + a_{n}s\}} - \mathds{1}_{\{X \leq x_{0}\}})] \\
&= \bE[(\Phi_{n}(X)-\Phi_{n}(x_{0}))(\mathds{1}_{\{X \leq x_{0} + a_{n}s\}} - \mathds{1}_{\{X \leq x_{0}\}})] \\
&= \bE[|\Phi_{n}(X)-\Phi_{n}(x_{0})||\mathds{1}_{\{X \leq x_{0} + a_{n}s\}} - \mathds{1}_{\{X \leq x_{0}\}}|], 
\end{align*}
where we used monotonicity of $\Phi_{n}$. Next, fix some neighborhood $U(0)$ of zero and note that, as long as $n$ is large enough, 
\[
\inf_{x\in U(0)}\Phi_{0}^{(\beta)}(x)>0
\] 
and $p_{X}>0$ on its 
support $[-T,T]$. Assuming $s \geq 0$ for the moment, a Taylor expansion of $\Phi_{n}$ around $x_{0}$ reveals for some $\xi_{n}$ between $X$ and $x_{0}$ and some 
constant $C > 0$, which may change from line to line, that 
\begin{align*}
\fZ_{n}^{2}(s) 
&= b_{n}\bE[(\Phi_{n}(X)-\Phi_{n}(x_{0}))(\mathds{1}_{\{X \leq x_{0} + a_{n}s\}} - \mathds{1}_{\{X \leq x_{0}\}})] \\
&= b_{n}\bE[\Phi_{n}^{(\beta)}(\xi_{n})(X-x_{0})^{\beta}
								(\mathds{1}_{\{X \leq x_{0} + a_{n}s\}} - \mathds{1}_{\{X \leq x_{0}\}})] \\
&\geq Cb_{n}\delta_{n}^{\beta}\bE[(X-x_{0})^{\beta}(\mathds{1}_{\{X \leq x_{0} + a_{n}s\}} - \mathds{1}_{\{X \leq x_{0}\}})] \\
&= Cb_{n}\delta_{n}^{\beta}\int_{0}^{a_{n}s}x^{\beta}p_{X}(x_{0}+x)dx \\
&\geq Cb_{n}\delta_{n}^{\beta}a_{n}^{\beta+1}s^{\beta+1} 
= Cs^{\beta+1}.
\end{align*}
By similar arguments, the same result holds for $s \leq 0$. To show uniform tightness of $\hat{s}_{n}$, we use a slicing 
argument similar to the proof of Theorem~3.2.5 in \cite{VaartWellner2023}. For $j\in\bN$, define the slices 
\[
S_{j,n} \defeq \big\{s \in \bR \mid 2^{j-1} < |s|^{\beta+1} \leq 2^{j} \big\}.
\]
Using that $\fZ_{n}(s_{n}) - \fZ_{n}(\hat{s}_{n}) \geq 0$ by the property of $\hat{s}_{n}$, we obtain by $\sigma$-subadditivity, as well as 
$\fZ_{n}(s_{n}) = 0$,
\begin{align*}
\bP\big(|\hat{s}_{n}| > 2^{K}\big) 
&\leq \sum_{j=K+1}^{\infty}\bP\bigg(\sup_{s \in S_{j,n}}(\fZ_{n}(s_{n}) - \fZ_{n}(s)) \geq 0\bigg) \\
&= \sum_{j=K+1}^{\infty}\bP\bigg(\sup_{s \in S_{j,n}}(\fZ_{n}^{2}(s) - \fZ_{n}(s) - \fZ_{n}^{2}(s)) \geq 0\bigg) \\
&\leq \sum_{j=K+1}^{\infty}\bP\bigg(\sup_{s \in S_{j,n}}(\fZ_{n}^{2}(s) - \fZ_{n}(s)) 
													\geq \inf_{s \in S_{j,n}}\fZ_{n}^{2}(s)\bigg) \\
&\leq \sum_{j=K+1}^{\infty}\bP\Big(\|\fZ_{n}^{2} - \fZ_{n}\|_{S_{j,n}} 
													\geq C2^{j-1}\Big) \\
&\leq \frac{4}{C}\sum_{j=K+1}^{\infty}\frac{1}{2^{j}}
													\bE[\|\fZ_{n} - \fZ_{n}^{2}\|_{S_{j,n}}],
\end{align*}
where we used Markov's inequality in the last step. 
Let us now define 
\begin{align*}
\fZ_{n,j}^{1} &\defeq \bigg\|\frac{b_{n}}{n}\sum_{i=1}^{n}f_{n}(X_{i},Y_{i}^{n},x_{0}+a_{n}\boldcdot) 
													- \bE[f_{n}(X_{i},Y_{i}^{n},x_{0}+a_{n}\boldcdot)]\bigg\|_{S_{j,n}}, \\
\fZ_{n,j}^{3} &\defeq \bigg\|v\frac{b_{n}}{nr_{n}}\sum_{i=1}^{n}g(X_{i},x_{0}+a_{n}\boldcdot)\bigg\|_{S_{j,n}} 
\end{align*}
and note that 
\[
\bE[\|\fZ_{n} - \fZ_{n}^{2}\|_{S_{j,n}}] 
\leq \bE[\fZ_{n,j}^{1}] + \bE[\fZ_{n,j}^{3}].
\]
As an immediate consequence, we have 
\[
\bE[\fZ_{n,j}^{3}] 
\leq v (n\delta_n^{2\beta})^{\frac{1}{2\beta+1}}\bE\mathds{1}_{(x_0,x_0+a_nS]}\leq v\| p_X\|_{[-T,T]}.
\]
Define
\[
f_{n,s} \colon [-S,S] \times \{0,1\} \to \bR, 
\quad 
f_{n,s}(x,y) \defeq (y-\Phi_{n}(x_{0}))\big(\mathds{1}_{\{x \leq x_{0}+a_{n}s\}}-\mathds{1}_{\{x \leq x_{0}\}}\big)
\]
for $s \in \bR$ and set $\cF_{n,j}^{\beta} \defeq \{f_{n,s} \mid s \in \bR, 2^{j} < |s|^{\beta+1} \leq 2^{j+1}\}$. 
For any $\varepsilon_{n} > 0$ and $M_{n} > 0$ satisfying $\bE[f^{2}] < \varepsilon_{n}^{2}$ and $\|f\|_{\infty} \leq M_{n}$ 
for every $f \in \cF_{n,j}^{\beta}$, Theorem~2.14.17' of \cite{VaartWellner2023} reveals for a universal constant $C > 0$, 
\begin{align}\label{eq: 7.5}
\bE[\fZ_{n,j}^{1}]\nonumber 
&= b_{n}n^{-1/2}\bE\bigg[\sup_{f_{n} \in \cF_{n,j}^{\beta}}\bigg|\frac{1}{\sqrt{n}}\sum_{i=1}^{n}f_{n}(X_{i},Y_{i}^{n}) 
																	- \bE[f_{n}(X_{i},Y_{i}^{n})] \bigg|\bigg] \nonumber\\
&\leq Cb_{n}n^{-1/2}J_{[]}\big(\varepsilon_{n},\cF_{n,j}^{\beta},L^{2}(P_{\Phi_n})\big)
		\bigg(1 + \frac{J_{[]}\big(\varepsilon_{n},\cF_{n,j}^{\beta},L^{2}(P_{\Phi_n})\big)}
																	{\varepsilon_{n}^{2}n^{1/2}}M_{n}\bigg) 
\end{align}
with 
$J_{[]}\big(\varepsilon_{n},\cF_{n,j}^{\beta},L^{2}(P_{\Phi_{n}})\big) 
\defeq \int_{0}^{\varepsilon_{n}}\sqrt{1+\log(N_{[]}(\nu,\cF_{n,j}^{\beta},L^{2}(P_{\Phi_{n}}))}d\nu$. It 
remains to specify $M_n$, $\varepsilon_n$ and a bound for the entropy with bracketing. For arbitrary 
$f \in \cF_{n,j}^{\beta}$, there exist $s \in [-S,S]$, satisfying $|s|^{\beta+1} \leq 2^{j+1}$, such that 
\begin{align*}
\bE[f(X,Y^{n})^{2}] 
\leq \bE[(\mathds{1}_{\{X \leq x_{0}+a_{n}s\}}-\mathds{1}_{\{X \leq x_{0}\}})^{2}] 
= |F_{X}(x_{0} + a_{n}s) - F_{X}(x_{0})| 
\leq a_{n}2^{(j+1)/(\beta+1)}\|p_{X}\|_{\infty} 
\end{align*}
and $\|f\|_{\infty} = \|f_{n,s}\|_{\infty} \leq 1$. Thus, with 
$\varepsilon_{n} = \sqrt{a_{n}\|p_{X}\|_{\infty}}2^{\frac{j+1}{2(\beta+1)}}$ and 
$M_{n} = 1$, \eqref{eq: 7.5} equals
\begin{align*}
b_{n}n^{-1/2}J_{[]}\big(\sqrt{a_{n}\|p_{X}\|_{\infty}}2^{\frac{j+1}{2(\beta+1)}},&\cF_{n,j}^{\beta},L^{2}(P_{\Phi_n})\big) 
	\bigg(1 + \frac{J_{[]}\big(\sqrt{a_{n}\|p_{X}\|_{\infty}}2^{\frac{j+1}{2(\beta+1)}},
		\cF_{n,j}^{\beta},L^{2}(P_{\Phi_n})\big)}{\sqrt{a_{n}\|p_{X}\|_{\infty}}2^{\frac{j+1}{2(\beta+1)}}\sqrt{n}}\bigg).
\end{align*}
By Lemma~\ref{lem:bracketing} (ii), it follows that for some constant $L > 0$ independent from the variable parameters in 
the following expressions and which may changes from line to line, 
\[
N_{[]}\big(\nu,\cF_{n,j}^{\beta},L^{2}(P_{\Phi_n})\big) 
\leq 2^{\frac{j+1}{\beta+1}}a_{n}\frac{L}{\nu^{2}}, 
\]
Thus, by utilizing that $\frac{d}{d x}x(\log(K/x^{2}) + 2) = \log(K/x^{2})$, 
\begin{align*}
J_{[]}\big(\sqrt{a_{n}\|p_{X}\|_{\infty}}2^{\frac{j+1}{2(\beta+1)}},\cF_{n,j}^{\beta},L^{2}(P_{\Phi_{n}})\big) 
\leq L\int_{0}^{\sqrt{a_{n}\|p_{X}\|_{\infty}}2^{\frac{j+1}{2(\beta+1)}}}\log\Big(2^{\frac{j+1}{\beta+1}}a_{n}\frac{L}{\nu^{2}}\Big)d\nu 
\leq L\sqrt{a_{n}}2^{\frac{j+1}{2(\beta+1)}}
\end{align*}
and we have 
\[
\bE[\fZ_{n,j}^{1}]
\leq Lb_{n}n^{-1/2}\sqrt{a_{n}}2^{\frac{j+1}{2(\beta+1)}}
= L2^{\frac{j+1}{2(\beta+1)}}.
\]
Summarizing, we have shown 
\[
\bP\big(|\hat{s}_{n}| > 2^{K}\big) 
\leq v\| p_X\|_{[-T,T]}\frac{4L}{C}\sum_{j=K+1}^{\infty}\frac{2^{\frac{j}{2(\beta+1)}}}{2^{j}} 
= v\| p_X\|_{[-T,T]}\frac{4L}{C}\sum_{j=K+1}^{\infty}\frac{1}{2^{j\frac{(2\beta+1)}{2(\beta+1)}}} 
\longrightarrow 0
\]
as $K\longrightarrow\infty$ and the assertion follows.
\end{proof}

\begin{proposition} \label{prop:pointwise argmin}
Under the same assumptions as in Lemma~\ref{lem:pointwise inside argmin}, the sequence of minimizers $\hat{s}_{n}$ of 
$\fZ_{n}(s)$ converges weakly to the minimizer $\hat{s}$ of $\fZ(s)$ for $n \longrightarrow \infty$.
\end{proposition}

\begin{proof}
From Lemma~\ref{lem:pointwise inside argmin}, we have for every compact set $\cK \subset \bR$ that $(\fZ_{n}(s))_{s\in \cK}$ 
converges weakly to $(\fZ(s))_{s\in \cK}$ in $\ell^{\infty}(\cK)$. Moreover, the sample paths 
$s \mapsto \fZ(s)$ are continuous and $\hat{s}$ is unique a.s.~and tight (cf.~\cite{Wright1981asymptotic}). By 
Lemma~\ref{lem:pointwise argmin bounded}, $\hat{s}_{n}$ is uniformly tight and consequently, by Theorem~3.2.2 of 
\cite{VaartWellner2023}, 
$\hat{s}_{n} \longrightarrow_{\cL} \hat{s}$ as $n \longrightarrow \infty$.
\end{proof}

For the results related to the proof of Theorem~\ref{thm:pointwise rate} (ii) and (iii), let us recall the definitions 
\[
h_{n} \colon [-T,T] \times \{0,1\} \times [-T,T] \to \bR, 
\quad h_{n}(x,y,t) \defeq (y-\Phi_{n}(x_{0}))\mathds{1}_{\{x \leq t\}} 
\]
and $H_{n}(t) \defeq \bE[h_{n}(X,Y^{n},t)]$ for every $t \in [-T,T]$. Further, let $(W(s))_{s \in [0,1]}$ denote a standard 
Brownian motion on $[0,1]$, define the stochastic processes 
\begin{align*}
\fW_{n}^{1}(s) 
	&\defeq \frac{1}{\sqrt{n}}\sum_{i=1}^{n}\big(h_{n}(X_{i},Y_{i}^{n},F_{X}^{-1}(s)) - H_{n}(F_{X}^{-1}(s))\big) \\
\fW_{n}^{2}(s) 
	&\defeq \sqrt{n}H_{n}(F_{X}^{-1}(s)) \\
\fW_{n}^{3}(s) 
	&\defeq v\frac{1}{n}\sum_{i=1}^{n}\mathds{1}_{\{X_{i} \leq F_{X}^{-1}(s)\}} \\
\fW^{1}(s) 
	&\defeq \sqrt{\Phi_{0}(0)(1-\Phi_{0}(0))}B(s) \\
\fW_{c}^{2}(s) 
	&\defeq \sqrt{c}\Phi_{0}^{(\beta)}(0)\bE\big[(X-x_{0})^{\beta}\mathds{1}_{\{X \leq F_{X}^{-1}(s)\}}\big] \\
\fW^{3}(s) 
	&\defeq vs
\end{align*}
and set for $s \in [0,1]$, 
\[
\fW_{n}(s) \defeq \fW_{n}^{1}(s) + \fW_{n}^{2}(s) - \fW_{n}^{3}(s), 
\quad 
\fW_{c}(s) \defeq \fW^{1}(s) + \fW_{c}^{2}(s) - \fW^{3}(s).
\]
Further, we redefine 
$\hat{s}_{n} \defeq \argminp_{s \in [0,1]} \fW_{n}(s)$ and $\hat{s}_{c} \defeq \argmin_{s \in [0,1]} \fW_{c}(s)$ from the previous 
proofs to now denote the minimizers of $\fW_{n}$ and $\fW_{c}$ respectively.

\begin{lemma} \label{lem:pointwise inside argmin - fast}
Let $\beta \in \bN$, $x_{0}$ an interior point of $ \cX$ and assume $\Phi_{0}$ to be $\beta$-times continuously 
differentiable in a neighborhood of $0$ with the $\beta$th derivative being the first non-vanishing derivative in $0$. 
Then, 
\[
\fW_{n}\longrightarrow_{\cL}\fW_{c} \ \text{ in }\, \ell^{\infty}([0,1]) 
\]
as 
$n \longrightarrow \infty$ and $n\delta_{n}^{2\beta} \longrightarrow c \in [0,\infty)$.
\end{lemma}

\begin{proof}
\textsc{Claim I: } $\| \fW_{n}^{2}-\fW_{c}^{2}\|_{[0,1]}\longrightarrow_{\bP}0$.

\smallskip
\noindent
\textit{Proof of Claim I. } By assumption on $\Phi_{n}$, 
a Taylor expansion of $\Phi_{n}$ around $x_{0}$ of order $\beta$ with Lagrange remainder yields the existence of some 
$\xi_{n}$ between $x$ and $x_{0}$ such that
\begin{equation}\label{eq:Taylor}
\Phi_{n}(x)-\Phi_{n}(x_{0}) 
= \Phi_{n}^{(\beta)}(\xi_{n})(x-x_{0})^{\beta} 
= \delta_{n}^{\beta}\Phi_{0}^{(\beta)}(\delta_{n}\xi_{n})(x-x_{0})^{\beta}.
\end{equation}
From $n\delta_{n}^{2\beta} \longrightarrow c$, we obtain 
\begin{align*}
\sup_{s \in [0,1]} |\fW_{n}^{2}(s) - \fW_{c}^{2}(s)| 
&= \sup_{s \in [0,1]} \big|\bE\big[n^{1/2}\delta_{n}^{\beta}\Phi_{0}^{(\beta)}(\delta_{n}\xi_{n})(X-x_{0})^{\beta}
													\mathds{1}_{\{X \leq F_{X}^{-1}(s)\}}\big] - \fW_{c}^{2}(s)\big| \\
&\leq \bE\big[\big|n^{1/2}\delta_{n}^{\beta}\Phi_{0}^{(\beta)}(\delta_{n}\xi_{n})
														-\sqrt{c}\Phi_{0}^{(\beta)}(0)\big||X-x_{0}|^{\beta}\big] 
\end{align*}
which converges to zero, as $n \longrightarrow \infty$. 

\medskip
\noindent
\textsc{Claim II.} $\| \fW_{n}^{3}-\fW^3\|_{[0,1]}\longrightarrow_{\bP}0$.

\smallskip
\noindent
\textit{Proof of Claim II. }The convergence $\fW_{n}^{3} \longrightarrow_{\bP} vF_{X} \circ F_{X}^{-1} = \fW^{3}$ in 
$\ell^{\infty}([0,1])$ as
$n \longrightarrow \infty$ is exactly the classical Glivenko-Cantelli result. 

\medskip
\noindent
\textsc{Claim III: } $\fW_{n}^{1}\longrightarrow_{\cL}\fW^1$ in $\ell^{\infty}([0,1])$.

\smallskip
\noindent
\textit{Proof of Claim III. } By Theorem~1.5.4 of \cite{VaartWellner2023}, it is sufficient to show that the sequence of 
stochastic processes $\fW_{n}^{1}$ is asymptotically tight and 
that for every finite subset $\{s_{1},\dots,s_{k}\} \subset [0,1]$, the marginals 
$(\fW_{n}^{1}(s_{1}),\dots,\fW_{n}^{1}(s_{k}))$ converge weakly to $(\fW^{1}(s_{1}),\dots,\fW^{1}(s_{k}))$.

\smallskip
\noindent
\textit{Convergence of finite-dimensional distributions. }For $k \in \bN$, let $\{s_{1},\dots,s_{k}\} \subset [0,1]$ denote some arbitrary finite subset of $[0,1]$ 
and note that 
\begin{align*}
\begin{pmatrix}
\fW_{n}^{1}(s_{1}) \\
\vdots \\
\fW_{n}^{1}(s_{k})
\end{pmatrix} 
= 
\sum_{i=1}^{n}\frac{1}{\sqrt{n}}
\begin{pmatrix}
h_{n}(X_{i},Y_{i}^{n},F_{X}^{-1}(s_{1})) - \bE[h_{n}(X_{i},Y_{i}^{n},F_{X}^{-1}(s_{1}))] \\
\vdots \\
h_{n}(X_{i},Y_{i}^{n},F_{X}^{-1}(s_{k})) - \bE[h_{n}(X_{i},Y_{i}^{n},F_{X}^{-1}(s_{k}))]
\end{pmatrix}
.
\end{align*}
As a shorthand notation, let us introduce 
\begin{align*}
V_{i}^{n}
\defeq 
\frac{1}{\sqrt{n}}
\begin{pmatrix}
h_{n}(X_{i},Y_{i}^{n},F_{X}^{-1}(s_{1})) \\
\vdots \\
h_{n}(X_{i},Y_{i}^{n},F_{X}^{-1}(s_{k}))
\end{pmatrix}
\end{align*}
for $i=1,\dots,n$. Then, $\|V_{i}^{n}\|_2^{2} \leq k/n$ by definition of $h_{n}$ and we have for every 
$\varepsilon > 0$, 
\[
\sum_{i=1}^{n}\bE[\|V_{i}^{n}\|_2^{2}\mathds{1}_{\{\|V_{i}^{n}\|_2 > \varepsilon\}}] 
\leq \frac{k}{n}\sum_{i=1}^{n}\bE[\mathds{1}_{\{\|V_{i}^{n}\|_2^{2} > \varepsilon^{2}\}}] 
\leq \frac{k}{n}\sum_{i=1}^{n}\bE[\mathds{1}_{\{k > n\varepsilon^{2}\}}] 
= k\mathds{1}_{\{k > n\varepsilon^{2}\}} 
\longrightarrow 0
\]
as $n \longrightarrow \infty$. As concerns the covariance matrices of $\sum_iV_{i}^n$, note that for 
$j,\ell \in \{1,\dots,k\}$, 
\begin{align*}
\bigg(\sum_{i=1}^{n}\Cov(V_{i}^{n})\bigg)_{j\ell} 
&= \bE\big[(Y^{n}-\Phi_{n}(x_{0}))^{2}\mathds{1}_{\{X \leq F_{X}^{-1}(s_{j})\}}
											\mathds{1}_{\{X \leq F_{X}^{-1}(s_{\ell})\}}\big] \\
&\quad- \bE\big[(Y^{n}-\Phi_{n}(x_{0}))\mathds{1}_{\{X \leq F_{X}^{-1}(s_{j})\}}\big]
						\bE\big[(Y^{n}-\Phi_{n}(x_{0}))\mathds{1}_{\{X \leq F_{X}^{-1}(s_{\ell})\}}\big] \\
&= \bE\big[(Y^{n}-\Phi_{n}(x_{0}))^{2}
							\mathds{1}_{\{X \leq \min{\{F_{X}^{-1}(s_{j}),F_{X}^{-1}(s_{\ell})\}}\}}\big] + \cO(\delta_n^{2\beta})
\end{align*}
by \eqref{eq:Taylor}. For the remaining summand, we observe 
\begin{align*}
&\bE\big[(Y^{n}-\Phi_{n}(x_{0}))^{2}\mathds{1}_{\{X \leq F_{X}^{-1}(\min{\{s_{j},s_{\ell}\}})\}}\big] \\
&\qquad= \bE\big[((1-\Phi_{n}(x_{0}))^{2}\Phi_{n}(X) + (\Phi_{n}(x_{0}))^{2}(1-\Phi_{n}(X)))
											\mathds{1}_{\{X \leq F_{X}^{-1}(\min{\{s_{j},s_{\ell}\}})\}}\big] \\
&\qquad= \bE\big[(\Phi_{n}(X) - 2\Phi_{n}(x_{0})\Phi_{n}(X) + \Phi_{n}(x_{0})^{2})
											\mathds{1}_{\{X \leq F_{X}^{-1}(\min{\{s_{j},s_{\ell}\}})\}}\big] \\
&\qquad\quad\longrightarrow \bE\big[(\Phi_{0}(0) - \Phi_{0}(0)^{2})
									\mathds{1}_{\{X \leq F_{X}^{-1}(\min{\{s_{j},s_{\ell}\}})\}}\big] 
\end{align*}
as $n \longrightarrow \infty$ by the theorem of dominated convergence. Thus, 
\begin{align*}
\bigg(\sum_{i=1}^{n}\Cov(V_{i}^{n})\bigg)_{j\ell} 
\longrightarrow \, &\bE\big[(\Phi_{0}(0) - \Phi_{0}(0)^{2})\mathds{1}_{\{X \leq F_{X}^{-1}(\min{\{s_{j},s_{\ell}\}})\}}\big] \\
&= \Phi_{0}(0)(1-\Phi_{0}(0))\min{\{s_{j},s_{\ell}\}} \\
&= \Cov(\fW^{1}(s_{j}),\fW^{1}(s_{\ell})),
\end{align*}
and by the Lindeberg-Feller central limit theorem, we conclude
\[
(\fW_{n}^{1}(s_{1}),\dots,\fW_{n}^{1}(s_{k})) 
\longrightarrow_{\cL} (\fW^{1}(s_{1}),\dots,\fW^{1}(s_{k})) 
\quad \text{as } n \longrightarrow \infty.
\]

\smallskip
\noindent
\textit{Asymptotic tightness.} By convergence of the finite dimensional distributions, it is sufficient by Theorem~1.5.7 of 
\cite{VaartWellner2023} to prove asymptotic uniform equicontinuity in probability. 
Define 
\[
h_{n,s,t} \colon [0,1] \times \{0,1\} \to \bR, 
\quad 
h_{n,s,t}(x,y) \defeq (y-\Phi_{n}(x_{0}))\big(\mathds{1}_{\{x \leq F_{X}^{-1}(s)\}}-\mathds{1}_{\{x \leq F_{X}^{-1}(t)\}}\big)
\]
for $s,t \in [0,1]$ and $\cH_{n,\eta} \defeq \{h_{n,s,t} \mid s,t \in [0,1], |s-t| < \eta\}$ for $\eta > 0$. For any 
$\varepsilon_{n} > 0$ and $M_{n} > 0$ satisfying $\bE[h^{2}] < \varepsilon_{n}^{2}$ and $\|h\|_{\infty} \leq M_{n}$ for every 
$h \in \cH_{n,\eta}$, Theorem~2.14.17' in \cite{VaartWellner2023} provides for a universal constant $C > 0$ the bound 
\begin{align*}
\bE\bigg[\sup_{|s-t|<\eta}|\fW_{n}^{1}(s)-\fW_{n}^{1}(t)|\bigg] 
&= \bE\bigg[\sup_{h_{n} \in \cH_{n,\eta}}\bigg|\frac{1}{\sqrt{n}}\sum_{i=1}^{n}h_{n}(X_{i},Y_{i}^{n}) 
																	- \bE[h_{n}(X_{i},Y_{i}^{n})] \bigg|\bigg] \\
&\leq CJ_{[]}\big(\varepsilon_{n},\cH_{n,\eta},L^{2}(P_{\Phi_n})\big)
		\bigg(1 + \frac{J_{[]}\big(\varepsilon_{n},\cH_{n,\eta},L^{2}(P_{\Phi_n})\big)}
																	{\varepsilon_{n}^{2}n^{1/2}}M_{n}\bigg), 
\end{align*}
with 
$J_{[]}\big(\varepsilon_{n},\cH_{n,\eta},L^{2}(P_{\Phi_{n}})\big) 
\defeq \int_{0}^{\varepsilon_{n}}\sqrt{1+\log(N_{[]}(\nu,\cH_{n,\eta},L^{2}(P_{\Phi_{n}}))}d\nu$.
For any $h \in \cH_{n,\eta}$, there exist $s,t \in [0,1]$ with $|s-t| < \eta$ such that 
\begin{align*}
\bE[h(X,Y^{n})^{2}] 
&\leq \bE[(\mathds{1}_{\{X \leq F_{X}^{-1}(s)\}}-\mathds{1}_{\{X \leq F_{X}^{-1}(t)\}})^{2}] \\
&= \bE[\mathds{1}_{\{F_{X}^{-1}(\min\{s,t\}) < X \leq F_{X}^{-1}(\max\{s,t\})\}}] \\
&= F_{X}(F_{X}^{-1}(\max\{s,t\})) - F_{X}(F_{X}^{-1}(\min\{s,t\})) 
= |s - t| 
< \eta 
\end{align*}
and $\|h\|_{\infty} \leq 1$.
Thus, by choosing $\varepsilon_{n} = \sqrt{\eta}$ and $M_{n} = 1$, we have 
\[
\bE\bigg[\sup_{|s-t|<\eta}|\fW_{n}^{1}(s)-\fW_{n}^{1}(t)|\bigg]
\leq CJ_{[]}\big(\eta^{1/2},\cH_{n,\eta},L^{2}(P_{\Phi_n})\big)
		\bigg(1 + \frac{J_{[]}\big(\sqrt{\eta},\cH_{n,\eta},L^{2}(P_{\Phi_n})\big)}{\eta \sqrt{n}}\bigg). 
\]
By Lemma~\ref{lem:bracketing} (iii), it follows that for some constant $K > 0$ which does not depend on 
the variables in the respective expressions and which may change from line to line, 
\[
N_{[]}(\nu,\cH_{n,\eta},L^{2}(P_{\Phi_n})) 
\leq N_{[]}(\nu,\cH_{n,1},L^{2}(P_{\Phi_n})) 
\leq \frac{K}{\nu^{4}}. 
\]
Thus, by utilizing that $\frac{d}{d x}x(\log(K/x^{4}) + 4) = \log(K/x^{4})$, 
\[
J_{[]}\big(\sqrt{\eta},\cH_{n,\eta},L^{2}(P_{\Phi_n})\big) 
\leq K\int_{0}^{\sqrt{\eta}}\log(K/\nu^{4})d\nu 
\leq K\sqrt{\eta}\log(1/\eta^{2})
\]
and Claim III follows by Markov's inequality and 
\[
\limsup_{n\to\infty}\bE\bigg[\sup_{|s-t|<\eta}|\fW_{n}^{1}(s)-\fW_{n}^{1}(t)|\bigg] 
\leq K\sqrt{\eta}\log(1/\eta^{2}).
\]

\medskip
\noindent
The 
assertion now follows from the fact that $\fW_{n} = \fW_{n}^{1}+ \fW_{n}^{2} - \fW_{n}^{3}$ as well as that 
$\fW_{n}^{2}$ and $\fW_{n}^{3}$ converge to nonrandom functions.
\end{proof}

\begin{proposition} \label{prop:pointwise argmin - fast}
Under the same conditions as in Lemma~\ref{lem:pointwise inside argmin - fast}, the sequence of minimizers $\hat{s}_{n}$ 
of $(\fW_{n}(s))_{s\in[0,1]}$ converges weakly to the minimizer $\hat{s}_{c}$ of $(\fW_{c}(s))_{s\in[0,1]}$ as 
$n \longrightarrow \infty$ as long as we have $n\delta_{n}^{2\beta} \longrightarrow c \in [0,\infty)$.
\end{proposition}

\begin{proof}
By Lemma~\ref{lem:pointwise inside argmin - fast}, $\fW_{n}\longrightarrow_{\cL}\fW$ in 
$\ell^{\infty}([0,1])$ as $n \longrightarrow \infty$. Further, the sample 
paths $s \mapsto \fW_{c}(s)$ are continuous and $\hat{s}_{c}$ is unique by Theorem~2 of \cite{Pimentel2014location}
and tight. As $\hat{s}_{n} \in [0,1]$ is uniformly tight, Theorem~3.2.2 of \cite{VaartWellner2023} reveals 
$\hat{s}_{n} \longrightarrow_{\cL} \hat{s}_{c}$ as $n \longrightarrow \infty$.
\end{proof}

\section{Remaining proofs of Section~\ref{sec:L1 limit}}
In this section, the proof of Theorem~\ref{thm:lower bound L1}, Proposition~\ref{prop: 4.2} and Theorem~\ref{thm:l1 rate} (i) are given.

\subsection{Proof of Theorem~\ref{thm:lower bound L1}} \label{proof:L1 lower bound}
We will calculate lower bounds for 
\[
\inf_{T_n^{\delta}}\sup_{\Phi\in\cF_{\delta}}\Big(\sqrt{n}\wedge \Big(\frac{n}{\delta}\Big)^{1/3}\Big)
				\bE_{\Phi}^{\otimes n}\bigg[\int_{-T}^{T}\big|T_n^{\delta}(x)-\Phi(x)\big|dx\bigg]
\]
separately for both $\delta \geq n^{-1/2}$ and $\delta < n^{-1/2}$, noting that 
$\max\{n^{-1/2},(\frac{n}{\delta})^{-1/3}\} = (\frac{n}{\delta})^{-1/3}$ if and only if $\delta \geq n^{-1/2}$. \\
$\bullet$ Let us start with the case $\delta < n^{-1/2}$. Let $C \leq \frac{1}{\sqrt{8T}}$, let 
$\eta_{n,\delta} \defeq 1/2 - \delta T - Cn^{-1/2}$ and define 
\begin{align*}
\Phi_{0,n} \colon \bR \to [0,1], 
&\quad \Phi_{0,n}|_{[-T,T]}(x) \defeq \delta(x+T) + \eta_{n,\delta} + 2Cn^{-1/2}, \\
\Phi_{1,n} \colon \bR \to [0,1], 
&\quad \Phi_{1,n}|_{[-T,T]}(x) \defeq \delta(x+T) + \eta_{n,\delta}, 
\end{align*}
where both functions are defined outside $[-T,T]$ by their values at the respective boundaries. 
Obviously, $\Phi_{0,n}, \Phi_{1,n} \in \cF_{\delta}$ and note that 
\[
\int_{-T}^{T}|\Phi_{0,n}(x)-\Phi_{1,n}(x)|dx 
= 4TCn^{-1/2}.
\]
Note further that for $n \geq 16(C+T)^{2}$ and all $x \in [-T,T]$, 
\begin{align*}
\Phi_{0,n}(x) &\geq \Phi_{1,n}(x) \geq \Phi_{1,n}(-T) = \eta_{n,\delta} \geq 1/2 - n^{-1/2}(C+T) \geq 1/4, \\
1-\Phi_{1,n}(x) &\geq 1-\Phi_{0,n}(x) \geq 1-\Phi_{0,n}(T) = 1 - 2T\delta - \eta_{n,\delta} - 2Cn^{-1/2} = \eta_{n,\delta} \geq 1/4.
\end{align*}
Writing $P_{0,n}^{\otimes n} \defeq P_{\Phi_{0,n}}^{\otimes n}$, 
$P_{1,n}^{\otimes n} \defeq P_{\Phi_{1,n}}^{\otimes n}$, we have for $\alpha = 4C^{2}$ 
\begin{align*}
&h^{2}\big(P_{0,n}^{\otimes n},P_{1,n}^{\otimes n}\big) 
\leq nh^{2}\big(P_{0,n},P_{1,n}\big) \\
&\quad= \frac{n}{2}\int_{-T}^{T}\Big(\sqrt{\Phi_{0,n}(x)} - \sqrt{\Phi_{1,n}(x)}\Big)^{2} 
								+ \Big(\sqrt{1-\Phi_{0,n}(x)} - \sqrt{1-\Phi_{1,n}(x)}\Big)^{2}dP_{X}(x) \\
&\quad= \frac{n}{2}\int_{-T}^{T}
		\bigg(\frac{\Phi_{0,n}(x)-\Phi_{1,n}(x)}{\sqrt{\Phi_{0,n}(x)} + \sqrt{\Phi_{1,n}(x)}}\bigg)^{2} 
		+ \bigg(\frac{\Phi_{0,n}(x)-\Phi_{1,n}(x)}{\sqrt{1-\Phi_{0,n}(x)} + \sqrt{1-\Phi_{1,n}(x)}}\bigg)^{2}dP_{X}(x) \\
&\quad\leq \frac{n}{8}\int_{-T}^{T}(\Phi_{0,n}(x)-\Phi_{1,n}(x))^{2}
		\bigg(\frac{1}{\Phi_{1,n}(x)} + \frac{1}{1-\Phi_{0,n}(x)}\bigg)dP_{X}(x) \\
&\quad\leq \frac{n}{8}\int_{-T}^{T}8(\Phi_{0,n}(x)-\Phi_{1,n}(x))^{2}dP_{X}(x) 
= 4C^{2} 
= \alpha 
< 2.
\end{align*}
From Chapter~2.2 in \cite{Tsybakov2008} and Theorem~2.2 (ii) of \cite{Tsybakov2008}, we have 
\[
\inf_{T_n^{\delta}}\sup_{\Phi\in\cF_{\delta}}C\Big(\sqrt{n}\wedge \Big(\frac{n}{\delta}\Big)^{1/3}\Big)
						\bE_{\Phi}^{\otimes n}\bigg[\int_{-T}^{T}\big|T_n^{\delta}(x)-\Phi(x)\big|dx\bigg]
> \frac{1}{2}\bigg(1-\sqrt{\frac{\alpha(1-\alpha)}{4}}\bigg)
> 0 
\]
for all $\delta \in [0,n^{-1/2})$ and all $n \geq 16(C+T)^{2}$. \\
$\bullet$ Let us now consider the case $\delta \geq n^{-1/2}$ and let 
$C \leq \big(\frac{1}{32\|p_{X}\|_{\infty}}\big)^{1/3}$. Following the idea of Chapter~2.6.1 in \cite{Tsybakov2008}, 
but with different hypotheses, we define $m \defeq \lfloor \frac{1}{4C}(n\delta^{2})^{1/3} \rfloor$, 
$h_{n} \defeq \frac{T}{m}$ and set 
\[
x_{k} \defeq -T + 2kh_{n}, 
\quad k = 0,\dots,m.
\]
Further, we define 
\begin{align*}
\varphi_{k,n}(x) \defeq 
\begin{cases}
0 \quad &x \in [x_{0},x_{k}] \\
\frac{\delta}{2}(x-x_{k}) &x \in [x_{k},x_{k}+h_{n}], \\
\frac{\delta}{2}h_{n} + \delta(x-(x_{k}+h_{n})) &x \in [x_{k}+h_{n},x_{k+1}] \\
\frac{\delta}{2}h_{n} + \delta(x_{k+1}-(x_{k}+h_{n})) &x \in [x_{k+1},x_{m}]
\end{cases}
\end{align*}
and 
\begin{align*}
\psi_{k,n}(x) \defeq 
\begin{cases}
0 \quad &x \in [x_{0},x_{k}] \\
\delta (x-x_{k}) \quad &x \in [x_{k},x_{k}+h_{n}], \\
\delta h_{n} + \frac{\delta}{2}(x-(x_{k}+h_{n})) &x \in [x_{k}+h_{n},x_{k+1}] \\
\delta h_{n} + \frac{\delta}{2}(x_{k+1}-(x_{k}+h_{n})) &x \in [x_{k+1},x_{m}]
\end{cases}
\end{align*}
for $k=0,\dots,m-1$. For $\gamma = (\gamma_{1},\dots,\gamma_{m}) \in \{0,1\}^{m}$, we define 
\[
\Phi_{\gamma,n} \colon \bR \to [0,1], 
\quad \Phi_{\gamma,n}|_{[-T,T]}(x) 
		\defeq \frac{1}{4} + \sum_{k=0}^{m-1}\gamma_{k+1}\varphi_{k,n}(x) + (1-\gamma_{k+1})\psi_{k,n}(x), 
\]
where the functions are defined outside $[-T,T]$ by their values at the respective boundaries 
and write $P_{\gamma,n}^{\otimes n} \defeq P_{\Phi_{\gamma,n}}^{\otimes n}$. Obviously, 
$\Phi_{0,n}, \Phi_{1,n} \in \cF_{\delta}$. A visualization of these hypotheses can be found in Figure~\ref{fig:hypotheses}. 
Note that for $n \geq 16T^{2}$, we have 
$n^{-1/2} \leq \frac{1}{4T}$ and so for $\delta \leq \frac{1}{4T}$ and all $x \in [-T,T]$, 
\begin{align*}
\Phi_{\gamma,n}(x) 
&\geq \Phi_{\gamma,n}(-T) = \Phi_{\gamma,n}(x_{0}) = 1/4, \\
1-\Phi_{\gamma,n}(x) 
&\geq 1-\Phi_{\gamma,n}(x_{m}) \geq 3/4 - 2T\delta \geq 1/4.
\end{align*}
Following Example~2.2 in \cite{Tsybakov2008}, let 
$\rho(\gamma',\gamma) \defeq \sum_{k=1}^{m}\mathds{1}_{\{\gamma_{k}' \neq \gamma_{k}\}}$, let 
\[
d_{k}(T_{n}^{\delta},\gamma_{k}) 
\defeq \int_{x_{k}}^{x_{k+1}}|T_{n}^{\delta}(x) - \gamma_{k}\varphi_{k,n}(x) - (1-\gamma_{k})\psi_{k,n}(x) - 1/4|dx
\]
and note that 
\begin{align*}
\bE_{\gamma,n}^{\otimes n}\bigg[\int_{-T}^{T}\big|T_n^{\delta}(x)-\Phi_{\gamma,n}(x)\big|dx\bigg] 
= \bE_{\gamma,n}^{\otimes n}\bigg[\sum_{k=0}^{m-1}\int_{x_{k}}^{x_{k+1}}\big|T_n^{\delta}(x)-\Phi_{\gamma,n}(x)\big|dx\bigg] 
= \sum_{k=0}^{m-1}\bE_{\gamma,n}^{\otimes n}\big[d_{k}(T_{n}^{\delta},\gamma_{k})\big]. 
\end{align*}
Defining $\hat{\gamma}_{k} \defeq \argmin_{t=0,1}d_{k}(T_{n}^{\delta},t)$, we have 
\begin{align*}
d_{k}(T_{n}^{\delta},\gamma_{k}) 
&\geq \frac{1}{2}d_{k}(\hat{\gamma}_{k}\varphi_{k,n} + (1-\hat{\gamma}_{k})\psi_{k,n} + 1/4,\gamma_{k}) \\
&= \frac{1}{2}|\hat{\gamma}_{k}-\gamma_{k}|\int_{x_{k}}^{x_{k+1}}|\varphi_{k,n}(x)-\psi_{k,n}(x)|dx 
\end{align*}
and so by noting that 
\[
\int_{x_{k}}^{x_{k+1}}|\varphi_{k,n}(x)-\psi_{k,n}(x)|dx 
= h_{n}|\varphi_{k,n}(x_{k}+h_{n})-\psi_{k,n}(x_{k}+h_{n})| 
\geq h_{n}2TC\Big(\frac{n}{\delta}\Big)^{-1/3}, 
\]
we obtain for all $\gamma \in \{0,1\}^{n}$, 
\begin{align*}
\bE_{\gamma,n}^{\otimes n}\bigg[\int_{-T}^{T}\big|T_n^{\delta}(x)-\Phi_{\gamma,n}(x)\big|dx\bigg] 
&\geq \frac{1}{2}\sum_{k=0}^{m-1}\bE_{\gamma,n}^{\otimes n}\bigg[|\hat{\gamma}_{k}-\gamma_{k}|\int_{x_{k}}^{x_{k+1}}|\varphi_{k,n}(x)-\psi_{k,n}(x)|dx\bigg] \\
&\geq h_{n}TC\Big(\frac{n}{\delta}\Big)^{-1/3}\bE_{\gamma,n}^{\otimes n}[\rho(\hat{\gamma},\gamma)].
\end{align*}
Consequently, for any $T_{n}^{\delta}$, 
\[
\max_{\gamma \in \{0,1\}^{n}}\bE_{\gamma,n}^{\otimes n}\bigg[\int_{-T}^{T}\big|T_n^{\delta}(x)-\Phi_{\gamma,n}(x)\big|dx\bigg] 
\geq h_{n}TC\Big(\frac{n}{\delta}\Big)^{-1/3}\inf_{\hat{\gamma}}\max_{\gamma \in \{0,1\}^{n}}\bE_{\gamma,n}^{\otimes n}[\rho(\hat{\gamma},\gamma)].
\]
By similar arguments as in the previous case, we have for $\alpha = 64C^{3}\|p_{X}\|_{\infty}$ and all 
$\gamma', \gamma \in \{0,1\}^{n}$ with $\rho(\gamma',\gamma) = 1$, 
\begin{align*}
h^{2}\big(P_{\gamma',n}^{\otimes n},P_{\gamma,n}^{\otimes n}\big) 
&\leq n\sum_{k=0}^{m-1}\int_{x_{k}}^{x_{k+1}}(\Phi_{\gamma',n}(x)-\Phi_{\gamma,n}(x))^{2}dP_{X}(x) \\
&= n\sum_{k=0}^{m-1}|\gamma_{k}'-\gamma_{k}|\int_{x_{k}}^{x_{k+1}}(\varphi_{k,n}(x)-\psi_{k,n}(x))^{2}dP_{X}(x) \\
&\leq n\sum_{k=0}^{m-1}|\gamma_{k}'-\gamma_{k}|\frac{1}{4}\int_{x_{k}}^{x_{k+1}}\delta^{2}h_{n}^{2}dP_{X}(x) \\
&\leq \frac{1}{2}n\delta^{2}h_{n}^{3}\|p_{X}\|_{\infty}\rho(\gamma',\gamma) \\
&\leq n\delta^{2}(4C(n\delta^{2})^{-1/3})^{3}\|p_{X}\|_{\infty} 
\leq 64C^{3}\|p_{X}\|_{\infty} = \alpha < 2.
\end{align*}
Thus, by Theorem~2.12 (iii) of \cite{Tsybakov2008}, 
\[
\inf_{\hat{\gamma}}\max_{\gamma \in \{0,1\}^{n}}\bE_{\gamma}[\rho(\hat{\gamma},\gamma)] 
\geq \frac{m}{2}\big(1-\sqrt{\alpha(1-\alpha)/4}\big) 
= \frac{1}{2h_{n}}\big(1-\sqrt{\alpha(1-\alpha)/4}\big) 
\]
and so we have for any $T_{n}^{\delta}$, 
\begin{align*}
\max_{\gamma \in \{0,1\}^{n}}\bE_{\gamma,n}^{\otimes n}\bigg[\int_{-T}^{T}\big|T_n^{\delta}(x)-\Phi_{\gamma,n}(x)\big|dx\bigg] 
&\geq h_{n}TC\Big(\frac{n}{\delta}\Big)^{-1/3}\inf_{\hat{\gamma}}\max_{\gamma \in \{0,1\}^{n}}\bE_{\gamma,n}^{\otimes n}[\rho(\hat{\gamma},\gamma)] \\
&\geq \Big(\frac{n}{\delta}\Big)^{-1/3}\frac{1}{2}TC\big(1-\sqrt{\alpha(1-\alpha)/4}\big) > 0 
\end{align*}
for all $\delta \in [n^{-1/2},\frac{1}{4T}]$ and all $n \geq 16T^{2}$, implying 
\[
\inf_{T_n^{\delta}}\sup_{\Phi\in\cF_{\delta}}\Big(\sqrt{n}\wedge \Big(\frac{n}{\delta}\Big)^{1/3}\Big)
						\bE_{\Phi}^{\otimes n}\bigg[\int_{-T}^{T}\big|T_n^{\delta}(x)-\Phi(x)\big|dx\bigg]
> \frac{1}{2}TC\big(1-\sqrt{\alpha(1-\alpha)/4}\big)
> 0 
\] 
for all $\delta \in [n^{-1/2},\frac{1}{4T}]$ and all $n \geq 16T^{2}$. \par
In summary, we have shown for any $C \leq \min\big\{\frac{1}{\sqrt{8T}},\big(\frac{1}{32\|p_{X}\|_{\infty}}\big)^{1/3}\big\}$, 
\begin{align*}
\inf_{T_n^{\delta}}\sup_{\Phi\in\cF_{\delta}}\Big(\sqrt{n}\wedge \Big(\frac{n}{\delta}\Big)^{1/3}\Big)
						\bE_{\Phi}^{\otimes n}\bigg[\int_{-T}^{T}\big|T_n^{\delta}(x)-\Phi(x)\big|dx\bigg] 
> \frac{1}{2}\big(1-\sqrt{\alpha(1-\alpha)/4}\big)\max\{TC,1\}
> 0 
\end{align*}
for $\alpha = \max\{16T^{2}C^{2},64C^{3}\|p_{X}\|_{\infty}\}$, all $\delta \in [0,\frac{1}{4T}]$ and 
$n > \max\{12^{3}C^{3}, 16T^{2}\}$ and so the assertion follows. \hfill \qed

\subsection{Proof of Proposition \ref{prop: 4.2}} \label{proof:4.2}
Proposition~\ref{prop: 4.2} is an immediate consequence of an application of Fubini's theorem and the following Lemma, 
which yields an upper bound on the convergence rate for $\bE[|\hat{\Phi}_{n}(t) - \Phi_{n}(t)|]$ for all $t \in (-T,T)$. For 
Proposition~\ref{prop: 4.2}, we utilize that in the subsequent result, 
the maximum is equal to $(n/\delta_{n})^{-1/3}$ if and only if 
$-T+(n\delta_{n}^{2})^{-1/3} \leq t \leq T - (n\delta_{n}^{2})^{-1/3}$.
Note that the following result is also used in the proof of Theorem~\ref{thm:l1 rate} (i).

\begin{lemma} \label{lem:rate expectation}
Assume $\Phi_{0}$ to be continuously differentiable in a neighborhood of zero and let $\Phi_{0}'(0) > 0$. Then, for $n$ 
large enough, there exists a constant $K > 0$ depending only on $\Phi_{0}$, $F_{X}$ and the bounds on its derivatives, such 
that for all $t \in (-T,T)$, 
\[
\bE\big[|\hat{\Phi}_{n}(t) - \Phi_{n}(t)|\big] 
\leq K\max\Big\{\Big(\frac{n}{\delta_{n}}\Big)^{-1/3}, \big(n(T-t)\big)^{-1/2}, \big(n(T+t)\big)^{-1/2}\Big\}.
\]
\end{lemma}

\begin{proof}
Conceptually, the proof follows the idea of the proof of Theorem~1 in \cite{Durot2008monotone}. From a technical point of 
view, however, the $n$-dependence of $\Phi_{n}$ and its vanishing derivative required us to do some adjustments.\par 
For ease of notation, let $x_{+} \defeq \max\{x,0\}$ denote the positive part of $x$ for every $x \in \bR$, let $K > 0$ 
denote a constant which may changes from line to line depending only on $\Phi_{0}'$ in a neighborhood of zero and on 
$p_{X}$ and define 
\[
I_{1}(t) \defeq \bE\big[(\hat{\Phi}_{n}(t) - \Phi_{n}(t))_{+}\big] 
\quad \text{and} \quad
I_{2}(t) \defeq \bE\big[(\Phi_{n}(t) - \hat{\Phi}_{n}(t))_{+}\big]
\]
for $t \in [-T,T]$, implying that 
\[
\bE\big[|\hat{\Phi}_{n}(t) - \Phi_{n}(t)|\big] 
= I_{1}(t) + I_{2}(t).
\]
From $|\hat{\Phi}_{n}(t) - \Phi_{n}(t)| \leq 1$, we obtain 
\[
I_{1}(t) 
= \int_{0}^{1}\bP\big(\hat{\Phi}_{n}(t) - \Phi_{n}(t) > x\big)dx 
\]
and from the fact that $\hat{\Phi}_{n}$ maps into $[0,1]$, we observe 
\[
I_{1}(t) 
= \int_{0}^{1}\bP\big(\hat{\Phi}_{n}(t) > x + \Phi_{n}(t)\big)dx 
= \int_{0}^{1-\Phi_{n}(t)}\bP\big(\hat{\Phi}_{n}(t) > x + \Phi_{n}(t)\big)dx. 
\]
By the switch relation (Lemma~\ref{lem:switch relation}), this implies 
\[
I_{1}(t) 
= \int_{0}^{1-\Phi_{n}(t)}\hspace{-0.2em}\bP\big(F_{n}^{-1} \circ \tilde{U}_{n}(\Phi_{n}(t) + x) < F_{n}^{-1}(F_{n}(t))\big)dx 
\leq \int_{0}^{1-\Phi_{n}(t)}\hspace{-0.2em}\bP\big(F_{n}^{-1} \circ \tilde{U}_{n}(\Phi_{n}(t) + x) < t\big)dx. 
\]
Now note that for every $x > 0$ which satisfies $\Phi_{n}(t) + x < \Phi_{n}(T)$, a Taylor expansion with Lagrange remainder 
of $\Phi_{n}^{-1}$ around $\Phi_{n}(t)$ yields for some $\nu_{n} \in (\Phi_{n}(t),\Phi_{n}(t)+x)$ that 
\[
\Phi_{n}^{-1}(\Phi_{n}(t) + x) 
= t + \frac{1}{\Phi_{n}'(\Phi_{n}^{-1}(\nu_{n}))}x 
\geq t + K\delta_{n}^{-1}x 
\]
for $n$ large enough. By an addition of zero, by using that $t - \Phi_{n}^{-1}(\Phi_{n}(t) + x) < 0$ and by 
Lemma~\ref{lem:tail bounds - 1} (ii), we find 
\begin{align*}
\bP\big(F_{n}^{-1} \circ \tilde{U}_{n}(\Phi_{n}(t) + x) < t\big) 
&= \bP\big(F_{n}^{-1} \circ \tilde{U}_{n}(\Phi_{n}(t) + x) - \Phi_{n}^{-1}(\Phi_{n}(t) + x) 
										< t - \Phi_{n}^{-1}(\Phi_{n}(t) + x)\big) \\
&\leq \bP\big(|F_{n}^{-1} \circ \tilde{U}_{n}(\Phi_{n}(t) + x) - \Phi_{n}^{-1}(\Phi_{n}(t) + x)|
																	\geq \Phi_{n}^{-1}(\Phi_{n}(t) + x) - t\big) \\
&\leq \bP\big(|F_{n}^{-1} \circ \tilde{U}_{n}(\Phi_{n}(t) + x) - \Phi_{n}^{-1}(\Phi_{n}(t) + x)|
																	\geq K\delta_{n}^{-1}x\big) \\
&\leq \mathds{1}_{\{x < K(\delta_{n}/n)^{1/3}\}} 
					+ \mathds{1}_{\{x \geq K(\delta_{n}/n)^{1/3}\}}K\delta_{n}n^{-1}x^{-3}. 
\end{align*}
Because we assumed $x < \Phi_{n}(T) - \Phi_{n}(t)$, we now have 
\begin{align*}
I_{1}(t) 
\leq K\Big(\frac{\delta_{n}}{n}\Big)^{1/3} 
	&+ K\Big(\frac{\delta_{n}}{n}\Big)\int_{\bR}x^{-3}\mathds{1}_{\{x \in [K(\frac{\delta_{n}}{n})^{1/3},\Phi_{n}(T) - \Phi_{n}(t)]\}}dx \\
	&+ \int_{\Phi_{n}(T) - \Phi_{n}(t)}^{1-\Phi_{n}(t)}
							\bP\big(F_{n}^{-1} \circ \tilde{U}_{n}(\Phi_{n}(t) + x) < t\big)dx. 
\end{align*}
Note that $\Phi_{n}(T) - \Phi_{n}(t) = \cO(\delta_{n})$ and so we can choose $n$ large enough, such that 
\begin{align*}
K\Big(\frac{\delta_{n}}{n}\Big)\int_{\bR}x^{-3}\mathds{1}_{\{x \in [K(\frac{\delta_{n}}{n})^{1/3},\Phi_{n}(T) - \Phi_{n}(t)]\}}dx
&\leq K\Big(\frac{\delta_{n}}{n}\Big)^{1/3} 
\end{align*}
and consequently, 
\[
I_{1}(t) 
\leq K\Big(\frac{\delta_{n}}{n}\Big)^{1/3} + \int_{\Phi_{n}(T) - \Phi_{n}(t)}^{1-\Phi_{n}(t)}
									\bP\big(F_{n}^{-1} \circ \tilde{U}_{n}(\Phi_{n}(t) + x) < t\big)dx.
\]
To derive an upper bound for the remaining integral, note that for every $x \geq \Phi_{n}(T) - \Phi_{n}(t)$, we have 
$\Phi_{n}^{-1}(\Phi_{n}(t) + x) = T$. So again by Lemma~\ref{lem:tail bounds - 1} (ii) and for $n$ large enough, 
\begin{align*}
\bP\big(F_{n}^{-1} \circ \tilde{U}_{n}(\Phi_{n}(t) + x) < t\big) 
&= \bP\big(F_{n}^{-1} \circ \tilde{U}_{n}(\Phi_{n}(t) + x) - T < t - T\big) \\
&\leq \bP\big(|F_{n}^{-1} \circ \tilde{U}_{n}(\Phi_{n}(t) + x) - T| \geq T - t\big) \\
&\leq K(n\delta_{n}^2)^{-1}(T-t)^{-3}
\end{align*}
for $1 - \Phi_{n}(t) > x \geq \Phi_{n}(T) - \Phi_{n}(t)$. \par
To finish the proof, we will consider the cases $T-t \geq (n\delta_{n}^{2})^{-1/3}$ and $T-t \leq (n\delta_{n}^{2})^{-1/3}$ 
separately. \\
$\bullet$ Let us start by assuming $T-t \geq (n\delta_{n}^{2})^{-1/3}$. Then, 
\begin{align*}
&\int_{\Phi_{n}(T) - \Phi_{n}(t)}^{1-\Phi_{n}(t)}\bP\big(F_{n}^{-1} \circ \tilde{U}_{n}(\Phi_{n}(t) + x) < t\big)dx \\
&\qquad\qquad\leq \int_{\Phi_{n}(T) - \Phi_{n}(t)}^{\Phi_{n}(T) - \Phi_{n}(t) + (\frac{\delta_{n}}{n})^{1/3}}
							K(n\delta_{n}^2)^{-1}(T-t)^{-3}dx \\
&\qquad\qquad\qquad\qquad+ \int_{\Phi_{n}(T) - \Phi_{n}(t) + (\frac{\delta_{n}}{n})^{1/3}}^{1-\Phi_{n}(t)}
							\bP\big(F_{n}^{-1} \circ \tilde{U}_{n}(\Phi_{n}(t) + x) < t\big)dx, 
\end{align*}
where 
\[
\int_{\Phi_{n}(T)-\Phi_{n}(t)}^{\Phi_{n}(T)-\Phi_{n}(t) + (\frac{\delta_{n}}{n})^{1/3}}K(n\delta_{n}^2)^{-1}(T-t)^{-3}dx 
= K\Big(\frac{\delta_{n}}{n}\Big)^{1/3}(n\delta_{n}^2)^{-1}(T-t)^{-3} 
\leq K\Big(\frac{\delta_{n}}{n}\Big)^{1/3} 
\]
and by Lemma~\ref{lem:tail bounds - 2}, we have
\begin{align*}
&\int_{\Phi_{n}(T) - \Phi_{n}(t) + (\frac{\delta_{n}}{n})^{1/3}}^{1-\Phi_{n}(t)}
							\bP\big(F_{n}^{-1} \circ \tilde{U}_{n}(\Phi_{n}(t) + x) < t\big)dx \\
&\qquad\leq \int_{\Phi_{n}(T) - \Phi_{n}(t) + (\frac{\delta_{n}}{n})^{1/3}}^{1-\Phi_{n}(t)}
							\bP\big(|F_{n}^{-1} \circ \tilde{U}_{n}(\Phi_{n}(t) + x) - T| \geq T - t\big)dx \\
&\qquad\leq K\int_{\Phi_{n}(T) - \Phi_{n}(t) + (\frac{\delta_{n}}{n})^{1/3}}^{1-\Phi_{n}(t)}
							(n(T-t))^{-1}(\Phi_{n}(T) - \Phi_{n}(t) - x)^{-2}dx \\
&\qquad= K(n(T-t))^{-1}\big[(\Phi_{n}(T)-\Phi_{n}(t)-x)^{-1}
							\big]_{x=\Phi_{n}(T) - \Phi_{n}(t) + (\frac{\delta_{n}}{n})^{1/3}}^{1-\Phi_{n}(t)} \\
&\qquad= K(n(T-t))^{-1}\Big((1-\Phi_{n}(T))^{-1} + \Big(\frac{n}{\delta_{n}}\Big)^{1/3}\Big) \\
&\qquad\leq K\Big(\frac{\delta_{n}}{n}\Big)^{2/3} + K\Big(\frac{\delta_{n}}{n}\Big)^{1/3} \\
&\qquad\leq K\Big(\frac{\delta_{n}}{n}\Big)^{1/3}. 
\end{align*}
So we have shown 
\[
I_{1}(t) \leq K\Big(\frac{\delta_{n}}{n}\Big)^{1/3}
\]
for $T-t \geq (n\delta_{n}^{2})^{-1/3}$. \\
$\bullet$ Now assume $T-t \leq (n\delta_{n}^{2})^{-1/3}$. Then, 
\[
(n(T-t))^{-1/2} \geq \Big(\frac{\delta_{n}}{n}\Big)^{1/3}
\]
and we have 
\begin{align*}
I_{1}(t) 
\leq K(n(T-t))^{-1/2}
			&+ \int_{\Phi_{n}(T) - \Phi_{n}(t)}^{\Phi_{n}(T) - \Phi_{n}(t) + (n(T-t))^{-1/2}}1dx \\
			&+ \int_{\Phi_{n}(T) - \Phi_{n}(t) + (n(T-t))^{-1/2}}^{1-\Phi_{n}(t)}
							\bP\big(F_{n}^{-1} \circ \tilde{U}_{n}(\Phi_{n}(t) + x) < t\big)dx.
\end{align*} 
As before, we know from Lemma~\ref{lem:tail bounds - 2} that 
\begin{align*}
&\int_{\Phi_{n}(T) - \Phi_{n}(t) + (n(T-t))^{-1/2}}^{1-\Phi_{n}(t)}
							\bP\big(F_{n}^{-1} \circ \tilde{U}_{n}(\Phi_{n}(t) + x) < t\big)dx \\
&\qquad\leq K(n(T-t))^{-1}\big[(\Phi_{n}(T)-\Phi_{n}(t)-x)^{-1}
							\big]_{x=\Phi_{n}(T) - \Phi_{n}(t) + (n(T-t))^{-1/2}}^{1-\Phi_{n}(t)} \\
&\qquad= K(n(T-t))^{-1}\big((1-\Phi_{n}(T))^{-1} + (n(T-t))^{1/2}\big) \\
&\qquad\leq K(n(T-t))^{-1} + K(n(T-t))^{-1/2} \\
&\qquad\leq K(n(T-t))^{-1/2}
\end{align*}
and so we have shown 
\[
I_{1}(t) \leq K(n(T-t))^{-1/2}
\]
for $T-t \leq (n\delta_{n}^{2})^{-1/3}$. \par
Summarizing the results, we have 
\[
I_{1}(t) \leq K\Big(\Big(\frac{n}{\delta_{n}}\Big)^{-1/3} + (n(T-t))^{-1/2}\Big) 
\]
and by similar arguments, 
\[
I_{2}(t) \leq K\Big(\Big(\frac{n}{\delta_{n}}\Big)^{-1/3} + (n(T+t))^{-1/2}\Big).
\]
Thus, 
\[
\bE\big[|\hat{\Phi}_{n}(t) - \Phi_{n}(t)|\big] 
\leq K\max\Big\{\Big(\frac{n}{\delta_{n}}\Big)^{-1/3}, (n(T-t))^{-1/2}, (n(T+t))^{-1/2}\Big\} 
\]
for $n$ large enough which proves the assertion.
\end{proof}

\subsection{Proof of Theorem~\ref{thm:l1 rate} (i)}\label{proof:4.4i}
The concept of proof presented below, namely to employ the switch relation to move over from $\hat{\Phi}_{n}$ and $\Phi_{n}$ to their inverse counterparts and to analyze the 
$L^{1}$-limit of these inverse counterparts, appeared first in 
\cite{Groeneboom1985monotone}, was made rigorous in Corollary~2.1 of \cite{Groeneboom1999L1} and was later considerably generalized in \cite{Durot2007Lp} and \cite{Durot2008monotone}.

\medskip
\textit{Further notation.} Throughout this section, we use the notation introduced in Section~\ref{sec:L1 limit} and
recall $\lambda_{n} = \Phi_{n} \circ F_{X}^{-1}$. Next, we assume $\tilde{U}_{n}$ on $[0,1]$ and $\Phi_{n}^{-1}$, $\lambda_{n}^{-1}$ on $[\Phi_n(-T),\Phi_n(T)]$ 
 to be continuously extended to functions on the real line by their values at the respective boundary points of their original domains. 
Note that this extension satisfies $\lambda_{n}^{-1} = F_{X} \circ \Phi_{n}^{-1}$. Finally, we abbreviate 
\[
\sigma_{n}^{2}(t) \defeq \Phi_{n}(t)(1-\Phi_{n}(t)), 
\quad 
\Lambda_{n}(s) \defeq \int_{0}^{t}\lambda_{n}(u)du 
\quad \text{and} \quad 
\cJ_{n} \defeq \int_{-T}^{T}|\hat{\Phi}_{n}(t) - \Phi_{n}(t)|dt
\]
for $t \in [-T,T]$ and $s \in [0,1]$. Recall that throughout, 
$P_{X}$ is compactly supported on 
$\cX = [-T,T]$ for some $T > 0$ with continuous, strictly positive Lebesgue density $p_{X}$ on $\cX$. 

\begin{proof}[Proof of Theorem~\ref{thm:l1 rate} (i)]
The proof is subdivided into six claims. Right before stating a claim, additional notation will be introduced if required. 
Throughout the proof, $K$ denotes a universal constant which may changes from line to line.

\smallskip
\noindent
\textsc{Claim I: } $\cJ_{n} = \cJ_{n,1} + o_{\bP}(n^{-1/2})$ with $\displaystyle \cJ_{n,1}
\defeq \int_{\Phi_{n}(-T)}^{\Phi_{n}(T)}|F_{n}^{-1} \circ \tilde{U}_{n}(a) - \Phi_{n}^{-1}(a)|da$.

\smallskip
\noindent
Note that here, as compared to the classical asymptotics, the integration domain is 
$n$-dependent with length of order $\delta_{n}$. 

\smallskip
\noindent
\textit{Proof of Claim I. } 
The subsequent proof is based on the tail bound of the inverse process given in 
Lemma~\ref{lem:tail bounds - 1} (ii). In that way, the proof hinges on the convergence rate of the inverse process, which is 
again highlighted by the necessary localizations of the integration domain.
Let
$I_{1} \defeq \int_{-T}^{T}\big(\hat{\Phi}_{n}(t) - \Phi_{n}(t)\big)_{+}dt$, 
$I_{2} \defeq \int_{-T}^{T}\big(\Phi_{n}(t) - \hat{\Phi}_{n}(t)\big)_{+}dt$ and 
\[
J_{1} \defeq \int_{-T}^{T}\int_{0}^{\Phi_{n}(T)-\Phi_{n}(t)}\mathds{1}_{\{\hat{\Phi}_{n}(t) > \Phi_{n}(t) + u\}}du \, dt.
\]
By Cavalieri's principle applied to $I_{1}$, 
\begin{align*}
I_{1} - J_{1} 
&= \int_{-T}^{T}\int_{0}^{1}\mathds{1}_{\{\hat{\Phi}_{n}(t) > \Phi_{n}(t) + u\}}du \, dt 
	- \int_{-T}^{T}\int_{0}^{\Phi_{n}(T)-\Phi_{n}(t)}\mathds{1}_{\{\hat{\Phi}_{n}(t) > \Phi_{n}(t) + u\}}du \, dt \\
&= \int_{-T}^{T}\int_{\Phi_{n}(T)-\Phi_{n}(t)}^{1}\mathds{1}_{\{\hat{\Phi}_{n}(t) > \Phi_{n}(t) + u\}}du \, dt \\
&= \int_{F_{n}^{-1} \circ \tilde{U}_{n}(\Phi_{n}(T))}^{T}
					\int_{\Phi_{n}(T)-\Phi_{n}(t)}^{1}\mathds{1}_{\{\hat{\Phi}_{n}(t) > \Phi_{n}(t) + u\}}du \, dt,
\end{align*}
where the last equality is based on $\hat{\Phi}_{n}(t) > \Phi_{n}(T)$ 
if and only if $F_{n}^{-1}(F_{n}(t)) > F_{n}^{-1} \circ \tilde{U}_{n}(\Phi_{n}(T))$ by the switch relation 
(Lemma~\ref{lem:switch relation}). Thus, $I_{1} - J_{1} \geq 0$ and again by Cavalieri’s principle, 
\begin{align*}
I_{1} - J_{1} 
&\leq \int_{F_{n}^{-1} \circ \tilde{U}_{n}(\Phi_{n}(T))}^{T}
					\int_{0}^{1}\mathds{1}_{\{\hat{\Phi}_{n}(t) > \Phi_{n}(t) + u\}}du \, dt \\
&\leq \int_{T-(n\delta_{n}^{2})^{-1/3}\log(n\delta_{n}^{2})}^{T}|\hat{\Phi}_{n}(t)-\Phi_{n}(t)|dt 
					+ 2T\mathds{1}_{\big\{F_{n}^{-1} \circ \tilde{U}_{n}(\Phi_{n}(T)) 
												\leq T-\frac{\log(n\delta_{n}^{2})}{(n\delta_{n}^{2})^{1/3}}\big\}} 
\end{align*}
where we used without loss of generality that $n\delta_{n}^{2} \geq 1$ for $n$ large enough. For $\varepsilon > 0$, 
Lemma~\ref{lem:tail bounds - 1} (ii) provides for $n$ large enough, 
\begin{align*}
\bP\bigg(\sqrt{n}\mathds{1}_{\big\{F_{n}^{-1} \circ \tilde{U}_{n}(\Phi_{n}(T)) 
								\leq T-\frac{\log(n\delta_{n}^{2})}{(n\delta_{n}^{2})^{1/3}}\big\}} \geq \varepsilon \bigg) 
&\leq \bP\bigg(|F_{n}^{-1} \circ \tilde{U}_{n}(\Phi_{n}(T))-T| \geq \frac{\log(n\delta_{n}^{2})}{(n\delta_{n}^{2})^{1/3}}\bigg) \\
&\leq K\big(n\delta_{n}^{2}\big((n\delta_{n}^{2})^{-1/3}\log(n\delta_{n}^{2})\big)^{3}\big)^{-1} 
\end{align*}
which is bounded by $K\log(n\delta_{n}^{2})^{-3}$ and so we have 
\[
I_{1} - J_{1} 
\leq \int_{T-(n\delta_{n}^{2})^{-1/3}\log(n\delta_{n}^{2})}^{T}|\hat{\Phi}_{n}(t)-\Phi_{n}(t)|dt + o_{\bP}(n^{-1/2}).
\]
By Markov's inequality, Fubini's theorem and Proposition~\ref{lem:rate expectation}, 
\begin{align*}
&\bP\bigg(\sqrt{n}\int_{T-(n\delta_{n}^{2})^{-1/3}\log(n\delta_{n}^{2})}^{T}
									|\hat{\Phi}_{n}(t)-\Phi_{n}(t)|dt > \varepsilon\bigg) \\
&\quad\leq K\frac{\sqrt{n}}{\varepsilon}\Big(\frac{n}{\delta_{n}}\Big)^{-1/3}(n\delta_{n}^{2})^{-1/3}
																		\big(\log(n\delta_{n}^{2})-1\big) 
		+ K\frac{\sqrt{n}}{\varepsilon}\int_{T-(n\delta_{n}^{2})^{-1/3}}^{T}n^{-1/2}(T-t)^{-1/2}dt \\
&\quad= \frac{K}{\varepsilon}(n\delta_{n}^{2})^{-1/6}\big(1+\log(n\delta_{n}^{2})\big)
\end{align*}
and so we have shown that $I_{1} = J_{1} + o_{\bP}(n^{-1/2})$. Note further that by the change of variable 
$a = \Phi_{n}(t) + u$, Fubini's theorem and the switch relation (Lemma~\ref{lem:switch relation}), 
\begin{align*}
J_{1} 
&= \int_{\Phi_{n}(-T)}^{\Phi_{n}(T)}\int_{-T}^{T}\mathds{1}_{\{a \geq \Phi_{n}(t)\}}
									\mathds{1}_{\{\hat{\Phi}_{n}(t) > a\}}dt \, da \\
&= \int_{\Phi_{n}(-T)}^{\Phi_{n}(T)}\int_{F_{n}^{-1} \circ \tilde{U}_{n}(a)}^{T}\mathds{1}_{\{\Phi_{n}^{-1}(a) \geq t\}}
									\mathds{1}_{\{\hat{\Phi}_{n}(t) > a\}}dt \, da \\
&= \int_{\Phi_{n}(-T)}^{\Phi_{n}(T)}\int_{F_{n}^{-1} \circ \tilde{U}_{n}(a)}^{\Phi_{n}^{-1}(a)}
									\mathds{1}_{\{F_{n}^{-1} \circ \tilde{U}_{n}(a) < t\}}dt \, da \\
&= \int_{\Phi_{n}(-T)}^{\Phi_{n}(T)}\big(\Phi_{n}^{-1}(a) - F_{n}^{-1} \circ \tilde{U}_{n}(a)\big)
									\mathds{1}_{\{F_{n}^{-1} \circ \tilde{U}_{n}(a) < \Phi_{n}^{-1}(a)\}}da
\end{align*}
and so we have 
\[
I_{1} 
= \int_{\Phi_{n}(-T)}^{\Phi_{n}(T)}\big(\Phi_{n}^{-1}(a) - F_{n}^{-1} \circ \tilde{U}_{n}(a)\big)
							\mathds{1}_{\{F_{n}^{-1} \circ \tilde{U}_{n}(a) < \Phi_{n}^{-1}(a)\}}da + o_{\bP}(n^{-1/2}).
\]
By similar arguments, 
\[
I_{2} 
= \int_{\Phi_{n}(-T)}^{\Phi_{n}(T)}\big(F_{n}^{-1} \circ \tilde{U}_{n}(a) - \Phi_{n}^{-1}(a)\big)
						\mathds{1}_{\{F_{n}^{-1} \circ \tilde{U}_{n}(a) \geq \Phi_{n}^{-1}(a)\}}da + o_{\bP}(n^{-1/2})
\]
and Claim I follows.

\medskip
\noindent
\textsc{Claim II: } There exist Brownian bridges $B_{n}$ on $[0,1]$, such that $\cJ_{n,1} = \cJ_{n,2} + o_{\bP}(n^{-1/2})$, 
with 
\[
\cJ_{n,2} \defeq \int_{\lambda_{n}(0)}^{\lambda_{n}(1)}\bigg|\tilde{U}_{n}(a) - \lambda_{n}^{-1}(a) 
			- \frac{B_{n}(\lambda_{n}^{-1}(a))}{\sqrt{n}}\bigg|\frac{1}{p_{X}(\Phi_{n}^{-1}(a))}da.
\]

\smallskip
\noindent
\textit{Proof of Claim II. } Note first that $\Phi_{n}(T) = \lambda_{n}(1)$ and 
$\Phi_{n}(-T) = \lambda_{n}(0)$ by definition and that by Theorem~3 of \cite{KMT1975}, there 
exist Brownian bridges $B_{n}$ on $[0,1]$, such that 
\begin{align}
\bE\bigg[\sup_{t \in [0,1]}\bigg|F_{n} \circ F_{X}^{-1}(t) - t - \frac{B_{n}(t)}{\sqrt{n}}\bigg|^{r}\bigg]^{1/r} 
= \cO\Big(\frac{\log(n)}{n}\Big) \label{eq:KMT}
\end{align}
for $r \geq 1$. By definition of $\cJ_{n,1}$ and rewriting $\Phi_{n}^{-1} = F_{X}^{-1} \circ \lambda_{n}^{-1}$, 
\begin{align*}
\cJ_{n,1}
&= \int_{\lambda_{n}(0)}^{\lambda_{n}(1)}\bigg|F_{n}^{-1} \circ \tilde{U}_{n}(a) - F_{X}^{-1} \circ \tilde{U}_{n}(a) 
						+ \frac{B_{n}(\lambda_{n}^{-1}(a))}{\sqrt{n}p_{X}(\Phi_{n}^{-1}(a))} \\
	&\qquad\qquad\qquad\quad	+ F_{X}^{-1} \circ \tilde{U}_{n}(a) - F_{X}^{-1} \circ \lambda_{n}^{-1}(a) 
						- \frac{B_{n}(\lambda_{n}^{-1}(a))}{\sqrt{n}p_{X}(\Phi_{n}^{-1}(a))}\bigg|da.
\end{align*}
A Taylor expansion of $F_{X}^{-1}$ around $\lambda_{n}^{-1}(a)$ yields 
\[
F_{X}^{-1} \circ \tilde{U}_{n}(a) - F_{X}^{-1} \circ \lambda_{n}^{-1}(a) 
= \frac{\tilde{U}_{n}(a)-\lambda_{n}^{-1}(a)}{p_{X}(\Phi_{n}^{-1}(a))}
				 + \frac{1}{2}(F_{X}^{-1})''(\nu_{n})\big(\tilde{U}_{n}(a)-\lambda_{n}^{-1}(a)\big)^{2}
\]
for some $\nu_{n}$ between $\lambda_{n}^{-1}(a)$ and $\tilde{U}_{n}(a)$. But 
$(F_{X}^{-1})'' = -\frac{p_{X}' \circ F_{X}^{-1}}{(p_{X} \circ F_{X}^{-1})^{3}}$ is bounded as $p_{X}$ is continuously 
differentiable and $p_{X}$ is bounded away from zero, whence
\[
\bE\bigg[\bigg|F_{X}^{-1} \circ \tilde{U}_{n}(a) - F_{X}^{-1} \circ \lambda_{n}^{-1}(a) 
				- \frac{\tilde{U}_{n}(a)-\lambda_{n}^{-1}(a)}{p_{X}(\Phi_{n}^{-1}(a))}\bigg|\bigg] 
\leq K(n\delta_{n}^{2})^{-2/3} 
\]
by Corollary~\ref{cor:tail bounds - 1} for $Z_{1,n} = Z_{2,n} = 0$. 
Combined with $|\lambda_{n}(1)-\lambda_{n}(0)| = \cO(\delta_{n})$ and 
$n\delta_{n}^{2} \longrightarrow \infty$, an application of Markov's inequality 
and Fubini's theorem imply 
\begin{align*}
\cJ_{n} 
&= \int_{\lambda_{n}(0)}^{\lambda_{n}(1)}\bigg|F_{n}^{-1} \circ \tilde{U}_{n}(a) - F_{X}^{-1} \circ \tilde{U}_{n}(a) 
							+ \frac{B_{n}(\lambda_{n}^{-1}(a))}{\sqrt{n}p_{X}(\Phi_{n}^{-1}(a))} \\
	&\qquad\qquad\qquad+ \bigg(\tilde{U}_{n}(a)-\lambda_{n}^{-1}(a) 
							- \frac{B_{n}(\lambda_{n}^{-1}(a))}{\sqrt{n}}\bigg)\frac{1}{p_{X}(\Phi_{n}^{-1}(a))}\bigg|da 
							+ o_{\bP}(n^{-1/2}).
\end{align*}
Now before we bring the KMT approximation \eqref{eq:KMT} into play, we need to approximate the difference 
$F_{n}^{-1} \circ \tilde{U}_{n}(a) - F_{X}^{-1} \circ \tilde{U}_{n}(a)$ within the integral appropriately. To make this precise, observe 
\begin{align*}
&\int_{\lambda_{n}(0)}^{\lambda_{n}(1)}\bigg|F_{n}^{-1} \circ \tilde{U}_{n}(a) - F_{X}^{-1} \circ \tilde{U}_{n}(a) 
								+ \frac{B_{n}(\lambda_{n}^{-1}(a))}{\sqrt{n}p_{X}(\Phi_{n}^{-1}(a))}\bigg|da \\
&\qquad\qquad\leq K\delta_{n}\sup_{u \in [0,1]}\bigg|F_{n}^{-1}(u) - F_{X}^{-1}(u) 
								+ \frac{B_{n}(F_{X} \circ F_{n}^{-1}(u))}{\sqrt{n}p_{X}(F_{n}^{-1}(u))}\bigg| \\
&\qquad\qquad\qquad\qquad\qquad	+ \frac{1}{\sqrt{n}}\int_{\lambda_{n}(0)}^{\lambda_{n}(1)}
											\bigg|\frac{B_{n}(\lambda_{n}^{-1}(a))}{p_{X}(\Phi_{n}^{-1}(a))} 
								- \frac{B_{n}(F_{X} \circ F_{n}^{-1}(\tilde{U}_{n}(a)))}
																		{p_{X}(F_{n}^{-1}(\tilde{U}_{n}(a)))}\bigg|da \\
&\qquad\qquad= K\delta_{n}\sup_{u \in [0,1]}\bigg|F_{X}^{-1}(u) - F_{n}^{-1}(u) 
								- \frac{B_{n}(F_{X} \circ F_{n}^{-1}(u))}{\sqrt{n}p_{X}(F_{n}^{-1}(u))}\bigg| 
								+ o_{\bP}(n^{-1/2}), 
\end{align*}
where the last equality follows from Lemma~\ref{lem:tail bounds - 1} (ii) and the classical bound on the expected modulus 
of continuity of the Brownian bridge (e.g. formula (2) in \cite{Fischer2010}, rewriting the Brownian bridge in terms of an independent standard Gaussian random variable and a Brownian motion). Now by decomposing 
$\sup_{u \in [0,1]} = \max_{i}\sup_{u \in [i/n,(i+1)/n)}$ and by utilizing that we can write 
$\max_{i}\sup_{u \in [i/n,(i+1)/n)}|F_{X}^{-1}(u) - F_{X}^{-1}(i/n)| = \cO(n^{-1})$, we find 
\begin{align*}
&\sup_{u \in [0,1]}\bigg|F_{X}^{-1}(u) - F_{n}^{-1}(u) 
				- \frac{B_{n}(F_{X} \circ F_{n}^{-1}(u))}{\sqrt{n}p_{X}(F_{n}^{-1}(u))}\bigg| \\
&\ \leq \sup_{u \in [0,1]}\bigg|F_{X}^{-1}(F_{n} \circ F_{X}^{-1}(u)) - F_{X}^{-1}(u) 
				- \frac{B_{n}(u)}{\sqrt{n}p_{X}(F_{X}^{-1}(u))}\bigg| + \cO(n^{-1}) \\
&\ \leq \sup_{u \in [0,1]}\bigg|(F_{X}^{-1})'(u)(F_{n} \circ F_{X}^{-1}(u) - u)
								- \frac{B_{n}(u)}{\sqrt{n}p_{X}(F_{X}^{-1}(u))}\bigg| + \cO(n^{-1}) + K\sup_{u \in [0,1]}\big|(F_{n} \circ F_{X}^{-1}(u) - u)^{2}\big| \\
&\ \leq \sup_{u \in [0,1]}\bigg|F_{n} \circ F_{X}^{-1}(u) - u - \frac{B_{n}(u)}{\sqrt{n}}\bigg|\frac{1}{p_{X}(F_{X}^{-1}(u))} + \cO_{\bP}(n^{-1}) 
= o_{\bP}(n^{-1/2})
\end{align*}
by \eqref{eq:KMT} and the fact that 
$\sup_{u \in [0,1]}|F_{n} \circ F_{X}^{-1}(u) - u| = \cO_{\bP}(n^{-1/2})$.

\medskip
\noindent
The punch line of the next claim is to incorporate the Brownian bridges $B_{n}$ from 
Claim II into the inverse process.

\smallskip
\noindent
\textsc{Further notation. } For $a \in [\lambda_{n}(0),\lambda_{n}(1)]$, let $i_{n}(a)$ denote the integer part of the expression
$(a-\lambda_{n}(0))(n\delta_{n}^{2})^{1/3}/(\delta_{n}\log(n\delta_{n}^{2}))$, define 
$a_{n} \defeq \lambda_{n}(0) + i_{n}(a)\delta_{n}(n\delta_{n}^{2})^{-1/3}\log(n\delta_{n}^{2})$ and 
\begin{equation}\label{eq:anB}
a_{n}^{B} 
\defeq a - \frac{B_{n}(\lambda_{n}^{-1}(a_{n}))}{n^{1/2}(\lambda_{n}^{-1})'(a_{n})}. 
\end{equation}
For $i \in \bN_{0}$, let $k_{i}^{n} \defeq \lambda_{n}(0) + i\delta_{n}(n\delta_{n}^{2})^{-1/3}\log(n\delta_{n}^{2})$, 
\[
I_{i}^{n} 
\defeq \big[k_{i}^{n},\min\{k_{i+1}^{n},\lambda_{n}(1)\}\big), 
\]
let $N^{n} \defeq (\lambda_{n}(1)-\lambda_{n}(0))\frac{(n\delta_{n}^{2})}{\delta_{n}\log(n\delta_{n}^{2})}$ and note that 
$\bigcup_{i=0}^{N^{n}} I_{i}^{n} = [\lambda_{n}(0),\lambda_{n}(1)]$.
Let us also define the interval $J_{n} \defeq \bigcup_{i=1}^{N^{n}-2} I_{i}^{n}$. The definition of the intervals and the 
behavior of $a_{n}^{B}$ is visualized in Figure~\ref{fig:interval}.

\begin{figure}
\begin{tikzpicture}

\draw[thick] (-6.2,0) -- (-0.4,0); 
\draw[thick, dotted] (-0.4,0) -- (1.0,0);
\draw[thick] (1.0,0) -- (6.2,0); 
\draw[thick] (-6,-1.25) -- (-6,1.25);
\draw[thick] (6,-1.25) -- (6,1.25);

\foreach \x in {-4.2, -2.4, -0.6, 1.2, 3.0, 4.8} {
    \draw[thick] (\x,0.25) -- (\x,-0.25);
}

\draw (-2.0,0.15) -- (-2.0,-0.15);
\draw (-1.6,0.15) -- (-1.6,-0.15);

\node[below] at (-6,-1.25) {$\lambda_{n}(0)$};
\node[below] at (6,-1.25) {$\lambda_{n}(1)$};

\node[below] at (-4.2,-0.35) {$k_{1}^{n}$};
\node[below] at (-2.4,-0.35) {$k_{2}^{n}$};
\node[below] at (-0.6,-0.35) {$k_{3}^{n}$};

\node[below] at (1.2,-0.35) {$k_{N^{n}-2}^{n}$};
\node[below] at (3.0,-0.35) {$k_{N^{n}-1}^{n}$};
\node[below] at (4.8,-0.35) {$k_{N^{n}}^{n}$};

\node[below] at (-5.1,-0.75) {$I_{0}^{n}$};
\node[below] at (-3.3,-0.75) {$I_{1}^{n}$};
\node[below] at (-1.5,-0.75) {$I_{2}^{n}$};
\node[below] at (2.3,-0.75) {$I_{N^{n}-2}^{n}$};
\node[below] at (3.9,-0.75) {$I_{N^{n}-1}^{n}$};
\node[below] at (5.4,-0.75) {$I_{N^{n}}^{n}$};

\node[below] at (-2.0,-0.2) {$a$};
\node[below] at (-1.5,-0.02) {$a_n^{B}$};

\draw[decorate,decoration={brace,amplitude=5pt},yshift=5pt] (1.2,0.2) -- (3.0,0.2) node[midway,yshift=12pt] 
    {$\delta_{n}(n\delta_{n}^{2})^{-1/3}\log(n\delta_{n}^{2})$};

\draw[decorate,decoration={brace,amplitude=5pt},yshift=5pt] (-2.6,0.2) -- (-1.4,0.2) node[midway,yshift=12pt] 
    {$2C\delta_{n}n^{-1/2}\log(n\delta_{n}^{2})$};

\fill[color=blue!30,opacity=0.2] (-2.6,0.3) rectangle (-1.4,-0.3);
\draw[color=blue!30] (-2.6,0.3) -- (-2.6,-0.3);
\draw[color=blue!30] (-1.4,0.3) -- (-1.4,-0.3);

\end{tikzpicture}
\caption{For $a \in I_{i}^{n}$ for some $i \in \{0,\dots,N^{n}\}$, we have $|a_{n}^{B} - a| \leq C\delta_{n}(\log(n\delta_{n}^{2})/n)^{1/2}$ 
for $n$ large enough, which is smaller than the length of an interval $I_{i}^{n}$, bounded by 
$\delta_{n}(n\delta_{n}^{2})^{-1/3}\log(n\delta_{n}^{2})$.}\label{fig:interval}
\end{figure}

\smallskip
\noindent
\textsc{Claim III: } $\cJ_{n,2} = \cJ_{n,3} + o_{\bP}(n^{-1/2})$, with 
$\displaystyle \cJ_{n,3} \defeq \int_{\lambda_{n}(0)}^{\lambda_{n}(1)}\frac{|\tilde{U}_{n}(a_{n}^{B}) 
										- \lambda_{n}^{-1}(a)|}{p_{X}(\Phi_{n}^{-1}(a))}da.$

\smallskip
\noindent
\textit{Proof of Claim III. } Let 
$\Omega_{n}\defeq \big\{ \sup_{u \in [0,1]}|B_{n}(u)| \leq \sqrt{\log(n\delta_{n}^{2})}\big\} \subset \Omega$ and note that 
$\bP(\Omega_{n}) \longrightarrow 1$ as $n \longrightarrow \infty$. Then, 
\[
\cJ_{n,2}\mathds{1}_{\Omega_{n}}
= \int_{J_{n}}\bigg|\tilde{U}_{n}(a)-\lambda_{n}^{-1}(a) - \frac{B_{n}(\lambda_{n}^{-1}(a))}{\sqrt{n}}\bigg|\mathds{1}_{\Omega_{n}}da + o_{\bP}(n^{-1/2}), 
\]
where we used Corollary~\ref{cor:tail bounds - 1} with $c_{n} = 0$. Note further that $i_{n}(a) = i$ for every 
$a \in I_{i}^{n}$ and so in this case, 
$a_{n} = \lambda_{n}(0) + i\delta_{n}(n\delta_{n}^{2})^{-1/3}\log(n\delta_{n}^{2}) = k_{i}^{n}$
on $I_{i}^{n}$. Thus, for $a \in I_{i}^{n}$, $a_{n}^{B}$ is just a translation of $a$ by 
\[
B_{i}^{n} \defeq \frac{B_{n}(\lambda_{n}^{-1}(k_{i}^{n}))}{\sqrt{n}(\lambda_{n}^{-1})'(k_{i}^{n})}.
\]
Let $I_{i}^{n,B} \defeq I_{i}^{n} + B_{i}^{n} \defeq \{x + B_{i}^{n} \mid x \in I_{i}^{n}\}$. 
Then, a change of variable 
inside the integral, where $a$ is replaced by $a_{n}^{B}$ on each interval $I_{i}^{n}$, proves that 
$\cJ_{n,2}\mathds{1}_{\Omega_{n}}$ is equal to 
\begin{align*}
&\sum_{i=1}^{N^{n}-2}\hspace*{-1mm}\int_{I_{i}^{n}}\bigg|\tilde{U}_{n}(a) - \lambda_{n}^{-1}(a) 
			- \frac{B_{n}(\lambda_{n}^{-1}(a))}{\sqrt{n}}\bigg|\frac{1}{p_{X}(\Phi_{n}^{-1}(a))}\mathds{1}_{\Omega_{n}}da + o_{\bP}(n^{-1/2}) \\
&= \hspace*{-1mm}\sum_{i=1}^{N^{n}-2}\hspace*{-1mm}\int_{I_{i}^{n,B}}\bigg|\tilde{U}_{n}(a_{n}^{B}) - \lambda_{n}^{-1}(a_{n}^{B}) 
			- \frac{B_{n}(\lambda_{n}^{-1}(a_{n}^{B}))}{\sqrt{n}}\bigg|\frac{1}{p_{X}(\Phi_{n}^{-1}(a_{n}^{B}))}\mathds{1}_{\Omega_{n}}da + o_{\bP}(n^{-1/2}) \\
&= \hspace*{-1mm}\sum_{i=1}^{N^{n}-2}\hspace*{-1mm}\int_{I_{i}^{n}}\bigg|\tilde{U}_{n}(a_{n}^{B}) - \lambda_{n}^{-1}(a_{n}^{B}) 
			- \frac{B_{n}(\lambda_{n}^{-1}(a_{n}^{B}))}{\sqrt{n}}\bigg|\frac{1}{p_{X}(\Phi_{n}^{-1}(a_{n}^{B}))}\mathds{1}_{\Omega_{n}}da 
			+ R_{n} + o_{\bP}(n^{-1/2}), 
\end{align*}
where, with $I_{i}^{n,B} \triangle I_{i}^{n}$ denoting the symmetric difference of the sets $I_{i}^{n,B}$ and $I_{i}^{n}$,
\begin{align}\label{eq: Rn3}
|R_{n}|
\leq \sum_{i=1}^{N^{n}-2}\int_{I_{i}^{n,B} \triangle I_{i}^{n}}
									\bigg|\tilde{U}_{n}(a_{n}^{B}) - \lambda_{n}^{-1}(a_{n}^{B}) 
			- \frac{B_{n}(\lambda_{n}^{-1}(a_{n}^{B}))}{\sqrt{n}}\bigg|\frac{1}{p_{X}(\Phi_{n}^{-1}(a_{n}^{B}))}\mathds{1}_{\Omega_{n}}da. 
\end{align}
By definition of $\Omega_{n}$, we have $|a_{n}^{B} - a| \leq C\delta_{n}\sqrt{\log(n\delta_{n}^{2})/n}$ on $\Omega_n$ for 
some constant $C>0$ that does not depend on $a\in[\lambda_n(0),\lambda_n(1)]$ and thus on $\Omega_{n}$, $a_{n}^{B}$ is in 
fact contained in $[\lambda_{n}(0),\lambda_{n}(1)]$ for $a \in J_{n}$ and $n$ large enough. Let 
$D_{i}^{n}$ denote the symmetric difference of $I_i^n$ and the union 
$(I_{i}^{n} + C\delta_{n}(\log(n\delta_{n}^{2})/n)^{1/2})\cup(I_{i}^{n} - C\delta_{n}(\log(n\delta_{n}^{2})/n)^{1/2})$. 
Then, $I_{i}^{n,B} \triangle I_{i}^{n} \subset D_{i}^{n}$ on $\Omega_n$ and we obtain with \eqref{eq: Rn3}, 
\begin{align*}
\bE[|R_n|]
&\leq K\sum_{i=1}^{N^{n}-2}\int_{D_{i}^{n}}
				\bE\big[\big|\tilde{U}_{n}(a_{n}^{B}) - \lambda_{n}^{-1}(a_{n}^{B})\big|\mathds{1}_{\Omega_{n}}\big] + \bE\bigg[\bigg|\frac{B_{n}(\lambda_{n}^{-1}(a_{n}^{B}))}{\sqrt{n}}\bigg|\mathds{1}_{\Omega_{n}}\bigg]da \\
&\leq K(\lambda_{n}(1)-\lambda_{n}(0))\frac{(n\delta_{n}^{2})^{1/3}}{\delta_{n}\log(n\delta_{n}^{2})}
					\big((n\delta_{n}^{2})^{-1/3} + n^{-1/2}\big)\delta_{n}\sqrt{\frac{\log(n\delta_{n}^{2})}{n}} 
\leq K\delta_{n}\frac{n^{-1/2}}{\sqrt{\log(n\delta_{n}^{2})}}, 
\end{align*}
where we used Corollary~\ref{cor:tail bounds - 1} with 
$c_{n} = \delta_{n}\sqrt{\log(n\delta_{n}^{2})/n}$. Thus, 
\[
\cJ_{n,2}
= \int_{J_{n}}\bigg|\tilde{U}_{n}(a_{n}^{B}) - \lambda_{n}^{-1}(a_{n}^{B}) 
			- \frac{B_{n}(\lambda_{n}^{-1}(a_{n}^{B}))}{\sqrt{n}}\bigg|\frac{1}{p_{X}(\Phi_{n}^{-1}(a_{n}^{B}))}da 
			+ o_{\bP}(n^{-1/2}). 
\]
Subsequently, we show that we can replace $a_{n}^{B}$ by $a$ in the argument of the Brownian bridge and the density 
$p_{X}$ in the previous expression. By a Taylor expansion of $1/p_{X}(\Phi_{n}^{-1}(a_{n}^{B}))$ around $a$, we find for 
some $\nu_{n}$ between $a$ and $a_{n}^{B}$, that 
\[
\frac{1}{p_{X}(\Phi_{n}^{-1}(a_{n}^{B}))} 
= \frac{1}{p_{X}(\Phi_{n}^{-1}(a))} 
	+ \frac{p_{X}'(\Phi_{n}^{-1}(\nu_{n}))}{(p_{X}(\Phi_{n}^{-1}(\nu_{n}))^{2}\Phi_{0}'(\delta_{n}\Phi_{n}^{-1}(\nu_{n}))}
						\frac{( a_{n}^{B}-a)}{\delta_{n}}.
\]
Similar as before and by the same application of Corollary~\ref{cor:tail bounds - 1}, 
\begin{align*}
&\bigg|\int_{J_{n}}\bE\bigg[\bigg|\tilde{U}_{n}(a_{n}^{B}) - \lambda_{n}^{-1}(a_{n}^{B}) 
		- \frac{B_{n}(\lambda_{n}^{-1}(a_{n}^{B}))}{\sqrt{n}}\bigg|
				\bigg(\frac{1}{p_{X}(\Phi_{n}^{-1}(a_{n}^{B}))}-\frac{1}{p_{X}(\Phi_{n}^{-1}(a))}\bigg)\mathds{1}_{\Omega_{n}}\bigg]da\bigg| \\
&\qquad\leq K(n\delta_{n}^{2})^{-1/3}\big(\lambda_{n}(1)-\lambda_{n}(0)\big)n^{-1/2}\log(n\delta_{n}^{2}), 
\end{align*}
which shows that 
\[
\cJ_{n,2} = \int_{J_{n}}\bigg|\tilde{U}_{n}(a_{n}^{B}) - \lambda_{n}^{-1}(a_{n}^{B}) 
			- \frac{B_{n}(\lambda_{n}^{-1}(a_{n}^{B}))}{\sqrt{n}}\bigg|\frac{1}{p_{X}(\Phi_{n}^{-1}(a))}da 
			+ o_{\bP}(n^{-1/2}). 
\]
Next, we observe 
\begin{align*}
&\bigg|\int_{J_{n}}\bE\bigg[\bigg(\bigg|\tilde{U}_{n}(a_{n}^{B}) - \lambda_{n}^{-1}(a_{n}^{B}) 
			- \frac{B_{n}(\lambda_{n}^{-1}(a_{n}^{B}))}{\sqrt{n}}\bigg| \\
&\qquad\qquad\qquad\qquad- \bigg|\tilde{U}_{n}(a_{n}^{B}) - \lambda_{n}^{-1}(a_{n}^{B}) 
			- \frac{B_{n}(\lambda_{n}^{-1}(a))}{\sqrt{n}}\bigg|\bigg)\frac{1}{p_{X}(\Phi_{n}^{-1}(a))}\mathds{1}_{\Omega_{n}}\bigg]da\bigg| \\
&\qquad\leq Kn^{-1/2}\int_{J_{n}}\bE\big[\big|B_{n}(\lambda_{n}^{-1}(a)) 
			- B_{n}(\lambda_{n}^{-1}(a_{n}^{B}))\big|\mathds{1}_{\Omega_{n}}\big]da
\end{align*}
and obtain by the classical bound on the expected modulus of continuity of the Brownian bridge, 
\[
\cJ_{n,2} = \int_{J_{n}}\bigg|\tilde{U}_{n}(a_{n}^{B}) - \lambda_{n}^{-1}(a_{n}^{B}) 
			- \frac{B_{n}(\lambda_{n}^{-1}(a))}{\sqrt{n}}\bigg|\frac{1}{p_{X}(\Phi_{n}^{-1}(a))}da 
			+ o_{\bP}(n^{-1/2}).
\]
To complete the proof of Claim III, it is sufficient to verify that
\begin{align}\label{eq: Claim3c}
\sqrt{n}(\lambda_n(0)-\lambda_n(1))\sup_{a \in J_{n}}\bigg|\lambda_{n}^{-1}(a_{n}^{B}) 
								+ \frac{B_{n}(\lambda_{n}^{-1}(a_{n}))}{\sqrt{n}}
								- \lambda_{n}^{-1}(a)\bigg| \longrightarrow_{\bP} 0 
\end{align}
as $n \longrightarrow \infty$. A Taylor expansion of $\lambda_{n}^{-1}$ around $a \in J_{n}$ reveals for some 
$\nu_{n} = \nu_{n}(a,a_{n}^{B})$ between $a_{n}^{B}$ and $a$ the identity 
\[
\lambda_{n}^{-1}(a_{n}^{B}) + \frac{B_{n}(\lambda_{n}^{-1}(a_{n}))}{\sqrt{n}}-\lambda_{n}^{-1}(a)
= \bigg(1-\frac{(\lambda_{n}^{-1})'(\nu_{n})}{(\lambda_{n}^{-1})'(a_{n})}\bigg)
								\frac{B_{n}(\lambda_{n}^{-1}(a_{n}))}{\sqrt{n}}.
\]
Evaluation of the right-hand side on $\Omega_n$ in terms of $\Phi_0$, $\delta_n$ and $p_X$ together with
\[
\sup_{a \in J_{n}}|\nu_{n}(a,a_{n}^{B}) - a_{n}| \mathds{1}_{\Omega_{n}}\leq 
\sup_{a \in J_{n}}|a_{n}^{B} - a|\mathds{1}_{\Omega_{n}} \leq K\delta_{n}(\log(n\delta_{n}^{2})/n)^{1/2}
\]
finally yields \eqref{eq: Claim3c} and Claim III follows.

\smallskip
\noindent
\textsc{Further notation: } Let $\displaystyle L_{n} \colon [0,1] \to \bR, 
L_{n}(t) \defeq \int_{0}^{t}\sigma_{n}^{2} \circ F_{X}^{-1}(u)du$ and define 
\[
U_{n}^{L} \colon [\lambda_{n}(0),\lambda_{n}(1)] \to [0,1], 
\quad 
U_{n}^{L}(a) \defeq L_{n}\big(\tilde{U}_{n}(a_{n}^{B})\big) - L_{n}\big(\lambda_{n}^{-1}(a)\big). 
\]

\smallskip
\noindent
\textsc{Claim IV: } $\cJ_{n,3} = \tilde{\cJ}_{n} + o_{\bP}(n^{-1/2})$, with 
$\displaystyle \tilde{\cJ}_{n} 
\defeq \int_{J_{n}}\bigg|\frac{U_{n}^{L}(a)}{L_{n}'(\lambda_{n}^{-1}(a))}\bigg|
																		\frac{1}{p_{X}(\Phi_{n}^{-1}(a))}da.$

\smallskip
\noindent
\textit{Proof of Claim IV. } It suffices to show that 
\[
\int_{J_{n}}\bigg(\big|\tilde{U}_{n}(a_{n}^{B}) - \lambda_{n}^{-1}(a)\big|
				- \bigg|\frac{U_{n}^{L}(a)}{L_{n}'(\lambda_{n}^{-1}(a))}\bigg|\bigg)
										\frac{1}{p_{X}(\Phi_{n}^{-1}(a))}da = o_{\bP}(n^{-1/2}). 
\]
As in the previous claim, we argue on $\Omega_{n}\defeq \big\{ \sup_{u \in [0,1]}|B_{n}(u)| \leq \sqrt{\log(n\delta_{n}^{2})}\big\} \subset \Omega$. A Taylor expansion of $L_{n}$ around $\lambda_{n}^{-1}(a)$ provides the equality 
\[
U_{n}^{L}(a)
= L_{n}'(\lambda_{n}^{-1}(a))\big(\tilde{U}_{n}(a_{n}^{B}) - \lambda_{n}^{-1}(a)\big) 
			+ \frac{1}{2}L_{n}''(\nu_{n})\big(\tilde{U}_{n}(a_{n}^{B}) - \lambda_{n}^{-1}(a)\big)^{2}
\]
for some $\nu_{n}$ between $\lambda_{n}^{-1}(a)$ and $\tilde{U}_{n}(a_{n}^{B})$. Recalling the definition
$\sigma_{n}^{2}(t) = \Phi_{n}(t)(1-\Phi_{n}(t))$, 
\[
L_{n}'(\lambda_{n}^{-1}(a)) 
= \Phi_{n}(\Phi_{n}^{-1}(a))\big(1-\Phi_{n}(\Phi_{n}^{-1}(a))\big) 
\geq \Phi_{n}(-T)(1-\Phi_{n}(T)) 
\geq K
\]
for all $a \in [\lambda_{n}(0),\lambda_{n}(1)]$, while
\[
|L_{n}''(\nu_{n})| 
= \big|(\sigma_{n}^{2} \circ F_{X}^{-1})'(\nu_{n})\big| 
= \bigg|\Phi_{n}'(F_{X}^{-1}(\nu_{n}))\frac{1-2\Phi_{n}(F_{X}^{-1}(\nu_{n}))}{p_{X}(F_{X}^{-1}(\nu_{n}))}\bigg|
\leq \delta_{n}K.
\]
Thus, by the reverse triangle inequality, 
\begin{align*}
\bigg|\big|\tilde{U}_{n}(a_{n}^{B}) &- \lambda_{n}^{-1}(a)\big| - \bigg|\frac{U_{n}^{L}(a)}{L_{n}'(\lambda_{n}^{-1}(a))}\bigg| \bigg| 
\leq \bigg|\tilde{U}_{n}(a_{n}^{B}) - \lambda_{n}^{-1}(a) - \frac{U_{n}^{L}(a)}{L_{n}'(\lambda_{n}^{-1}(a))}\bigg| \\
&= \bigg|\tilde{U}_{n}(a_{n}^{B}) - \lambda_{n}^{-1}(a) 
				- \big(\tilde{U}_{n}(a_{n}^{B}) - \lambda_{n}^{-1}(a)\big) - \frac{1}{2}\frac{L_{n}''(\nu_{n})}{L_{n}'(\lambda_{n}^{-1}(a))}\big(\tilde{U}_{n}(a_{n}^{B}) - \lambda_{n}^{-1}(a)\big)^{2}\bigg| \\
&\leq \delta_{n}K\big(\tilde{U}_{n}(a_{n}^{B}) - \lambda_{n}^{-1}(a)\big)^{2}
\end{align*}
for all $a \in J_{n}$. Consequently,
\[
\bE\bigg[\bigg| \big|\tilde{U}_{n}(a_{n}^{B}) - \lambda_{n}^{-1}(a)\big|
	- \bigg|\frac{U_{n}^{L}(a)}{L_{n}'(\lambda_{n}^{-1}(a))}\bigg| \bigg|\mathds{1}_{\Omega_{n}}\bigg] 
\leq \delta_{n}K\bE\Big[\big(\tilde{U}_{n}(a_{n}^{B}) - \lambda_{n}^{-1}(a)\big)^{2}\mathds{1}_{\Omega_{n}}\Big],
\]
which is bounded by $K(n\delta_{n}^{2})^{-2/3}\delta_{n}$ by Corollary~\ref{cor:tail bounds - 1}, applied with 
$c_{n} = \delta_{n}(\log(n\delta_{n}^{2})/n)^{1/2}$ and $Z_{2,n} = 0$. 
Markov's inequality, Fubini's theorem and $\lambda_{n}(1)-\lambda_{n}(0) = \cO(\delta_{n})$ then reveal
for any $\varepsilon > 0$ the bound
\[
\bP\bigg(\sqrt{n}\int_{J_{n}}\hspace{-1.5mm}\bigg(\big|\tilde{U}_{n}(a_{n}^{B}) - \lambda_{n}^{-1}(a)\big|
		- \bigg|\frac{U_{n}^{L}(a)}{L_{n}'(\lambda_{n}^{-1}(a))}\bigg|\bigg)
											\frac{da}{p_{X}(\Phi_{n}^{-1}(a))} > \varepsilon, \,\Omega_n\bigg) 
\leq \frac{K\delta_{n}}{\varepsilon(n\delta_{n}^{2})^{1/6}}.
\]

\medskip
\noindent
The goal of the next claim is to bring another strong Gaussian approximation into play, namely standard Brownian 
motions $W_{n}$ given $X_1,\dots, X_n$, conditionally independent of the Brownian bridges $B_{n}$ of the KMT approximation, 
that satisfy for some constant $A>0$
\begin{align}\label{eq: Sak85}
\bE\bigg[\sup_{t \in [0,1]}\bigg|\Upsilon_{n}(t) - \int_{0}^{t}\Phi_{n} \circ F_{n}^{-1}(u)du 
															- \frac{W_{n}(L^{n}(t))}{\sqrt{n}}\bigg|^{q}\bigg| X_1,\dots, X_n\bigg]
\leq An^{1-q}.
\end{align}
Noting that $Y_{i}^{n}-\Phi_{n}(X_{i})$, $i=1,\dots, n$, are bounded and conditionally centered and independent given $X_1,\dots, X_n$, existence of such $W_n$'s is guaranteed by \cite{Sakhanenko1985invariance}.

\smallskip
\noindent
\textsc{Further notation: } Define 
\[
L^{n}(t) \defeq \int_{0}^{t}\sigma_{n}^{2} \circ F_{n}^{-1}(u)du, 
\qquad 
\psi_{n}(t) \defeq \frac{L_{n}''(t)}{\sqrt{n}L_{n}'(t)}B_{n}(t), 
\qquad
d_{n}(t) \defeq \sqrt{\delta_{n}}\frac{|\lambda_{n}'(t)|}{2L_{n}'(t)^{2}} 
\]
for $t \in [0,1]$ and let $\bP^{|X}$ denote the conditional measure given $(X_{1},\dots,X_{n})$. 
For 
$n$ large enough to ensure the subsequent expression being well-defined for $|u| \leq (\frac{n}{\delta_{n}})^{1/3}L^{n}(t)$ 
and any $t \in (0,1)$, define for the $\bP^{|X}$-Brownian motions $W_n$ fulfilling \eqref{eq: Sak85}, 
\[
W_{t}^{n}(u) 
\defeq \frac{1}{\sqrt{1-\psi_{n}(t)}}\Big(\frac{n}{\delta_{n}}\Big)^{1/6}\Big(W_{n}\Big(L^{n}(t) 
							+ \Big(\frac{n}{\delta_{n}}\Big)^{-1/3}u(1-\psi_{n}(t))\Big) - W_{n}\big(L^{n}(t)\big)\Big).
\]
Note that $W_{t}^{n}$ is therefore 
distributed as a standard two-sided Brownian motion under $\bP^{|X}$ for every $t \in (0,1)$. In addition, define 
$\tilde{V}_{n}(t) 
\defeq \argmin_{|u| \leq \delta_{n}^{-1}\log(n\delta_{n}^{2})}\{W_{t}^{n}(u) + d_{n}(t)u^{2}\}$
and set 
\[
\tilde{\cJ}_{n,1} 
\defeq \int_{0}^{1}|\tilde{V}_{n}(t)|\frac{|\Phi_{n}' \circ F_{X}^{-1}(t)|}{(p_{X} \circ F_{X}^{-1}(t))^{2}|L_{n}'(t)|}dt.
\]

\smallskip
\noindent
\textsc{Claim V: } For the $\bP^{|X}$-standard Brownian motions $W_{n}$ from \eqref{eq: Sak85}, the distribution of both $(n\delta_{n}^{2})^{1/6}(\tilde{\cJ}_{n,1} - \mu_{n})$ and 
$(n\delta_{n}^{2})^{1/6}((\frac{n}{\delta_{n}})^{1/3}\tilde{\cJ}_{n} - \mu_{n})$ with respect to $\bP^{|X}$ have the same weak limit in probability. 

\smallskip
\noindent
\textit{Proof of Claim V. } 
Let us define 
\begin{align}\label{eq:Tn}
T_{n} \defeq \delta_{n}^{-1}(n\delta_{n}^{2})^\frac{1}{3(3q-5)}
\end{align}
for some $q \geq 12$ and let $\Omega_{n}' \subset \Omega$ denote the measurable set on which the following inequalities hold 
\begin{align*}
\sup_{u \in [0,1]}&|B_{n}(u)| 
\leq \log(n\delta_{n}^{2}), \quad
\sup_{u \in [0,1]}\bigg|F_{X}^{-1}(u) - F_{n}^{-1}(u) 
									- \frac{B_{n}(F_{X} \circ F_{n}^{-1}(u))}{\sqrt{n}p_{X}(F_{n}^{-1}(u))}\bigg| 
\leq \frac{\log(n)^{2}}{n}, \\
&\sup_{|u-v| \leq T_{n}(\frac{n}{\delta_{n}})^{-1/3}\sqrt{\log(n)}}|B_{n}(u) - B_{n}(v)| 
\leq \sqrt{T_{n}}\Big(\frac{n}{\delta_{n}}\Big)^{-1/6}\log(n),
\end{align*}
where $B_{n}$ denote the Brownian bridges from Claim II. Note that 
$\bP(\Omega_{n}') \longrightarrow 1$ for $n \longrightarrow \infty$, so w.l.o.g.~we prove the 
claim on $\Omega_{n}'$. For readability, the proof is divided into multiple steps. \\
$\bullet$ For every $a \in J_{n}$, let us introduce 
\[
I_{n}(a) 
\defeq \Big[\Big(\frac{n}{\delta_{n}}\Big)^{1/3}\big(L_{n}(0)-L_{n}(\lambda_{n}^{-1}(a))\big),
							\Big(\frac{n}{\delta_{n}}\Big)^{1/3}\big(L_{n}(1)-L_{n}(\lambda_{n}^{-1}(a))\big)\Big], 
\]
which is the subset over which the $\argmin$ in the definition of $U_{n}^{L}(a)$ is considered, as shown subsequently.
By using elementary properties of the $\argmin$, 
\begin{align*}
\Big(\frac{n}{\delta_{n}}\Big)^{1/3}U_{n}^{L}(a) 
&= \Big(\frac{n}{\delta_{n}}\Big)^{1/3}\bigg(L_{n}\bigg(\argmin_{u \in [0,1]}
				\{\Upsilon_{n}(u) - a_{n}^{B}u\}\bigg) - L_{n}(\lambda_{n}^{-1}(a))\bigg) \\
&= \Big(\frac{n}{\delta_{n}}\Big)^{1/3}\bigg(\argmin_{v \in [L_{n}(0),L_{n}(1)]}
				\big\{\big(\Upsilon_{n} \circ L_{n}^{-1} - a_{n}^{B}L_{n}^{-1}\big)(v)\big\} - L_{n}(\lambda_{n}^{-1}(a))\bigg) \\
&= \Big(\frac{n}{\delta_{n}}\Big)^{1/3}\hspace{-2mm}\argmin_{v: \, v + L_{n}(\lambda_{n}^{-1}(a)) \in [L_{n}(0),L_{n}(1)]}
				\big\{\big(\Upsilon_{n} \circ L_{n}^{-1} - a_{n}^{B}L_{n}^{-1}\big)\big(v + L_{n}(\lambda_{n}^{-1}(a))\big)\big\} \\
&= \argmin_{v \in I_{n}(a)}\Big\{\big(\Upsilon_{n} \circ L_{n}^{-1} 
							- a_{n}^{B}L_{n}^{-1}\big)\Big(\Big(\frac{n}{\delta_{n}}\Big)^{-1/3}v + L_{n}(\lambda_{n}^{-1}(a))\Big)\Big\} \\
&= \argmin_{v \in I_{n}(a)}\bigg\{\frac{n^{2/3}}{\delta_{n}^{1/6}}\big(\Upsilon_{n} \circ L_{n}^{-1}
							- a_{n}^{B}L_{n}^{-1}\big)\Big(\Big(\frac{n}{\delta_{n}}\Big)^{-1/3}v + L_{n}(\lambda_{n}^{-1}(a))\Big) \\
	&\qquad\qquad\qquad - \frac{n^{2/3}}{\delta_{n}^{1/6}}\big(\Lambda_{n}(\lambda_{n}^{-1}(a)) - a\lambda_{n}^{-1}(a)\big) 
											- \frac{n^{2/3}}{\delta_{n}^{1/6}}(a-a_{n}^{B})\lambda_{n}^{-1}(a)\bigg\}.
\end{align*}
Defining further for $a \in J_{n}$ and $u \in I_{n}(a)$, 
\begin{align*}
D_{n}(a,u) 
&\defeq \frac{n^{2/3}}{\delta_{n}^{1/6}}\big(\Lambda_{n} \circ L_{n}^{-1} - aL_{n}^{-1}\big)
									\Big(\Big(\frac{n}{\delta_{n}}\Big)^{-1/3}u + L_{n}(\lambda_{n}^{-1}(a))\Big) - \frac{n^{2/3}}{\delta_{n}^{1/6}}\big(\Lambda_{n}(\lambda_{n}^{-1}(a)) - a\lambda_{n}^{-1}(a)\big), \\
R_{n}(a,u) 
&\defeq \frac{n^{2/3}}{\delta_{n}^{1/6}}\int_{\lambda_{n}^{-1}(a)}^{L_{n}^{-1}((\frac{n}{\delta_{n}})^{-1/3}u 
				+ L_{n}(\lambda_{n}^{-1}(a)))}\Phi_{n} \circ F_{n}^{-1}(x) - \Phi_{n} \circ F_{X}^{-1}(x)dx \\
	&\quad + \frac{n^{2/3}}{\delta_{n}^{1/6}}(a-a_{n}^{B})
					\Big(L_{n}^{-1}\Big(\Big(\frac{n}{\delta_{n}}\Big)^{-1/3}u + L_{n}(\lambda_{n}^{-1}(a))\Big) - \lambda_{n}^{-1}(a)\Big), \\
\tilde{R}_{n}(a,u) 
&\defeq \frac{n^{2/3}}{\delta_{n}^{1/6}}\Upsilon_{n} \circ L_{n}^{-1}\Big(\Big(\frac{n}{\delta_{n}}\Big)^{-1/3}u 
				+ L_{n}(\lambda_{n}^{-1}(a))\Big) \\
	&\quad - \frac{n^{2/3}}{\delta_{n}^{1/6}}\int_{0}^{L_{n}^{-1}((\frac{n}{\delta_{n}})^{-1/3}u 
														+ L_{n}(\lambda_{n}^{-1}(a)))}\hspace{-0.40em}\Phi_{n} \circ F_{n}^{-1}(x)dx - W_{\lambda_{n}^{-1}(a)}^{n}(u) - \frac{n^{1/6}}{\delta_{n}^{1/6}}W_{n}(L_{n}(\lambda_{n}^{-1}(a))), 
\end{align*} 
we see for every $a \in J_{n}$, 
\[
\Big(\frac{n}{\delta_{n}}\Big)^{1/3}U_{n}^{L}(a) 
= \argmin_{u \in I_{n}(a)}\big\{D_{n}(a,u) + W_{\lambda_{n}^{-1}(a)}^{n}(u) + R_{n}(a,u) + \tilde{R}_{n}(a,u)\big\}, 
\]
where the expressions in the $\argmin$ on the right-hand side deviate from the expressions in the $\argmin$ on the 
left-hand side only by a term which does not depend on $u$. \\
$\bullet$ Before we show in the next step that both, $R_{n}$ and $\tilde{R}_{n}$, are negligible for the 
location of the $\argmin$, we localize. So let 
\[
\hat{U}_{n}^{L}(a) 
\defeq \argmin_{|u| \leq T_{n}}\big\{D_{n}(a,u) + W_{\lambda_{n}^{-1}(a)}^{n}(u) + R_{n}(a,u) + \tilde{R}_{n}(a,u)\big\}
\]
for $a \in J_{n}$ and note that $[-T_{n},T_{n}] \subset I_{n}(a)$ at least for $n$ large 
enough, with $T_{n}$ defined in \eqref{eq:Tn}. This follows from 
\[
T_{n} 
= n^{\frac{1}{3(3q-5)}}\delta_{n}^{-\frac{(9q-17)}{(9q-15)}} 
= n^{\frac{1}{3(3q-5)}}\delta_{n}^{-\frac{1}{3(3q-5)}}\delta_{n}^{-\frac{(3q-6)}{(3q-5)}} 
= \Big(\frac{n}{\delta_{n}}\Big)^{\frac{1}{3(3q-5)}}\delta_{n}^{-\frac{(3q-6)}{(3q-5)}} 
\]
and 
\begin{equation}\label{eq:Tn rate}
\Big(\frac{n}{\delta_{n}}\Big)^{-1/3}T_{n} 
= \Big(\frac{n}{\delta_{n}}\Big)^{-\frac{(3q-6)}{3(3q-5)}}\delta_{n}^{-\frac{(3q-6)}{(3q-5)}}
= (n\delta_{n}^{2})^{-\frac{(3q-6)}{3(3q-5)}}
= (n\delta_{n}^{2})^{-\frac{q-2}{3q-5}}. 
\end{equation}
Note further that $(\frac{n}{\delta_{n}})^{1/3}U_{n}^{L}(a)$ 
differs from $\hat{U}_{n}^{L}(a)$ if and only if $(\frac{n}{\delta_{n}})^{1/3}|U_{n}^{L}(a)| > T_{n}$. But then, by a Taylor expansion 
of $L_{n}$ around $\lambda_{n}^{-1}(a)$, Corollary~\ref{cor:tail bounds - 1} and by definition of $T_{n}$, 
\begin{align*}
\bP\Big(\Big(\frac{n}{\delta_{n}}\Big)^{1/3}U_{n}^{L}(a) \neq \hat{U}_{n}^{L}(a), \Omega_{n}'\Big) 
&= \bP\Big(|L_{n}(\tilde{U}_{n}(a_{n}^{B})) - L_{n}(\lambda_{n}^{-1}(a))| > T_{n}\Big(\frac{n}{\delta_{n}}\Big)^{-1/3}, \Omega_{n}'\Big) \\
&\leq K(\delta_{n}T_{n})^{-3q/2} 
= K(n\delta_{n}^{2})^{-\frac{q}{2(3q-5)}}.
\end{align*}
Using this inequality, we have for any $\varepsilon > 0$ and $n$ large enough, by Markov's inequality, Fubini's 
theorem, Hölder's inequality and Minkowski's inequality, that 
\begin{align*}
&\bP\bigg((n\delta_{n}^{2})^{1/6}\bigg|\int_{J_{n}}
						\bigg(\bigg|\frac{(n/\delta_{n})^{1/3}U_{n}^{L}(a)}{L_{n}'(\lambda_{n}^{-1}(a))}\bigg| 
								- \bigg|\frac{\hat{U}_{n}^{L}(a)}{L_{n}'(\lambda_{n}^{-1}(a))}\bigg|\bigg)
												\frac{1}{p_{X}(\Phi_{n}^{-1}(a))}da\bigg| > \varepsilon, \Omega_{n}'\bigg) \\
&\qquad\leq \frac{(n\delta_{n}^{2})^{1/6}}{\varepsilon}\int_{J_{n}}
						\bE\big[\mathds{1}_{\{(n/\delta_{n})^{1/3}U_{n}^{L}(a) \neq \hat{U}_{n}^{L}(a)\}}\mathds{1}_{\Omega_{n}'}\big]^{(r-1)/r} \\
&\qquad\qquad\qquad\cdot 
			\bigg(\bE\bigg[\bigg|\frac{(n/\delta_{n})^{1/3}U_{n}^{L}(a)}{L_{n}'(\lambda_{n}^{-1}(a))}\bigg|^{r}\mathds{1}_{\Omega_{n}'}\bigg]^{1/r} 
					+ \bE\bigg[\bigg|\frac{\hat{U}_{n}^{L}(a)}{L_{n}'(\lambda_{n}^{-1}(a))}\bigg|^{r}\mathds{1}_{\Omega_{n}'}\bigg]^{1/r}\bigg)
																		\frac{1}{p_{X}(\Phi_{n}^{-1}(a))}da, 
\end{align*}
which is bounded by $\frac{K}{\varepsilon}(n\delta_{n}^{2})^{-\frac{q(r-1)}{2(3q-5)r}}$ and 
where we used Corollary~\ref{cor:tail bounds - 1}, resulting in expectation bounds of 
respective order $\delta_{n}^{-1}$, compensated by an upper bound on the length of the integration domain 
$\lambda_{n}(1)-\lambda_{n}(0) = \cO(\delta_{n})$. Choosing $r$ smaller but sufficiently close to 
$3q/2$, this is bounded by $\frac{K}{\varepsilon}(n\delta_{n}^{2})^{-\frac{q(3q-2)}{4(3q-5)}}$, whence
\[
\Big(\frac{n}{\delta_{n}}\Big)^{1/3}\tilde{\cJ}_{n}
= \int_{J_{n}}\bigg|\frac{\hat{U}_{n}^{L}(a)}{L_{n}'(\lambda_{n}^{-1}(a))}\bigg|
												\frac{1}{p_{X}(\Phi_{n}^{-1}(a))}da + o_{\bP}((n\delta_{n}^{2})^{-1/6}). 
\]
$\bullet$ Now we show that $R_{n}$ and $\tilde{R}_{n}$ are actually negligible, i.e.~we prove that $\hat{U}_{n}^{L}$ can 
be replaced in the previous integral by the following process, where $S_{n} \defeq \delta_{n}^{-1}\log(n\delta_{n}^{2})$, 
\[
\hat{V}_{n} \colon [\lambda_{n}(0),\lambda_{n}(1)] \to \bR, 
\quad \hat{V}_{n}(a) \defeq \argmin_{|u| \leq S_{n}}\big\{D_{n}(a,u) + W_{\lambda_{n}^{-1}(a)}^{n}(u)\big\}. 
\]
For ease of notation, let us also introduce 
\[
\hat{V}_{n}^{*} \colon [\lambda_{n}(0),\lambda_{n}(1)] \to \bR, 
\quad \hat{V}_{n}^{*}(a) \defeq \argmin_{|u| \leq T_{n}}\big\{D_{n}(a,u) + W_{\lambda_{n}^{-1}(a)}^{n}(u)\big\}, 
\]
where \eqref{eq:Tn rate} guarantees $|u| \leq (\frac{n}{\delta_{n}})^{1/3}L^{n}(t)$ for all $|u| \leq T_{n}$.
Note that $\hat{V}_{n}^{*}(a)$ differs from $\hat{V}_{n}$ if and only if $|\hat{V}_{n}^{*}(a)| > S_{n}$ and it 
follows from Proposition~1 of \cite{Durot2002sharp} together with the comments just before this Proposition, that there 
exists $K > 0$, such that for every $(x,\alpha)$, satisfying 
$\alpha \in \big(0,S_{n}\big]$, $x > 0$ and $K\delta_{n}^{3}S_{n}^{2} \leq -(\alpha\log(2x\alpha))^{-1}$, 
\begin{equation}\label{eq:localized tails}
\begin{aligned}
\bP^{|X}\big(|&\hat{U}_{n}^{L}(a) - \hat{V}_{n}(a)| > \alpha, \Omega_{n}'\big) \\
&\leq \bP^{|X}\big(|\hat{U}_{n}^{L}(a) - \hat{V}_{n}^{*}(a)| > \alpha/2, \Omega_{n}'\big) 
					+ \bP^{|X}\big(|\hat{V}_{n}(a) - \hat{V}_{n}^{*}(a)| > \alpha/2, \Omega_{n}'\big) \\
&\leq \bP^{|X}\bigg(2\sup_{|u| \leq T_{n}}|R_{n}(a,u) + \tilde{R}_{n}(a,u)| > x(\alpha/2)^{3/2}, \Omega_{n}'\bigg) \\
	&\qquad\qquad+ KS_{n}x + \bP^{|X}\big(|\hat{V}_{n}^{*}(a)| > S_{n}, \Omega_{n}'\big) + \bP^{|X}\big(|\hat{V}_{n}^{*}(a)| > S_{n}, \Omega_{n}'\big) \\
&\leq K(x\alpha^{3/2})^{-q}\bE^{|X}\bigg[\sup_{|u| \leq T_{n}}|R_{n}(a,u) + \tilde{R}_{n}(a,u)|^{q}\mathds{1}_{\Omega_{n}'}\bigg] + KS_{n}x + 2\bP^{|X}\big(|\hat{V}_{n}^{*}(a)| > S_{n}, \Omega_{n}'\big), 
\end{aligned}
\end{equation}
where we also applied Markov's inequality in the last step. 
Before deriving an upper bound on the expectation involving $R_{n}$ and $\tilde{R}_{n}$, let us consider the 
probability involving $\hat{V}_{n}^{*}$. Noting that $D_{n}(a,0) = 0$ and that a Taylor expansion of 
$\Lambda_{n} \circ L_{n}^{-1} - aL_{n}^{-1}$ around $L_{n}(\lambda_{n}^{-1})$ reveals 
$|D_{n}(a,u)| \geq \delta_{n}^{3/2}\kappa u^{2}$ for some $\kappa > 0$ and $|u| \leq S_{n}$, using that the first 
expansion term vanishes, Theorem~4 of \cite{Durot2002sharp} yields 
\begin{align}\label{eq:exp inequality}
\bP^{|X}\big(|\hat{V}_{n}^{*}(a)| > S_{n},\Omega_{n}'\big) 
\leq K\exp(-\kappa^{2}\delta_{n}^{3}S_{n}^{3}/2) 
\leq K\exp(-\kappa^{2}\log(n\delta_{n}^{2})^{3}/2).
\end{align}
By Lemma~\ref{lem:uniform bound}, 
\[
\bE^{|X}\bigg[\sup_{|u| \leq T_{n}}|R_{n}(a,u) + \tilde{R}_{n}(a,u)|^{q}\mathds{1}_{\Omega_{n}'}\bigg] 
\leq Kn^{1-q/3}\delta_{n}^{-q/6} 
\]
and we obtain together with \eqref{eq:localized tails} and \eqref{eq:exp inequality}, 
\[
\bP^{|X}\big(|\hat{U}_{n}^{L}(a) - \hat{V}_{n}(a)| > \alpha, \Omega_{n}'\big) 
\leq K(x\alpha^{3/2})^{-q}n^{1-q/3}\delta_{n}^{-q/6} + KS_{n}x 
\]
for every $(x,\alpha)$, satisfying $\alpha \in \big(0,S_{n}\big]$, $x > 0$ and 
$K\delta_{n}^{3}S_{n}^{2} \leq -(\alpha\log(2x\alpha))^{-1}$. Now note that for any $\varepsilon > 0$, for every 
$\alpha \in ((n\delta_{n}^{2})^{-1/6}\delta_{n}^{-1}/\log(n\delta_{n}^{2}),(n\delta_{n}^{2})^{-\varepsilon}\delta_{n}^{-1}]$ 
and 
\[
x_{\alpha,n} \defeq S_{n}^{-1/(q+1)}\alpha^{-3q/(2(q+1))}n^{(3-q)/(3(q+1))}\delta_{n}^{-q/(6(q+1))}, 
\]
we have $(x_{\alpha,n}\alpha^{3/2})^{-q}n^{1-q/3}\delta_{n}^{-q/6} \leq S_{n}x_{\alpha,n}$ and 
$\alpha x_{\alpha,n} \longrightarrow 0$ for $n \longrightarrow \infty$ and so $(\alpha,x_{\alpha,n})$ does in fact satisfy 
$-(\alpha\log(2x_{\alpha,n}\alpha)^{-1} \geq K\delta_{n}^{3}S_{n}^{2}$. Thus, 
$\bP^{|X}\big(|\hat{U}_{n}^{L}(a) - \hat{V}_{n}(a)| > \alpha, \Omega_{n}'\big) 
\leq KS_{n}x_{\alpha,n}$. By definition, $|\hat{U}_{n}^{L}(a) - \hat{V}_{n}(a)|$ is bounded by $2T_{n}$ and thus, using that $q > 12$, 
\begin{align*}
\int_{J_{n}}\bE^{|X}&\big[\big|\hat{U}_{n}^{L}(a) - \hat{V}_{n}(a)\big|\mathds{1}_{\Omega_{n}'}\big]da
= \int_{J_{n}}\int_{0}^{2T_{n}}\bP^{|X}\big(|\hat{U}_{n}^{L}(a) - \hat{V}_{n}(a)| > \alpha,\Omega_{n}'\big)d\alpha da \\
&\leq K\delta_{n}\bigg((n\delta_{n}^{2})^{-1/6}\delta_{n}^{-1}/\log(n\delta_{n}^{2}) 
							 + KT_{n}S_{n}x_{(n\delta_{n}^{2})^{-\varepsilon}\delta_{n}^{-1}} + K\int_{(n\delta_{n}^{2})^{-1/6}\delta_{n}^{-1}/\log(n\delta_{n}^{2})}^{(n\delta_{n}^{2})^{-\varepsilon}\delta_{n}^{-1}}S_{n}x_{\alpha,n}d\alpha\bigg) \\
&\leq K(n\delta_{n}^{2})^{-1/6}/\log(n\delta_{n}^{2}). 
\end{align*} 
Consequently, for any $\varepsilon > 0$, by Markov's inequality and Fubini's theorem, 
\[
\bP^{|X}\bigg((n\delta_{n}^{2})^{1/6}\int_{J_{n}}\bigg|\frac{|\hat{U}_{n}^{L}(a)|-|\hat{V}_{n}(a)|}{L_{n}'(\lambda_{n}^{-1}(a))}\bigg|
													\frac{1}{p_{X}(\Phi_{n}^{-1}(a))}da > \varepsilon, \Omega_{n}'\bigg) 
= o_{\bP}(1).
\]
$\bullet$ In the last step, we approximate the integral over $\hat{V}_{n}$ by the integral over 
$\tilde{V}_{n} \circ \lambda_{n}^{-1}$, where first the integration domain $\cJ_{n}$ can be easily replaced by 
$[\lambda_{n}(0),\lambda_{n}(1)]$. As the remaining proof is very similar to the one in the previous step, we defer it to 
Lemma~\ref{lem:auxiliary}, showing that for any $\varepsilon > 0$,
\[
\bP^{|X}\bigg((n\delta_{n}^{2})^{1/6}\int_{\lambda_{n}(0)}^{\lambda_{n}(1)}\bigg|\frac{|\hat{V}_{n}(a)|-|\tilde{V}_{n}(\lambda_{n}^{-1}(a))|}{L_{n}'(\lambda_{n}^{-1}(a))}\bigg|\frac{1}{p_{X}(\Phi_{n}^{-1}(a))}da > \varepsilon, \Omega_{n}'\bigg) 
= o_{\bP}(1).
\]
A change of variable, where $a$ is replaced by $\lambda_{n}(a)$, then proves Claim V.

\medskip
\noindent
\textsc{Claim VI: } The distribution of $(n\delta_{n}^{2})^{1/6}\big(\tilde{\cJ}_{n} - \mu_{n}\big)$ under $\bP^{|X}$ converges weakly in probability to a normal distribution with mean zero and variance $\sigma^{2}$, defined in \eqref{eq:L1var}.

\smallskip
\noindent
\textit{Proof of Claim VI. } 
As in the proof of Claim V, we show the assertion without loss of generality on $\Omega_{n}'$, as 
$\bP(\Omega_{n}') \longrightarrow 1$ for $n \longrightarrow \infty$, with $\Omega_{n}'$ defined at the beginning of the proof 
of Claim V. Let 
\[
V_{n} \colon [0,1] \to \bR, \quad
V_{n}(t) \defeq \argmin_{u \in \bR}\{W_{t}^{n}(u) + d_{n}(t)u^{2}\}, 
\]
denote a variation of $\tilde{V}_{n}$ where the $\argmin$ is now considered over the whole real line instead of 
$[-S_{n},S_{n}]$, recalling $S_{n} = \delta_{n}^{-1}\log(n\delta_{n}^{2})$. Further, define 
\[
\eta_{n} \colon [0,1] \to [0,\infty), \quad 
\eta_{n}(t) \defeq \frac{|\Phi_{n}' \circ F_{X}^{-1}(t)|}{(p_{X} \circ F_{X}^{-1}(t))^{2}}
\]
and set 
\[
Y_{n}(t) 
\defeq \bigg(\bigg|\frac{\tilde{V}_{n}(t)}{L_{n}'(t)}\bigg| 
						- \bE^{|X}\bigg[\bigg|\frac{\tilde{V}_{n}(t)}{L_{n}'(t)}\bigg|\bigg]\bigg)\eta_{n}(t) 
\]
for $t \in [0,1]$. Note that $V_{n}(t)$ can differ from $\tilde{V}_{n}(t)$ only if 
$V_{n}(t) > S_{n}$ and so we have by Theorem~4 of \cite{Durot2002sharp} that there exists 
$\kappa > 0$, such that
\begin{align*}
\bP^{|X}\big(V_{n}(t) \neq \tilde{V}_{n}(t)\big) 
&\leq \bP^{|X}(V_{n}(t) > S_{n})
\leq 2\exp(-\kappa^{2}\delta_{n}^{3}S_{n}^{3}/2) 
= 2\exp(-\kappa^{2}\log(n\delta_{n}^{2})^{3}/2). 
\end{align*}
Note further that under $\bP^{|X}$, both $\tilde{V}_{n}(t)$ and $V_{n}(t)$
have finite second moments and that $\eta_{n}(t)$ is bounded. So by Fubini's theorem and the Cauchy-Schwarz inequality, 
\[
\bE^{|X}\bigg[\int_{0}^{1}\bigg(\bigg|\frac{\tilde{V}_{n}(t)}{L_{n}'(t)}\bigg| 
								- \bigg|\frac{V_{n}(t)}{L_{n}'(t)}\bigg|\bigg)\eta_{n}(t)dt\bigg] 
\leq K\int_{0}^{1}\bP^{|X}\big(V_{n}(t) \neq \tilde{V}_{n}(t)\big)^{1/2} dt
\leq K(n\delta_{n}^{2})^{-1/6}/\log(n\delta_{n}^{2}).
\]
Combining this with the fact that $d_{n}(t)^{2/3}V_{n}(t)$ is distributed as $X(0)$ for any $t$, we have shown that 
\begin{align*}
&\bE^{|X}\bigg[\int_{0}^{1}\bigg|\frac{\tilde{V}_{n}(t)}{L_{n}'(t)}\bigg|\eta_{n}(t)dt\bigg] \\
&\quad= \bE[|X(0)|]\int_{0}^{1}\delta_{n}^{-1/3}(L_{n}'(t))^{1/3}\bigg(\frac{2}{|\lambda_{n}'(t)|}\bigg)^{2/3}\eta_{n}(t)dt 
			+ o_{\bP}((n\delta_{n}^{2})^{-1/6}) \\
&\quad= \int_{0}^{1}\delta_{n}^{-1/3}\big(4\sigma_{n}^{2} \circ F_{X}^{-1}(t)\big)^{1/3}
						\bigg(\frac{p_{X} \circ F_{X}^{-1}(t)}{|\Phi_{n}' \circ F_{X}^{-1}(t)|}\bigg)^{2/3} 
									\frac{|\Phi_{n}' \circ F_{X}^{-1}(t)|}{(p_{X} \circ F_{X}^{-1}(t))^{2}}dt 
			+ o_{\bP}((n\delta_{n}^{2})^{-1/6}) \\
&\quad= \int_{-T}^{T}\delta_{n}^{-1/3}\big(4\sigma_{n}^{2}(t)\Phi_{n}'(t)\big)^{1/3}p_{X}(t)^{-1/3}dt 
			+ o_{\bP}((n\delta_{n}^{2})^{-1/6}) \\
&\quad= \mu_{n} + o_{\bP}((n\delta_{n}^{2})^{-1/6}).
\end{align*}
It remains to prove that the distribution of 
$(n\delta_{n}^{2})^{1/6}\int_{0}^{1}Y_{n}(t)dt$
under $\bP^{|X}$ converges weakly to $\cN(0,\sigma^{2})$ in probability, as $n \longrightarrow \infty$. For this, we introduce 
\[
v_{n} 
\defeq \Var^{|X}\bigg((n\delta_{n}^{2})^{1/6}\int_{0}^{1}Y_{n}(t)dt\bigg) 
= (n\delta_{n}^{2})^{1/3}\Var^{|X}\bigg((n\delta_{n}^{2})^{1/6}\int_{0}^{1}Y_{n}(t)dt\bigg) 
\]
and note that similar as in the calculation of $\mu_{n}$ in the previous display and by virtually the same arguments as in Step~5 of \cite{Durot2008monotone}, we obtain 
$
v_{n} = \sigma^{2} + o_{\bP}(1). 
$
Asymptotic normality of $(n\delta_{n}^{2})^{1/6}\int_{0}^{1}Y_{n}(t)dt$ can now be deduced as in Step~6 of 
\cite{Durot2007Lp} by Bernstein's method of big blocks and small blocks, where the only difference lies in the 
replacement of $n$ by $n\delta_{n}^{2}$. 

\medskip
\noindent
\textsc{Conclusion: } By combining Claims I -- VI, 
\begin{align*}
(n\delta_{n}^{2})^{1/6}\Big(\Big(\frac{n}{\delta_{n}}\Big)^{1/3}\tilde{\cJ}_{n} - \mu_{n}\Big) 
&= (n\delta_{n}^{2})^{1/6}\Big(\Big(\frac{n}{\delta_{n}}\Big)^{1/3}\big(\tilde{\cJ}_{n} + o_{\bP}(n^{-1/2})\big) - \mu_{n}\Big) \\
&= (n\delta_{n}^{2})^{1/6}\Big(\Big(\frac{n}{\delta_{n}}\Big)^{1/3}\tilde{\cJ}_{n} - \mu_{n}\Big ) + o_{\bP}(1) 
\longrightarrow_{\cL} \cN(0,\sigma^{2}) 
\end{align*}
for $n \longrightarrow \infty$, unconditionally under $\bP$. 
\end{proof}

\section{Proof of the auxiliary result of Section~\ref{proof:4.4ii}}
\begin{lemma}\label{lem: 4.5}
Let $A_{n} \colon [-T,T] \to \bR$ denote the continuous, piecewise linear process satisfying 
\[
A_{n}(X_{i}) 
= \frac{1}{\sqrt{n}}\sum_{\ell=1}^{n}(Y_{\ell}^{n}-\Phi_{0}(0))
						\big(1-2\mathds{1}_{\{X_{\ell} \leq X_{i}\}}\big)
\]
for $i \in \{1,\dots,n\}$. Then, as $n\delta_n^2\longrightarrow 0$ and $n\longrightarrow\infty$, 
\[
A_{n} \longrightarrow_{\cL} A \ \text{ in } \cC([-T,T]),
\] 
where $A$ is defined in Theorem~\ref{thm:l1 rate} (ii) and where $\cC([-T,T])$
denotes the space of continuous functions $\cC([-T,T])$ endowed with the topology of uniform convergence.
\end{lemma}

\begin{proof}
As the processes $A_{n}$ are already continuous, we may rely on the classical theory of weak convergence on 
Polish spaces. To this aim, we need to prove convergence of finite-dimensional distributions to $A$, as well as tightness of the sequence 
$(A_{n})_{n\in\bN}$ in $\cC([-T,T])$ (cf. Theorem~7.3 in \cite{Billingsley1999}). But first, 
for any $s \in [-T,T]$, let $i(s) \in \{0,\dots,n\}$ be the random index that satisfies 
$X_{i(s)} \leq s < X_{i(s)+1}$. Then, 
\begin{align*}
A_{n}(s) 
&= \frac{1}{\sqrt{n}}\sum_{\ell=1}^{n}\Big\{(Y_{\ell}^{n}-\Phi_{0}(0))
						\big(1-2\mathds{1}_{\{X_{\ell} \leq X_{i(s)}\}}\big) + (s-X_{i(s)})\big(A_{n}(X_{i(s)+1}) - A_{n}(X_{i(s)})\big)\Big\} \\
&= \frac{1}{\sqrt{n}}\sum_{\ell=1}^{n}(Y_{\ell}^{n}-\Phi_{0}(0))
						\big(1-2\mathds{1}_{\{X_{\ell} \leq s\}}\big) \\
	&\qquad\qquad\qquad\qquad	+ 2\frac{(s-X_{i(s)})}{\sqrt{n}}\sum_{\ell=1}^{n}(Y_{\ell}^{n}-\Phi_{0}(0))
						\big(\mathds{1}_{\{X_{\ell} \leq X_{i(s)}\}} - \mathds{1}_{\{X_{\ell} \leq X_{i(s)+1}\}}\big).
\end{align*}
Note that 
\[
\sup_{s \in [-T,T]}\bigg|\frac{(s-X_{i(s)})}{\sqrt{n}}\sum_{\ell=1}^{n}(Y_{\ell}^{n}-\Phi_{0}(0))
					\big(\mathds{1}_{\{X_{\ell} \leq X_{i(s)}\}} - \mathds{1}_{\{X_{\ell} \leq X_{i(s)+1}\}}\big)\bigg| 
\leq \frac{2T}{\sqrt{n}} 
= o(1).
\]
Hence, it suffices to show convergence of the finite-dimensional distributions and uniform stochastic equicontinuity of the 
processes 
\[
\fA_{n}(\boldcdot) \defeq 
\frac{1}{\sqrt{n}}\sum_{\ell=1}^{n}(Y_{\ell}^{n}-\Phi_{0}(0))
						\big(1-2\mathds{1}_{\{X_{\ell} \leq \boldcdot\}}\big),
\]
implying convergence of the finite dimensional distributions and tightness of $A_{n}$. 

\medskip
\noindent
\textit{Convergence of finite-dimensional distributions. }
For any $k \in \bN$, let $\{s_{1},\dots,s_{k}\} \subset [-T,T]$ denote a subset of 
cardinality $k$, define
\[
f_{n} \colon [-T,T] \times \{0,1\} \times [-T,T] \to \bR, 
\quad f_{n}(x,y,s) \defeq (y-\Phi_{0}(0))(1-2\mathds{1}_{\{x \leq s\}})
\]
and note that 
\begin{align*}
\begin{pmatrix}
\fA_{n}(s_{1}) \\
\vdots \\
\fA_{n}(s_{k})
\end{pmatrix} 
= 
\sum_{i=1}^{n}\frac{1}{\sqrt{n}}
\begin{pmatrix}
f_{n}(X_{i},Y_{i}^{n},s_{1}) \\
\vdots \\
f_{n}(X_{i},Y_{i}^{n},s_{k})
\end{pmatrix}
,
\end{align*}
as well as that $\bE[f_{n}(X_{i},Y_{i}^{n},s)] = o(n^{-1/2})$. 
As a shorthand notation, let us introduce 
\begin{align*}
V_{i}^{n}
\defeq 
\frac{1}{\sqrt{n}}
\begin{pmatrix}
f_{n}(X_{i},Y_{i}^{n},s_{1}) \\
\vdots \\
f_{n}(X_{i},Y_{i}^{n},s_{k})
\end{pmatrix}
\end{align*}
for $i=1,\dots,n$. Note that $\|V_{i}^{n}\|_2^{2} \leq k/n$ by definition of $f_{n}$ 
and $b_{n}$ and so for every $\varepsilon > 0$, 
\[
\sum_{i=1}^{n}\bE[\|V_{i}^{n}\|_2^{2}\mathds{1}_{\{\|V_{i}^{n}\|_2 > \varepsilon\}}] 
\leq \frac{k}{n}\sum_{i=1}^{n}\bE[\mathds{1}_{\{\|V_{i}^{n}\|_2^{2} > \varepsilon^{2}\}}] 
\leq \frac{k}{n}\sum_{i=1}^{n}\bE[\mathds{1}_{\{k > n\varepsilon^{2}\}}] 
= k\mathds{1}_{\{k > n\varepsilon^{2}\}} 
\longrightarrow 0
\]
as $n \longrightarrow \infty$. Next, for
$j,\ell \in \{1,\dots,k\}$, we evaluate 
\begin{align*}
\bigg(\sum_{i=1}^{n}\Cov(&V_{i}^{n})\bigg)_{j\ell} 
= \bE\big[(Y^{n}-\Phi_{0}(0))^{2}(1-2\mathds{1}_{\{X \leq s_{j}\}}\mathds{1}_{\{X \leq s_{\ell}\}}\big] + o(1) \\
&= \Phi_{0}(0)(1-\Phi_{0}(0))\bE[1-2\mathds{1}_{\{X \leq s_{j}\}}-2\mathds{1}_{\{X \leq s_{\ell}\}} 
										+4\mathds{1}_{\{X \leq \min\{s_{j},s_{\ell}\}\}}] + o(1) \\
&= \Phi_{0}(0)(1-\Phi_{0}(0))(1 + 4F_{X}(\min\{s_{j},s_{\ell}\}) - 2F_{X}(s_{j}) - 2F_{X}(s_{\ell})) + o(1) \\
&= \Phi_{0}(0)(1-\Phi_{0}(0))(1 + 2F_{X}(\min\{s_{j},s_{\ell}\}) - 2F_{X}(\max\{s_{j},s_{\ell}\})) + o(1) \\
&= \Phi_{0}(0)(1-\Phi_{0}(0))(1 - 2|F_{X}(s_{\ell})-F_{X}(s_{j})|) + o(1) \\
&= \Cov(A(s_{j}),A(s_{\ell})) + o(1).
\end{align*}
The Lindeberg-Feller central limit theorem then implies
\[
(\fA_{n}(s_{1}),\dots,\fA_{n}(s_{k})) 
\longrightarrow_{\cL} (\fA(s_{1}),\dots,\fA(s_{k})) \quad \text{as } n \longrightarrow \infty.
\]

\medskip
\noindent
\textit{Uniform stochastic equicontinuity of $(\fA_{n})_{n\in\bN}$. }
Note that convergence of finite dimensional distributions implies tightness of $(\fA_{n}(s))_{n\in\bN}$ for every 
$s \in [-T,T]$. 
To show that $(\fA_{n})_{n\in\bN}$ is asymptotically uniformly equicontinuous in probability, let $\varepsilon > 0$ and $\eta > 0$ 
and note that by Markov's inequality, 
\begin{align}\label{eq:markov}
\bP\bigg(\sup_{|s-t|<\eta}|\fA_{n}(s)-\fA_{n}(t)| > \Delta\bigg) 
\leq \frac{1}{\Delta}\bE\bigg[\sup_{|s-t|<\eta}|\fA_{n}(s)-\fA_{n}(t)|\bigg].
\end{align}
Defining 
\[
h_{n,s,t} \colon [-T,T] \times \{0,1\} \to \bR, 
\qquad 
h_{n,s,t}(x,y) \defeq 2(y-\Phi_{n}(x_{0}))(\mathds{1}_{\{x \leq s\}}-\mathds{1}_{\{x \leq t\}})
\]
for $s,t \in [-T,T]$, setting $\cH_{n,\eta} \defeq \{h_{n,s,t} \mid s,t \in [-T,T], |s-t| < \eta\}$ and choosing 
$\varepsilon_{n} > 0$ and $M_{n} > 0$, such that $\bE[h^{2}] < \varepsilon_{n}^{2}$ and $\|h\|_{\infty} \leq M_{n}$ for every 
$h \in \cH_{n,\eta}$, we have by Theorem~2.14.17' of \cite{VaartWellner2023}
\begin{align*}
\bE\bigg[\sup_{|s-t|<\eta}|\fA_{n}(s)-\fA_{n}(t)|\bigg] 
&= \bE\bigg[\sup_{h_{n} \in \cH_{n,\eta}}\bigg|\frac{1}{\sqrt{n}}\sum_{i=1}^{n}h_{n}(X_{i},Y_{i}^{n}) 
																	- \bE[h_{n}(X_{i},Y_{i}^{n})] \bigg|\bigg] \\
&\leq J_{[]}\big(\varepsilon_{n},\cH_{n,\eta},L^{2}(P_{\Phi_{n}})\big)
		\bigg(1 + \frac{J_{[]}\big(\varepsilon_{n},\cH_{n,\eta},L^{2}(P_{\Phi_{n}})\big)}
																	{\varepsilon_{n}^{2}n^{1/2}}M_{n}\bigg).
\end{align*}
For arbitrary $h \in \cH_{n,\eta}$, there exists $s,t \in [-T,T]$, satisfying $|s-t| < \eta$, such that 
\begin{align*}
\bE[h(X,Y^{n})^{2}] 
&= \bE[h_{n,s,t}(X,Y^{n})^{2}] 
\leq \bE[(\mathds{1}_{\{X \leq s\}}-\mathds{1}_{\{X \leq t\}})^{2}] \\
&= \bE[\mathds{1}_{\{\min\{s,t\} < X \leq \max\{s,t\}\}}] \\
&= F_{X}(\max\{s,t\}) - F_{X}(\min\{s,t\}) \\
&\leq \|p_{X}\|_{\infty}(\max\{s,t\} - \min\{s,t\}) \\
&= \|p_{X}\|_{\infty}|s - t| < 
\|p_{X}\|_{\infty}\eta 
\end{align*}
and $\|h\|_{\infty} = \|h_{n,s,t}\|_{\infty} \leq 1$. Thus, by choosing $\varepsilon_{n} = \sqrt{\eta}$ and $M_{n} = 1$, we 
have 
\[
\bE\bigg[\sup_{|s-t|<\eta}|\fA_{n}(s)-\fA_{n}(t)|\bigg] 
\leq J_{[]}\big(\eta^{1/2},\cH_{n,\eta},L^{2}(P_{\Phi_{n}})\big)
		\bigg(1 + \frac{J_{[]}\big(\sqrt{\eta},\cH_{n,\eta},L^{2}(P_{\Phi_{n}})\big)}{\eta \sqrt{n}}\bigg). 
\]
By similar arguments as in Lemma~\ref{lem:bracketing} (iv), it follows that for some constant $K > 0$, which may change from 
line to line, 
\[
N_{[]}(\nu,\cH_{n,\eta},L^{2}(P_{\Phi_{n}})) 
\leq \frac{K}{\nu^{4}}.
\]
Thus, by utilizing that $\frac{d}{d x}x(\log(K/x^{4}) + 4) = \log(K/x^{4})$, 
\[
J_{[]}\big(\sqrt{\eta},\cH_{n,\eta},L^{2}(P_{\Phi_{n}})\big) 
\leq K\int_{0}^{\sqrt{\eta}}\log\Big(\frac{K}{\nu^{4}}\Big)d\nu 
\leq K\sqrt{\eta}\log\Big(\frac{1}{\eta^{2}}\Big).
\]
Therefore, the assertion follows from \eqref{eq:markov} combined with 
\[
\limsup_{n \to \infty}\bE\bigg[\sup_{|s-t|<\eta}|\fA_{n}(s)-\fA_{n}(t)|\bigg] 
\leq K\sqrt{\eta}\log\Big(\frac{1}{\eta^{2}}\Big).
\]
\end{proof}

\section{Proofs of Section~\ref{sec:inverse process}}\label{proofs:inverse process}
This section contains the proofs of Lemma~\ref{lem:tail bounds - 1} and Corollary~\ref{cor:tail bounds - 1}.

\subsection{Proof of Lemma~\ref{lem:tail bounds - 1}} \label{proof:tail bound - 1}
The result follows immediately for the case $n\delta_{n}^{2} \longrightarrow 0$, as in this case, for $n$ large enough, 
the right-hand side is greater than $1$. For $n\delta_{n}^{2} \longrightarrow c \in (0,\infty]$, the proof follows
the route of Theorem~1 in \cite{Durot2007Lp} but incorporates the explicit dependence on $\Phi_n'$ and thus 
reveals the convergence rate of the inverse process in the weak-feature-impact scenario. 
The first and last inequalities in both cases are obviously true. So for the proof of (i), let us consider 
$x \in [(n\delta_{n}^{2})^{-1/3},1]$ and let $K > 0$ denote a constant which may changes from line to line. Let us define 
$M_{n} \colon [0,1] \to \bR$, $M_{n}(t) \defeq \Upsilon_{n}(t) - \Lambda_{n}(t)$, set 
$\varepsilon_{(j)}^{n} \defeq Y_{j}^{n} - \Phi_{n}(X_{(j)})$ and note that by definition of $\Upsilon_{n}$, we have 
\[
\Upsilon_{n}(u) 
= \Upsilon_{n}\bigg(\frac{\lfloor nu \rfloor}{n}\bigg) 
	+ \bigg(u - \frac{\lfloor nu \rfloor}{n}\bigg)
		\bigg(\Upsilon_{n}\bigg(\frac{\lfloor nu \rfloor + 1}{n}\bigg) 
				- \Upsilon_{n}\bigg(\frac{\lfloor nu \rfloor}{n}\bigg)\bigg), 
\]
where 
\[
\Upsilon_{n}(i/n) 
= \frac{1}{n}\sum_{j=1}^{i}Y_{(j)}^{n}
= \frac{1}{n}\sum_{j=1}^{i}\varepsilon_{(j)}^{n} + \int_{0}^{i/n}\Phi_{n}(F_{n}^{-1}(u))du, \quad i=1,\dots,n. 
\]
Now fix $a \in \bR$ and note that by definition of $\tilde{U}_{n}$, 
\[
\big\{|\tilde{U}_{n}(a) - \lambda_{n}^{-1}(a)| \geq x\big\} 
			\subset \bigg\{\inf_{|u - \lambda_{n}^{-1}(a)| \geq x}\Upsilon_{n}(u) - au 
								\leq \Upsilon_{n}(\lambda_{n}^{-1}(a)) - a\lambda_{n}^{-1}(a)\bigg\}.
\]
Consequently, 
\begin{align*}
&\bP\big(|\tilde{U}_{n}(a) - \lambda_{n}^{-1}(a)| > x\big) \\
&\leq \bP\bigg(\inf_{|u - \lambda_{n}^{-1}(a)| \geq x}\Upsilon_{n}(u) - au 
									\leq \Upsilon_{n}(\lambda_{n}^{-1}(a)) - a\lambda_{n}^{-1}(a)\bigg) \\
&= \bP\bigg(\inf_{|u - \lambda_{n}^{-1}(a)| \geq x}\Upsilon_{n}(u) - \Upsilon_{n}(\lambda_{n}^{-1}(a)) 
													+ a\lambda_{n}^{-1}(a) - au \leq 0\bigg) \\
&= \bP\bigg(\sup_{|u - \lambda_{n}^{-1}(a)| \geq x}\Upsilon_{n}(\lambda_{n}^{-1}(a)) - \Upsilon_{n}(u)
													+ au - a\lambda_{n}^{-1}(a) \geq 0\bigg) \\
&= \bP\bigg(\sup_{|u - \lambda_{n}^{-1}(a)| \geq x}M_{n}(\lambda_{n}^{-1}(a)) - M_{n}(u) 
							+ \Lambda_{n}(\lambda_{n}^{-1}(a)) - \Lambda_{n}(u) + au - a\lambda_{n}^{-1}(a) \geq 0\bigg).
\end{align*}
From a Taylor expansion of $\Lambda_{n}(u)$ around $\lambda_{n}^{-1}(a)$ with Lagrange remainder, we obtain 
\[
\Lambda_{n}(u) 
= \Lambda_{n}(\lambda_{n}^{-1}(a)) 
		+ \lambda_{n}(\lambda_{n}^{-1}(a))(u-\lambda_{n}^{-1}(a)) 
		+ \frac{1}{2}\lambda_{n}'(\xi_{n})(u-\lambda_{n}^{-1}(a))^{2}
\]
for some $\xi_{n}$ between $u$ and $\lambda_{n}^{-1}(a)$ and by assumption, we know that at least for $n$ large enough, 
\[
\lambda_{n}'(t) = \delta_{n}\Phi_{0}'(\delta_{n}F_{X}^{-1}(t))(F_{X}^{-1})'(t) > \delta_{n}K. 
\]
Thus, 
\begin{align*}
\Lambda_{n}(\lambda_{n}^{-1}(a)) - \Lambda_{n}(u) 
&= -\lambda_{n}(\lambda_{n}^{-1}(a))(u-\lambda_{n}^{-1}(a)) 
								- \frac{1}{2}\lambda_{n}'(\xi_{n})(u-\lambda_{n}^{-1}(a))^{2} \\
&\leq -\lambda_{n}(\lambda_{n}^{-1}(a))(u-\lambda_{n}^{-1}(a)) 
								- K\delta_{n}(u-\lambda_{n}^{-1}(a))^{2}.
\end{align*}
Now note that if $\lambda_{n}(\lambda_{n}^{-1}(a)) \neq a$, then either 
\[
a < \lambda_{n}(\lambda_{n}^{-1}(a)) 
\quad \text{and} \quad 
\lambda_{n}^{-1}(a) = F_{X}(-T) = 0, 
\]
or
\[
a > \lambda_{n}(\lambda_{n}^{-1}(a)) 
\quad \text{and} \quad 
\lambda_{n}^{-1}(a) = F_{X}(T) = 1. 
\]
Thus, $(a-\lambda_{n}(\lambda_{n}^{-1}(a)))(u-\lambda_{n}^{-1}(a)) \leq 0$ for every $a$ and we obtain 
\begin{align*}
&\Lambda_{n}(\lambda_{n}^{-1}(a)) - \Lambda_{n}(u) + au - a\lambda_{n}^{-1}(a) \\
&\qquad\qquad\leq -\lambda_{n}(\lambda_{n}^{-1}(a))(u-\lambda_{n}^{-1}(a)) 
					- K\delta_{n}(u-\lambda_{n}^{-1}(a))^{2} + a(u - \lambda_{n}^{-1}(a)) \\
&\qquad\qquad\leq - K\delta_{n}(u-\lambda_{n}^{-1}(a))^{2}.
\end{align*}
Consequently, by a slicing argument, the union bound and Markov's inequality, 
\begin{align*}
\bP\big(|\tilde{U}_{n}(a) - \lambda_{n}^{-1}(a)| \geq x\big) 
&\leq \bP\bigg(\sup_{|u - \lambda_{n}^{-1}(a)| \geq x}M_{n}(\lambda_{n}^{-1}(a)) - M_{n}(u) 
											- K\delta_{n}(u-\lambda_{n}^{-1}(a))^{2} \geq 0\bigg) \\
&\leq \sum_{k \geq 0}\bP\bigg(\sup_{|u - \lambda_{n}^{-1}(a)| \in [x2^{k},x2^{k+1}]}
										M_{n}(\lambda_{n}^{-1}(a)) - M_{n}(u) \geq K\delta_{n}(x2^{k})^{2}\bigg) \\
&\leq K(\delta_{n}x^{2})^{-q}\sum_{k \geq 0}2^{-2kq}
									\bE\bigg[\sup_{|u - \lambda_{n}^{-1}(a)| \in [x2^{k},x2^{k+1}]}
															|M_{n}(\lambda_{n}^{-1}(a)) - M_{n}(u)|^{q}\bigg].
\end{align*}
Now we want to determine an upper bound for the expectation in the previous inequality. For $u \in [0,1]$ and without loss 
of generality for $t \in [0,1]$, we have 
\begin{align*}
\Upsilon_{n}(t+u) - \Upsilon_{n}(u) 
&= \Upsilon_{n}\bigg(\frac{\lfloor n(t+u) \rfloor}{n}\bigg) - \Upsilon_{n}\bigg(\frac{\lfloor nu \rfloor}{n}\bigg) \\
	&\qquad\quad+ \bigg(t+u - \frac{\lfloor n(t+u) \rfloor}{n}\bigg)
		\bigg(\Upsilon_{n}\bigg(\frac{\lfloor n(t+u) \rfloor + 1}{n}\bigg) 
				- \Upsilon_{n}\bigg(\frac{\lfloor n(t+u) \rfloor}{n}\bigg)\bigg) \\
	&\qquad\quad- \bigg(u - \frac{\lfloor nu \rfloor}{n}\bigg)
		\bigg(\Upsilon_{n}\bigg(\frac{\lfloor nu \rfloor + 1}{n}\bigg) 
				- \Upsilon_{n}\bigg(\frac{\lfloor nu \rfloor}{n}\bigg)\bigg).
\end{align*}
As an immediate consequence, we see 
\[
\bigg(t+u - \frac{\lfloor n(t+u) \rfloor}{n}\bigg) \leq \frac{1}{n} 
\quad \text{and} \quad 
\bigg(u - \frac{\lfloor nu \rfloor}{n}\bigg) \leq \frac{1}{n}.
\]
By definition of $\Upsilon_{n}$, we find 
\[
\Upsilon_{n}\bigg(\frac{\lfloor n(t+u) \rfloor}{n}\bigg) - \Upsilon_{n}\bigg(\frac{\lfloor nu \rfloor}{n}\bigg) 
= \frac{1}{n}\sum_{j=\lfloor nu \rfloor + 1}^{\lfloor n(t+u) \rfloor} \varepsilon_{(j)}^{n} 
		+ \int_{\lfloor nu \rfloor/n}^{\lfloor n(t+u) \rfloor/n}\Phi_{n}(F_{n}^{-1}(s))ds 
\]
and in particular, 
\[
\Upsilon_{n}\bigg(\frac{\lfloor n(t+u) \rfloor + 1}{n}\bigg) - \Upsilon_{n}\bigg(\frac{\lfloor n(t+u) \rfloor}{n}\bigg) 
= \frac{1}{n}Y_{(\lfloor n(t+u) \rfloor + 1)}^{n}
\]
and 
\[
\Upsilon_{n}\bigg(\frac{\lfloor nu \rfloor + 1}{n}\bigg) - \Upsilon_{n}\bigg(\frac{\lfloor nu \rfloor}{n}\bigg) 
= \frac{1}{n}Y_{(\lfloor nu \rfloor + 1)}^{n}.
\]
Putting all of this together, we have by definition of $\Lambda_{n}$ and the mean value theorem, 
\begin{align*}
|M_{n}(t+u) - M_{n}(u)| 
&= |\Upsilon_{n}(t+u) - \Upsilon_{n}(u) - (\Lambda_{n}(t+u) - \Lambda_{n}(u))| \\
&\leq \Bigg|\frac{1}{n}\sum_{j=\lfloor nu \rfloor + 1}^{\lfloor n(t+u) \rfloor} \varepsilon_{(j)}^{n} 
		+ \int_{\lfloor nu \rfloor/n}^{\lfloor n(t+u) \rfloor/n}\Phi_{n}(F_{n}^{-1}(s))ds
		- \int_{u}^{t+u}\lambda_{n}(s)ds \Bigg| \\
	&\qquad\qquad+ \frac{1}{n^{2}}(Y_{(\lfloor n(t+u) \rfloor + 1)}^{n} + Y_{(\lfloor nu \rfloor + 1)}^{n}) \\
&\leq \frac{1}{n}\Bigg|\sum_{j=\lfloor nu \rfloor + 1}^{\lfloor n(t+u) \rfloor} \varepsilon_{(j)}^{n}\Bigg|
		+ \bigg|\int_{u}^{t+u}\Phi_{n}(F_{n}^{-1}(s)) - \Phi_{n}(F_{X}^{-1}(s))ds\bigg| + \frac{2}{n} + \frac{2}{n^{2}} \\
&\leq \frac{1}{n}\Bigg|\sum_{j=\lfloor nu \rfloor + 1}^{\lfloor n(t+u) \rfloor} \varepsilon_{(j)}^{n}\Bigg|
		+ \delta_{n}t\sup_{s \in [-T,T]}|\Phi_{0}'(\delta_{n}s)|
							\sup_{v \in [0,1]}|F_{n}^{-1}(v) - F_{X}^{-1}(v)| + \frac{4}{n}.
\end{align*}
For $x \geq 1/n$, which follows from $x \geq (n\delta_{n}^{2})^{-1/3}$, we observe from the previous inequality, 
\begin{align*}
\sup_{t \in [0,x]}|M_{n}(t+u) - M_{n}(u)| 
\leq \frac{1}{n}\sup_{t \in [0,x]}
					\Bigg|\sum_{j=\lfloor nu \rfloor + 1}^{\lfloor n(t+u) \rfloor}\varepsilon_{(j)}^{n}\Bigg|
		+ K\delta_{n}x\sup_{v \in [0,1]}|F_{n}^{-1}(v) - F_{X}^{-1}(v)| + \frac{4}{n}.
\end{align*}
From \cite{DKW1956}, 
\[
\bE\bigg[\sup_{v \in [0,1]}|F_{n}^{-1}(v) - F_{X}^{-1}(v)|^{q}\bigg] 
\leq Kn^{-q/2}.
\]
The following bound 
\[
\bE\Bigg[\sup_{t \in [0,x]}\Bigg|\sum_{j=\lfloor nu \rfloor + 1}^{\lfloor n(t+u) \rfloor}\varepsilon_{(j)}^{n}\Bigg|^{q}\Bigg] 
\leq K(nx)^{q/2} 
\]
is virtually the same as on p.~333 in \cite{Durot2008monotone} and can be likewise deduced by Doob's inequality together 
with Theorem~3 of \cite{Rosenthal1973Sub}, noting that the arguments do not involve the level of feature impact 
$\delta_{n}$. Finally, this shows 
\[
\bE\bigg[\sup_{t \in [0,x]}\big|M_{n}(t+u) - M_{n}(u)\big|^{q}\bigg]^{1/q}
\leq K\frac{x^{1/2}}{\sqrt{n}} + K\frac{x\delta_{n}}{\sqrt{n}} + \frac{4}{n}
\]
and for $n$ sufficiently large, we have $1/n \leq \delta_{n}/\sqrt{n} \leq 1/\sqrt{n}$ and by utilizing 
$x \geq 1/n$, 
\[
\bE\bigg[\sup_{t \in [0,x]}\big|M_{n}(t+u) - M_{n}(u)\big|^{q}\bigg]
\leq K\Big(\frac{x}{n}\Big)^{q/2}.
\]
By using the same arguments, this also holds for $t \in [-1,0]$ and so we have 
\begin{align*}
&\bE\bigg[\sup_{|u-\lambda_{n}^{-1}(a)| \in [x2^{k},x2^{k+1}]}
								\big|M_{n}(\lambda_{n}^{-1}(a)) - M_{n}(u)\big|^{q}\bigg] \\ 
&\qquad\qquad\leq \bE\bigg[\sup_{|\lambda_{n}^{-1}(a)-u| \leq x2^{k+1}}
								\big|M_{n}(\lambda_{n}^{-1}(a)-u+u) - M_{n}(u)\big|^{q}\bigg] 
\leq K\bigg(\frac{x2^{k+1}}{n}\bigg)^{q/2}.
\end{align*}
Combining this with the previous results, we obtain 
\begin{align*}
\bP\big(|\tilde{U}_{n}(a) - \lambda_{n}^{-1}(a)| \geq x\big) 
&\leq K(\delta_{n}x^{2})^{-q}\sum_{k \geq 0}2^{-2kq}\bigg(\frac{x2^{k+1}}{n}\bigg)^{q/2} 
\leq K(n\delta_{n}^{2}x^{3})^{-q/2}\sum_{k \geq 0}2^{-3kq/2} \\
&\leq K(n\delta_{n}^{2}x^{3})^{-q/2}
\end{align*}
and statement (i) follows. \par
For the proof of statement (ii), let $x \in [(n\delta_{n}^{2})^{-1/3},2T]$ and let again $K > 0$ denote a constant that may 
changes from line to line. Note further that a Taylor expansion with Lagrange remainder of $F_{X}^{-1}$ around 
$\tilde{U}_{n}(a)$ yields 
\[
F_{X}^{-1}(\lambda_{n}^{-1}(a)) 
= F_{X}^{-1}(\tilde{U}_{n}(a)) + (F_{X}^{-1})'(\xi_{n})(\lambda_{n}^{-1}(a)-\tilde{U}_{n}(a)) 
\]
for some $\xi_{n}$ between $\lambda_{n}^{-1}(a)$ and $\tilde{U}_{n}(a)$. Consequently, 
\begin{align*}
|F_{n}^{-1}(\tilde{U}_{n}(a)) - F_{X}^{-1}(\lambda_{n}^{-1}(a))| 
&\leq |F_{n}^{-1}(\tilde{U}_{n}(a)) - F_{X}^{-1}(\tilde{U}_{n}(a))|
					+ |(F_{X}^{-1})'(\xi_{n})(\lambda_{n}^{-1}(a)-\tilde{U}_{n}(a))| \\
&\leq \sup_{u \in [0,1]}|F_{n}^{-1}(u) - F_{X}^{-1}(u)|
					+ \frac{1}{p_{X}(F_{X}^{-1}(\xi_{n}))}|\lambda_{n}^{-1}(a)-\tilde{U}_{n}(a))| \\
&\leq \sup_{u \in [0,1]}|F_{n}^{-1}(u) - F_{X}^{-1}(u)| 
					+ \Big(\inf_{t \in [-T,T]}p_{X}(t)\Big)^{-1}|\lambda_{n}^{-1}(a)-\tilde{U}_{n}(a)| 
\end{align*}
and we obtain from statement (i), Markov's inequality and \cite{DKW1956}
\begin{align*}
&\bP\big(|F_{n}^{-1}(\tilde{U}_{n}(a)) - \Phi_{n}^{-1}(a)| \geq x\big) \\
&\qquad= \bP\big(|F_{n}^{-1}(\tilde{U}_{n}(a)) - F_{X}^{-1}(\lambda_{n}^{-1}(a))| \geq x\big) \\
&\qquad\leq \bP\bigg(\sup_{u \in [0,1]}\big|F_{n}^{-1}(u) - F_{X}^{-1}(u)\big| \geq x/2\bigg) 
		+ \bP\big(\big|F_{X}(\Phi_{n}^{-1}(a))-\tilde{U}_{n}(a)\big| \geq Kx\big) \\
&\qquad\leq Kx^{-3q/2}\bE\bigg[\sup_{u \in [0,1]}\big|F_{n}^{-1}(u) - F_{X}^{-1}(u)\big|^{3q/2}\bigg] 
		+ K(n\delta_{n}^{2}x^{3})^{-q/2} \\
&\qquad\leq Kx^{-3q/2}\big(n^{-3q/4} + (n\delta_{n}^{2})^{-q/2}\big) \\ 
&\qquad\leq K(n\delta_{n}^{2}x^{3})^{-q/2}, 
\end{align*}
which proves statement (ii). \hfill \qed

\, \\
The next result is a variation of Lemma~\ref{lem:tail bounds - 1} for the case that 
$a \in [0,1] \setminus \Phi_{n}([-T,T])$. The proof follows exactly the lines of Lemma~2 in \cite{Durot2008monotone} and 
is therefore omitted.

\begin{lemma} \label{lem:tail bounds - 2}
There exist constants $C = C(\Phi_{0},F_{X}) > 0$ and $N_{0} = N_{0}(\Phi_{0}) \in \bN$, such 
that for every $n \geq N_{0}$, $a \in [0,1] \setminus \Phi_{n}([-T,T])$ and $x > 0$, 
\begin{itemize}
\item[(i)] $\bP^{|X}\big(\big|\tilde{U}_{n}(a) - \lambda_{n}^{-1}(a)\big| \geq x\big) 
					\leq K(nx)^{-1}(\Phi_{n} \circ \Phi_{n}^{-1}(a) - a)^{-2}$,

\item[(ii)] $\bP\big(\big|F_{n}^{-1} \circ \tilde{U}_{n}(a) - \Phi_{n}^{-1}(a)\big| \geq x\big) 
					\leq K(nx)^{-1}(\Phi_{n} \circ \Phi_{n}^{-1}(a) - a)^{-2}$.
\end{itemize}
\end{lemma}

\subsection{Proof of Corollary~\ref{cor:tail bounds - 1}}\label{proof:cor tail bounds - 1}
Note first that by monotonicity of $\tilde{U}_{n}$ and $\lambda_{n}^{-1}$, 
\begin{align*}
\big|\tilde{U}_{n}(a+Z_{1,n})\hspace{-0.015em} - \hspace{-0.015em}\lambda_{n}^{-1}(a+Z_{2,n})\big| 
&= \max\big\{\tilde{U}_{n}(a+Z_{1,n}) - \lambda_{n}^{-1}(a+Z_{2,n}),
					\lambda_{n}^{-1}(a+Z_{2,n}) - \tilde{U}_{n}(a+Z_{1,n})\big\} \\
&\leq \max\big\{\tilde{U}_{n}(a+c_{n}) - \lambda_{n}^{-1}(a-c_{n}),
					\lambda_{n}^{-1}(a+c_{n}) - \tilde{U}_{n}(a-c_{n})\big\}.
\end{align*}
Thus, for any $x > 0$, 
\begin{align*}
&\bP\big(\big|\tilde{U}_{n}(a+Z_{1,n}) - \lambda_{n}^{-1}(a+Z_{2,n})\big| > x\big) \\
&\qquad\leq \bP\big(\big|\tilde{U}_{n}(a+c_{n}) - \lambda_{n}^{-1}(a-c_{n})\big| > x/2\big) 
				+ \bP\big(\big|\tilde{U}_{n}(a-c_{n}) - \lambda_{n}^{-1}(a+c_{n})\big| > x/2\big).
\end{align*}
Note further that by a Taylor expansion of $\lambda_{n}^{-1}$ around $a-c_{n}$, 
\[
|\lambda_{n}^{-1}(a+c_{n}) - \lambda_{n}^{-1}(a-c_{n})| 
\leq Kc_{n}\delta_{n}^{-1}
\]
for some $K > 0$, depending only on the bounds on $\Phi_{0}'(0)$ and $p_{X}$. Thus, by a suitable redefinition of $K$, 
\begin{align*}
&\bP\big(\big|\tilde{U}_{n}(a+c_{n}) - \lambda_{n}^{-1}(a-c_{n})\big| > x/2\big) \\
&\qquad\leq \bP\big(\big|\tilde{U}_{n}(a+c_{n}) - \lambda_{n}^{-1}(a-c_{n})\big| > x/4\big) 
				+ \bP\big(\big|\lambda_{n}^{-1}(a+c_{n}) - \lambda_{n}^{-1}(a-c_{n})\big| > x/4\big) \\
&\qquad\leq \mathds{1}_{\{x \in [0,4(n\delta_{n}^{2})^{-1/3})\}} 
			+ K(n\delta_{n}^{2}x^{3})^{-q/2}\mathds{1}_{\{x \in [4(n\delta_{n}^{2})^{-1/3},1]\}} 
			+ \mathds{1}_{\{x \in [0,Kc_{n}\delta_{n}^{-1}]\}}, 
\end{align*}
by Lemma~\ref{lem:tail bounds - 1} (i) and similarly, 
\begin{align*}
&\bP\big(\big|\tilde{U}_{n}(a-c_{n}) - \lambda_{n}^{-1}(a+c_{n})\big| > x/2\big) \\
&\qquad\leq \mathds{1}_{\{x \in [0,4(n\delta_{n}^{2})^{-1/3})\}} 
			+ K(n\delta_{n}^{2}x^{3})^{-q/2}\mathds{1}_{\{x \in [4(n\delta_{n}^{2})^{-1/3},1]\}} 
			+ \mathds{1}_{\{x \in [0,Kc_{n}\delta_{n}^{-1}]\}}. 
\end{align*}
Integrating $\bP(|\tilde{U}_{n}(a+Z_{1,n}) - \lambda_{n}^{-1}(a+Z_{2,n})| > x)$ over $x$ now yields, again for a redefined 
$K$, 
\[
\bE\big[\big|\tilde{U}_{n}(a+Z_{1,n}) - \lambda_{n}^{-1}(a+Z_{2,n})\big|^{r}\big] 
\leq K\min\Big\{(n\delta_{n}^{2})^{-r/3} + \Big(\frac{c_{n}}{\delta_{n}}\Big)^{r},1\Big\}. 
\] \hfill \qed

\section{Proofs of the auxiliary results of Section~\ref{proof:4.4i}} \label{sec:auxiliary7}

\begin{lemma}\label{lem:uniform bound}
Under the same assumptions as in Theorem~\ref{thm:l1 rate} (i) and by using the same notations as in Section~\ref{proof:4.4i}, we have 
\[
\bE^{|X}\bigg[\sup_{|u| \leq T_{n}}|R_{n}(a,u) + \tilde{R}_{n}(a,u)|^{q}\mathds{1}_{\Omega_{n}'}\bigg] 
\leq Kn^{1-q/3}\delta_{n}^{-q/6}.
\]
\end{lemma}

\begin{proof}
By definition of $\tilde{R}_{n}$, by definition of $\Omega_{n}'$, by the Minkowski's inequality and by the classical bound 
on the expected modulus of continuity of Brownian motion (e.g.~formula (2) in \cite{Fischer2010}), 
\begin{align*}
&\bE^{|X}\bigg[\sup_{|u| \leq T_{n}}|\tilde{R}_{n}(a,u)|^{q}\mathds{1}_{\Omega_{n}'}\bigg]^{1/q} \\
&\leq \frac{n^{2/3}}{\delta_{n}^{1/6}}\bE^{|X}\bigg[\sup_{t \in [0,1]}\bigg|\Upsilon_{n}(t) 
						- \int_{0}^{t}\Phi_{n} \circ F_{n}^{-1}(x)dx 
						- \frac{W_{n}(L^{n}(t))}{\sqrt{n}}\bigg|^{q}\bigg]^{1/q} \\ 
				&\qquad + \bE^{|X}\bigg[\sup_{|u| \leq T_{n}}
				\bigg|\frac{n^{1/6}}{\delta_{n}^{1/6}}W_{n}\Big(L^{n}\Big(L_{n}^{-1}\Big(\Big(\frac{n}{\delta_{n}}\Big)^{-1/3}u 
																			+ L_{n}(\lambda_{n}^{-1}(a))\Big)\Big)\Big) \\
	&\qquad\qquad\qquad\qquad\qquad - \frac{n^{1/6}}{\delta_{n}^{1/6}}W_{n}\big(L_{n}(\lambda_{n}^{-1}(a))\big) 
				 - \big(1-\psi_{n}(\lambda_{n}^{-1}(a))\big)^{1/2}W_{\lambda_{n}^{-1}(a)}^{n}(u) \\
	&\qquad\qquad\qquad\qquad\qquad - \big(1 - (1-\psi_{n}(\lambda_{n}^{-1}(a)))^{1/2}\big)
													W_{\lambda_{n}^{-1}(a)}^{n}(u)\bigg|^{q}\mathds{1}_{\Omega_{n}'}\bigg]^{1/q} \\
&\leq A\frac{n^{2/3}}{\delta_{n}^{1/6}}n^{(1-q)/q} 
		+ \bE^{|X}\bigg[\sup_{|u| \leq T_{n}}
				\bigg|\frac{n^{1/6}}{\delta_{n}^{1/6}}W_{n}\Big(L^{n}\Big(L_{n}^{-1}\Big(\Big(\frac{n}{\delta_{n}}\Big)^{-1/3}u 
																			+ L_{n}(\lambda_{n}^{-1}(a))\Big)\Big)\Big) \\ 
	&\qquad\qquad\qquad\qquad - \frac{n^{1/6}}{\delta_{n}^{1/6}}W_{n}\Big(L^{n}(\lambda_{n}^{-1}(a)) 
									+ \Big(\frac{n}{\delta_{n}}\Big)^{-1/3}u\big(1-\psi_{n}(\lambda_{n}^{-1}(a))\big)\Big)\bigg|^{q}\mathds{1}_{\Omega_{n}'}\bigg]^{1/q} \\ 
	&\qquad\qquad\qquad + \bE^{|X}\bigg[\sup_{|u| \leq T_{n}}\bigg|\psi_{n}(\lambda_{n}^{-1}(a))
													W_{\lambda_{n}^{-1}(a)}^{n}(u)\bigg|^{q}\mathds{1}_{\Omega_{n}'}\bigg]^{1/q} \\
&\leq A\frac{n^{1/q-1/3}}{\delta_{n}^{1/6}} 
	+ \frac{n^{1/6}}{\delta_{n}^{1/6}}\bE^{|X}\bigg[\sup_{|u-v| \leq (n/\delta_{n})^{-1/3}T_{n}(\log(n)/\sqrt{n})\delta_{n}}
											\big|W_{n}(v) - W_{n}(u)\big|^{q}\mathds{1}_{\Omega_{n}'}\bigg]^{1/q} \\ 
	&\qquad\qquad\qquad + \bE^{|X}\bigg[\sup_{u \in [0,1]}\big|W_{\lambda_{n}^{-1}(a)}^{n}(u)\big|^{q}\mathds{1}_{\Omega_{n}'}\bigg]^{1/q}
											\frac{K\log(n)}{n^{1/2}}\delta_{n} \\
&\leq K\frac{n^{1/q-1/3}}{\delta_{n}^{1/6}} + T_{n}^{1/2}\frac{\log(n)}{n^{1/2}}\delta_{n} 
			+ \frac{K\log(n)}{n^{1/2}}\delta_{n} \\ 
&\leq Kn^{1/q-1/3}\delta_{n}^{-1/6}
\end{align*}
and by definition of $a_{n}^{B}$ in \eqref{eq:anB} and Minkowski's inequality, 
\begin{align*}
&\bE^{|X}\bigg[\sup_{|u| \leq T_{n}}|R_{n}(a,u)|^{q}\mathds{1}_{\Omega_{n}'}\bigg]^{1/q} \\
&\qquad= \frac{n^{2/3}}{\delta_{n}^{1/6}}\bE^{|X}\bigg[\sup_{|u| \leq T_{n}}
						\bigg|\int_{\lambda_{n}^{-1}(a)}^{L_{n}^{-1}((\frac{n}{\delta_{n}})^{-1/3}u + L_{n}(\lambda_{n}^{-1}(a)))}
							\Phi_{n} \circ F_{X}^{-1}(x) - \Phi_{n} \circ F_{n}^{-1}(x) \\
	&\qquad\qquad\qquad\qquad\qquad\qquad\qquad- \Phi_{n}'(x)\frac{B_{n}(F_{X} \circ F_{n}^{-1}(x))}
																			{\sqrt{n}p_{X} \circ F_{n}^{-1}(x)} 
							+ \Phi_{n}'(x)\frac{B_{n}(F_{X} \circ F_{n}^{-1}(x))}{\sqrt{n}p_{X} \circ F_{n}^{-1}(x)} \\
	&\qquad\qquad\qquad\qquad\qquad\qquad\qquad- \frac{B_{n}(\lambda_{n}^{-1}(a_{n}))}
																{\sqrt{n}(\lambda_{n}^{-1})'(a_{n})}dx\bigg|^{q}\mathds{1}_{\Omega_{n}'}\bigg]^{1/q} \\
&\qquad\leq \frac{n^{1/3}}{\delta_{n}^{1/6}}\sup_{|u| \leq T_{n}}\Big|L_{n}^{-1}\Big(\Big(\frac{n}{\delta_{n}}\Big)^{-1/3}u 
														+ L_{n}(\lambda_{n}^{-1}(a))\Big) - \lambda_{n}^{-1}(a)\Big| \\
	&\qquad\qquad\qquad\cdot 
		\bigg(\bE^{|X}\bigg[K\delta_{n}^{q}\sup_{t \in [0,1]}
					\bigg|F_{X}^{-1}(t) - F_{n}^{-1}(t) 
				- \frac{B_{n}(F_{X} \circ F_{n}^{-1}(t))}{\sqrt{n}p_{X} \circ F_{n}^{-1}(t)}\bigg|^{q}\mathds{1}_{\Omega_{n}'}\bigg]^{1/q} \\
	&\qquad\qquad\qquad\qquad\quad+ \bE^{|X}\bigg[K\delta_{n}^{q}n^{-q/2}\sup_{|v-w| \leq (n/\delta_{n})^{-1/3}T_{n}}
																			\big|B_{n}(w)-B_{n}(v)\big|^{q}\mathds{1}_{\Omega_{n}'}\bigg]^{1/q}\bigg).
\end{align*}
By a Taylor expansion of $L_{n}^{-1}$ around $L_{n}(\lambda_{n}^{-1}(a))$ and the definition of $\Omega_{n}'$, the 
right-hand side of the previous display is bounded by
\[
Kn^{1/3}\delta_{n}^{1/6}\delta_{n}T_{n}\Big(n^{-1}\log(n)^{2} 
				+ n^{-1/2}\Big(\frac{n}{\delta_{n}}\Big)^{-1/6}T_{n}^{1/2}\log(n\delta_{n}^{2})^{1/2}\Big) 
\leq Kn^{1/q-1/3}\delta_{n}^{-1/6}.
\]
\end{proof}

\begin{lemma}\label{lem:auxiliary}
Under the same assumptions as in Theorem~\ref{thm:l1 rate} (i) and by using the same notations as in Section~\ref{proof:4.4i}, we have for any $\varepsilon > 0$, 
\[
\bP^{|X}\bigg((n\delta_{n}^{2})^{1/6}\int_{\lambda_{n}(0)}^{\lambda_{n}(1)}\bigg|\frac{|\hat{V}_{n}(a)|-|\tilde{V}_{n}(\lambda_{n}^{-1}(a))|}{L_{n}'(\lambda_{n}^{-1}(a))}\bigg|\frac{1}{p_{X}(\Phi_{n}^{-1}(a))}da > \varepsilon, \Omega_{n}'\bigg) 
= o_{\bP}(1).
\]
\end{lemma}

\begin{proof}
By a Taylor expansion, there exists $K>0$, such that for all $|u| \leq S_{n}$, 
\[
|D_{n}(a,u) - d_{n}(\lambda_{n}^{-1}(a))u^{2}| 
\leq Kn^{-1/3}\delta_{n}^{-1/6}\delta_{n}^{2}S_{n}^{3}. 
\]
By similar arguments as in the third step of the proof of Claim~V in the proof of Theorem~\ref{thm:l1 rate} (i), we have by Proposition~1 of \cite{Durot2002sharp} and Theorem~4 of \cite{Durot2002sharp}, 
for every $(x,\alpha)$, satisfying $\alpha \in \big(0,S_{n}\big]$, $x > 0$ and 
$K\delta_{n}^{3}S_{n}^{2} \leq -(\alpha\log(2x\alpha))^{-1}$, that 
\begin{align*}
\bP^{|X}\big(|\hat{V}_{n}(a) &- \tilde{V}_{n}(\lambda_{n}^{-1}(a))| > \alpha, \Omega_{n}'\big) \\
&\leq \bP^{|X}\bigg(2\sup_{|u| \leq S_{n}}|D_{n}(a,u) - d_{n}(\lambda_{n}^{-1}(a))u^{2}| > x\alpha^{3/2}, \Omega_{n}'\bigg) \\
	&\qquad\qquad\qquad\qquad+ KS_{n}x + \bP^{|X}\big(|\hat{V}_{n}(a)| > S_{n}, \Omega_{n}'\big) \\
&\leq \mathds{1}_{\{Kn^{-1/3}\delta_{n}^{-1/6}\delta_{n}^{2}S_{n}^{3} > x\alpha^{3/2}\}} + KS_{n}x 
				+ K\exp(-\kappa^{2}\delta_{n}^{3}S_{n}^{3}/2).
\end{align*}
For any $\varepsilon > 0$, every 
$\alpha \in ((n\delta_{n}^{2})^{-1/6}\delta_{n}^{-1}/\log(n\delta_{n}^{2}),(n\delta_{n}^{2})^{-\varepsilon}\delta_{n}^{-1}]$ 
and 
\[
x_{\alpha,n} \defeq 2K\alpha^{-3/2}n^{-1/3}\delta_{n}^{-1/6}\delta_{n}^{2}S_{n}^{3}, 
\]
we have $\alpha x_{\alpha,n} \longrightarrow 0$ for $n \longrightarrow \infty$ and so $(\alpha,x_{\alpha,n})$ satisfies 
$-(\alpha\log(2x_{\alpha,n}\alpha))^{-1} \geq K\delta_{n}^{3}S_{n}^{2}$
for $n$ large enough. Thus, again for $n$ large enough, 
\[
\bP^{|X}\big(|\hat{V}_{n}(a) - \tilde{V}_{n}(\lambda_{n}^{-1}(a))| > \alpha, \Omega_{n}'\big) 
\leq KS_{n}x_{\alpha,n} 
\]
for every 
$\alpha \in ((n\delta_{n}^{2})^{-1/6}\delta_{n}^{-1}/\log(n\delta_{n}^{2}),(n\delta_{n}^{2})^{-\varepsilon}\delta_{n}^{-1}]$. 
By definition, $|\hat{V}_{n}(a) - \tilde{V}_{n}(\lambda_{n}^{-1}(a))|$ is bounded by $2S_{n}$ and so we obtain, 
\begin{align*}
\bE^{|X}\big[|\hat{V}_{n}(a) - \tilde{V}_{n}(\lambda_{n}^{-1}(a))|\mathds{1}_{\Omega_{n}'}\big] 
&= \int_{0}^{2S_{n}}\bP^{|X}\big(|\hat{V}_{n}(a) - \tilde{V}_{n}(\lambda_{n}^{-1}(a))| > \alpha, \Omega_{n}'\big)d\alpha \\
&\leq K(n\delta_{n}^{2})^{-1/6}\delta_{n}^{-1}/\log(n\delta_{n}^{2}) 
				+ KS_{n}x_{(n\delta_{n}^{2})^{-\varepsilon}\delta_{n}^{-1}} \\
	&\qquad+ K\int_{(n\delta_{n}^{2})^{-1/6}\delta_{n}^{-1}/\log(n\delta_{n}^{2})}^{(n\delta_{n}^{2})^{-\varepsilon}
																			\delta_{n}^{-1}}S_{n}x_{\alpha,n}d\alpha \\
&\leq K(n\delta_{n}^{2})^{-1/6}\delta_{n}^{-1}/\log(n\delta_{n}^{2}). 
\end{align*}
Thus, 
\begin{align*}
(n\delta_{n}^{2})^{1/6}\int_{\lambda_{n}(0)}^{\lambda_{n}(1)}
										\bE^{|X}\big[|\hat{V}_{n}(a) &- \tilde{V}_{n}(\lambda_{n}^{-1}(a))|\mathds{1}_{\Omega_{n}'}\big]da 
\leq K(n\delta_{n}^{2})^{1/6}\delta_{n}(n\delta_{n}^{2})^{-1/6}\delta_{n}^{-1}/\log(n\delta_{n}^{2}), 
\end{align*}
which is bounded by $K\log(n\delta_{n}^{2})^{-1}$ and, as desired for any $\varepsilon > 0$, 
\[
\bP^{|X}\bigg((n\delta_{n}^{2})^{1/6}\int_{\lambda_{n}(0)}^{\lambda_{n}(1)}\bigg|\frac{|\hat{V}_{n}(a)|-|\tilde{V}_{n}(\lambda_{n}^{-1}(a))|}{L_{n}'(\lambda_{n}^{-1}(a))}\bigg|\frac{1}{p_{X}(\Phi_{n}^{-1}(a))}da > \varepsilon, \Omega_{n}'\bigg) 
= o_{\bP}(1).
\]
\end{proof}

\section{Auxiliary results} \label{app:auxiliary}
In this section, we summarize some necessary technical auxiliary results. Throughout, we use the notations introduced in Sections~\ref{sec:setting}--\ref{sec:auxiliary7}.

\begin{lemma} \label{lem:hellinger inequalities}
For $a > 0$ and $b \geq 0$, we have
\[
\frac{\sqrt{a} - \sqrt{b}}{\sqrt{\frac{a+b}{2}} - \sqrt{b}} 
= 2\frac{\sqrt{\frac{a+b}{2}} + \sqrt{b}}{\sqrt{a} + \sqrt{b}}, \quad 
\frac{\sqrt{\frac{a+b}{2}} + \sqrt{b}}{\sqrt{a} + \sqrt{b}} \leq 2, \quad 
\big|\sqrt{a} - \sqrt{b}\big|^{2} \leq 16\Bigg|\sqrt{\frac{a+b}{2}} - \sqrt{b}\Bigg|^{2}.
\]
\end{lemma}

\begin{proof}
The first statement follows from $(\sqrt{a} - \sqrt{b})(\sqrt{a} + \sqrt{b}) = a-b$ and 
\[
2\Bigg(\sqrt{\frac{a+b}{2}} + \sqrt{b}\Bigg)\Bigg(\sqrt{\frac{a+b}{2}} - \sqrt{b}\Bigg) 
= 2\Big(\frac{a+b}{2} - b\Big) 
= a + b - 2b 
= a - b.
\]
For the second statement, note that 
\[
\sqrt{\frac{a+b}{2}} \leq \sqrt{\frac{a}{2}} + \sqrt{\frac{b}{2}} \leq \sqrt{a} + \sqrt{b}.
\]
Thus, 
\[
\frac{\sqrt{\frac{a+b}{2}} + \sqrt{b}}{\sqrt{a} + \sqrt{b}} 
\leq \frac{\sqrt{a} + 2\sqrt{b}}{\sqrt{a} + \sqrt{b}} 
\leq 2\frac{\sqrt{a} + \sqrt{b}}{\sqrt{a} + \sqrt{b}} 
= 2.
\]
By a combination of the first two statements, we obtain 
\[
\big|\sqrt{a} - \sqrt{b}\big| 
= 2\Bigg|\sqrt{\frac{a+b}{2}} - \sqrt{b}\Bigg|\Bigg(\frac{\sqrt{\frac{a+b}{2}} + \sqrt{b}}{\sqrt{a} + \sqrt{b}}\Bigg) 
\leq 4\Bigg|\sqrt{\frac{a+b}{2}} - \sqrt{b}\Bigg|
\]
and the third statement follows from taking squares on both sides.
\end{proof}

\begin{lemma} \label{lem:log inequalities}
For $a,b \in [0,\infty)$, we have 
\begin{itemize}
\item[(i)] $| \log(1/2 + b) - \log(1/2 + a) | \leq 2|b-a|$, 
\item[(ii)] For $a \in (0,\infty)$, we have $\log(a) \leq 2(\sqrt{a} - 1)$ and for $a \geq 1$, we have $\log(a)^{2} \leq 4(\sqrt{a} - 1)^{2}$, 
\item[(iii)] For $0 < a \leq 1$, we have $\log(a)^{2} \leq (1-\frac{1}{a})^{2}$.
\end{itemize}
\end{lemma}

\begin{proof}
\begin{itemize}
\item[(i)] Without loss of generality, we assume $a \leq b$. Then, 
\begin{align*}
&|\log(1/2 + b) - \log(1/2 + a)| 
= \log\Big( \frac{1/2 + b}{1/2 + a} \Big) 
= \log\Big(1 + \Big( \frac{1/2 + b}{1/2 + a} - 1 \Big) \Big) \\
&\qquad\qquad\leq \frac{1/2 + b}{1/2 + a} - 1 
= \frac{1}{1/2 + a}(1/2 + b - (1/2 + a)) 
= \frac{1}{1+2a}2(b-a) \\
&\qquad\qquad\leq 2|b-a|,
\end{align*}
where we used $\log(1+x) \leq x$ for $x \in [0,\infty)$. 

\item[(ii)] Let $g \colon (0,\infty) \to \bR$, $g(a) \defeq \log(a) - 2(\sqrt{a} - 1)$. Then, $g(1) = 0$ and 
\[
g'(a) 
= \frac{1}{a} - 2\frac{1}{2\sqrt{a}} 
= \frac{1 - \sqrt{a}}{a}.
\]
Thus, $g'(a) > 0$ for $a \in (0,1)$, $g'(1) = 0$ and $g'(a) < 0$ for $a > 1$, implying 
\[
\log(a) \leq 2(\sqrt{a} - 1).
\]
The assertion now follows from the fact that $\log(a) \geq 0$ and $2(\sqrt{a} - 1) \geq 0$ for all $a \geq 1$.

\item[(iii)] Let $g \colon (0,1] \to \bR$, $g(a) \defeq \log(a) - 1 + \frac{1}{a}$. Then, $g(1) = 0$ and 
\[
g'(a) 
= \frac{1}{a} - \frac{1}{a^{2}} 
= \frac{a - 1}{a^{2}}
\leq 0.
\]
Thus, $g(a) \geq 0$ for all $a \in (0,1]$, implying 
\[
\log(a) \geq 1 - \frac{1}{a}.
\]
The assertion now follows from the fact that $\log(a) \leq 0$ and $2(\sqrt{a} - 1) \leq 0$ for all $a \leq 1$.
\end{itemize}
\end{proof}

\begin{lemma} \label{lem:jumping point}
Let $G \colon \bR \to \bR$ and assume there exists $T \in \bR$ with $G|_{(-\infty,T)} < 0$ and $G|_{[T,\infty)} \geq 0$. 
Then, for every $s \in \bR$, 
\[
\int_{s}^{\infty}G(x)dx - \int_{-\infty}^{s}G(x)dx 
\leq \int_{T}^{\infty}G(x)dx - \int_{-\infty}^{T}G(x)dx. 
\]
In particular, 
\[
\max_{s \in \bR} \Big\{\int_{s}^{\infty}G(x)dx - \int_{-\infty}^{s}G(x)dx\Big\} 
= \int_{T}^{\infty}G(x)dx - \int_{-\infty}^{T}G(x)dx.
\]
\end{lemma}

\begin{proof}
Consider $s \geq T$. Then, 
\[
\int_{T}^{\infty}G(x)dx - \int_{s}^{\infty}G(x)dx + \int_{-\infty}^{s}G(x)dx - \int_{-\infty}^{T}G(x)dx 
= \int_{T}^{s}G(x)dx + \int_{T}^{s}G(x)dx, 
\]
which is greater than or equal to zero. The case $s < T$ follows similarly.
\end{proof}

The following result is stated as an exercise in \cite{VaartWellner2023} (cf.~Problem~3.2.5). For completeness, we 
decided to give the proof as well.

\begin{lemma}\label{lem:argmin brownian motion transformation}
Let $(Z(s))_{s \in \bR}$ be a standard (two-sided) Brownian motion and let $a,b \in (0,\infty)$ and $c \in \bR$. Then, 
\[
\argmin_{s \in \bR}\{aZ(s) + bs^{2} - cs\}
=_{\cL} \Big(\frac{a}{b}\Big)^{2/3}\argmin_{s \in \bR}\{Z(s) + s^{2}\} + \frac{c}{2b}
\]
\end{lemma}

\begin{proof}
By replacing $s$ with $h(s) \defeq (a/b)^{2/3}s + c/2b$, we obtain 
\begin{align*}
\argmin_{s \in \bR}\{aZ(s) + bs^{2} - cs\} 
&= \argmin_{h(s) \in \bR}\{aZ(h(s)) + bh(s)^{2} - ch(s)\} \\
&= \Big(\frac{a}{b}\Big)^{2/3}\argmin_{s \in \bR}\{aZ(h(s)) + bh(s)^{2} - ch(s)\} + \frac{c}{2b}.
\end{align*}
Using the properties of Brownian motion, we have 
\[
aZ(h(s)) 
=_{\cL} a\Big(\frac{a}{b}\Big)^{1/3}Z(s) + aZ\Big(\frac{c}{2b}\Big)
=_{\cL} \frac{a^{4/3}}{b^{1/3}}Z(s) + aZ\Big(\frac{c}{2b}\Big)
\]
and simple calculations yield
\begin{align*}
bh(s)^{2} - ch(s) 
&= b\Big(\Big(\frac{a}{b}\Big)^{4/3}s^{2} + \frac{c}{b}\Big(\frac{a}{b}\Big)^{2/3}s + \frac{c^{2}}{4b^{2}}\Big) 
																- c\Big(\frac{a}{b}\Big)^{2/3}s - \frac{c^{2}}{2b} \\
&= \frac{a^{4/3}}{b^{1/3}}s^{2} + c\Big(\frac{a}{b}\Big)^{2/3}s + \frac{c^{2}}{4b} 
																- c\Big(\frac{a}{b}\Big)^{2/3}s - \frac{c^{2}}{2b} \\
&= \frac{a^{4/3}}{b^{1/3}}s^{2} + \frac{c^{2}}{4b} - \frac{c^{2}}{2b}.
\end{align*}
By a combination of these results, 
\begin{align*}
\argmin_{s \in \bR}\big\{aZ(h(s)) + bh(s)^{2} - ch(s)\big\} 
&=_{\cL} \argmin_{s \in \bR}\Big\{\frac{a^{4/3}}{b^{1/3}}Z(s) + aZ\Big(\frac{c}{2b}\Big) 
											+ \frac{a^{4/3}}{b^{1/3}}s^{2} + \frac{c^{2}}{4b} - \frac{c^{2}}{2b}\Big\} \\
&=_{\cL} \argmin_{s \in \bR}\Big\{\frac{a^{4/3}}{b^{1/3}}Z(s) + \frac{a^{4/3}}{b^{1/3}}s^{2}\Big\} \\
&=_{\cL} \argmin_{s \in \bR}\big\{Z(s) + s^{2}\big\}
\end{align*}
and the assertion follows.
\end{proof}

\begin{lemma} \label{lem:pointwise convergence}
Let $\Phi_{n}$ be defined as in Section~\ref{subsec: 1.2} and $\Phi_{0}$ continuous. 
Then, for every $\varepsilon > 0$ and for every $x, y \in \bR$, there exists $N \in \bN$, such that 
\[
|\Phi_{n}(y)-\Phi_{n}(x)| < \varepsilon \quad \text{for all } n > N.
\]
\end{lemma}

\begin{proof}
For every $\varepsilon > 0$, we know from continuity of $\Phi_{0}$ in a neighborhood of zero that there exists $\delta > 0$, such that 
\[
|\Phi_{0}(z) - \Phi_{0}(0)| < \frac{\varepsilon}{2} \quad \text{for all } |z| < \delta.
\]
Now for arbitrary $x, y \in \bR$, choose $N \in \bN$ such that both, $|\delta_{n}y| < \delta$ and $|\delta_{n}x| < \delta$ for 
all $n > N$. Then, 
\[
|\Phi_{n}(y)-\Phi_{n}(x)| 
= |\Phi_{0}(\delta_{n}y)-\Phi_{0}(\delta_{n}x)|
\leq |\Phi_{0}(\delta_{n}y) - \Phi_{0}(0)| + |\Phi_{0}(0) - \Phi_{0}(\delta_{n}x)|
< \varepsilon
\]
and the assertion follows.
\end{proof}

We conclude this section with bounds on bracketing numbers for various function classes. Recall $\Phi_n(\boldcdot)=\Phi_0(\delta_n\boldcdot)$ for some strictly increasing, continuous $\Phi_0:[-T,T]\rightarrow[0,1]$, where $\delta_n\searrow 0$ denotes the level of feature impact.

\begin{lemma}\label{lem:bracketing monotone functions}
For $\eta > 0$ and $a, b \in \bR$ with $-T \leq a < b \leq T$, let 
\[
\cG_{n,\eta} \defeq \big\{g \colon [a,b] \to [0,1] \mid g=|f-\Phi_{n}| \text{ for } f \in \cF, \ \|g\|_{[a,b]} \leq \eta\big\}. 
\]
Then, there exist universal constants $L > 0$ and $C > 0$, such that for any $\nu > 0$, 
\[
N_{[]}(\nu,\cG_{n,\eta},L^{2}(P_{X})) 
\leq L^{C(\eta + \delta_{n})/\nu}.
\] 
\end{lemma}

\begin{proof}
Let $g = |f-\Phi_{n}| \in \cG_{n,\eta}$ and $[f_{L},f^{U}]$ denote a $\nu$-bracket for $f \in \cF$ with respect to the 
$L^{2}(P_{X})$-distance, i.e.~$f_{L}(x) \leq f(x) \leq f^{U}(x)$ for all $x \in [a,b]$ and $\|f^{U}-f_{L}\|_{L^2(P_X)} \leq \nu$. 
Let $K > 0$ denote a universal constant which may changes from line to line and note that 
\begin{align*}
f(b)-f(a) 
&= f(b)-\Phi_{n}(b)+\Phi_{n}(b)-\Phi_{n}(a)+\Phi_{n}(a)-f(a) \\
&\leq 2\|g\|_{[a,b]} + K\delta_{n} 
\leq K(\eta + \delta_{n}).
\end{align*}
Then, 
\begin{align*}
(f_{L}(x)-\Phi_{n}(x))_{+} 
&\leq (f(x)-\Phi_{n}(x))_{+} 
\leq (f^{U}(x)-\Phi_{n}(x))_{+}, \\
(f^{U}(x)-\Phi_{n}(x))_{-} 
&\leq (f(x)-\Phi_{n}(x))_{-} 
\leq (f_{L}(x)-\Phi_{n}(x))_{-}
\end{align*}
and 
$
g(x) = (f(x)-\Phi_{n}(x))_{+} + (f(x)-\Phi_{n}(x))_{-} 
$, whence
\begin{align*}
g_{L} \colon [a,b] \to [0,1], \quad &g_{L}(x) \defeq (f_{L}(x)-\Phi_{n}(x))_{+} + (f^{U}(x)-\Phi_{n}(x))_{-}\ \text{ and}\\
g^{U} \colon [a,b] \to [0,1], \quad &g^{U}(x) \defeq (f_{L}(x)-\Phi_{n}(x))_{-} + (f^{U}(x)-\Phi_{n}(x))_{+} 
\end{align*}
satisfy $
g_{L}(x) \leq g(x) \leq g^{U}(x) 
$ for every $x \in [a,b]$. 
Furthermore,
\begin{align*}
&g^{U}(x) - g_{L}(x) \\
&= (f_{L}(x)-\Phi_{n}(x))_{-} + (f^{U}(x)-\Phi_{n}(x))_{+} 
	- \big((f_{L}(x)-\Phi_{n}(x))_{+} + (f^{U}(x)-\Phi_{n}(x))_{-}\big) \\
&= (f^{U}(x)-\Phi_{n}(x))_{+} - (f^{U}(x)-\Phi_{n}(x))_{-} 
	- \big((f_{L}(x)-\Phi_{n}(x))_{+} - (f_{L}(x)-\Phi_{n}(x))_{-}\big) \\
&= f^{U}(x) - \Phi_{n}(x) - (f_{L}(x) - \Phi_{n}(x)) \\
&= f^{U}(x) - f_{L}(x)
\end{align*}
and consequently, for $\cF_{n,\eta} \defeq \{f \in \cF \mid f(b) - f(a) \leq K(\eta + \delta_{n})\}$, 
\[
N_{[]}(\nu,\cG_{n,\eta},L^{2}(P_{X})) 
\leq N_{[]}(\nu,\cF_{n,\eta},L^{2}(P_{X})) 
= N_{[]}\Big(\frac{\nu}{K(\eta + \delta_{n})},\cF,L^{2}(P_{X})\Big).
\]
By Theorem~2.7.9 of \cite{VaartWellner2023}, we obtain the existence of universal constants $L > 0$ and $C > 0$, such that 
the $L^{2}(P_X)$-bracketing number of the class of monotone functions is bounded by 
$L^{C(\eta+\delta_{n})/\nu}$ and the assertion follows.
\end{proof}

\begin{lemma}\label{lem:bracketing}
Let $S > 0$.
\begin{itemize}
\item[(i)] Let $\cF_{n} \defeq \{f_{n,s,t} \mid s,t \in [-S,S]\}$, where 
\[
f_{n,s,t} \colon [-S,S] \times \{0,1\} \to \bR, 
\quad 
f_{n,s,t}(x,y) \defeq (y-\Phi_{n}(x_{0}))\big(\mathds{1}_{\{x \leq x_{0}+a_{n}s\}}-\mathds{1}_{\{x \leq x_{0}+a_{n}t\}}\big)
\]
for $s,t \in [-S,S]$. Then, there exists a universal constant $K > 0$, such that for any $\nu > 0$, 
\[
N_{[]}\big(\nu,\cF_{n},L^{2}(P_{\Phi_n})\big) 
\leq a_{n}^{2}\frac{K}{\nu^{4}}.
\]

\item[(ii)] For $j \in \bN$, let 
$\cF_{n,j}^{\beta} \defeq \{f_{n,s} \mid s \in \bR, 2^{j} < |s|^{\beta+1} \leq 2^{j+1}\}$, where 
\[
f_{n,s} \colon [-S,S] \times \{0,1\} \to \bR, 
\quad 
f_{n,s}(x,y) \defeq (y-\Phi_{n}(x_{0}))\big(\mathds{1}_{\{x \leq x_{0}+a_{n}s\}}-\mathds{1}_{\{x \leq x_{0}\}}\big)
\]
for $s \in \bR$. Then, there exists a universal constant $K > 0$, such that for any $\nu > 0$, 
\[
N_{[]}\big(\nu,\cF_{n,j}^{\beta},L^{2}(P_{\Phi_n})\big) 
\leq a_{n}\frac{K}{\nu^{2}}2^{\frac{j+1}{\beta+1}}.
\]

\item[(iii)] Let $\cH_{n} \defeq \{h_{n,s,t} \mid s,t \in [0,1]\}$, where 
\[
h_{n,s,t} \colon [0,1] \times \{0,1\} \to \bR, 
\quad 
h_{n,s,t}(x,y) \defeq (y-\Phi_{n}(x_{0}))\big(\mathds{1}_{\{x \leq F_{X}^{-1}(s)\}}-\mathds{1}_{\{x \leq F_{X}^{-1}(t)\}}\big)
\]
for $s,t \in [0,1]$. Then, there exists a universal constant $K > 0$, such that for any $\nu > 0$, 
\[
N_{[]}(\nu,\cH_{n},L^{2}(P_{\Phi_n})) 
\leq \frac{K}{\nu^{4}}.
\]

\item[(iv)] Let $\cH_{n} \defeq \{h_{n,s,t} \mid s,t \in [0,1]\}$, where 
\[
h_{n,s,t} \colon [0,1] \times \{0,1\} \to \bR, 
\quad 
h_{n,s,t}(x,y) \defeq (y-\Phi_{n}(x_{0}))\big(\mathds{1}_{\{x \leq s\}}-\mathds{1}_{\{x \leq t\}}\big)
\]
for $s,t \in [0,1]$. Then, there exists a universal constant $K > 0$, such that for any $\nu > 0$, 
\[
N_{[]}(\nu,\cH_{n},L^{2}(P_{\Phi_n})) 
\leq \frac{K}{\nu^{4}}.
\]
\end{itemize}
\end{lemma}

\begin{proof}
Note that for deriving an upper bound on the bracketing number, we can omit the factor $(y-\Phi_{n}(x_{0}))$ in the 
definition of each function, as shown now exemplary for (i). For this, define 
$
\cG_{n} \defeq \{g_{n,s,t} \mid s,t \in [-S,S]\}$, 
where 
\[
g_{n,s,t} \colon [-S,S] \to \bR, 
\quad 
g_{n,s,t}(x) \defeq \mathds{1}_{\{x \leq x_{0}+a_{n}s\}}-\mathds{1}_{\{x \leq x_{0}+a_{n}t\}}.
\]
Considering a function $f \in \cF_{n}$, there exists $g \in \cG_{n}$, such that 
$f(x,y) = (y-\Phi_{n}(x))g(x)$. Now let $[g_{L},g^{U}]$ denote a $\nu$-bracket for $g$ in $L^{2}(P_{\Phi_n})$, i.e.~for 
every $x \in [-S,S]$, we have $g_{L}(x) \leq g(x) \leq g^{U}(x)$ and $\bE[|g^{U}(X)-g_{L}(X)|^{2}]^{1/2} \leq \nu$. Defining 
\begin{align*}
f_{L} \colon [-S,S] \times \{0,1\} \to \bR, \quad 
&f_{L}(x,y) \defeq -(1-y)\Phi_{n}(x_{0})g^{U}(x) + y(1-\Phi_{n}(x_{0}))g_{L}(x) \\
f^{U} \colon [-S,S] \times \{0,1\} \to \bR, \quad 
&f^{U}(x,y) \defeq -(1-y)\Phi_{n}(x_{0})g_{L}(x) + y(1-\Phi_{n}(x_{0}))g^{U}(x), 
\end{align*}
note that 
\begin{align*}
f_{L}(x,0) = -\Phi_{n}(x_{0})g^{U}(x) &\leq -\Phi_{n}(x_{0})g(x) = f(x,0), \\
f_{L}(x,1) = (1-\Phi_{n}(x_{0}))g_{L}(x) &\leq (1-\Phi_{n}(x_{0}))g(x) = f(x,1)
\end{align*}
and similarly, 
\begin{align*}
f_{U}(x,0) = -\Phi_{n}(x_{0})g_{L}(x) &\geq -\Phi_{n}(x_{0})g(x) = f(x,0), \\
f_{U}(x,1) = (1-\Phi_{n}(x_{0}))g^{U}(x) &\geq (1-\Phi_{n}(x_{0}))g(x) = f(x,1).
\end{align*}
Further, we have 
\begin{align*}
f^{U}(x,y) - f_{L}(x,y) 
&= -(1-y)\Phi_{n}(x_{0})g_{L}(x) + y(1-\Phi_{n}(x_{0}))g^{U}(x) \\
	&\qquad\qquad\qquad+ (1-y)\Phi_{n}(x_{0})g^{U}(x) - y(1-\Phi_{n}(x_{0}))g_{L}(x) \\
&= (1-y)\Phi_{n}(x_{0})(g^{U}(x)-g_{L}(x)) + y(1-\Phi_{n}(x_{0}))(g^{U}(x)-g_{L}(x)) \\
&= (g^{U}(x)-g_{L}(x))(\Phi_{n}(x_{0}) - y\Phi_{n}(x_{0}) + y - \Phi_{n}(x_{0})y) \\
&= (g^{U}(x)-g_{L}(x))(\Phi_{n}(x_{0}) - 2y\Phi_{n}(x_{0}) + y).
\end{align*}
Thus, 
\begin{align*}
\bE\big[|f^{U}(X,Y^{n}) - f_{L}(X,Y^{n})|^{2}\big]^{1/2} 
&= \bE\big[|(g^{U}(X)-g_{L}(X))(\Phi_{n}(x_{0}) - 2Y^{n}\Phi_{n}(x_{0}) + Y^{n})|^{2}\big]^{1/2} \\
&\leq \bE\big[|g^{U}(X)-g_{L}(X)|^{2}|\Phi_{n}(x_{0}) - 2Y^{n}\Phi_{n}(x_{0}) + Y^{n}|^{2}\big]^{1/2} \\
&\leq \bE\big[|g^{U}(X)-g_{L}(X)|^{2}\big]^{1/2} \leq \nu 
\end{align*}
and so we have 
\[
N_{[]}\big(\nu,\cF_{n},L^{2}(P_{\Phi_n})\big) 
\leq N_{[]}\big(\nu,\cG_{n},L^{2}(P_{\Phi_n})\big).
\]
Analogously, this also follows for (ii), (iii) and (iv). \par 
To construct the brackets for statement (i), note that by the previous result, it suffices to construct brackets for the function class $\cF_{n}' \defeq \{f_{n,s,t} | s,t \in [-S,S]\}$, where 
$f_{n,s,t}(x) \defeq \mathds{1}_{\{x \leq x_{0}+a_{n}s\}}-\mathds{1}_{\{x \leq x_{0}+a_{n}t\}}$ for 
$x \in [-S,S]$. For this, let $\nu > 0$, set $N(\nu) \defeq \frac{2Sa_{n}}{\nu^{2}}4\|p_{X}\|_{\infty}$ 
and define for $i=1,\dots,\lfloor N(\nu) \rfloor$, 
\[
s_{0}^{n} \defeq -S, 
\qquad 
s_{i}^{n} \defeq s_{i-1}^{n} + \frac{\nu^{2}}{4\|p_{X}\|_{\infty}a_{n}}, 
\qquad 
s_{\lfloor N(\nu) \rfloor + 1}^{n} \defeq S.
\]
Then, $-S = s_{0}^{n} < s_{1}^{n} < \cdots < s_{\lfloor N(\nu) \rfloor + 1}^{n} = S$, 
$s_{i}^{n}-s_{i-1}^{n} \leq \frac{\nu^{2}}{4\|p_{X}\|_{\infty}a_{n}}$ for $1 \leq i \leq \lfloor N(\nu) \rfloor + 1$, and for 
every $s,t \in [-S,S]$, there exists $i,j \in \{1,\dots,\lfloor N(\nu) \rfloor + 1\}$, such that 
$s_{i-1}^{n} \leq s \leq s_{i}^{n}$ and $s_{j-1}^{n} \leq t \leq s_{j}^{n}$. Hence, 
$f_{n,s_{i-1}^{n},s_{j}^{n}}(x) \leq f_{n,s,t}(x) \leq f_{n,s_{i}^{n},s_{j-1}^{n}}(x)$ for every $x \in \bR$ and 
\begin{align*}
&\bigg(\int_{\bR}|f_{n,s_{i}^{n},s_{j-1}^{n}}(x) - f_{n,s_{i-1}^{n},s_{j}^{n}}(x)|^{2}dP_{X}(x)\bigg)^{1/2} \\
&=\bigg(\int_{\bR}|\mathds{1}_{\{x \leq x_{0} + a_{n}s_{i}^{n}\}} 
								- \mathds{1}_{\{x \leq x_{0} + a_{n}s_{i-1}^{n}\}} 
		+ \mathds{1}_{\{x \leq x_{0} + a_{n}s_{j-1}^{n}\}} 
								- \mathds{1}_{\{x \leq x_{0} + a_{n}s_{j}^{n}\}}|^{2}dP_{X}(x)\bigg)^{1/2} \\
&\leq (a_{n}(s_{i}^{n}-s_{i-1}^{n})\|p_{X}\|_{\infty})^{1/2} + (a_{n}(s_{j}^{n}-s_{j-1}^{n})\|p_{X}\|_{\infty})^{1/2} \\
&\leq 2\Big(\|p_{X}\|_{\infty}a_{n}\frac{\nu^{2}}{4\|p_{X}\|_{\infty}a_{n}}\Big)^{1/2}
= \nu, 
\end{align*}
whence $[f_{n,s_{i-1}^{n},s_{j}^{n}},f_{n,s_{i}^{n},s_{j-1}^{n}}]_{i,j=1,\dots,\lfloor N(\nu) \rfloor + 1}$ define $\nu$-brackets for $\cG_{n}$ in $L^{2}(P_{X})$ and 
\[
N_{[]}(\nu,\cG_{n},L^{2}(P_{X})) 
\leq (\lfloor N(\nu) \rfloor + 1)^{2}
\leq \Big(1 + \frac{2Sa_{n}}{\nu^{2}}4\|p_{X}\|_{\infty}\Big)^{2}.
\]
Analogously, the brackets for the classes in (ii), (iii) and (iv) can be obtained.
\end{proof}


\bibliographystyle{imsart-nameyear} 
\bibliography{literature}       

\end{document}